\documentclass[11pt]{preprint}
\usepackage[full]{textcomp}
\usepackage[osf]{newtxtext} 
\usepackage[cal=boondoxo]{mathalfa}
\usepackage{colortbl}

\usepackage{comment}

\usepackage{amssymb}
\usepackage{mathtools}
\usepackage{hyperref}
\usepackage{breakurl}
\usepackage{mhenvs}
\usepackage{mhequ} 
\usepackage{mhsymb}
\usepackage{booktabs}
\usepackage{tikz}
\usepackage{tcolorbox}
\usepackage{mathrsfs}
\usepackage[utf8]{inputenc}
\usepackage{longtable}
\usepackage{wrapfig}
\usepackage{subcaption}
\usepackage[font=small]{caption}
\usepackage{mathrsfs}
\usepackage{epsfig}
\usepackage{microtype} 
\usepackage{comment}
\usepackage{wasysym}
\usepackage{centernot}
\usepackage{enumitem}
\usepackage{bm}
\usepackage{stackrel}
\usepackage{graphicx}
\usepackage{tabularx}
\usepackage{array}
\makeatletter
\newcommand{\globalcolor}[1]{%
	\color{#1}\global\let\default@color\current@color
}
\makeatother

\usetikzlibrary{calc}
\usetikzlibrary{decorations}
\usetikzlibrary{positioning}
\usetikzlibrary{shapes}
\usetikzlibrary{external}

\definecolor{blush}{rgb}{0.87, 0.36, 0.51}
\definecolor{brightcerulean}{rgb}{0.11, 0.67, 0.84}
\definecolor{greenryb}{rgb}{0.4, 0.69, 0.2}

\newif\ifdark
\darkfalse

\ifdark
\definecolor{darkred}{rgb}{0.9,0.2,0.2}
\definecolor{darkblue}{rgb}{0.7,0.3,1}
\definecolor{darkgreen}{rgb}{0.1,0.9,0.1}
\definecolor{franck}{rgb}{0,0.8,1}
\definecolor{pagebackground}{rgb}{.15,.21,.18}
\definecolor{pageforeground}{rgb}{.84,.84,.85}
\pagecolor{pagebackground}
\AtBeginDocument{\globalcolor{pageforeground}}
\definecolor{symbols}{rgb}{0,0.7,1}
\colorlet{connection}{red!80!black}
\colorlet{boxcolor}{blue!50}

\else

\definecolor{darkred}{rgb}{0.7,0.1,0.1}
\definecolor{darkblue}{rgb}{0.4,0.1,0.8}
\definecolor{darkgreen}{rgb}{0.1,0.7,0.1}
\definecolor{franck}{rgb}{0,0,1}
\definecolor{pagebackground}{rgb}{1,1,1}
\definecolor{pageforeground}{rgb}{0,0,0}
\colorlet{symbols}{blue!90!black}
\colorlet{connection}{red!30!black}
\colorlet{boxcolor}{blue!50!black}

\fi

\def\slash{\leavevmode\unskip\kern0.18em/\penalty\exhyphenpenalty\kern0.18em}
\def\dash{\leavevmode\unskip\kern0.18em--\penalty\exhyphenpenalty\kern0.18em}

\DeclareMathAlphabet{\mathbbm}{U}{bbm}{m}{n}

\DeclareFontFamily{U}{BOONDOX-calo}{\skewchar\font=45 }
\DeclareFontShape{U}{BOONDOX-calo}{m}{n}{
	<-> s*[1.05] BOONDOX-r-calo}{}
\DeclareFontShape{U}{BOONDOX-calo}{b}{n}{
	<-> s*[1.05] BOONDOX-b-calo}{}
\DeclareMathAlphabet{\mcb}{U}{BOONDOX-calo}{m}{n}
\SetMathAlphabet{\mcb}{bold}{U}{BOONDOX-calo}{b}{n}

\setlist{noitemsep,topsep=4pt,leftmargin=1.5em}

\DeclareMathAlphabet{\mathbbm}{U}{bbm}{m}{n}

\DeclareMathAlphabet{\mcb}{U}{BOONDOX-calo}{m}{n}
\SetMathAlphabet{\mcb}{bold}{U}{BOONDOX-calo}{b}{n}
\DeclareFontFamily{U}{mathx}{\hyphenchar\font45}
\DeclareFontShape{U}{mathx}{m}{n}{
	<5> <6> <7> <8> <9> <10>
	<10.95> <12> <14.4> <17.28> <20.74> <24.88>
	mathx10
}{}
\DeclareSymbolFont{mathx}{U}{mathx}{m}{n}
\DeclareMathSymbol{\bigtimes}{1}{mathx}{"91}

\providecommand{\figures}{false}
{ \ifthenelse{\equal{\figures}{false}} {#1}{\[ {\rm Figure \ missing !} \]} }{}
\def\id{\mathrm{id}}

\def\CH{\mathcal{H}}
\def\CP{\mathcal{P}}

\def\CC{\mathcal{C}}
\def\CQ{\mathcal{Q}}

\def\CT{\mathcal{T}}

\tikzstyle{tinydots}=[dash pattern=on \pgflinewidth off \pgflinewidth]
\tikzstyle{superdense}=[dash pattern=on 4pt off 1pt]

\newcommand{\beq}{\begin{equation}}
	\newcommand{\eeq}{\end{equation}}

\usepackage{empheq}

\newcommand{\T}{\mathbf{T}}

\def\Labp{\mathfrak{p}}
\def\Labe{\mathfrak{e}}

\def\Labn{\mathfrak{n}}
\def\Labo{\mathfrak{o}}

\def\Labhom{\mathfrak{t}}

\def\Lab{\mathfrak{L}}

\def\Deltap{\Delta^{\!+}}

\def\${|\!|\!|}

\newcounter{theorem}
\newtheorem{defi}[theorem]{Definition}
\newtheorem{ex}[theorem]{Example}

\newenvironment{DIFnomarkup}{}{} 

\theorembodyfont{\rmfamily}
\newtheorem{example}[theorem]{Example}

\newfont{\indic}{bbmss12}

\def\Nabla_#1{\nabla_{\!#1}}

\makeatletter
\pgfdeclareshape{crosscircle}
{
	\inheritsavedanchors[from=circle] 
	\inheritanchorborder[from=circle]
	\inheritanchor[from=circle]{north}
	\inheritanchor[from=circle]{north west}
	\inheritanchor[from=circle]{north east}
	\inheritanchor[from=circle]{center}
	\inheritanchor[from=circle]{west}
	\inheritanchor[from=circle]{east}
	\inheritanchor[from=circle]{mid}
	\inheritanchor[from=circle]{mid west}
	\inheritanchor[from=circle]{mid east}
	\inheritanchor[from=circle]{base}
	\inheritanchor[from=circle]{base west}
	\inheritanchor[from=circle]{base east}
	\inheritanchor[from=circle]{south}
	\inheritanchor[from=circle]{south west}
	\inheritanchor[from=circle]{south east}
	\inheritbackgroundpath[from=circle]
	\foregroundpath{
		\centerpoint%
		\pgf@xc=\pgf@x%
		\pgf@yc=\pgf@y%
		\pgfutil@tempdima=\radius%
		\pgfmathsetlength{\pgf@xb}{\pgfkeysvalueof{/pgf/outer xsep}}%
		\pgfmathsetlength{\pgf@yb}{\pgfkeysvalueof{/pgf/outer ysep}}%
		\ifdim\pgf@xb<\pgf@yb%
		\advance\pgfutil@tempdima by-\pgf@yb%
		\else%
		\advance\pgfutil@tempdima by-\pgf@xb%
		\fi%
		\pgfpathmoveto{\pgfpointadd{\pgfqpoint{\pgf@xc}{\pgf@yc}}{\pgfqpoint{-0.707107\pgfutil@tempdima}{-0.707107\pgfutil@tempdima}}}
		\pgfpathlineto{\pgfpointadd{\pgfqpoint{\pgf@xc}{\pgf@yc}}{\pgfqpoint{0.707107\pgfutil@tempdima}{0.707107\pgfutil@tempdima}}}
		\pgfpathmoveto{\pgfpointadd{\pgfqpoint{\pgf@xc}{\pgf@yc}}{\pgfqpoint{-0.707107\pgfutil@tempdima}{0.707107\pgfutil@tempdima}}}
		\pgfpathlineto{\pgfpointadd{\pgfqpoint{\pgf@xc}{\pgf@yc}}{\pgfqpoint{0.707107\pgfutil@tempdima}{-0.707107\pgfutil@tempdima}}}
	}
}
\makeatother

\def\symbol#1{\textcolor{symbols}{#1}}

\def\decorate#1#2{
	\ifnum#2>0
	\foreach \count in {1,...,#2}{
		let
		\p1 = (sourcenode.center),
		\p2 = (sourcenode.east),
		\n1 = {\x2-\x1},
		\n2 = {1mm},
		\n3 = {(1.3+0.6*(\count-1))*\n1},
		\n4 = {0.7*\n1}
		in 
		node[rectangle,fill=symbols,rotate=30,inner sep=0pt,minimum width=0.2*\n2,minimum height=\n2] at ($(sourcenode.center) + (\n3,\n4)$) {}
	}
	\fi
	\ifnum#1>0
	\foreach \count in {1,...,#1}{
		let
		\p1 = (sourcenode.center),
		\p2 = (sourcenode.east),
		\n1 = {\x2-\x1},
		\n2 = {1mm},
		\n3 = {(1.3+0.6*(\count-1))*\n1},
		\n4 = {0.7*\n1}
		in 
		node[rectangle,fill=symbols,rotate=-30,inner sep=0pt,minimum width=0.2*\n2,minimum height=\n2] at ($(sourcenode.center) + (-\n3,\n4)$) {}
	}
	\fi
}

\tikzset{
	dectriangle/.style 2 args={
		triangle,
		alias=sourcenode,
		append after command={\decorate{#1}{#2}}
	},
	dectriangle/.default={0}{0},
}

\tikzset{
	cross/.style={path picture={ 
			\draw[symbols]
			(path picture bounding box.south east) -- (path picture bounding box.north west) (path picture bounding box.south west) -- (path picture bounding box.north east);
	}},
	root/.style={circle,fill=green!50!black,inner sep=0pt, minimum size=1.2mm},
	dot/.style={circle,fill=pageforeground,inner sep=0pt, minimum size=1mm},
	dotred/.style={circle,fill=pageforeground!50!pagebackground,inner sep=0pt, minimum size=2mm},
	var/.style={circle,fill=pageforeground!10!pagebackground,draw=pageforeground,inner sep=0pt, minimum size=3mm},
	var1/.style={circle,fill=pageforeground!10!pagebackground,draw=pageforeground,inner sep=0pt, minimum size=4mm},
	kernel/.style={semithick,shorten >=2pt,shorten <=2pt},
	kernels/.style={snake=zigzag,shorten >=2pt,shorten <=2pt,segment amplitude=1pt,segment length=4pt,line before snake=2pt,line after snake=5pt,},
	rho/.style={densely dashed,semithick,shorten >=2pt,shorten <=2pt},
	testfcn/.style={dotted,semithick,shorten >=2pt,shorten <=2pt},
	renorm/.style={shape=circle,fill=pagebackground,inner sep=1pt},
	labl/.style={shape=rectangle,fill=pagebackground,inner sep=1pt},
	xic/.style={very thin,circle,draw=symbols,fill=symbols,inner sep=0pt,minimum size=1.2mm},
	g/.style={very thin,rectangle,draw=symbols,fill=symbols!10!pagebackground,inner sep=0pt,minimum width=2.5mm,minimum height=1.2mm},
	xi/.style={very thin,circle,draw=symbols,fill=symbols!10!pagebackground,inner sep=0pt,minimum size=1.2mm},
	xies/.style={very thin,rectangle,fill=green!50!black!25,draw=symbols,inner sep=0pt,minimum size=1.1mm},
	xiesf/.style={very thin,rectangle,fill=green!50!black,draw=symbols,inner sep=0pt,minimum size=1.1mm},
	xix/.style={very thin,crosscircle,fill=symbols!10!pagebackground,draw=symbols,inner sep=0pt,minimum size=1.2mm},
	X/.style={very thin,cross,rectangle,fill=pagebackground,draw=symbols,inner sep=0pt,minimum size=1.2mm},
	xib/.style={thin,circle,fill=symbols!10!pagebackground,draw=symbols,inner sep=0pt,minimum size=1.6mm},
	xie/.style={thin,circle,fill=green!50!black,draw=symbols,inner sep=0pt,minimum size=1.6mm},
	xid/.style={thin,circle,fill=symbols,draw=symbols,inner sep=0pt,minimum size=1.6mm},
	xibx/.style={thin,crosscircle,fill=symbols!10!pagebackground,draw=symbols,inner sep=0pt,minimum size=1.6mm},
	kernels2/.style={very thick,draw=connection,segment length=12pt},
	keps/.style={thin,draw=symbols,->},
	kepspr/.style={thick,draw=connection,->},
	krho/.style={thin,draw=symbols,superdense,->},
	krhopr/.style={thick,draw=connection,superdense},
	triangle/.style = { regular polygon, regular polygon sides=3},
	not/.style={thin,circle,draw=connection,fill=connection,inner sep=0pt,minimum size=0.5mm},
	diff/.style = {very thin,draw=symbols,triangle,fill=red!50!black,inner sep=0pt,minimum size=1.6mm},
	diff1/.style = {very thin,dectriangle={1}{0},fill=red!50!black,draw=symbols,inner sep=0pt,minimum size=1.6mm},
	diff2/.style = {very thin,dectriangle={1}{1},fill=red!50!black,draw=symbols,inner sep=0pt,minimum size=1.6mm},
	diffmini/.style = {very thin,rectangle,fill=black,draw=black,inner sep=0pt,minimum size=0.75mm},
	kernelsmod/.style={very thick,draw=connection,segment length=12pt},
	rec/.style = {very thin,rectangle,fill=black,draw=black,inner sep=0pt,minimum size=2mm},
	cerc/.style={very thin,circle,draw=black,fill=symbols,inner sep=0pt,minimum size=2mm},
	stars/.style={very thin,star,star points=6,star point ratio=0.5, draw=black,fill=red,inner sep=0pt,minimum size=0.7mm},
	>=stealth,
}
\tikzset{
	root/.style={circle,fill=black!50,inner sep=0pt, minimum size=3mm},
	circ/.style={circle,fill=white,draw=black,very thin,inner sep=.5pt, minimum size=1.2mm},
	round1/.style={fill=white,outer sep = 0,inner sep=2pt,rounded corners=1mm,draw,text=black,thin,minimum size=1.2mm},
	circ1/.style={circle,fill=red!10,draw=red,very thin,inner sep=.5pt, minimum size=1.2mm},
	rect/.style={fill=white,outer sep = 0,inner sep=2pt,rectangle,draw,text=black,thin,minimum size=1.2mm},
	rect1/.style={fill=white,outer sep = 0,inner sep=2pt,rectangle,draw,text=black,thin,minimum size=1.2mm},
	round2/.style={fill=red!10,outer sep = 0,inner sep=2pt,rounded corners=1mm,draw,text=black,thin,minimum size=1.2mm},
	round3/.style={fill=blue!10,outer sep = 0,inner sep=2pt,rounded corners=1mm,draw,text=black,thin,minimum size=1.2mm}, 
	rect2/.style={fill=black!10,outer sep = 0,inner sep=2pt,rectangle,draw,text=black,thin,minimum size=1.2mm},
	dot/.style={circle,fill=black,inner sep=0pt, minimum size=1.2mm},
	dotred/.style={circle,fill=black!50,inner sep=0pt, minimum size=2mm},
	var/.style={circle,fill=black!10,draw=black,inner sep=0pt, minimum size=3mm},
	kernel/.style={semithick,shorten >=2pt,shorten <=2pt},
	diag/.style={thin,shorten >=4pt,shorten <=4pt},
	kernel1/.style={thick},
	kernels/.style={snake=zigzag,shorten >=2pt,shorten <=2pt,segment amplitude=1pt,segment length=4pt,line before snake=2pt,line after snake=5pt,},
	kernels1/.style={snake=zigzag,segment amplitude=0.5pt,segment length=2pt},
	rho1/.style={densely dotted,semithick},
	rho/.style={densely dashed,semithick,shorten >=2pt,shorten <=2pt},
	testfcn/.style={dotted,semithick,shorten >=2pt,shorten <=2pt},
	visible/.style={draw, circle, fill, inner sep=0.25ex},
	renorm/.style={shape=circle,fill=white,inner sep=1pt},
	labl/.style={shape=rectangle,fill=white,inner sep=1pt},
	xic/.style={very thin,circle,fill=symbols,draw=black,inner sep=0pt,minimum size=1.2mm},
	xi/.style={very thin,circle,fill=blue!10,draw=black,inner sep=0pt,minimum size=1.2mm},
	xib/.style={very thin,circle,fill=blue!10,draw=black,inner sep=0pt,minimum size=1.6mm},
	xie/.style={very thin,circle,fill=green!50!black,draw=black,inner sep=0pt,minimum size=1mm},
	xid/.style={very thin,circle,fill=symbols,draw=black,inner sep=0pt,minimum size=1.6mm},
	edgetype/.style={very thin,circle,draw=black,inner sep=0pt,minimum size=5mm},
	nodetype/.style={very thick,circle,draw=black,inner sep=0pt,minimum size=5mm},
	kernels2/.style={very thick,draw=connection,segment length=12pt},
	clean/.style={thin,circle,fill=black,inner sep=0pt,minimum size=1mm},	not/.style={thin,circle,fill=symbols,draw=connection,fill=connection,inner sep=0pt,minimum size=0.8mm},
	>=stealth,
}

\makeatletter
\def\DeclareSymbol#1#2#3{%
	\expandafter\gdef\csname MH@symb@#1\endcsname{\tikzsetnextfilename{symbol#1}%
		\tikz[baseline=#2,scale=0.15,draw=symbols,line join=round]{#3}}%
	\expandafter\gdef\csname MH@symb@#1s\endcsname{\scalebox{0.75}{\tikzsetnextfilename{symbol#1}%
			\tikz[baseline=#2,scale=0.15,draw=symbols,line join=round]{#3}}}%
	\expandafter\gdef\csname MH@symb@#1ss\endcsname{\scalebox{0.65}{\tikzsetnextfilename{symbol#1}%
			\tikz[baseline=#2,scale=0.15,draw=symbols,line join=round]{#3}}}%
}
\def\<#1>{\ifthenelse{\boolean{mmode}}{\mathchoice{\csname MH@symb@#1\endcsname}{\csname MH@symb@#1\endcsname}{\csname MH@symb@#1s\endcsname}{\csname MH@symb@#1ss\endcsname}}{\csname MH@symb@#1\endcsname}}
\makeatother

\DeclareSymbol{Xi22}{0.5}{\draw (0,0) node[xi] {} -- (-1,1) node[not] {} -- (0,2) node[xi] {};} 
\DeclareSymbol{Xi2}{-2}{\draw (-1,-0.25) node[xi] {} -- (0,1) node[xi] {};} 
\DeclareSymbol{Xi2b}{-2}{\draw (-1,-0.25) node[xic] {} -- (0,1) node[xic] {};} 
\DeclareSymbol{Xi2g}{-2}{\draw (-1,-0.25) node[xies] {} -- (0,1) node[xi] {};} 
\DeclareSymbol{Xi2g2}{-2}{\draw (-1,-0.25) node[xi] {} -- (0,1) node[xies] {};} 
\DeclareSymbol{cXi2}{-2}{\draw (0,-0.25) node[xi] {} -- (-1,1) node[xic] {};}
\DeclareSymbol{Xi3}{0}{\draw (0,0) node[xi] {} -- (-1,1) node[xi] {} -- (0,2) node[xi] {};}
\DeclareSymbol{XiIIXi}{0}{\draw (0,0) node[xi] {} -- (-1,1); \draw[kernels2] (-1,1) node[not] {} -- (0,2) node[xi] {};}

\DeclareSymbol{Xi4}{2}{\draw (0,0) node[xi] {} -- (-1,1) node[xi] {} -- (0,2) node[xi] {} -- (-1,3) node[xi] {};}
\DeclareSymbol{Xi4_1}{2}{\draw (0,0) node[xic] {} -- (-1,1) node[xic] {} -- (0,2) node[xi] {} -- (-1,3) node[xi] {};}
\DeclareSymbol{Xi4_2}{2}{\draw (0,0) node[xic] {} -- (-1,1) node[xi] {} -- (0,2) node[xi] {} -- (-1,3) node[xic] {};}
\DeclareSymbol{Xi2X}{-2}{\draw (0,-0.25) node[xi] {} -- (-1,1) node[xix] {};}
\DeclareSymbol{XXi2}{-2}{\draw (0,-0.25) node[xix] {} -- (-1,1) node[xi] {};}
\DeclareSymbol{IIXi}{0}{\draw (0,-0.25) node[not] {} -- (-1,1) node[xi] {} -- (0,2) node[xi] {};}
\DeclareSymbol{IXi^2}{-1}{\draw (-1,1) node[xi] {} -- (0,0) node[not] {} -- (1,1) node[xi] {};}
\DeclareSymbol{IIXi^2}{-4}{\draw (0,-1.5) node[not] {} -- (0,0);
	\draw[kernels2] (-1,1) node[xi] {} -- (0,0) node[not] {} -- (1,1) node[xi] {};}
\DeclareSymbol{XiX}{-2.8}{\node[xibx] {};}
\DeclareSymbol{tauX}{-2.8}{ \node[X] {};}
\DeclareSymbol{Xi}{-2.8}{\node[xib] {};}

\DeclareSymbol{IXiX}{-1}{\draw (0,-0.25) node[not] {} -- (-1,1) node[xix] {};}
\DeclareSymbol{IXi3}{2}{\draw (0,-0.25) node[not] {} -- (-1,1) node[xi] {} -- (0,2) node[xi] {} -- (-1,3) node[xi] {};}
\DeclareSymbol{IXi}{-2}{\draw (0,-0.25) node[not] {} -- (-1,1) node[xi] {};}
\DeclareSymbol{XiI}{-2}{\draw (0,-0.25) node[xi] {} -- (-1,1) node[not] {};}

\DeclareSymbol{Xi4b}{0}{\draw(0,1.5) node[xi] {} -- (0,0); \draw (-1,1) node[xi] {} -- (0,0) node[xi] {} -- (1,1) node[xi] {};}
\DeclareSymbol{Xi4b'}{0}{\draw(0,1.5) node[xi] {} -- (0,-0.2); \draw (-1,1) node[xi] {} -- (0,-0.2) node[not] {} -- (1,1) node[xi] {};}
\DeclareSymbol{Xi4c}{0}{\draw (0,1) -- (0.8,2.2) node[xi] {};\draw (0,-0.25) node[xi] {} -- (0,1) node[xi] {} -- (-0.8,2.2) node[xi] {};}
\DeclareSymbol{Xi4d}{-4.5}{\draw (0,-1.5) node[not] {} -- (0,0); \draw (-1,1) node[xi] {} -- (0,0) node[xi] {} -- (1,1) node[xi] {};}
\DeclareSymbol{Xi4e}{0}{\draw (0,2) node[xi] {} -- (-1,1) node[xi] {} -- (0,0) node[xi] {} -- (1,1) node[xi] {};}
\DeclareSymbol{Xi4e'}{0}{\draw (0,2) node[xi] {} -- (-1,1) node[xi] {} -- (0,-0.2) node[not] {} -- (1,1) node[xi] {};}

\DeclareSymbol{Xitwo}
{0}{\draw[kernels2] (0,0) node[not] {} -- (-1,1) node[not] {}
	-- (-2,2) node[not]{} -- (-3,3) node[xi]  {};
	\draw[kernels2] (0,0) -- (1,1) node[xi] {};
	\draw[kernels2] (-1,1) -- (0,2) node[xi] {};
	\draw[kernels2] (-2,2) -- (-1,3) node[xi] {};}

\DeclareSymbol{IXitwo}
{0}{\draw (-.7,1.2) node[xi] {} -- (0,-0.2) -- (.7,1.2) node[xi] {};}
\DeclareSymbol{I1Xitwo}
{0}{\draw[kernels2] (0,0) node[not] {} -- (-1,1) node[xi] {};
	\draw[kernels2] (0,0) -- (1,1) node[xi] {};}

\DeclareSymbol{I1Xitwobis}
{0}{\draw[kernels2] (0,0) node[not] {} -- (-1,1) node[xies] {};
	\draw[kernels2] (0,0) -- (1,1) node[xies] {};}

\DeclareSymbol{I1Xitwog}
{0}{\draw[kernels2] (0,0) node[not] {} -- (-1,1) node[xies] {};
	\draw[kernels2] (0,0) -- (1,1) node[xi] {};}

\DeclareSymbol{cI1Xitwo}
{0}{\draw[kernels2] (0,0) node[not] {} -- (-1,1) node[xic] {};
	\draw[kernels2] (0,0) -- (1,1) node[xi] {};}

\DeclareSymbol{I1IXi3}{0}{\draw (0,0) node[xi] {} -- (-1,1) ; 
	\draw[kernels2] (-1,1) node[not] {} -- (0,2) node[xi] {};
	\draw[kernels2] (-1,1) node[not] {} -- (-2,2) node[xi] {};}

\DeclareSymbol{I1Xi3c}{-1}{\draw[kernels2](0,1.5) node[xi] {} -- (0,0) node[not] {}; \draw (-1,1) node[xi] {} -- (0,0) ; \draw[kernels2] (0,0) -- (1,1) node[xi] {};}

\DeclareSymbol{I1Xi3cbis}{-1}{\draw[kernels2](0,1.5) node[xies] {} -- (0,0) node[not] {}; \draw (-1,1) node[xies] {} -- (0,0) ; \draw[kernels2] (0,0) -- (1,1) node[xies] {};}

\DeclareSymbol{I1IXi3b}{0}{\draw[kernels2] (0,0) node[not] {} -- (-1,1) ; \draw[kernels2] (0,0)   -- (1,1) node[xi] {} ;
	\draw (-1,1) node[xi] {} -- (0,2) node[xi] {};
}

\DeclareSymbol{I1IXi3c}{0}{\draw[kernels2] (0,0) node[not] {} -- (-1,1) ; \draw[kernels2] (0,0)   -- (1,1) node[xi] {} ;
	\draw[kernels2] (-1,1) node[not] {} -- (0,2) node[xi] {};
	\draw[kernels2] (-1,1) node[not] {} -- (-2,2) node[xi] {};}

\DeclareSymbol{I1IXi3cbis}{0}{\draw[kernels2] (0,0) node[not] {} -- (-1,1) ; \draw[kernels2] (0,0)   -- (1,1) node[xies] {} ;
	\draw[kernels2] (-1,1) node[not] {} -- (0,2) node[xies] {};
	\draw[kernels2] (-1,1) node[not] {} -- (-2,2) node[xies] {};}

\DeclareSymbol{I1Xi}{0}{\draw[kernels2] (0,0) node[not] {} -- (-1,1)  node[xi] {} ;}

\DeclareSymbol{I1Xi4a}{2}{\draw[kernels2] (0,0) node[not] {} -- (-1,1) ; \draw[kernels2] (0,0) node[not] {} -- (1,1) node[xi] {} ;
	\draw (-1,1) node[xi] {} -- (0,2) node[xi] {} -- (-1,3) node[xi] {};}

\DeclareSymbol{cI1Xi4a}{2}{\draw[kernels2] (0,0) node[not] {} -- (-1,1) ; \draw[kernels2] (0,0) node[not] {} -- (1,1) node[xic] {} ;
	\draw (-1,1) node[xic] {} -- (0,2) node[xi] {} -- (-1,3) node[xi] {};}

\DeclareSymbol{I1Xi4b}{2}{\draw (0,0) node[xi] {} -- (-1,1) node[xi] {} -- (0,2) ; \draw[kernels2] (0,2) node[not] {} -- (-1,3) node[xi] {};\draw[kernels2] (0,2)  -- (1,3) node[xi] {};
}

\DeclareSymbol{cI1Xi4b}{2}{\draw (0,0) node[xic] {} -- (-1,1) node[xic] {} -- (0,2) ; \draw[kernels2] (0,2) node[not] {} -- (-1,3) node[xi] {};\draw[kernels2] (0,2)  -- (1,3) node[xi] {};
}

\DeclareSymbol{I1Xi4c}{2}{\draw (0,0) node[xi] {} -- (-1,1) node[not] {}; \draw[kernels2] (-1,1) -- (0,2) ; 
	\draw[kernels2] (-1,1) -- (-2,2) node[xi] {} ;
	\draw (0,2) node[xi] {} -- (-1,3) node[xi] {};}

\DeclareSymbol{cI1Xi4c}{2}{\draw (0,0) node[xic] {} -- (-1,1) node[not] {}; \draw[kernels2] (-1,1) -- (0,2) ; 
	\draw[kernels2] (-1,1) -- (-2,2) node[xic] {} ;
	\draw (0,2) node[xi] {} -- (-1,3) node[xi] {};}

\DeclareSymbol{I1Xi4ab}{2}{\draw[kernels2] (0,0) node[not] {} -- (-1,1) ; \draw[kernels2] (0,0) node[not] {} -- (1,1) node[xi] {};\draw (-1,1) node[xi] {} -- (0,2) ; \draw[kernels2] (0,2) node[not] {} -- (-1,3) node[xi] {};\draw[kernels2] (0,2)  -- (1,3) node[xi] {}; }

\DeclareSymbol{cI1Xi4ab}{2}{\draw[kernels2] (0,0) node[not] {} -- (-1,1) ; \draw[kernels2] (0,0) node[not] {} -- (1,1) node[xic] {};\draw (-1,1) node[xic] {} -- (0,2) ; \draw[kernels2] (0,2) node[not] {} -- (-1,3) node[xi] {};\draw[kernels2] (0,2)  -- (1,3) node[xi] {}; }

\DeclareSymbol{I1Xi4bc}{2}{\draw (0,0) node[xi] {} -- (-1,1) node[not] {}; \draw[kernels2] (-1,1) -- (0,2) ; 
	\draw[kernels2] (-1,1) -- (-2,2) node[xi] {} ; \draw[kernels2] (0,2) node[not] {} -- (-1,3) node[xi] {};\draw[kernels2] (0,2)  -- (1,3) node[xi] {};
}

\DeclareSymbol{cI1Xi4bc}{2}{\draw (0,0) node[xic] {} -- (-1,1) node[not] {}; \draw[kernels2] (-1,1) -- (0,2) ; 
	\draw[kernels2] (-1,1) -- (-2,2) node[xic] {} ; \draw[kernels2] (0,2) node[not] {} -- (-1,3) node[xi] {};\draw[kernels2] (0,2)  -- (1,3) node[xi] {};
}

\DeclareSymbol{I1Xi4abcc1}{2}{\draw[kernels2] (0,0) node[not] {} -- (-1,1) node[not] {}
	-- (-2,2) node[not]{} -- (-3,3) node[xic]  {};
	\draw[kernels2] (0,0) -- (1,1) node[xic] {};
	\draw[kernels2] (-1,1) -- (0,2) node[xi] {};
	\draw[kernels2] (-2,2) -- (-1,3) node[xi] {};
}

\DeclareSymbol{I1Xi4abcc1b}{2}{\draw[kernels2] (0,0) node[not] {} -- (-1,1) node[not] {}
	-- (-2,2) node[not]{} -- (-3,3) node[xi]  {};
	\draw[kernels2] (0,0) -- (1,1) node[xic] {};
	\draw[kernels2] (-1,1) -- (0,2) node[xic] {};
	\draw[kernels2] (-2,2) -- (-1,3) node[xi] {};
}

\DeclareSymbol{I1Xi4abcc2}{2}{\draw[kernels2] (0,0) node[not] {} -- (-1,1) node[not] {}
	-- (-2,2) node[not]{} -- (-3,3) node[xic]  {};
	\draw[kernels2] (0,0) -- (1,1) node[xi] {};
	\draw[kernels2] (-1,1) -- (0,2) node[xi] {};
	\draw[kernels2] (-2,2) -- (-1,3) node[xic] {};
}

\DeclareSymbol{I1Xi4ac}{2}{\draw[kernels2] (0,0) node[not] {} -- (-1,1) ; \draw[kernels2] (0,0) node[not] {} -- (1,1) node[xi] {}; 
	\draw[kernels2] (-1,1) node[not] {} -- (0,2) ; 
	\draw[kernels2] (-1,1) -- (-2,2) node[xi] {} ;
	\draw (0,2) node[xi] {} -- (-1,3) node[xi] {};}

\DeclareSymbol{cI1Xi4ac}{2}{\draw[kernels2] (0,0) node[not] {} -- (-1,1) ; \draw[kernels2] (0,0) node[not] {} -- (1,1) node[xic] {}; 
	\draw[kernels2] (-1,1) node[not] {} -- (0,2) ; 
	\draw[kernels2] (-1,1) -- (-2,2) node[xic] {} ;
	\draw (0,2) node[xi] {} -- (-1,3) node[xi] {};}

\DeclareSymbol{I1Xi4acc1}{2}{\draw[kernels2] (0,0) node[not] {} -- (-1,1) ; \draw[kernels2] (0,0) node[not] {} -- (1,1) node[xic] {}; 
	\draw[kernels2] (-1,1) node[not] {} -- (0,2) ; 
	\draw[kernels2] (-1,1) -- (-2,2) node[xi] {} ;
	\draw (0,2) node[xic] {} -- (-1,3) node[xi] {};}

\DeclareSymbol{I1Xi4acc2}{2}{\draw[kernels2] (0,0) node[not] {} -- (-1,1) ; \draw[kernels2] (0,0) node[not] {} -- (1,1) node[xic] {}; 
	\draw[kernels2] (-1,1) node[not] {} -- (0,2) ; 
	\draw[kernels2] (-1,1) -- (-2,2) node[xi] {} ;
	\draw (0,2) node[xi] {} -- (-1,3) node[xic] {};}

\DeclareSymbol{2I1Xi4}{2}{\draw[kernels2] (0,0) node[not] {} -- (-1,1) node[not] {};
	\draw[kernels2] (0,0) -- (1,1) node[not] {};
	\draw[kernels2] (-1,1) -- (-1.5,2.5) node[xi] {};
	\draw[kernels2] (-1,1) -- (-0.5,2.5) node[xi] {};
	\draw[kernels2] (1,1) -- (0.5,2.5) node[xi] {};
	\draw[kernels2] (1,1) -- (1.5,2.5) node[xi] {};
}

\DeclareSymbol{2I1Xi4dis}{2}{\draw[kernels2] (0,0) node[not] {} -- (-1,1) node[not] {};
	\draw[kernels2] (0,0) -- (1,1) node[not] {};
	\draw[kernels2] (-1,1) -- (-1.5,2.5) node[xies] {};
	\draw[kernels2] (-1,1) -- (-0.5,2.5) node[xies] {};
	\draw[kernels2] (1,1) -- (0.5,2.5) node[xies] {};
	\draw[kernels2] (1,1) -- (1.5,2.5) node[xies] {};
}

\DeclareSymbol{2I1Xi4c1}{2}{\draw[kernels2] (0,0) node[not] {} -- (-1,1) node[not] {};
	\draw[kernels2] (0,0) -- (1,1) node[not] {};
	\draw[kernels2] (-1,1) -- (-1.5,2.5) node[xic] {};
	\draw[kernels2] (-1,1) -- (-0.5,2.5) node[xi] {};
	\draw[kernels2] (1,1) -- (0.5,2.5) node[xic] {};
	\draw[kernels2] (1,1) -- (1.5,2.5) node[xi] {};
}

\DeclareSymbol{2I1Xi4c2}{2}{\draw[kernels2] (0,0) node[not] {} -- (-1,1) node[not] {};
	\draw[kernels2] (0,0) -- (1,1) node[not] {};
	\draw[kernels2] (-1,1) -- (-1.5,2.5) node[xic] {};
	\draw[kernels2] (-1,1) -- (-0.5,2.5) node[xic] {};
	\draw[kernels2] (1,1) -- (0.5,2.5) node[xi] {};
	\draw[kernels2] (1,1) -- (1.5,2.5) node[xi] {};
}

\DeclareSymbol{2I1Xi4b}{2}{\draw[kernels2] (0,0) node[not] {} -- (-1,1) ;
	\draw[kernels2] (0,0) -- (1,1);
	\draw (-1,1) node[xi] {} -- (-1,2.5) node[xi] {};
	\draw (1,1)  node[xi] {} -- (1,2.5) node[xi] {};
}

\DeclareSymbol{2I1Xi4bb}{2}{\draw[kernels2] (0,0) node[not] {} -- (-1,1) ;
	\draw[kernels2] (0,0) -- (1,1);
	\draw (-1,1) node[xi] {} -- (-1,2.5) node[xiesf] {};
	\draw (1,1)  node[xi] {} -- (1,2.5) node[xic] {};
}

\DeclareSymbol{2I1Xi4c}{2}{\draw[kernels2] (0,0) node[not] {} -- (-1,1);
	\draw[kernels2] (0,0) -- (1,1) node[not] {};
	\draw (-1,1)  node[xi] {} -- (-1,2.5) node[xi] {};
	\draw[kernels2] (1,1) -- (0.4,2.5) node[xi] {};
	\draw[kernels2] (1,1) -- (1.6,2.5) node[xi] {};
}

\DeclareSymbol{2I1Xi4cc1}{2}{\draw[kernels2] (0,0) node[not] {} -- (-1,1);
	\draw[kernels2] (0,0) -- (1,1) node[not] {};
	\draw (-1,1)  node[xic] {} -- (-1,2.5) node[xi] {};
	\draw[kernels2] (1,1) -- (0.4,2.5) node[xic] {};
	\draw[kernels2] (1,1) -- (1.6,2.5) node[xi] {};
}

\DeclareSymbol{2I1Xi4cc2}{2}{\draw[kernels2] (0,0) node[not] {} -- (-1,1);
	\draw[kernels2] (0,0) -- (1,1) node[not] {};
	\draw (-1,1)  node[xic] {} -- (-1,2.5) node[xic] {};
	\draw[kernels2] (1,1) -- (0.4,2.5) node[xi] {};
	\draw[kernels2] (1,1) -- (1.6,2.5) node[xi] {};
}

\DeclareSymbol{Xi4ba}{0}{\draw(-0.5,1.5) node[xi] {} -- (0,0); \draw (-1.5,1) node[xi] {} -- (0,0) node[not] {}; \draw[kernels2] (0,0) -- (1.5,1) node[xi] {};
	\draw[kernels2] (0,0) -- (0.5,1.5) node[xi] {} ;}

\DeclareSymbol{Xi4badis}{0}{\draw(-0.5,1.5) node[xies] {} -- (0,0); \draw (-1.5,1) node[xies] {} -- (0,0) node[not] {}; \draw[kernels2] (0,0) -- (1.5,1) node[xies] {};
	\draw[kernels2] (0,0) -- (0.5,1.5) node[xies] {} ;}

\DeclareSymbol{Xi4ba1}{0}{\draw(-0.5,1.5) node[xi] {} -- (0,0); \draw (-1.5,1) node[xi] {} -- (0,0) node[not] {}; \draw[kernels2] (0,0) -- (1.5,1) node[xic] {};
	\draw[kernels2] (0,0) -- (0.5,1.5) node[xic] {} ;}

\DeclareSymbol{Xi4ba1b}{0}{\draw(-0.5,1.5) node[xic] {} -- (0,0); \draw (-1.5,1) node[xic] {} -- (0,0) node[not] {}; \draw[kernels2] (0,0) -- (1.5,1) node[xi] {};
	\draw[kernels2] (0,0) -- (0.5,1.5) node[xi] {} ;}

\DeclareSymbol{Xi4ba1bdiff}{0}{\draw(-0.5,1.5) node[xic] {} -- (0,0); \draw (-1.5,1) node[xic] {} -- (0,0) node[not] {}; \draw (0,0) -- (1.5,1) node[xi] {};
	\draw (0,0) -- (0.5,1.5) node[xi] {};
	\draw(0,0) node[diff] {};}

\DeclareSymbol{Xi4ba1bb}{0}{\draw(-0.5,1.5) node[xic] {} -- (0,0); \draw (-1.5,1) node[xiesf] {} -- (0,0) node[not] {}; \draw[kernels2] (0,0) -- (1.5,1) node[xi] {};
	\draw[kernels2] (0,0) -- (0.5,1.5) node[xi] {} ;}

\DeclareSymbol{Xi4ba2}{0}{\draw(-0.5,1.5) node[xi] {} -- (0,0); \draw (-1.5,1) node[xic] {} -- (0,0) node[not] {}; \draw[kernels2] (0,0) -- (1.5,1) node[xi] {};
	\draw[kernels2] (0,0) -- (0.5,1.5) node[xic] {} ;}

\DeclareSymbol{Xi4ba2b}{0}{\draw(-0.5,1.5) node[xi] {} -- (0,0); \draw (-1.5,1) node[xic] {} -- (0,0) node[not] {}; \draw[kernels2] (0,0) -- (1.5,1) node[xi] {};
	\draw[kernels2] (0,0) -- (0.5,1.5) node[xiesf] {} ;}

\DeclareSymbol{Xi4ca}{0}{\draw (0,1) -- (-1,2.2) node[xi] {};\draw (0,-0.25) node[xi] {} -- (0,1) ; \draw[kernels2] (0,1) node[not] {} -- (1,2.2) node[xi] {};
	\draw[kernels2] (0,1) {} -- (0,2.7) node[xi] {};
}

\DeclareSymbol{Xi4cb}{0}{\draw (-1,1) -- (-2,2) node[xi] {};\draw[kernels2] (0,0)  -- (-1,1) node[xi] {} ; \draw[kernels2] (0,0) node[not] {} -- (1,1) node[xi] {} ; 
	\draw (-1,1) node[xi] {} -- (0,2) node[xi] {};}

\DeclareSymbol{Xi4cbb}{0}{\draw (-1,1) -- (-2,2) node[xiesf] {};\draw[kernels2] (0,0)  -- (-1,1) node[xi] {} ; \draw[kernels2] (0,0) node[not] {} -- (1,1) node[xi] {} ; 
	\draw (-1,1) node[xi] {} -- (0,2) node[xic] {};}

\DeclareSymbol{Xi4cbc1}{0}{\draw (-1,1) -- (-2,2) node[xic] {};\draw[kernels2] (0,0)  -- (-1,1) node[xic] {} ; \draw[kernels2] (0,0) node[not] {} -- (1,1) node[xi] {} ; 
	\draw (-1,1) node[xic] {} -- (0,2) node[xi] {};}

\DeclareSymbol{Xi4cbc2}{0}{\draw (-1,1) -- (-2,2) node[xi] {};\draw[kernels2] (0,0)  -- (-1,1) node[xi] {} ; \draw[kernels2] (0,0) node[not] {} -- (1,1) node[xic] {} ; 
	\draw (-1,1) node[xic] {} -- (0,2) node[xi] {};}

\DeclareSymbol{Xi4cab}{0}{\draw (-1,1) -- (-2,2) node[xi] {};\draw[kernels2] (0,0)  -- (-1,1); \draw[kernels2] (0,0) node[not] {} -- (1,1) node[xi] {} ; 
	\draw[kernels2] (-1,1)  {} -- (0,2) node[xi] {};
	\draw[kernels2] (-1,1) node[not] {} -- (-1,2.5) node[xi] {};
}

\DeclareSymbol{Xi4cabdis}{0}{\draw (-1,1) -- (-2,2) node[xies] {};\draw[kernels2] (0,0)  -- (-1,1); \draw[kernels2] (0,0) node[not] {} -- (1,1) node[xies] {} ; 
	\draw[kernels2] (-1,1)  {} -- (0,2) node[xies] {};
	\draw[kernels2] (-1,1) node[not] {} -- (-1,2.5) node[xies] {};
}

\DeclareSymbol{Xi4cabc1}{0}{\draw (-1,1) -- (-2,2) node[xi] {};\draw[kernels2] (0,0)  -- (-1,1); \draw[kernels2] (0,0) node[not] {} -- (1,1) node[xic] {} ; 
	\draw[kernels2] (-1,1)  {} -- (0,2) node[xic] {};
	\draw[kernels2] (-1,1) node[not] {} -- (-1,2.5) node[xi] {};
}

\DeclareSymbol{Xi4cabc2}{0}{\draw (-1,1) -- (-2,2) node[xic] {};\draw[kernels2] (0,0)  -- (-1,1); \draw[kernels2] (0,0) node[not] {} -- (1,1) node[xic] {} ; 
	\draw[kernels2] (-1,1)  {} -- (0,2) node[xi] {};
	\draw[kernels2] (-1,1) node[not] {} -- (-1,2.5) node[xi] {};
}

\DeclareSymbol{Xi4ea}{1.5}{\draw (-1,2.5) node[xi] {} -- (-1,1) node[xi] {} -- (0,0); 
	\draw[kernels2] (0,0)  -- (1,1) node[xi] {};
	\draw[kernels2] (0,0) node[not] {} -- (0,1.5) node[xi] {}; }

\DeclareSymbol{Xi4eac1}{1.5}{\draw (-1,2.5) node[xic] {} -- (-1,1) node[xi] {} -- (0,0); 
	\draw[kernels2] (0,0)  -- (1,1) node[xic] {};
	\draw[kernels2] (0,0) node[not] {} -- (0,1.5) node[xi] {}; }

\DeclareSymbol{Xi4eac1b}{1.5}{\draw (-1,2.5) node[xic] {} -- (-1,1) node[xi] {} -- (0,0); 
	\draw[kernels2] (0,0)  -- (1,1) node[xiesf] {};
	\draw[kernels2] (0,0) node[not] {} -- (0,1.5) node[xi] {}; }

\DeclareSymbol{Xi4eac2}{1.5}{\draw (-1,2.5) node[xic] {} -- (-1,1) node[xic] {} -- (0,0); 
	\draw[kernels2] (0,0)  -- (1,1) node[xi] {};
	\draw[kernels2] (0,0) node[not] {} -- (0,1.5) node[xi] {}; }

\DeclareSymbol{Xi4eact1}{1.5}{\draw (-1,2.5) node[xic] {} -- (-1,1) node[xi] {} -- (0,0); 
	\draw (0,0)  -- (1,1) node[xic] {};
	\draw[rho] (0,0) node[not] {} -- (0,1.5) node[xi] {}; }

\DeclareSymbol{Xi4eact2}{1.5}{\draw[rho] (-1,2.5) node[xic] {} -- (-1,1) node[xi] {} -- (0,0); 
	\draw (0,0)  -- (1,1) node[xic] {};
	\draw (0,0) node[not] {} -- (0,1.5) node[xi] {}; }

\DeclareSymbol{Xi4eabis}{1.5}{\draw (-1,2.5) node[xi] {} -- (-1,1) ; \draw[kernels2] (-1,1) node[xi] {} -- (0,0); 
	\draw (0,0)  -- (1,1) node[xi] {};
	\draw[kernels2] (0,0) node[not] {} -- (0,1.5) node[xi] {}; }

\DeclareSymbol{Xi4eabisc1}{1.5}{\draw (-1,2.5) node[xic] {} -- (-1,1) ; \draw[kernels2] (-1,1) node[xi] {} -- (0,0); 
	\draw (0,0)  -- (1,1) node[xi] {};
	\draw[kernels2] (0,0) node[not] {} -- (0,1.5) node[xic] {}; }

\DeclareSymbol{Xi4eabisc1b}{1.5}{\draw (-1,2.5) node[xic] {} -- (-1,1) ; \draw[kernels2] (-1,1) node[xi] {} -- (0,0); 
	\draw (0,0)  -- (1,1) node[xi] {};
	\draw[kernels2] (0,0) node[not] {} -- (0,1.5) node[xiesf] {}; }

\DeclareSymbol{Xi4eabisc1bis}{1.5}{\draw (-1,2.5) node[xi] {} -- (-1,1) ; \draw[kernels2] (-1,1) node[xi] {} -- (0,0); 
	\draw (0,0)  -- (1,1) node[xi] {};
	\draw[kernels2] (0,0) node[not] {} -- (0,1.5) node[xi] {};
	\draw (-2,1) node[] {\tiny{$i$}};
	\draw (-2,2.5) node[] {\tiny{$\ell$}};
	\draw (2,1) node[] {\tiny{$k$}};
	\draw (0,2.5) node[] {\tiny{$j$}};
}

\DeclareSymbol{Xi4eabisc1tris}{1.5}{\draw (-1,2.5) node[xi] {} -- (-1,1) ; \draw[kernels2] (-1,1) node[xi] {} -- (0,0); 
	\draw (0,0)  -- (1,1) node[xi] {};
	\draw[kernels2] (0,0) node[not] {} -- (0,1.5) node[xi] {};
	\draw (-2,1) node[] {\tiny{i}};
	\draw (-2,2.5) node[] {\tiny{j}};
	\draw (2,1) node[] {\tiny{j}};
	\draw (0,2.5) node[] {\tiny{i}};
}

\DeclareSymbol{Xi4eabisc1quater}{1.5}{\draw (-1,2.5) node[xic] {} -- (-1,1) ; \draw[kernels2] (-1,1) node[xi] {} -- (0,0); 
	\draw (0,0)  -- (1,1) node[xic] {};
	\draw[kernels2] (0,0) node[not] {} -- (0,1.5) node[xi] {};
}

\DeclareSymbol{Xi4eabisc2}{1.5}{\draw (-1,2.5) node[xic] {} -- (-1,1) ; \draw[kernels2] (-1,1) node[xi] {} -- (0,0); 
	\draw (0,0)  -- (1,1) node[xic] {};
	\draw[kernels2] (0,0) node[not] {} -- (0,1.5) node[xi] {}; }

\DeclareSymbol{Xi4eabisc2l}{1.5}{\draw (-1,2.5) node[xiesf] {} -- (-1,1) ; \draw[kernels2] (-1,1) node[xi] {} -- (0,0); 
	\draw (0,0)  -- (1,1) node[xic] {};
	\draw[kernels2] (0,0) node[not] {} -- (0,1.5) node[xi] {}; }

\DeclareSymbol{Xi4eabisc2r}{1.5}{\draw (-1,2.5) node[xic] {} -- (-1,1) ; \draw[kernels2] (-1,1) node[xi] {} -- (0,0); 
	\draw (0,0)  -- (1,1) node[xiesf] {};
	\draw[kernels2] (0,0) node[not] {} -- (0,1.5) node[xi] {}; }

\DeclareSymbol{Xi4eabisc3}{1.5}{\draw (-1,2.5) node[xic] {} -- (-1,1) ; \draw[kernels2] (-1,1) node[xic] {} -- (0,0); 
	\draw (0,0)  -- (1,1) node[xi] {};
	\draw[kernels2] (0,0) node[not] {} -- (0,1.5) node[xi] {}; }

\DeclareSymbol{Xi4eb}{0}{
	\draw[kernels2] (0,2) node[xi] {} -- (-1,1) ; \draw[kernels2] (-2,2)  node[xi] {} -- (-1,1) ; \draw (-1,1)  node[not] {} -- (0,0); 
	\draw (0,0) node[xi] {}  -- (1,1) node[xi] {};
}

\DeclareSymbol{Xi4eab}{1.5}{\draw[kernels2] (-1,2.5) node[xi] {} -- (-1,1) ; \draw[kernels2] (-2,2)  node[xi] {} -- (-1,1) ; \draw (-1,1)  node[not] {} -- (0,0); 
	\draw[kernels2] (0,0)  -- (1,1) node[xi] {};
	\draw[kernels2] (0,0) node[not] {} -- (0,1.5) node[xi] {}; 
}

\DeclareSymbol{Xi4eabdis}{1.5}{\draw[kernels2] (-1,2.5) node[xies] {} -- (-1,1) ; \draw[kernels2] (-2,2)  node[xies] {} -- (-1,1) ; \draw (-1,1)  node[not] {} -- (0,0); 
	\draw[kernels2] (0,0)  -- (1,1) node[xies] {};
	\draw[kernels2] (0,0) node[not] {} -- (0,1.5) node[xies] {}; 
}

\DeclareSymbol{Xi4eabc1}{1.5}{\draw[kernels2] (-1,2.5) node[xic] {} -- (-1,1) ; \draw[kernels2] (-2,2)  node[xi] {} -- (-1,1) ; \draw (-1,1)  node[not] {} -- (0,0); 
	\draw[kernels2] (0,0)  -- (1,1) node[xic] {};
	\draw[kernels2] (0,0) node[not] {} -- (0,1.5) node[xi] {}; 
}

\DeclareSymbol{Xi4eabc2}{1.5}{\draw[kernels2] (-1,2.5) node[xi] {} -- (-1,1) ; \draw[kernels2] (-2,2)  node[xi] {} -- (-1,1) ; \draw (-1,1)  node[not] {} -- (0,0); 
	\draw[kernels2] (0,0)  -- (1,1) node[xic] {};
	\draw[kernels2] (0,0) node[not] {} -- (0,1.5) node[xic] {}; 
}

\DeclareSymbol{Xi4eabbis}{1.5}{\draw[kernels2] (-1,2.5) node[xi] {} -- (-1,1) ; \draw[kernels2] (-2,2)  node[xi] {} -- (-1,1) ; \draw[kernels2] (-1,1)  node[not] {} -- (0,0); 
	\draw (0,0)  -- (1,1) node[xi] {};
	\draw[kernels2] (0,0) node[not] {} -- (0,1.5) node[xi] {}; 
}

\DeclareSymbol{Xi4eabbisc1}{1.5}{\draw[kernels2] (-1,2.5) node[xic] {} -- (-1,1) ; \draw[kernels2] (-2,2)  node[xi] {} -- (-1,1) ; \draw[kernels2] (-1,1)  node[not] {} -- (0,0); 
	\draw (0,0)  -- (1,1) node[xic] {};
	\draw[kernels2] (0,0) node[not] {} -- (0,1.5) node[xi] {}; 
}

\DeclareSymbol{Xi4eabbisc1perm}{1.5}{\draw[kernels2] (-1,2.5) node[xi] {} -- (-1,1) ; \draw[kernels2] (-2,2)  node[xic] {} -- (-1,1) ; \draw[kernels2] (-1,1)  node[not] {} -- (0,0); 
	\draw (0,0)  -- (1,1) node[xic] {};
	\draw[kernels2] (0,0) node[not] {} -- (0,1.5) node[xi] {}; 
}

\DeclareSymbol{Xi4eabbisc2}{1.5}{\draw[kernels2] (-1,2.5) node[xi] {} -- (-1,1) ; \draw[kernels2] (-2,2)  node[xi] {} -- (-1,1) ; \draw[kernels2] (-1,1)  node[not] {} -- (0,0); 
	\draw (0,0)  -- (1,1) node[xic] {};
	\draw[kernels2] (0,0) node[not] {} -- (0,1.5) node[xic] {}; 
}

\DeclareSymbol{Xi2cbis}{0}{\draw[kernels2] (0,1) -- (0.8,2.2) node[xi] {};\draw[kernels2] (0,-0.25) node[not] {} -- (0,1); \draw[kernels2] (0,1) node[not] {} -- (-0.8,2.2) node[xi] {};}

\DeclareSymbol{Xi2cbis1}{0}{\draw (0,1) -- (-0.8,2.2) node[xi] {};\draw[kernels2] (0,-0.25) node[not] {} -- (0,1) node[xi] {}; }

\DeclareSymbol{Xi2Xbis}{-2}{\draw[kernels2] (0,-0.25)  -- (-1,1) ; \draw (-1,1) node[xix] {};
	\draw[kernels2] (0,-0.25) node[not] {} -- (1,1) node[xi] {};}

\DeclareSymbol{XXi2bis}{-2}{\draw[kernels2] (0,-0.25) -- (-1,1) node[xi] {};
	\draw[kernels2] (0,-0.25) node[X] {} -- (1,1) node[xi] {};}

\DeclareSymbol{I1XiIXi}{0}{\draw[kernels2] (0,-0.25) -- (1,1) node[xi] {};
	\draw (0,-0.25) node[not] {} -- (-1,1) node[xi] {};}

\DeclareSymbol{I1XiIXib}{0}{\draw  (0,-0.25) node[xi] {} -- (0,1) node[not] {};
	\draw[kernels2] (0,1) -- (0,2.25) ; \draw (0,2.25) node[xi]{}; }

\DeclareSymbol{I1XiIXic}{0}{
	\draw[kernels2] (0,0) -- (1,1) node[xi] {} ; 
	\draw[kernels2] (0,0) node[not] {}  -- (-1,1) node[not] {} -- (0,2) node[xi] {};
}

\DeclareSymbol{thin}{1.4}{\draw[pagebackground] (-0.3,0) -- (0.3,0); \draw  (0,0) -- (0,2);}
\DeclareSymbol{thin2}{1.4}{\draw[pagebackground] (-0.3,0) -- (0.3,0); \draw[tinydots]  (0,0) -- (0,2);}

\DeclareSymbol{thick}{1.4}{\draw[pagebackground] (-0.3,0) -- (0.3,0); \draw[kernels2]  (0,0) -- (0,2);}

\DeclareSymbol{thick2}{1.4}{\draw[pagebackground] (-0.3,0) -- (0.3,0); \draw[kernels2,tinydots]  (0,0) -- (0,2);}

\DeclareSymbol{Xi4ind}{2}{\draw (0,0) node[xi,label={[label distance=-0.2em]right: \scriptsize  $ i $}]  { } -- (-1,1) node[xi,label={[label distance=-0.2em]left: \scriptsize  $ j $}] {} -- (0,2) node[xi,label={[label distance=-0.2em]right: \scriptsize  $ k $}] {} -- (-1,3) node[xi,label={[label distance=-0.2em]left: \scriptsize  $ \ell $}] {};}

\DeclareSymbol{Xi4c1}{2}{\draw (0,0) node[xic] {} -- (-1,1) node[xi] {} -- (0,2) node[xic] {} -- (-1,3) node[xi] {};} 
\DeclareSymbol{IXi2ex}{0}{\draw (0,-0.25) node[xie] {} -- (-1,1) node[xi] {} ; \draw (0,-0.25)-- (1,1) node[xi] {};}
\DeclareSymbol{IXi2ex1}{0}{\draw (0,-0.25) node[xie] {} -- (-1,1) node[xi] {} -- (0,2) node[xi] {};}

\DeclareSymbol{Xi4b1}{0}{\draw(0,1.5) node[xic] {} -- (0,0); \draw (-1,1) node[xic] {} -- (0,0) node[xi] {} -- (1,1) node[xi] {};}

\DeclareSymbol{Xi4ec1}{0}{\draw (0,2) node[xi] {} -- (-1,1) node[xic] {} -- (0,0) node[xic] {} -- (1,1) node[xi] {};}
\DeclareSymbol{Xi4ec2}{0}{\draw (0,2) node[xic] {} -- (-1,1) node[xi] {} -- (0,0) node[xic] {} -- (1,1) node[xi] {};}
\DeclareSymbol{Xi4ec3}{0}{\draw (0,2) node[xic] {} -- (-1,1) node[xic] {} -- (0,0) node[xi] {} -- (1,1) node[xi] {};}

\DeclareSymbol{I1Xi4ac1}{2}{\draw[kernels2] (0,0) node[not] {} -- (-1,1) ; \draw[kernels2] (0,0) node[not] {} -- (1,1) node[xic] {} ;
	\draw (-1,1) node[xi] {} -- (0,2) node[xic] {} -- (-1,3) node[xi] {};}

\DeclareSymbol{I1Xi4ac2}{2}{\draw[kernels2] (0,0) node[not] {} -- (-1,1) ; \draw[kernels2] (0,0) node[not] {} -- (1,1) node[xic] {} ;
	\draw (-1,1) node[xi] {} -- (0,2) node[xi] {} -- (-1,3) node[xic] {};}

\DeclareSymbol{I1Xi4bp}{2}{\draw (0,0) node[not] {} -- (-1,1) node[xi] {} -- (0,2) ; \draw[kernels2] (0,2) node[not] {} -- (-1,3) node[xi] {};\draw[kernels2] (0,2)  -- (1,3) node[xi] {};
}

\DeclareSymbol{I1Xi4bc1}{2}{\draw (0,0) node[xic] {} -- (-1,1) node[xi] {} -- (0,2) ; \draw[kernels2] (0,2) node[not] {} -- (-1,3) node[xi] {};\draw[kernels2] (0,2)  -- (1,3) node[xic] {};
}

\DeclareSymbol{I1Xi4bc2}{2}{\draw (0,0) node[xic] {} -- (-1,1) node[xi] {} -- (0,2) ; \draw[kernels2] (0,2) node[not] {} -- (-1,3) node[xic] {};\draw[kernels2] (0,2)  -- (1,3) node[xi] {};
}

\DeclareSymbol{I1Xi4cp}{2}{\draw (0,0) node[not] {} -- (-1,1) node[not] {}; \draw[kernels2] (-1,1) -- (0,2) ; 
	\draw[kernels2] (-1,1) -- (-2,2) node[xi] {} ;
	\draw (0,2) node[xi] {} -- (-1,3) node[xi] {};}

\DeclareSymbol{I1Xi4cc1}{2}{\draw (0,0) node[xic] {} -- (-1,1) node[not] {}; \draw[kernels2] (-1,1) -- (0,2) ; 
	\draw[kernels2] (-1,1) -- (-2,2) node[xi] {} ;
	\draw (0,2) node[xic] {} -- (-1,3) node[xi] {};}

\DeclareSymbol{I1Xi4cc2}{2}{\draw (0,0) node[xic] {} -- (-1,1) node[not] {}; \draw[kernels2] (-1,1) -- (0,2) ; 
	\draw[kernels2] (-1,1) -- (-2,2) node[xi] {} ;
	\draw (0,2) node[xi] {} -- (-1,3) node[xic] {};}

\DeclareSymbol{I1Xi4abc1}{2}{\draw[kernels2] (0,0) node[not] {} -- (-1,1) ; \draw[kernels2] (0,0) node[not] {} -- (1,1) node[xic] {};\draw (-1,1) node[xi] {} -- (0,2) ; \draw[kernels2] (0,2) node[not] {} -- (-1,3) node[xic] {};\draw[kernels2] (0,2)  -- (1,3) node[xi] {}; }

\DeclareSymbol{I1Xi4abc2}{2}{\draw[kernels2] (0,0) node[not] {} -- (-1,1) ; \draw[kernels2] (0,0) node[not] {} -- (1,1) node[xic] {};\draw (-1,1) node[xi] {} -- (0,2) ; \draw[kernels2] (0,2) node[not] {} -- (-1,3) node[xi] {};\draw[kernels2] (0,2)  -- (1,3) node[xic] {}; }

\DeclareSymbol{R1}{0}{\draw (-1,1) node[xi] {} -- (0,0) node[not] {};
	\draw[kernels2] (0,1.5) node[xic] {} -- (0,0) -- (1,1) node[xic] {};}
\DeclareSymbol{R2}{0}{\draw (-1,1) node[xic] {} -- (0,0) node[not] {};
	\draw[kernels2] (0,1.5)  {} -- (0,0) -- (1,1)  {};
	\draw (0,1.5) node[xi] {};
	\draw (1,1) node[xic] {};
}
\DeclareSymbol{R3}{1}{\draw[kernels2] (-1,1.5)  {} -- (0,0) node[not] {} -- (1,1.5);
	\draw (-1,1.5) node[xi] {};
	\draw[kernels2] (0,3) {} -- (1,1.5) -- (2,3)  {};
	\draw  (0,3) node[xic] {} ;
	\draw (2,3) node[xic] {};}
\DeclareSymbol{R4}{1}{\draw[kernels2] (-1,1.5) node[xic] {} -- (0,0) node[not] {} -- (1,1.5);
	\draw[kernels2] (0,3) {} -- (1,1.5) -- (2,3) node[xic] {};
	\draw (0,3) node[xi] {};}

\DeclareSymbol{I1Xi4bcp}{2}{\draw (0,0) node[not] {} -- (-1,1) node[not] {}; \draw[kernels2] (-1,1) -- (0,2) ; 
	\draw[kernels2] (-1,1) -- (-2,2) node[xi] {} ; \draw[kernels2] (0,2) node[not] {} -- (-1,3) node[xi] {};\draw[kernels2] (0,2)  -- (1,3) node[xi] {};
}

\DeclareSymbol{I1Xi4bcc1}{2}{\draw (0,0) node[xic] {} -- (-1,1) node[not] {}; \draw[kernels2] (-1,1) -- (0,2) ; 
	\draw[kernels2] (-1,1) -- (-2,2) node[xi] {} ; \draw[kernels2] (0,2) node[not] {} -- (-1,3) node[xi] {};\draw[kernels2] (0,2)  -- (1,3) node[xic] {};
}

\DeclareSymbol{I1Xi4bcc2}{2}{\draw (0,0) node[xic] {} -- (-1,1) node[not] {}; \draw[kernels2] (-1,1) -- (0,2) ; 
	\draw[kernels2] (-1,1) -- (-2,2) node[xi] {} ; \draw[kernels2] (0,2) node[not] {} -- (-1,3) node[xic] {};\draw[kernels2] (0,2)  -- (1,3) node[xi] {};
} 

\DeclareSymbol{2I1Xi4bc1}{2}{\draw[kernels2] (0,0) node[not] {} -- (-1,1) ;
	\draw[kernels2] (0,0) -- (1,1);
	\draw (-1,1) node[xic] {} -- (-1,2.5) node[xi] {};
	\draw (1,1)  node[xic] {} -- (1,2.5) node[xi] {};
}

\DeclareSymbol{2I1Xi4bc2}{2}{\draw[kernels2] (0,0) node[not] {} -- (-1,1) ;
	\draw[kernels2] (0,0) -- (1,1);
	\draw (-1,1) node[xi] {} -- (-1,2.5) node[xic] {};
	\draw (1,1)  node[xic] {} -- (1,2.5) node[xi] {};
}

\DeclareSymbol{diff2I1Xi4bc2}{2}{\draw (0,0) node[diff] {} -- (-1,1) ;
	\draw (0,0) -- (1,1);
	\draw (-1,1) node[xi] {} -- (-1,2.5) node[xic] {};
	\draw (1,1)  node[xic] {} -- (1,2.5) node[xi] {};
}

\DeclareSymbol{2I1Xi4bc3}{2}{\draw[kernels2] (0,0) node[not] {} -- (-1,1) ;
	\draw[kernels2] (0,0) -- (1,1);
	\draw (-1,1) node[xic] {} -- (-1,2.5) node[xic] {};
	\draw (1,1)  node[xi] {} -- (1,2.5) node[xi] {};
}

\DeclareSymbol{Xi41}{0}{\draw (0,1) -- (0.8,2.2) node[xic] {};\draw (0,-0.25) node[xi] {} -- (0,1) node[xi] {} -- (-0.8,2.2) node[xic] {};} 

\DeclareSymbol{Xi42}{0}{\draw (0,1) -- (0.8,2.2) node[xi] {};\draw (0,-0.25) node[xic] {} -- (0,1) node[xi] {} -- (-0.8,2.2) node[xic] {};}

\DeclareSymbol{Xi4ca1}{0}{\draw (0,1) -- (-1,2.2) node[xic] {};\draw (0,-0.25) node[xi] {} -- (0,1) ; \draw[kernels2] (0,1) node[not] {} -- (1,2.2) node[xic] {};
	\draw[kernels2] (0,1) {} -- (0,2.7) node[xi] {};
}

\DeclareSymbol{Xi4ca2}{0}{\draw (0,1) -- (-1,2.2) node[xi] {};\draw (0,-0.25) node[xi] {} -- (0,1) ; \draw[kernels2] (0,1) node[not] {} -- (1,2.2) node[xic] {};
	\draw[kernels2] (0,1) {} -- (0,2.7) node[xic] {};
}

\DeclareSymbol{Xi4cap}{0}{\draw (0,1) -- (-1,2.2) node[xi] {};\draw (0,-0.25) node[not] {} -- (0,1) ; \draw[kernels2] (0,1) node[not] {} -- (1,2.2) node[xi] {};
	\draw[kernels2] (0,1) {} -- (0,2.7) node[xi] {};
}

\DeclareSymbol{Xi3a}{0}{
	\draw (-1,1)  node[xi] {} -- (0,0); 
	\draw (0,0) node[xi] {}  -- (1,1) node[xi] {};
}

\DeclareSymbol{Xi4ebc1}{0}{
	\draw[kernels2] (0,2) node[xi] {} -- (-1,1) ; \draw[kernels2] (-2,2)  node[xic] {} -- (-1,1) ; \draw (-1,1)  node[not] {} -- (0,0); 
	\draw (0,0) node[xic] {}  -- (1,1) node[xi] {};
}

\DeclareSymbol{Xi4ebc2}{0}{
	\draw[kernels2] (0,2) node[xi] {} -- (-1,1) ; \draw[kernels2] (-2,2)  node[xi] {} -- (-1,1) ; \draw (-1,1)  node[not] {} -- (0,0); 
	\draw (0,0) node[xic] {}  -- (1,1) node[xic] {};
}

\DeclareSymbol{Xi2cbispex}{0}{\draw[kernels2] (0,1) -- (0.8,2.2) node[xi] {};\draw (0,-0.25) node[xie] {} -- (0,1); \draw[kernels2] (0,1) node[not] {} -- (-0.8,2.2) node[xi] {};}

\DeclareSymbol{Xi2cbis1p}{0}{\draw (0,1) -- (-0.8,2.2) node[xi] {};\draw (0,-0.25) node[not] {} -- (0,1) node[xi] {}; }

\DeclareSymbol{Xi2Xp}{-2}{\draw (0,-0.25) node[not] {} -- (-1,1) node[xix] {};} 

\DeclareSymbol{I1XiIXib}{0}{\draw  (0,-0.25) node[xi] {} -- (0,1) node[not] {};
	\draw[kernels2] (0,1) -- (0,2.25) ; \draw (0,2.25) node[xi]{}; }

\DeclareSymbol{IXi2b}{0}{\draw  (0,-0.25) node[xi] {} -- (0,1) node[not] {};
	\draw (0,1) -- (0,2.25) ; \draw (0,2.25) node[xi]{}; }

\DeclareSymbol{IXi2bex}{0}{\draw  (0,-0.25) node[xi] {} -- (0,1) node[xie] {};
	\draw (0,1) -- (0,2.25) ; \draw (0,2.25) node[xi]{}; }

\def\1{\mathbf{\symbol{1}}}

\def\one{\mathbf{1}}

\DeclareSymbol{diff}{0}{
	\draw (0,0.5) node[diff] {};
}

\DeclareSymbol{diff1}{0}{
	\draw (0,0.5) node[diff1] {};
}

\DeclareSymbol{diff2}{0}{
	\draw (0,0.5) node[diff2] {};
}

\DeclareSymbol{geo}{0}{
	\draw (0,0) node[diff] {};
	\draw (0.3,0) node[diff] {};
}

\DeclareSymbol{generic}{0}{
	\draw (0,0.6) node[xi] {};
}

\DeclareSymbol{g}{0}{
	\draw (0,0.6) node[g] {};
}

\DeclareSymbol{Ito}{0}{
	\draw (0,0.6) node[xies] {};
}

\DeclareSymbol{Itob}{0}{
	\draw (0,0.6) node[xiesf] {};
}

\DeclareSymbol{greycirc}{0}{
	\draw (0,0.3) node[xi] {};
}

\DeclareSymbol{not}{0}{
	\draw (0,0.6) node[not] {};
	\draw[tinydots] (0,0.6) circle (0.8);
}

\DeclareSymbol{genericb}{0}{
	\draw (0,0.6) node[xic] {};
}

\DeclareSymbol{bluecirc}{0}{
	\draw (0,0.3) node[xic] {};
}

\DeclareSymbol{genericxix}{0}{
	\draw (0,0.6) node[xix] {};
}

\DeclareSymbol{genericX}{0}{
	\draw (0,0.6) node[X] {};
}

\DeclareSymbol{diffIto}{1}{
	\draw  (0,2.5) -- (0,0) ;
	\draw (0,-0.1) node[diff] {};
	\draw (0,2.5) node[xies] {};
}
\DeclareSymbol{Itodiff}{2}{
	\draw(0,2.9) -- (0,-0.2);
	\draw (0,2.9) node[diff] {};
	\draw (0,-0.1) node[xies] {};
}

\DeclareSymbol{diffgeneric}{1}{
	\draw  (0,2.5) -- (0,0) ;
	\draw (0,-0.1) node[diff] {};
	\draw (0,2.5) node[xi] {};
}

\DeclareSymbol{genericdiff}{2}{
	\draw(0,2.9) -- (0,-0.2);
	\draw (0,2.9) node[diff] {};
	\draw (0,-0.1) node[xi] {};
}

\DeclareSymbol{diffdot}{2}{
	\draw  (0,3) -- (0,-0.1) ;
	\draw (0,3) node[not] {};
	\draw (0,-0.1) node[diff] {};
}

\DeclareSymbol{diffdotmini}{0}{
	\draw  (0,0) -- (0,1.2) ;
	\draw (0,1.2) node[not] {};
	\draw (0,0) node[diffmini] {};
}

\DeclareSymbol{dotdiff}{2}{
	\draw[kernelsmod]  (0,3) -- (0,-0.1) ;
	\draw (0,3) node[diff] {};
	\draw (0,-0.1) node[not] {};
}

\DeclareSymbol{dotdiff1}{2}{
	\draw[kernelsmod]  (0,3) -- (0,-0.1) ;
	\draw (0,3) node[diff1] {};
	\draw (0,-0.1) node[not] {};
}

\DeclareSymbol{dotdiff1mini}{0}{
	\draw[kernelsmod]  (0,1.2) -- (0,0) ;
	\draw (0,1.2) node[diffmini] {};
	\draw (0,0) node[not] {};
}

\DeclareSymbol{dotdiff2}{2}{
	\draw (0,3) -- (0,-0.1) ;
	\draw (0,3) node[diff] {};
	\draw (0,-0.1) node[not] {};
}

\DeclareSymbol{dotdiff2mini}{0}{
	\draw (0,1.2) -- (0,0) ;
	\draw (0,1.2) node[diffmini] {};
	\draw (0,0) node[not] {};
}

\DeclareSymbol{dotdiffstraight}{0}{
	\draw  (0,3) -- (0,-0.1) ;
	\draw (0,3) node[diff] {};
	\draw (0,-0.1) node[not] {};
}

\DeclareSymbol{arbre1}{0}{
	\draw  (0,0) -- (1.5,1.5) ;
	\draw (1.5,1.5) node[not] {};
	\draw (0,0) node[not] {};
}

\DeclareSymbol{arbre2}{0}{
	\draw  (0,0) -- (1.5,1.5) ;
	\draw[kernelsmod] (0,0) -- (-1.5,1.5);
	\draw (1.5,1.5) node[not] {};
	\draw (0,0) node[not] {};
	\draw (-1.5,1.5) node[xi] {};
}

\DeclareSymbol{arbre3}{0}{
	\draw  (0,0) -- (1.5,1.5) ;
	\draw[kernelsmod] (1.5,1.5) -- (0,3);
	\draw (0,0) node[not] {};
	\draw (1.5,1.5) node[not] {};
	\draw (0,3) node[xi] {};
}

\DeclareSymbol{treeeval}{0}{
	\draw (0,0) -- (1,1);
	\draw (0,0) node[xi] {};
	\draw (1.25,1.25) node[xi] {};
	\draw (-0.6,0.6) node[]{\tiny{$i$}};
	\draw (0.65,1.85) node[]{\tiny{$j$}};
}

\DeclareSymbol{testeval}{0}{
	\draw (0,0) -- (1,1);
	\draw (0,0) -- (-1,1);
	\draw (0,0) node[xi] {};
	\draw (1.25,1.25) node[xi] {};
	\draw (-1.25,1.25) node[xi] {};
	\draw (-0.6,-0.6) node[]{\tiny{$i$}};
	\draw (0.65,1.85) node[]{\tiny{$j$}};
	\draw (-1.95,1.85) node[]{\tiny{$k$}};
}

\DeclareSymbol{treeeval2}{0}{
	\draw[kernelsmod] (-0.25,-1) -- (1,0.5) ;
	\draw[kernelsmod] (1,0.5) -- (-0.25,2);
	\draw (1,0.5) node[diff2] {};
	\draw (-0.25,-1) node[not] {};
	\draw (-0.25,2) node[xi] {};
	\draw (-0.6,1.2) node[]{\tiny{1}};
}

\DeclareSymbol{arbreact}{1}{
	\draw (0,0) node[not] {};
	\draw[kernelsmod] (0,0) -- (1,1);
	\draw[kernelsmod] (0,0) -- (-1,1);
	\draw (-1,1) node[xic] {};
	\draw  (0,2) -- (1,1) ;
	\draw (0,2) node[xic] {};
	\draw (1,1) node[xi] {};
}

\DeclareSymbol{arbreact1}{0}{
	\draw (0,-1.5) -- (0,0);
	\draw[kernelsmod] (0,0) -- (1,1);
	\draw[kernelsmod] (0,0) -- (-1,1);
	\draw  (0,2) -- (1,1) ;
	\draw (0,-1.5) node[diff] {};
	\draw (0,0) node[not] {};
	\draw (-1,1) node[xic] {};
	\draw (0,2) node[xic] {};
	\draw (1,1) node[xi] {};
}

\DeclareSymbol{arbreact2}{0}{
	\draw (0,-0.75) -- (-1,0.5); 
	\draw (0,-0.75) -- (1,0.5);
	\draw (0,1.5) -- (1,0.5);
	\draw (0,1.5) node[xic] {};
	\draw (1,0.5) node[xi] {};
	\draw (-1,0.5) node[xic] {};
	\draw (0,-0.75) node[diff] {};
}

\DeclareSymbol{arbreact3}{0}{
	\draw[kernelsmod] (0,-0.75) -- (-1,0.5); 
	\draw[kernelsmod] (0,-0.75) -- (1,0.5);
	\draw (0,1.75) -- (1,0.5);
	\draw (2,1.75) -- (1,0.5);
	\draw (0,1.75) node[xic] {};
	\draw (1,0.5) node[diff] {};
	\draw (-1,0.5) node[xic] {};
	\draw (2,1.75) node[xi] {};
	\draw (0,-0.75) node[not] {};
}

\DeclareSymbol{pre_im_I1Xitwo}{0}{
	\draw[kernels2] (0,-0.3) node[not] {} -- (-0.6,0.7) ;
	\draw[kernels2] (0,-0.3) -- (0.6,0.7);
	\draw (0,0.9) node[g] {};
}

\DeclareSymbol{pre_im_cI1Xi4ab}{2}{
	\draw[kernels2] (0,-1) node[not] {} -- (-0.6,0) ;
	\draw[kernels2] (0,-1) -- (0.6,0);
	\draw (0,0.2) node[g] {};
	\draw (0,0.6) -- (0,1.5);
	\draw[kernels2] (0,1.5) node[not] {} -- (-0.6,2.5) ;
	\draw[kernels2] (0,1.5) -- (0.6,2.5);
	\draw (0,2.7) node[g] {};
}

\DeclareSymbol{pre_im_I1Xi4acc2}{0}{
	\draw[kernels2] (-1,-0.5) node[not] {} -- (-1.6,0.5) ;
	\draw[kernels2] (-1,-0.5) -- (-0.4,0.5);
	\draw[kernels2] (-1,-0.5) -- (0.2,-1.5) node[not] {} ;
	\draw (-1,1.1) -- (-1,2);
	\draw[kernels2] (0.2,-1.5) -- (0.2,2);
	\draw (-1,0.7) node[g] {};
	\draw (-0.3,2.2) node[g] {};
}

\DeclareSymbol{pre_im_I1Xi4abcc2}{2}{
	\draw[kernels2] (0,-1) node[not] {} -- (-1,0) node[not] {};
	\draw[kernels2] (-1,1.2) node[not] {} -- (-1,0);
	\draw[kernels2] (-1,1.2) -- (-1.5,2.5);
	\draw[kernels2] (-1,1.2) -- (-0.5,2.5);
	\draw[kernels2] (-1,0) -- (0.7,2.5);
	\draw[kernels2] (0,-1) -- (1.5,2.5);
	\draw (-1,2.7) node[g] {};
	\draw (1,2.7) node[g] {};
}

\DeclareSymbol{pre_im_2I1Xi4c1}{2}{
	\draw[kernels2] (0,-0.5) node[not] {} -- (-1,0.5) node[not] {};
	\draw[kernels2] (0,-0.5) -- (1,0.5) node[not] {};
	\draw[kernels2] (-1,0.5) node[not] {}-- (-1.7,2);
	\draw[kernels2]  (-1,2) -- (1,0.5);
	\draw[kernels2] (-1,0.5) -- (1,2);
	\draw[kernels2] (1,0.5) -- (1.7,2);
	\draw (-1.2,2.2) node[g] {};
	\draw (1.2,2.2) node[g] {};
}

\DeclareSymbol{pre_im_Xi4eabisc2}{2}{
	\draw[kernels2] (1.2,-0.5) node[not] {} -- (-0.7,0.8) ;
	\draw[kernels2] (1.2,-0.5) -- (0.4,0.8);
	\draw (0,1.4)  -- (0,2.2);
	\draw (1.2,2.2) -- (1.2,-0.6);
	\draw (0,1) node[g] {};
	\draw (0.6,2.4) node[g] {};
}

\DeclareSymbol{pre_im_Xi4eabisc22}{2}{
	\draw (1.2,-0.5) node[not] {} -- (-0.7,0.8) ;
	\draw[kernels2] (1.2,-0.5) -- (0.4,0.8);
	\draw (0,1.4)  -- (0,2.2);
	\draw[kernels2] (1.2,2.2) -- (1.2,-0.6);
	\draw (0,1) node[g] {};
	\draw (0.6,2.4) node[g] {};
}

\DeclareSymbol{pre_im_Xi4eabisc222}{2}{
	\draw[kernels2] (0.4,-0.5) node[not] {} -- (-0.6,1) ;
	\draw[kernels2] (1.2,0) -- (0.3,1);
	\draw (0,1.1)  -- (0,2.5);
	\draw[kernels2] (1.2,2.5) -- (1.2,0) node[not] {} -- (0.4,-0.6);
	\draw (0,1.2) node[g] {};
	\draw (0.6,2.5) node[g] {};
}

\DeclareSymbol{pre_im_Xi4eabc2}{2}{
	\draw (0,-0.5) node[not] {} -- (-1,0.5) node[not] {};
	\draw[kernels2] (-1,0.5) -- (-1.5,2);
	\draw[kernels2] (0,-0.5)  -- (0.7,2);
	\draw[kernels2] (-1,0.5) -- (-0.5,2);
	\draw[kernels2] (0,-0.5) -- (1.5,2);
	\draw (-1,2.2) node[g] {};
	\draw (1,2.2) node[g] {};
}

\DeclareSymbol{pre_im_Xi4eabbisc2}{2}{
	\draw[kernels2] (0,-0.5) node[not] {} -- (-1,0.5) node[not] {};
	\draw[kernels2] (-1,0.5) -- (-1.5,2);
	\draw[kernels2] (0,-0.5)  -- (0.7,2);
	\draw[kernels2] (-1,0.5) -- (-0.5,2);
	\draw (0,-0.5) -- (1.5,2);
	\draw (-1,2.2) node[g] {};
	\draw (1,2.2) node[g] {};
}

\DeclareSymbol{pre_im_I1Xi4abcc1}{2}{
	\draw[kernels2] (0,-1) node[not] {} -- (-1,0) node[not] {};
	\draw[kernels2] (0,1.1) node[not] {} -- (-1,0);
	\draw[kernels2] (-1,0) -- (-1.5,2.5);
	\draw[kernels2] (0,1.1) node[not] {} -- (-0.5,2.5);
	\draw[kernels2] (0,1.1) -- (0.5,2.5);
	\draw[kernels2] (0,-1) -- (1.5,2.5);
	\draw (-1,2.7) node[g] {};
	\draw (1,2.7) node[g] {};
}

\DeclareSymbol{pre_im_Xi4eabc1}{2}{
	\draw (0,-0.5) node[not] {} -- (-1,0.5) node[not] {};
	\draw[kernels2] (-1,0.5) -- (-1.5,2);
	\draw[kernels2] (0,-0.5)  -- (-0.5,2);
	\draw[kernels2] (-1,0.5) -- (0.8,2);
	\draw[kernels2] (0,-0.5) -- (1.5,2);
	\draw (-1,2.2) node[g] {};
	\draw (1,2.2) node[g] {};
}

\DeclareSymbol{pre_im_Xi4ba1b}{2}{
	\draw[kernels2] (0,0) node[not] {}  -- (1.8,1.5);
	\draw[kernels2] (0,0) -- (0.8,1.5);
	\draw (0,-0.1) -- (-1.8,1.5);
	\draw (0,-0.1) -- (-0.8,1.5);
	\draw (-1,1.7) node[g] {};
	\draw (1,1.7) node[g] {};
}

\DeclareSymbol{pre_im_Xi4ba2}{2}{
	\draw (0,-0.1) node[not] {}  -- (1.8,1.5);
	\draw[kernels2] (0,0) -- (0.8,1.5);
	\draw (0,-0.1) -- (-1.8,1.5);
	\draw[kernels2] (0,0) -- (-0.8,1.5);
	\draw (-1,1.7) node[g] {};
	\draw (1,1.7) node[g] {};
}

\DeclareSymbol{pre_im_Xi4cabc2}{2}{
	\draw[kernels2] (0,-0.5) node[not] {} -- (-1,0.5) node[not] {};
	\draw[kernels2] (-1,0.5) -- (-1.5,2);
	\draw (-1,0.5)  -- (0.7,2);
	\draw[kernels2] (-1,0.5) -- (-0.5,2);
	\draw[kernels2] (0,-0.5) -- (1.5,2);
	\draw (-1,2.2) node[g] {};
	\draw (1,2.2) node[g] {};
}

\DeclareSymbol{pre_im_Xi4cabc1}{2}{
	\draw[kernels2] (0,-0.5) node[not] {} -- (-1,0.5) node[not] {};
	\draw (-1,0.5) -- (-1.5,2);
	\draw[kernels2] (-1,0.5)  -- (0.7,2);
	\draw[kernels2] (-1,0.5) -- (-0.5,2);
	\draw[kernels2] (0,-0.5) -- (1.5,2);
	\draw (-1,2.2) node[g] {};
	\draw (1,2.2) node[g] {};
}

\DeclareSymbol{pre_im_Xi4eabbisc1}{2}{
	\draw[kernels2] (0,-0.5) node[not] {} -- (-1,0.5) node[not] {};
	\draw[kernels2] (-1,0.5) -- (-1.5,2);
	\draw[kernels2] (0,-0.5)  -- (-0.5,2);
	\draw[kernels2] (-1,0.5) -- (0.8,2);
	\draw (0,-0.5) -- (1.5,2);
	\draw (-1,2.2) node[g] {};
	\draw (1,2.2) node[g] {};
}

\DeclareSymbol{pre_im_1}{0}{
	\draw[kernels2] (0,-0.5) node[not] {} -- (-0.6,0.5) ;
	\draw[kernels2] (0,-0.5) -- (0.6,0.5);
	\draw (0,1.1)  -- (-0.55,2);
	\draw (0,1.1)  -- (0.55,2);
	\draw (0,0.7) node[g] {};
	\draw (0,2.2) node[g] {};
}

\DeclareSymbol{disconnect}{0}{
	\draw[kernels2] (0,-0.5) node[not] {} -- (-0.6,0.5) ;
	\draw[kernels2] (0,-0.5) -- (0.6,0.5);
	\draw (-0.55,1.1)  -- (-0.55,2.3);
	\draw (0.55,2.3) -- (0.55,1.5) -- (1.2,1.5) -- (1.2,3.5) -- (0.55,3.5) -- (0.55,2.7);
	\draw (0,0.7) node[g] {};
	\draw (0,2.5) node[g] {};
}

\DeclareSymbol{pre_im_2}{2}{\draw[kernels2] (0,0) node[not] {} -- (-1,1) node[not] {};
	\draw[kernels2] (0,0) -- (1,1) node[not] {};
	\draw[kernels2] (-1,1) -- (-1.5,2.5);
	\draw[kernels2] (-1,1) -- (-0.5,2.5);
	\draw[kernels2] (1,1) -- (0.5,2.5);
	\draw[kernels2] (1,1) -- (1.5,2.5);
	\draw (-1,2.7) node[g] {};
	\draw (1,2.7) node[g] {};
}

\DeclareSymbol{CX_rec}{0}{
	\draw [black] (-0.3,1) to (-0.3,-0.3);
	\draw [black] (0.3,1) to (0.3,-0.3);
	\draw [black] (-0.3,1) to (-0.3,2.3);
	\draw [black] (0.3,1) to (0.3,2.3);
	\draw (0,1) node[rec] {};
}

\DeclareSymbol{CX_cerc}{0}{
	\draw [black] (0,1) to (0,-0.3);
	\draw (0,1) node[cerc] {};
}

\DeclareMathAlphabet{\mathpzc}{OT1}{pzc}{m}{it}

\def\eqref#1{(\ref{#1})}

\makeatletter 
\newcommand*{\bigcdot}{}
\DeclareRobustCommand*{\bigcdot}{%
	\mathbin{\mathpalette\bigcdot@{}}%
}
\newcommand*{\bigcdot@scalefactor}{.5}
\newcommand*{\bigcdot@widthfactor}{1.15}
\newcommand*{\bigcdot@}[2]{%
	\sbox0{$#1\vcenter{}$}
	\sbox2{$#1\cdot\m@th$}%
	\hbox to \bigcdot@widthfactor\wd2{%
		\hfil
		\raise\ht0\hbox{%
			\scalebox{\bigcdot@scalefactor}{%
				\lower\ht0\hbox{$#1\bullet\m@th$}%
			}%
		}%
		\hfil
	}%
}
\makeatother

\tcbset
{colframe=boxcolor,colback=symbols!7!pagebackground,coltext=pageforeground,
	fonttitle=\bfseries,nobeforeafter,center title,size=fbox,boxsep=1.5pt,
	top=0mm,bottom=0mm,boxsep=0mm,tcbox raise base}

\def\two{{\<generic>\kern0.05em\<genericb>}}
\def\twoI{{\<Ito>\kern0.05em\<Itob>}}

\def\mail#1{\burlalt{#1}{mailto:#1}}

\usepackage{thmtools}

\newcommand\georginline[1]{#1}

\begin{document}

	\title{Symmetric resonance based integrators and  \\ forest formulae}
	

	\author{Yvonne Alama Bronsard$^1$, Yvain Bruned$^2$, Georg Maierhofer$^3$, Katharina Schratz$^1$}
	\institute{IRMAR (UMR 6625), Universit\'e de Rennes, \and IECL (UMR 7502), Université de Lorraine\and DAMTP, University of Cambridge \and LJLL (UMR 7598), Sorbonne University\\
		Email:\ \begin{minipage}[t]{\linewidth}
			\mail{yvonne.alamabronsard@univ-rennes.fr},\\
			\mail{yvain.bruned@univ-lorraine.fr}, \\
			\mail{gam37@cam.ac.uk}, \\ \mail{katharina.schratz@sorbonne-universite.fr}.
	\end{minipage}} 
	
	\maketitle

	\begin{abstract}
\georginline{In the present work we introduce a unified framework that allows for the very first systematic construction of symmetric resonance-based integrators to approximate a wide class of  nonlinear dispersive equations at low-regularity. The inclusion of symmetries in the construction of resonance-based schemes presents serious challenges and induces a need for a significant extension of prior approaches to allow for sufficient number of degrees of freedom in the resulting schemes while preserving the favorable low-regularity convergence properties of prior constructions. Motivated by recent work \cite{BS}, we achieve this by introducing a novel formalism based on forest formulae that allows us to encode a wider range of possibilities of iterating Duhamel's formula and interpolatory approximations of lower order parts in the construction of these time-stepping methods. The forest formulae allow for a simple characterisation of symmetric schemes and provides a fascinating algebraic structure in its own right which echo those used in Quantum Field Theory for renormalising Feynman diagrams and those used for the renormalisation of singular SPDEs via the theory of Regularity Structures. Our constructions lead to the development of several new symmetric low-regularity integrators that exhibit remarkable structure preservation and convergence properties which are witnessed in numerical experiments.
\\[.4em]
\noindent {\scriptsize\textit{Communicated by Hans Munthe-Kaas.}}}
	\end{abstract}
	\setcounter{tocdepth}{2}
	\setcounter{secnumdepth}{4}
	\tableofcontents
	
	\section{Introduction}
		
We consider a general class of dispersive differential equations of the form 
\begin{equation}\label{dis}
\begin{aligned}
& i \partial_t u(t,x) +   \mathcal{L}\left(\nabla\right) u(t,x) =\vert \nabla\vert^\alpha p\left(u(t,x), \overline u(t,x)\right), \\
& u(0,x) = v(x),
\end{aligned}
\end{equation}
equipped with periodic boundary conditions $x\in \mathbb{T}^d$. \georginline{Throughout}, we assume that $p$ is a polynomial nonlinearity, and that the structure of \eqref{dis} implies at least local well-posedness of the problem on a finite time interval $]0,T]$, $T<\infty$, in an appropriate functional space. \georginline{This class of equations captures a number of physically important models, including the Korteweg de Vries (KdV) equation
\begin{equation}\label{kdvIntro}
	\partial_t u -i\mathcal{L}\left(\nabla\right) u = \frac12 \partial_x u^2, \quad \mathcal{L}\left(\nabla\right) = i\partial_x^3, \quad \vert \nabla\vert^\alpha = \partial_x, \quad x \in \mathbb{T},
\end{equation}
and the nonlinear Schr\"odinger (NLS) equation,
\begin{equation}\label{nlsIntro}
	i \partial_t u + \mathcal{L}\left(\nabla\right)  u = \vert u\vert^2 u, \quad \mathcal{L}\left(\nabla\right) = \Delta, \quad x\in \mathbb{T}^d.
\end{equation}
Like those two examples many physical equations in this class possess conservation laws, or are indeed integrable systems (for example the KdV equation is a completely integrable parity-time invariant system). It is known that for nonstiff Hamiltonian ODEs, symmetric numerical schemes have favourable long-time behaviour when applied to such reversible integrable systems, such as linear (slow) growth in error as a function of the integration time, and near conservation of first integrals over long times \cite{H04,HLW,BG94}. At the same time the numerical approximation of the Cauchy problem in low-regularity regimes requires the design of designated methods, amongst which resonance-based schemes have seen significant success over recent years. Firstly developed for specific equations, including the KdV equation \cite{HS17,WZhao22,LW22}, the NLS equation \cite{OS18,CLLsecondNLS,AB23,ORS19, ORS20,WuYao22,Neumann}, the Gross--Pitaevskii equation \cite{AB22} and the Navier--Stokes equations \cite{NS}, more recent work has started to establish a more general framework for resonance based low-regularity integrators \cite{BBS,ABS22, RSKDV}. The key  idea of  these schemes lies in embedding the underlying structure of resonances - triggered by the nonlinear frequency interactions between the leading differential operator $\mathcal{L}\left(\nabla\right)$ and the nonlinearity $p\left(u(t,x), \overline u(t,x)\right)$ - into the numerical discretisation. These nonlinear interactions are in general neglected by classical approximation techniques such as Runge-Kutta methods, splitting methods or exponential integrators. While for smooth solutions these nonlinear interactions are indeed negligible, they do play a central role at low regularity and high oscillations. The accurate resolution of these interactions has been achieved in broad generality only in the recent few years in \cite{BS,FS, BBS}. Yet, while the design of such schemes has seen a wide range of developments, prior work has focussed mostly on explicit schemes with desired convergence properties. A few recent results \cite{BMS22, AB23, FMS23, MS22} have introduced first or second order implicit symmetric integrators at low-regularity fitting the particular structure of the equation with better conservation properties. Nevertheless, a central question remained unanswered: \emph{Can we systematically construct structure preserving resonance based schemes up to arbitrary order which preserve central symmetries of the underlying continuous equation?} }

For ordinary differential equations (ODEs) the theory of \emph{structure preservation} in numerical schemes is thoroughly established \cite{HLW}, specifically there is an extensive amount of literature on the characterisation of symmetric and symplectic Runge--Kutta methods \cite{K03,S88,IZ00} and, more broadly, B-series methods \cite{BC94,CM07}; on the favourable long-time behaviour of such methods when applied for finite-dimensional integrable reversible systems and Hamiltonian systems \cite{H04,HLW,BG94} respectively; and even on the limitations on types of structure that can be preserved with B-series methods \cite{IQT07}. { The rigorous understanding of the long-time behaviour of such methods in the case of highly oscillatory ODEs \cite{HL13} and of PDEs \cite{Faou12, CHL08, GL10, FGP10, BK24} is much less straightforward, and remains an ongoing challenge. Nonetheless, numerical observations of favourable structure preservation properties have motived the study of symmetric methods for PDEs, for example in the classification of symmetric splitting methods \cite{McLacQ02} and symmetric exponential integrators \cite{CCO08}. The work presented here follows a similar direction: we build novel symmetric schemes for \eqref{dis} exhibiting very good long-time conservation properties as observed in the numerical experiments of Section \ref{sec:numerical_experiments}. This encourages also future studies of their rigorous understanding, which is at the moment a completely open problem.}

In general, resonance based schemes are \emph{not} structure preserving and do not preserve the symmetries in the system. We can consider for example the second order resonance based scheme introduced by \cite[Section~5.1.2]{BS}, referred to henceforth as `Bruned \& Schratz 2022', and given by
\begin{align}\begin{split}\label{eqn:bruned_schratz_22}
	u^{n+1}&=e^{i\tau\Delta}u^{n}-i\tau e^{i\tau\Delta}\left(\left(u^n\right)^2\left(\varphi_1(-2i\tau\Delta)-\varphi_2(-2i\tau\Delta)\right)\overline{u^n}\right)\\
	&\quad\quad\quad-i\tau\left(e^{i\tau\Delta}u^n\right)^2\varphi_{2}(-2i\tau\Delta)(e^{i\tau\Delta}\overline{u^n})-\frac{\tau^2}{2}e^{i\tau\Delta}\left(\left|u^n\right|^4u^n\right),
	\end{split}
\end{align}
where $\tau$  is the step size of the method, 
$	\varphi_1(\sigma) = \frac{e^{\sigma}-1}{\sigma}$ and $ \varphi_2(\sigma) = \frac{e^{\sigma}-\varphi_1(\sigma)}{\sigma}$.
Symmetry of a numerical scheme is defined by considering its adjoint method: For a given method $u^n\mapsto u^{n+1}=\Phi_{\tau}(u^n)$ its adjoint method is defined as $\widehat{\Phi}_{\tau}:=\Phi^{-1}_{-\tau}$.
\begin{defi}[{See for example Definition~V.1.4 in \cite{HLW}}]\label{def:time-symmetric_method} The method $\Phi_{\tau}$ is called \textit{symmetric} if $\Phi_{\tau}=\widehat{\Phi}_{\tau}$.
\end{defi}
The scheme \eqref{eqn:bruned_schratz_22} is not symmetric in the sense of definition~\ref{def:time-symmetric_method} because the adjoint method is given by
\begin{align*}
u^{n+1}& = e^{i\tau\Delta} u^n - i\tau \left(\left(u^{n+1}\right)^2\left(\varphi_1(2i\tau\Delta)-\varphi_2(2i\tau\Delta)\right)\overline{u^{n+1}}\right) \\
&\qquad\quad -i\tau e^{i\tau\Delta}\left(\left(e^{-i\tau\Delta}u^{n+1}\right)^2\varphi_{2}(2i\tau\Delta)(e^{-i\tau\Delta}\overline{u^{n+1}})\right)\\
&\qquad\quad+\frac{\tau^2}{2}\left(\left |u^{n+1}\right|^4u^{n+1}\right),
\end{align*}
which is implicit as opposed to the original scheme \eqref{eqn:bruned_schratz_22}, \georginline{which is explicit}.

The derivation of new schemes which are  structure preserving and at the same time allow for low-regularity approximations of the solution $u$ were first addressed in the specific case of the KdV  \cite{MS22}, Klein-Gordon (KG) \cite{WZ22}, NLS  \cite{AB23}, and the isotropic Landau--Lifschitz equation \cite{BMS22}. In the case of the NLS equations, a further symmetric low-regularity integrator with good long time behaviour was introduced in \cite{FMS23}; see also the very recent work \cite{FMW24} on higher-order explicit and symmetric integrators.
Let us also highlight the work \cite{WuYao22} which was the first low regularity method which allowed for high order mass conservation (for fixed time). However, all these results are yet again tailored to the particular structure of the equation, and their higher-order extensions are not optimal in the sense of regularity, see Remark \ref{rem:evenOrder}.

This motivates the study of systematic constructions of symmetric resonance based schemes that we address in the present work. { A natural starting point is to consider symmetrizing existing asymmetric resonance based schemes. Typically, given $\Phi_{\tau}$ an asymmetric resonance based scheme of order $p$, we expect the symmetrisation with its adjoint 
$$\widehat{\Phi}_{\tau}\circ \Phi_{\tau},$$
 to be a $p$-th order method with the same regularity requirements as for $\Phi_{\tau}$, and when $p$ is odd, one can show through appropriate additional regularity assumptions that the scheme is of order $p+1$, see \cite{AB23} and Remark \ref{rem:evenOrder}. This yields a {\it subclass} of high-order symmetric resonance-based schemes, built upon the asymmetric resonance-based schemes of \cite{BS}. Nonetheless, myriads of other possible symmetric approximations exist. The goal of this work is to offer a general characterisation for symmetric schemes, while putting up front a new subclass of high-order symmetric schemes based upon mid-point approximations.}

We extend the resonance based decorated trees approach \cite{BS} to a richer framework by exploring different ways of iterating  Duhamel's formula, capturing the dominant parts while interpolating the lower parts of the resonances in a symmetric manner. This gives a range of new numerical schemes with more degrees of freedom than the original framework from \cite{BS}. Our new framework allows us to recover previously constructed low-regularity symmetric schemes, such as the scheme \eqref{schemeNLS_first scheme} first introduced in \cite{AB23}, as well as new symmetric low regularity schemes which includes the previously mentioned mid-point based approximations, see  \eqref{schemeNLS2} and \eqref{schemeNLS3}. 

In addition, as opposed to the previous works \cite{BS,ABS22, BBS} the schemes we introduce here do not need to be accompanied by well-chosen filter functions in order to ensure stability of the scheme. Indeed, our construction based on interpolation rather than Taylor series expansion of the non-dominant parts of the nonlinear frequency interactions directly leads to stable schemes, see also \cite{FS, AB22}.

Our main result is the new general resonance based scheme presented in Definition \ref{main_scheme},  with its local error terms given in Theorem \ref{thm:genloc}\georginline{, the latter of} which is a consequence of \cite{BS}. We show that this scheme is symmetric in Theorem \ref{symmetric_schemme_low_regularity} and that it is contained within a forest formula
in Theorem \ref{forest_formula_symmetric scheme}. Our general framework is illustrated on concrete examples in Section \ref{sec:examples} and simulations show the better structure preserving properties as well as the convergence properties of the scheme. \georginline{This was only possible through a significant extension of} the algebraic structures proposed in \cite{BS} by \georginline{introducing} new forest formulae in Theorem \ref{forest_formula_2} and Theorem \ref{forest_formula_1_2}. These formulae are used for finding new symmetric schemes and have their own interest by providing a new parametrisation of low regularity schemes allowing for implicitness in the schemes and thus resembling more closely the formulation of classical schemes such as Runge--Kutta methods or exponential integrators. We derive conditions on the coefficients of these formulae for having a symmetric scheme, see Proposition \ref{prop:symmetry_conditions}.

\begin{remark}
Up to now we were faced with a choice between structure preservation and low-regularity approximation properties. This is exhibited in Figure~\ref{fig:intro_N_128_two_regularities} where we study the cubic NLS equation and compare the preservation of energy of the Strang splitting (a symmetric splitting method) against previous resonance based integrators (Bruned \& Schratz 2022 \cite{BS}) for smooth $C^\infty$ and $H^2$ data.
\begin{figure}[h!]
	\centering
	\begin{subfigure}{1\linewidth}
		\centering
		\includegraphics[width=0.73\linewidth]{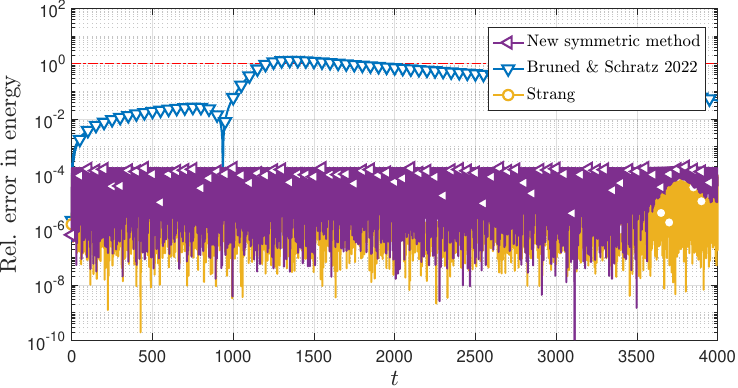}
		\caption{Smooth data, $u_0\in C^{\infty}$ and $M=64$.}
		\label{fig:intro_N_128_two_regularities_smooth}
	\end{subfigure}\\
	\begin{subfigure}{1\linewidth}
		\centering
		\includegraphics[width=0.73\linewidth]{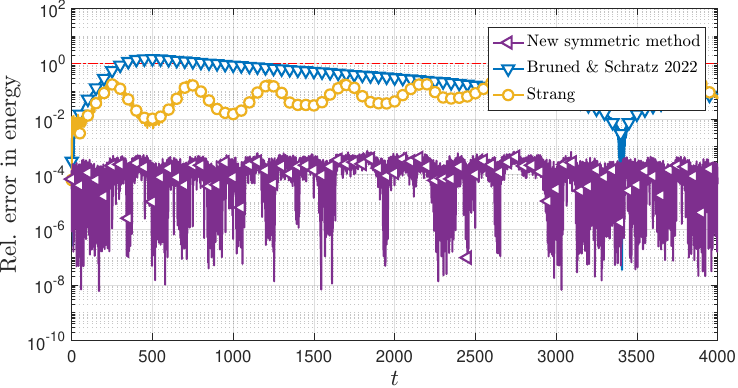}
		\caption{Low regularity data, $u_0\in H^{2}$ and $M=64$.}
		\label{fig:intro_N_128_two_regularities_H2}
	\end{subfigure}
\caption{Long-time relative error in the energy of the NLS equation (time-step $\tau=0.02$).}
\label{fig:intro_N_128_two_regularities}\vspace{-0.7cm}
\end{figure}

The Strang splitting almost preserves the energy over long times for smooth solutions, but suffers from numerical energy blow up  for rougher data. The resonance based integrator Bruned \& Schratz 2022 \cite{BS} on the other hand only \georginline{achieves approximate} energy preservation up to short  times (for both smooth and rougher data). Our novel resonance based midpoint method \eqref{schemeNLS2} bridges this gap allowing for numerical long-time approximate energy conservation even at low regularity, see Figure~\ref{fig:intro_N_128_two_regularities_H2}. 

Let us now take a closer look at smooth solutions, where we find a surprising additional characteristic of our new scheme \eqref{schemeNLS2}. Note that long-time structure preservation properties apply only subject to a CFL condition for Strang splitting methods applied to the NLS equation. More precisely the time step size $\tau$ has to be chosen such that $\tau\lesssim M^{-2}$ where $M$ is the number of degrees of freedom in the spatial discretisation, see for instance \cite{Faou12} and references therein for a detailed discussion. This step size restriction is not only a  theoretical  technicality, but also observed in numerical experiments. The long-time energy preservation in the Strang splitting drastically breaks down if we start to increase the number of Fourier modes $M$, i.e., move from ``ODE to PDE'', see Figure~\ref{fig:intro_N_128_two_regularities_smooth} versus Figure~\ref{fig:intro_N_256_smooth solutions}, where we double the Fourier modes in our discretisation. A very interesting feature of our new resonance-based constructions appears to be that in numerical experiments the long-time behaviour of the method seemingly does not depend on the number $M$ of spatial modes used. 
\begin{figure}[h!]
	\centering
\begin{subfigure}{1\linewidth}
	\centering
	\includegraphics[width=0.73\linewidth]{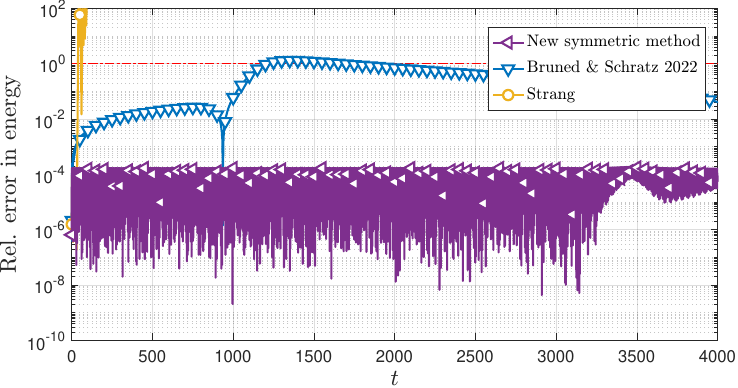}
	\caption{Long time interval $t\in[0,4000]$, smooth data $u_0\in C^\infty$ and $M=256$.}
\end{subfigure}\\
\begin{subfigure}{1\linewidth}
	\centering
	\includegraphics[width=0.73\linewidth]{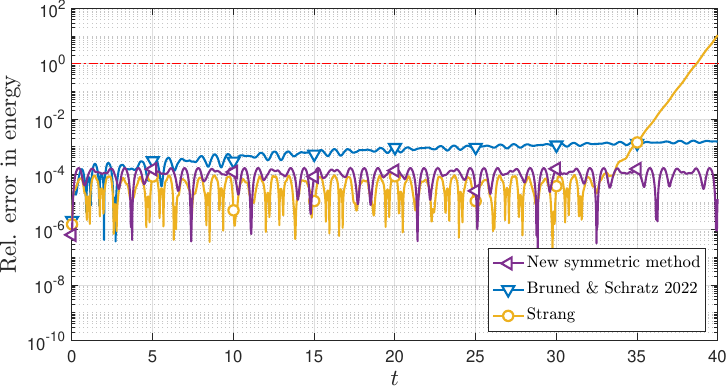}
	\caption{Zoom on time interval $t \in [0,40]$, smooth data $u_0\in C^\infty$ and $M=256$.}
\end{subfigure}
\caption{Long-time relative error in the energy of the NLS equation (time-step $\tau=0.02$).}\vspace{-0.5cm}
\label{fig:intro_N_256_smooth solutions}
\end{figure}

\begin{figure}[h!]
	\centering
	\begin{subfigure}{1\linewidth}
		\centering
		\includegraphics[width=0.73\linewidth]{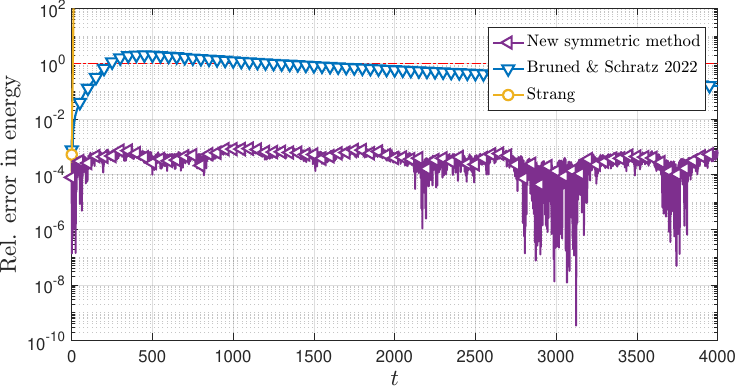}
		\caption{Long time interval $t\in[0,4000]$, low-regularity data $u_0\in H^2$ and $M=256$.}
	\end{subfigure}\\
	\begin{subfigure}{1\linewidth}
		\centering
		\includegraphics[width=0.73\linewidth]{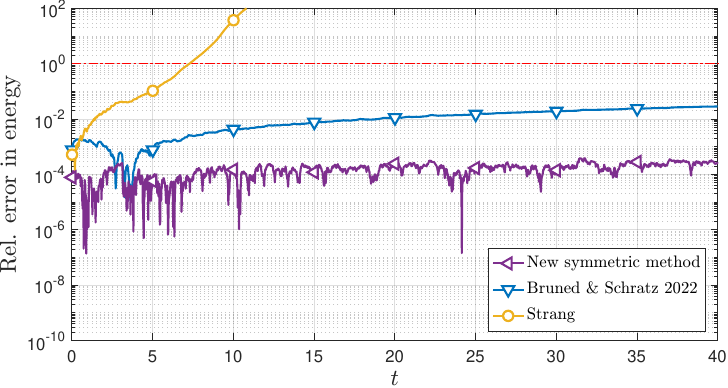}
		\caption{Zoom on time interval $t \in [0,40]$, low-regularity data $u_0\in H^2$ and $M=256$.}
	\end{subfigure}
	\caption{Long-time relative error in the energy of the NLS equation (time-step $\tau=0.02$).}
	\label{fig:intro_N_256_low-reg_solutions}\vspace{-0.5cm}
\end{figure}
In summary, the long-time dynamics shown in Figures~\ref{fig:intro_N_128_two_regularities}-\ref{fig:intro_N_256_low-reg_solutions} is representative of the behaviour of these methods and is observed in a large number of numerical experiments. Namely, we have that the Strang splitting is able to approximately preserve the energy over long times only for a small number of spatial discretisation points and for smooth initial data ($M\ll \tau^{-1/2}$). \georginline{In contrast}, our proposed symmetric low-regularity integrators can achieve this feat even for low regularity solutions and with a large number of spatial discretisation points.

\end{remark}\vspace{-0.3cm}

	\vspace{0.5em}
 {\bf{Outline of the article}}
\vspace{0.5em}	\\
The remainder of this manuscript is structured as follows. \georginline{To begin with, in section~\ref{sec:main_ideas_symmetric_scheme_derivation_introduction}, we outline the main ideas in the construction of symmetric resonance based schemes, before formalising those ideas more rigorously in the subsequent sections. In particular, }in section~\ref{sec:decorated_trees_generalised_res-based_schemes} we firstly recall the decorated tree framework introduced in \cite{BS} for non-symmetric resonance based schemes. We then generalise this framework in order to capture a broader class of resonance based integrators allowing for polynomial interpolation of lower order parts in the approximation (cf. section~\ref{sec:dominant_part_and_polynomial_interpolation}). This leads to a general framework taking the form of a forest formula that can capture a wide class of implicit and explicit resonance based schemes and which is introduced in section~\ref{sec:forest_formula} followed by a generalisation in section~\ref{sec:more_general_forest_formula}. This forest formula motivates our consideration of a particular way of iterating Duhamel's formula to generate a subclass of resonance based schemes described by this general framework in section~\ref{Duhamel_iteration_mid_point} which turns out to be sufficiently general to allow us to find symmetric resonance based schemes of arbitrary order in this class. In section~\ref{sec:symmetric_schemes} we then describe how symmetric interpolation in the construction from section~\ref{Duhamel_iteration_mid_point} leads to symmetric schemes before classifying all symmetric schemes captured by the general forest formula in section~\ref{sec:conditions_for_symmetry_general_framework}. We conclude the section with examples of the new symmetric resonance based integrators that can be found using our novel framework in section \ref{sec:examples}. In section~\ref{sec:numerical_experiments}, we provide numerical experiments demonstrating the favorable practical performance of the new symmetric resonance based schemes that we were able to develop using our formalism. \vspace{-0.1cm}
\subsection*{Acknowledgements}
{\small
K. S. and Y. A. B.  received financial
support from the European Research Council (ERC) under the European Union's Horizon 2020 research and innovation programme (grant agreement No. 850941). Y. B. received financial support from the European Research Council (ERC) (ERC Starting Grant grant agreement No.\ 101075208)
and from the ANR via the project LoRDeT (Dynamiques de faible régularité via les arbres décorés) from the projects call T-ERC\_STG. G. M. gratefully acknowledges funding from the European Union’s Horizon Europe research and innovation programme under the Marie Sklodowska--Curie grant agreement No.\ 101064261.
}  
\vspace{-0.3cm}
	\section{Main ideas of the derivation of symmetric resonance based schemes}\label{sec:main_ideas_symmetric_scheme_derivation_introduction}
	Before diving into a more abstract construction of the algebraic structures describing our novel resonance based schemes let us begin by outlining the main assumptions on the type of equation we consider as well as the blueprint for the construction of implicit (and specifically symmetric) resonance based integrators for equations of the form \eqref{dis}.
	\vspace{0.5em}	\\ 
	\noindent {\bf{Assumptions}}
\vspace{0.5em}	\\ 
	We impose periodic boundary conditions, i.e. $x \in \T^d$. We assume that the   differential operator $\mathcal{L}$ is real and that the differential operators $\mathcal{L}\left(\nabla \right) $ and $\vert \nabla\vert^\alpha$ shall cast in Fourier  space into the form 
	\begin{equs}\label{Lldef}
		\mathcal{L}\left(\nabla \right)(k) = k^\sigma + \sum_{\gamma : |\gamma| < \sigma} a_{\gamma} \prod_{j=1}^d k_j^{\gamma_j} ,\qquad \vert \nabla\vert^\alpha(k) =  \sum_{\gamma : |\gamma| \le\alpha} \prod_{j=1}^{d} k_j^{\gamma_j}
	\end{equs}
	for some $ \alpha \in \R $, $\sigma \in \N$, $ \gamma \in \Z^d $ and $ |\gamma| = \sum_i \gamma_i $,
	where  for $k = (k_1,\ldots,k_d)\in \Z^d$ and $m = (m_1, \ldots, m_d)\in \Z^d$ we set \begin{equs}
		k^\sigma  = k_1^\sigma + \ldots + k_d^\sigma, \qquad k \cdot m = k_1 m_1 + \ldots + k_d m_d.
	\end{equs}
\vspace{-1.5em} \\
	\textbf{Construction of implicit resonance based schemes}
	\vspace{0.5em} \\
	We first rewrite \eqref{dis} in  Duhamel's form
	\begin{equs}
		 u(t,x) =   e^{i t \mathcal{L}} v(0,x) - i e^{ it \mathcal{L}} \int^t_0  e^{- is  \mathcal{L}} \vert \nabla\vert^\alpha p\left(u(s,x), \bar{u}(s,x)  \right) ds
	\end{equs} 
where we have used $ \mathcal{L} = \mathcal{L}\left(\nabla\right) $ as a short hand notation. Then, if we move to Fourier space by denoting $ u_k $ and $ v_k $ the k-th Fourier coefficients of $ u $ and $ v $, one obtains:
\begin{equs} \label{duhamel_formula_fourier}
	u_k(t) =  e^{i t \mathcal{L}\left(k \right)} v_k(0) - i e^{ it \mathcal{L}\left(k\right)} \int^t_0  e^{- is  \mathcal{L}\left( k \right)} \vert \nabla\vert^\alpha(k) p_k\left(u(s,x), \bar{u}(s,x)  \right) ds\quad
\end{equs}
	where $ \mathcal{L}\left(k \right) $ and $ \vert \nabla\vert^\alpha(k) $ are the differential operators $ \mathcal{L} $ and $ \vert \nabla\vert^\alpha $ mapped in Fourier space. The term $ p_k\left(u(s,x), \bar{u}(s,x)  \right) $ stands for the Fourier transform of the product. For example, in the case of NLS we have $ \alpha = 0 $, $ \mathcal{L} = \Delta $, and $ p(u,\bar{u})  =u^2 \bar{u}$. The equation \eqref{duhamel_formula_fourier} becomes 
	\begin{equs}
		u_k(t) =  e^{-i t k^2} v_k(0) - i e^{ - it k^2} \int^t_0  e^{ is   k^2} \left( \sum_{k = -k_1 +k_2+k_3} \bar{u}_{k_1}(s) u_{k_2}(s) u_{k_3}(s) \right)  ds.
	\end{equs}
The product $ |u|^2 u $ becomes a convolution on the coefficients in Fourier space, where we note that the minus pre-multiplying $ k_1 $ is due to the conjugate. We have also used the fact that the Fourier transform of $ e^{i t \Delta} $ is $ e^{-i t k^2} $.
The first step in the construction of resonance based schemes consists in iterating \eqref{duhamel_formula_fourier} inside the nonlinearity which produces a sum of oscillatory integrals that can be described by decorated trees as introduced in \cite{BS}. Namely, in \cite{BS} we iterate only with \eqref{duhamel_formula_fourier} which corresponds to a left end point iteration, meaning that given a time step $ \tau $, we always take  the left end point approximation \georginline{in the linear part $\exp(-itk^2)v_k(0)$} on the interval $ [0,\tau] $. Hence, in general we do not obtain a symmetric scheme. Indeed, one has the possibility to \georginline{write Duhamel's formula} around any point in the interval $ [0,\tau] $. \georginline{In particular, if we set} for $ s \in [0,\tau] $
\begin{equs} \label{duhamel_formula_fourier_revisited}
I(k,u,s,t) 
& = e^{i (t-s) \mathcal{L}\left(k \right)} u_k(s) \\
& - i e^{ it \mathcal{L}\left(k\right)} \int^t_s  e^{- i \tilde{s}  \mathcal{L}\left( k \right)} \vert \nabla\vert^\alpha(k) p_k\left(u(\tilde{s},x), \bar{u}(\tilde{s},x)  \right) d\tilde{s}
\end{equs} 
\georginline{then we have} the identity:
\begin{equs}
	u_k(t) = I(k,u,s,t) .
\end{equs}
From the identity \eqref{duhamel_formula_fourier_revisited}, we will obtain implicit schemes. We can take a weighted sum of the various Duhamel's iterations \eqref{duhamel_formula_fourier_revisited} \georginline{ultimately to arrive at schemes with a large number of additional degrees of freedom}. The sum that we will use \georginline{for a large part of} this paper is the midpoint rule, that is:
\begin{equs} \label{mid_point}
		u_k(t) = \frac{1}{2} \left(  I(k,u,0,t) + I(k,u,\tau,t) \right).
\end{equs}
\georginline{For example, to construct a symmetric resonance based scheme of a desired order we can }start the first iteration for $ u_k(\tau) $ with the left end point Duhamel's formula and then we iterate the midpoint rule \eqref{mid_point}. \georginline{We can express this iteration in terms of the following tree series}

\begin{equs} \label{decoratedV1}
	U_{\text{\tiny{mid}},k}^{r}(\tau) & = e^{i \tau \mathcal{L} } u_k(0)  +  \sum_{T \in \CV^r_k} \frac{\Upsilon^{p}_{\text{\tiny{mid}}}(T)(u,\tau)}{S(T)} \left( \Pi_{\text{\tiny{mid}}} T \right)(\tau)
\end{equs}
where $ \CV^r_k $ is a set of decorated trees of size at most $ r+1 $ which incorporate the frequency $k$. These trees encode via $ \left( \Pi_{\text{\tiny{mid}}} T \right)(\tau)$, iterated integrals of depth at most $ r+1 $. The coefficient $ S(T) $ is the symmetry factor associated to the tree $ T $ and $ \Upsilon^{p}_{\text{\tiny{mid}}}(T)(u,\tau) $ is the coefficient appearing in the iteration of Duhamel's formulation depending on the nonlinearity $ p $. The coefficients $ \Upsilon^{p}_{\text{\tiny{mid}}}(T)(u,\tau) $  depend  on $ \tau $.
Indeed, when we iterate the mipoint Duhamel's formula, we make appear, in the iterated integrals, terms of the form
\begin{equs}
	\frac{1}{2} \left( e^{i s \mathcal{L}} u_{k_j}(0) +  e^{i (s-\tau) \mathcal{L}}u_{k_j}(\tau) \right) = 	e^{i s \mathcal{L}} \frac{1}{2} \left(  u_{k_j}(0) +  e^{-i \tau \mathcal{L}}u_{k_j}(\tau) \right).
	\end{equs}
	{This form of the coefficients is a central difference to direct explicit constructions (cf. \cite{BS}) in which $\Upsilon^{p}(T)(u)$ does not depend on $\tau$ and only involves terms of the form $u_{k_j}(0)$.} It is natural to absorb the term $ e^{i s \mathcal{L}}  $ in the definition of $ \left( \Pi_{\text{\tiny{mid}}} T \right)(\tau)  $   and 
\begin{equs}
 \frac{1}{2} \left(  u_{k_j}(0) +  e^{-i \tau \mathcal{L}}u_{k_j}(\tau) \right) 
 \end{equs}
   into the definition of $ \Upsilon^{p}_{\text{\tiny{mid}}}(T)(u,\tau) $. These \georginline{aforementioned} quantities are described in detail in Section \ref{Duhamel_iteration_mid_point}.  The sum \eqref{decoratedV1} can be viewed as a first numerical approximation by keeping only the iterated integrals of order below $ r $ of the infinite series describing formally the solution of \eqref{dis}. In Proposition~\ref{duhamel_symmetric_1}, we show that 
 \eqref{decoratedV1} is a symmetric scheme.

\georginline{In order for the scheme to have a suitable local error when applied to low-regularity solutions}, it is necessary to replace each oscillatory integral $ \left( \Pi_{\text{\tiny{mid}}} T \right)(\tau) $ appearing in the finite sum \eqref{decoratedV1} by a low regularity approximation that embeds the resonance structure into the  numerical discretization. Our novel approach is to try and perform this approximation in a symmetric manner. Let us explain briefly how it works. \georginline{Suppose we aim to }discretise an oscillatory integral of the form
 \begin{equs}
 	\int_0^{t} e^{is  \mathcal{L}} ds, \quad  \mathcal{L} = \mathcal{L}_{\text{\tiny{dom}}} +  \mathcal{L}_{\text{\tiny{low}}},
 	\end{equs}
 where we have split the operator into a dominant part $ \mathcal{L}_{\text{\tiny{dom}}} $ that we will integrate exactly and a lower part $ \mathcal{L}_{\text{\tiny{low}}} $ that we will approximate \georginline{(cf. section \ref{sec:dominant_part_and_polynomial_interpolation} for a definition of these quantities)}. A typical example arising in the case of NLS is 
 \begin{equs}
 	\mathcal{L} = k^2 + k^2_1 -k_2^2 -k_3^2, \quad  \mathcal{L}_{\text{\tiny{dom}}}  = 2 k_1^2 \quad  \mathcal{L}_{\text{\tiny{low}}} = - 2 k_1 (k_2 + k_3) + 2 k_2 k_3,
 \end{equs}
where $ k = -k_1 + k_2 + k_3 $.
We see that the exact integration of 
\begin{equs}
		\int_0^{t} e^{is  \mathcal{L}_{\text{\tiny{dom}}}} ds = \frac{e^{2 i s k_1^2} - 1}{2 i  k_1^2}
	\end{equs}
can be mapped back to physical space as $ 1/k^2 $ corresponds to $ \Delta^{-1} $. This property of being able to write the scheme in physical space is crucial for efficient numerical implementation: If we have an expression in physical space the differential operators can be computed quickly in \textit{frequency} space using the Fast Fourier Transform (FFT) while any polynomial-type nonlinear term can be computed quickly in \textit{physical} space since it corresponds to a local operation on the function values on a grid. Now, it remains to approximate the lower part. For this task, we use a polynomial interpolation
with $ m+1  $ points on $ [0,\tau] $ denoted by $ a_j \tau $.
We note that the error incurred by this polynomial interpolation will be one of the determining factors of the convergence order of our overall numerical scheme, and thus we highlight that this interpolation can be done to any given order. We take $ r +1 $ distinct interpolation points $0\leq a_0<a_1<\dots<a_r\leq 1$ which are symmetrically distributed such that $a_{j}=1-a_{r-j},j=0,\dots, r$. Let us denote the corresponding nodal polynomials by $ p_{j,r} $ such that
\begin{equs}
	p_{j,r}(a_m \tau) = \delta_{j,m}.
\end{equs}
Then, we define the following approximation
\begin{equs}
	\tilde{p}_r( f,\xi) = \sum_{j=0}^r f(a_j)  p_{j,r}(\xi), \quad f(a_j) =  e^{i a_j \tau \CL_{\text{\tiny{low}}}}.
\end{equs}
We have the following local error 
\begin{equs} \label{local_error_interpolation}
	f(\xi) - 	\tilde{p}_r(f,\xi) = \mathcal{O}( \prod_{j=0}^r (\xi-a_j \tau)  (i \CL_{\text{\tiny{low}}})^{r+1} ) 
\end{equs}
which \georginline{requires less regularity than if we had chosen to base our approximation on} $ \CL $ instead of $ \CL_{\text{\tiny{low}}} $ as classical schemes do.

\georginline{In order to arrive at a numerical scheme, we will in the following introduce} the low-regularity symmetric approximation operator of $  \Pi_{\text{\tiny{mid}}}$ \georginline{denoted by} by $ \Pi^{n,r}_{\text{\tiny{mid}}} $. Here, $ r $ corresponds to the order of the discretization and $ n $ is the a priori regularity assumed on the initial data $ v $. Namely, we assume $v\in H^n$, where $H^n$ is a Sobolev space.  The general scheme then takes the form:
\begin{equs} \label{general_scheme_BS}
	U_{\text{\tiny{mid}},k}^{n,r}(\tau, u) & = \sum_{T \in \CV^r_k} \frac{\Upsilon^{p}_{\text{\tiny{mid}}}(T)(u,\tau)}{S(T)} \left( \Pi^{n,r}_{\text{\tiny{mid}}} T \right)(\tau).
\end{equs}
We will show in Theorem~\ref{symmetric_schemme_low_regularity} that this scheme is symmetric.
The  local error structure for each approximated iterated integral is given by
\begin{equs}\label{eq:loci}
	\left(\Pi_{\text{\tiny{mid}}} T - \Pi^{n,r}_{\text{\tiny{mid}}} T \right)(\tau)  = \mathcal{O}\left( \tau^{r+2} \mathcal{L}^{r}_{\text{\tiny{low}}}(T,n) \right),
\end{equs}
where $\mathcal{L}^{r}_{\text{\tiny{low}}}$ involves only lower order derivatives. Its proof is exactly the same as in \cite{BS}. The local error does not depend on the choice of Duhamel's iteration and polynomial interpolations.
The form of the scheme draws its inspiration from the treatment of  singular stochastic partial differential equations (SPDEs) via Regularity Structures in \cite{reg,BHZ,BCCH,EMS}. These decorated tree expansions are generalization of the B-series widely used for ordinary differential equations, we refer to \cite{Butcher72,CCO08,HLW,Hans} and 
tree series used for dispersive equations \cite{Christ, oh1, Gub11,LO13,HLO20}.
In the end, one obtains  an approximation of $u $ under much lower regularity assumptions than classical methods (e.g., splitting methods, exponential integrators \cite{R1,Faou12,HLW,HochOst10,HLRS10,Law67,R3,LR04,McLacQ02,SanBook}) require, which in general introduce the local error 
\begin{equs}\label{eq:lociClass}
	\mathcal{O}\left( \tau^{r+2} \mathcal{L}^{r}(T,n) \right)
\end{equs}
involving the full high order differential operator $\mathcal{L}^{r}$. Indeed, \georginline{denoting by $\mathcal{D}(\,\cdot\,)$ the domain of a given operator,} we have that  $\mathcal{D}( \mathcal{L}_{\text{\tiny{low}}}) \supset \mathcal{D}( \mathcal{L})$, meaning that the local error structure \eqref{eq:loci} allows us to deal with a  {rougher} class of solutions than the classical error bound \eqref{eq:lociClass}.
Let us mention that the local error analysis can be nicely understood via a Birkhoff factorisation of the character $ \Pi^{n,r}_{\text{\tiny{mid}}} $ (see \cite{BS,BE}) that involves a deformed Butcher-Connes-Kreimer coproduct (see \cite{Butcher72,CK,CKI,BS,BM}).
 One can associate a pre-Lie product and a Grossman-Larson product with it. They are very much constrained as one can graft only at spots where the frequencies are the same (the frequency of the leaves of the tree on which the grafting is performed and the frequency at the root of the other tree should be the same). It is possible to define the composition of low regularity integrators with these structures. Something of this flavour has been done for deriving normal forms in \cite{B24}. Also, arborification maps built out of a Butcher-Connes-Kreimer type coproduct are used in \cite{BL24} to understand recent cancellations in dispersive PDEs. Let us also mention \cite{B23} where the composition is understood at the level of Regularity Structures B-series where similar decorated trees are used without the frequency decoration.

The B-series formalism introduced in \cite{BS} and used in this work is very much different from the previous combinatorial approach in numerical analysis. But as already mentioned many open questions could be addressed for a better understanding of the class of low regularity schemes such as
 composition  and substitution in connection with backward error analysis, see \cite{Butcher72,HW74,CHV05,CHV07,CHV10} for how they have been implemented in the context of B-series. For Lie group integrators \cite{Hans,ML08}, it is not clear how one can recover a similar structure in this dispersive context but one can perfectly imagine working with a planar version of the decorated trees used in this paper.
 Another important question that has been solved for B-series is the characterisation of B-series as affine equivariant methods (see \cite{MV16,MMMV16}) which is connected to volume preserving methods. In our context, the elementary differentials are not free as they are just monomials in the initial data and they do not reflect the tree structure. One can have some hope to get some freeness as the tree structure may be preserved in the low regularity discretisation but it is not obvious to have a general statement and may only work case by case.
 One option could be to find another combinatorial set like for example, the word series for dynamical systems in \cite{MSS17}, the random tensors \cite{DNY22}
  or some type of multi-indices like in \cite{OSSW,BEH}.
The decorated trees have been chosen in \cite{BS} because they are very natural for encoding iterated integrals and performing the resonance discretisation.

 In the present work, we push further the algebraic perspective by writing several forest formulae that \georginline{can be used to represent a larger class of} low regularity schemes. These forest formulae take the following form:
\begin{equs} \label{main_forest_formula} \begin{aligned}
	u_k^{\ell+1} & =  e^{i \tau \mathcal{L}} u_k^{\ell}+   e^{i \tau \mathcal{L}} \sum_{T \in \tilde{\CV}^r_k}  \sum_{\mathbf{a} \in [ 0, 1]^{\tilde{E}_T}}  \sum_{\chi \in \lbrace 0, 1\rbrace^{ L_T}}\sum_{T_0 \cdot T_1 ... \cdot T_{{m}} \subset T} C_T \\ &  b_{\mathbf{a}, \chi, T, T_0 \cdot ... \cdot T_{{m}}}(\tau,i \tau \mathscr{F}_{\text{\tiny{dom}}}(T_j), j \in \lbrace 0,...,{m} \rbrace)
\\ &	\prod_{j=0}^m \prod_{e \in \tilde{E}_{T_j}} e^{i \tau a_e\mathscr{F}_{\text{\tiny{low}}}(T_j^e)} \frac{\Upsilon^p_{\chi}(T)(u_{k_v}^{\ell +\chi_v}, v \in L_T, \tau)}{S(T)}.
\end{aligned}
\end{equs}
Below, we give a brief description of the notation of this forest formula \georginline{before introducing each term in full detail in section~\ref{sec:decorated_trees_generalised_res-based_schemes}.}
Here, $ \mathcal{L} $ si the full operator of \eqref{dis}, $ \tilde{\CV}^r_k $ is a finite set of decorated trees, $ \tilde{E}_T $ denotes the edges of $ T $ that correspond to a time integration. These time integrals are discretised with a low regularity approximation. Therefore, we have to use a map \textbf{a} on these edges that specifies which interpolation points have been used. This corresponds to the following term
 $$ \prod_{e \in \tilde{E}_{T_j}} e^{i \tau a_e\mathscr{F}_{\text{\tiny{low}}}(T_j^e)}$$
 Here $ \mathscr{F}_{\text{\tiny{low}}}(T_j^e) $ denote the lower part of the various discretisations where $ T_j^e $ is included into $ T_j $. The set $ L_T $ are the leaves of $ T $ associated to some $ u_{k_u} $ and the map $ \chi $ specifies if they are evaluated at the right ($ u_{k_u}^{\ell} $) or left end point ($ u_{k_u}^{\ell +1} $). The coefficients $ \Upsilon^p_{\chi}(T)(u_{k_v}^{\ell +\chi_v}, v \in L_T, \tau) $ depend on the structure of the equation and the way one iterates Duhamel's formula.
 One essential choice of this forest formula is the splitting of $ T $ into a forest $ T_0 \cdot ... \cdot T_m $ where the $ T_i $ are decorated trees. This allows us to encode all the lower parts of the resonances $  \mathscr{F}_{\text{\tiny{low}}}(T_j^e) $  and all their dominant parts $ \mathscr{F}_{\text{\tiny{dom}}}(T_j) $ that appear in the low regularity discretisation. \georginline{As we shall see below, this forest splitting is a crucial novelty necessary for the construction and classification of symmetric low-regularity integrators.} For the splitting, one can use a Butcher-Connes-Kreimer coproduct (see Section \ref{sec:forest_formula}) or a deformed Butcher-Connes-Kreimer coproduct used for the local error analysis (see Section \ref{sec:more_general_forest_formula}). The difference between the two is that one provides more terms in the deformed forest formula (deformed Butcher-Connes-Kreimer coproduct) and therefore more degrees of freedom for finding new schemes. In our applications so far, the forest formula without deformation is enough for finding symmetric schemes. We derive a condition on the coefficients $ b_{\mathbf{a}, \chi, T, T_0 \cdot ... \cdot T_{{m}}} $ in Proposition\ref{prop:symmetry_conditions} that allows to find symmetric schemes. The coefficients $  b_{\mathbf{a}, \chi, T, T_0 \cdot ... \cdot T_{{m}}} $  do not depend on the frequencies that are node decorations of the trees $ T, T_j $. 
 One can see them only as functions of the dominant parts of the various operators encountered during the discretisation.
 The term $ C_T $ is a structure term depending on the frequencies that encode the various operators that appear in the iterated integral given by $ T $. \georginline{We conclude this section with a few remarks concerning the structures introduced and the properties of the resulting low-regularity schemes.}

 { \begin{remark} 
 		One can wonder if one has always $ \mathcal{L}_{\text{\tiny{low}}} = \mathcal{L} $ then one obtains, with the interpolation described above, all  Runge-Kutta and exponential integrators. It seems to  be the case. In \cite{BS}, the authors have constructed many schemes that interpolate between low regularity schemes and classical exponential integrators see \cite[Prop. 4.11]{BS}. We have just extended this construction taking into account more degrees of freedom that should allow in principle to recover all classical schemes. 
 		\end{remark}
 }
 \begin{remark}
 	The forest formula appear in the BPHZ algorithm \cite{BP57,KH69,WZ69} for renormalising Feynman diagrams and was later used for renormalising singular SPDEs in \cite{BHZ,CH} with an extension of the algebraic structure.
 	\end{remark}
 
\begin{remark} 
The scheme \eqref{general_scheme_BS} has been generalized to non-polynomial nonlinearities  and to parabolic equations in \cite{BBS} with the use of nested commutators first introduced in \cite{FS}. The Birkhoff factorisation discovered in  \cite{BS} is not available in this case. It is also not obvious to translate forest formulae into this context. Indeed, due to the fact that the formula is written in Fourier space, there is no order on the operators written in Fourier space. This is not the case in physical space. {Potentially, the recursive proof given in Section \ref{sec:symmetric_interpolation} could be extended. It remains an open problem to provide general symmetric schemes in this context.}
	\end{remark}

\begin{remark}
	The schemes presented in \cite{BS} have been adapted to a probabilistic setting by proposing a low regularity approximation \cite{ABS22} of the second moment of the Fourier coefficient of the solution, i.e. $ \mathbb{E}(|u(v^\eta, \tau)|^2) $ where $ v^{\eta} $ is a random initial data. In this context, one has to work with paired decorated trees. It possible to write symmetric schemes for approximating this second moment using our approach. Also, one can set up an equivalent forest formula on these paired decorated trees. One open direction is to understand the connection between the algebraic tools developed for these numerical schemes and the tools used for the rigorous derivation of the wave kinetic equation (WKE) for NLS is performed in \cite{deng_hani_2021,deng2021derivation,ACG}
	\end{remark}

\begin{remark} \georginline{The central novelty of the present work is the structured understanding of implicit and, in particular, symmetric low-regularity integrators.} The local error bounds we use in this paper \georginline{often} rely on the previous local error derivations first introduced in \cite{BS}. Indeed, the scheme $ U_{k}^{n,r}(\tau, u) $ is of the form \eqref{main_forest_formula} but the local error analysis comes from the fact that it is defined recursively via the character $ \Pi^{n,r}_{\text{\tiny{mid}}} $ and therefore the tools from \cite{BS} are available. If one found a new scheme by choosing the coefficients $ b_{\mathbf{a}, \chi, T, T_0 \cdot ... \cdot T_{{m}}} $, it is not clear how to get directly the local error analysis and check that the scheme is optimal in terms of regularity. 

We also make the important remark that given that we derive schemes which are of implicit nature, an additional fixed-point argument needs to be performed on the numerical flow in order to rigorously buckle the local error bounds, we refer to the works of \cite{MS22, AB23} where this analysis is made in detail.
	\end{remark}

\begin{remark} On this forest formula, we have identified symmetric schemes and, in addition, we have provided a general recursive mechanism to derive symmetric schemes for a large class of PDEs. One can wonder if such an approach could be repeated for other symmetries. Indeed, we believe that our techniques are fairly general. The degrees of freedom offered by different Duhamel's iterations and interpolations should allow us to capture other symmetries at low regularity using variants of the recursive scheme $  U_{k}^{n,r}(\tau, u) $. One degree of freedom which has not been used in full generality is the splitting of the operator into dominant and lower part :
	\begin{equs}
		\mathcal{L} = \mathcal{L}_{\text{\tiny{dom}}} +  \mathcal{L}_{\text{\tiny{low}}}.
	\end{equs}
Right now, it is governed by Definition \ref{Dom_projection} that guarantees to get a resonance-based scheme and a scheme which can be written in physical space. For symplectic schemes, one expects to have symmetries between the frequencies of $ \mathcal{L}_{\text{\tiny{dom}}} $ and those $ \mathcal{L}_{\text{\tiny{low}}} $. One should have the possibility of refining this splitting for encapsulating some symmetries as has been done for the 1D NLSE and the KdV equation in recent work \cite{MS22}.
The rest of the construction of the scheme remains unchanged. Consequently, a natural line of future research is the study of such symmetries ($\rho$-reversibility, preservation of quadratic invariants, etc.) directly on a structured tree or forest expansion of the numerical schemes comparable to the use of B-series in the study of structure preservation properties of methods for ODEs. We believe that the forest formulae presented in the current work take a first step in this direction.
	\end{remark}
	
	\georginline{\begin{remark}\label{rem:evenOrder}
	Let us close this section with an interesting, and crucial observation:  In the context of ODEs it is well known that symmetric methods are  of {\it even order} (cf. \cite[Theorem IX.2.2]{HLW}). In general this is, however, \emph{not} the case for PDEs as the rate of convergence depends intrinsically on the regularity of the solution, and hence convergence at even order only holds if sufficient regularity requirements are met by the solution. 
	
Take as an example the resonance based midpoint method \eqref{schemeNLS3} or symmetrization scheme \eqref{schemeNLS_first scheme} for the NLS equation.
	These schemes are symmetric and of first order, with an optimal local error structure in $\mathcal{O}(\tau^2 \nabla u)$, which does not require more regularity on the solution than the asymmetric first order resonance based schemes of \cite{BS,OS18}. 	
	As these schemes are symmetric, for $C^\infty$ solutions they are naturally of even order, thereby not only of order one, but also of order two. However, a closer look shows that second order convergence is only attained for sufficiently regular solutions. Indeed, following \cite{AB23}, at second order the symmetric schemes introduce a local error  of type $\mathcal{O}(\tau^3 \nabla \Delta u)$, which requires the boundedness of three additional derivatives in order to attain second order convergence. For initial data in lower order spaces than $H^3$, one can obtain fractional convergence of order {\it less than two} \cite{AB23}. We also refer to the very recent work \cite{FMW24}, which corroborates the above results. We make the additional remark that in view of \cite{BS}, requiring a local error of $\mathcal{O}(\tau^3 \nabla \Delta u)$ is not optimal in the sense of regularity. 
	Indeed, we recall that the second-order non-symmetric resonance based integrators \cite{BS} obeys the favourable error structure $\mathcal{O}(\tau^3 \Delta u)$, hence asking for one less derivative on the solution. Thereby, we introduce in \eqref{schemeNLS2} a different second-order symmetric scheme, whose local error structure coincides with that of \cite{BS}.
\end{remark}}

	\section{Decorated trees and generalised resonance based schemes}\label{sec:decorated_trees_generalised_res-based_schemes}
	
	The main object of this manuscript is to formalise the construction of symmetric resonance based schemes as outlined in section~\ref{sec:main_ideas_symmetric_scheme_derivation_introduction}. To achieve this we resort to a new, generalised tree formalism which has already seen (in much simpler version) significant success in the construction of explicit (asymmetric) resonance based schemes (cf. \cite{BS}). In the present section we will begin by recalling some of the main definitions in this framework before generalising the construction to incorporate the possibility of implicit low-regularity integrators before ultimately culminating in a forest formula \eqref{forest_formula} \& \eqref{forest_formula_bis} which captures a broad class of resonance based numerical schemes in such way that we can later \georginline{characterise} those schemes in this class which are symmetric in the sense of definition~\ref{def:time-symmetric_method}.
	
	We recall briefly the structure of decorated trees introduced in \cite[Sec. 2]{BS}.
Let  $\Lab$ a finite set  and  frequencies $ k_1,...,k_m \in \Z^{d}$. We suppose we are given a fixed time step $ \tau >0 $. The set $\Lab$ parametrizes a set of differential
	operators with constant coefficients, whose symbols are given by the polynomials  $ (P_{\Labhom})_{\Labhom \in \Lab} $. These operators are given in Fourier space and therefore the polynomials will be evaluated in the frequencies $ k_i $.
	We define the set of decorated trees $ \hat \CT  $ 
	as elements of the form  $ 
	T_{\Labe}^{\Labn, \Labo} =  (T,\Labn,\Labo,\Labe) $ where 
	\begin{itemize}
		\item $ T $ is a non-planar rooted tree with root $ \varrho_T $, node set $N_T$ and edge set $E_T$. We denote the leaves of $ T $ by $ L_T $. $ T $ must also be a planted tree which means that there is only one edge connecting the root to the rest of the tree.
		\item the map $ \Labe : E_T \rightarrow \Lab \times \lbrace 0,1\rbrace$ are edge decorations. The set $ \{ 0,1\} $ encodes the action of taking the conjugate, and determines the sign of the frequencies at the top of this edge. Namely, we have that $1$ corresponds to a conjugate and to multiplying by $(-1)$ the frequency on the node above and adjacent to this edge.
		\item the map $ \Labn : N_T \setminus \lbrace\varrho_T \rbrace \rightarrow \N^2 $ are node decorations. For every inner node $ v$, this map encodes a monomial of the form $ \xi^{\Labn_1(v)} \tau^{\Labn_2(v)} $ where $ \xi $ is a free time variable belonging to $ [0,\tau] $. This is a novelty from \cite{BS} where we do not have factors in $ \tau  $. We need it as in the sequel, we will consider integrals of the form $ \int^{\xi}_\tau ... ds$. 
		\item the map $ \Labo : N_T \setminus \lbrace\varrho_T \rbrace \rightarrow \Z^{d}$ are node decorations. These decorations are frequencies that satisfy for every inner node $ u $:
		\begin{equs}\label{innerdecoration}
			(-1)^{\mathfrak{p}(e_u)}\Labo(u)  = \sum_{e=(u,v) \in E_T} (-1)^{\mathfrak{p}(e)} \Labo(v)
		\end{equs}
		where $ \Labe(e) = (\Labhom(e),\mathfrak{p}(e)) $ is the edge decoration of $e$ with $\Labhom(e)\in  \Lab$ and $\mathfrak{p}(e)\in \lbrace 0,1\rbrace$ and  $ e_u $ is the unique edge outgoing from $u$ which is part of the path connecting $u$ to the root. We denote this edge by  $ (v,u) $.
		From this definition, one can see that the node decorations at the leaves $ (\Labo(u))_{u \in L_T} $ determine the decoration of the inner nodes. One can call this identity Kirchhoff's law. We assume that the node decorations at the leaves are linear combinations of
		the $ k_i $ with coefficients in $ \lbrace -1,0,1 \rbrace $. 
		\item we assume that the root of $ T $ has no decoration.
	\end{itemize}
	
	When the node decoration $ \Labn $ is zero, we will denote the decorated trees $ T_{\Labe}^{\Labn,\Labo} $ as
	$ T_{\Labe}^{\Labo} = (T,\Labo,\Labe)  $. The set of decorated trees satisfying such a condition is denoted by $ \hat \CT_0 $. We set $\hat H $ (resp. $ \hat H_0 $) the (unordered) forests composed of trees in $ \hat \CT $ (resp. $ \hat \CT_0 $) with  linear spans $\hat \CH $ and  $ \hat \CH_0 $.  The forest product is denoted by $ \cdot $, the empty forest by $ \one $. Elements in $ \hat \CT $ are abstract representation of iterated time integrals and elements in $ \hat{H} $ are a product of them.
	
	We now introduce how one can represent uniquely decorated trees by using symbolic notations. We denote by $\CI_{o}$, an edge decorated by $o=(\mathfrak{t}, \mathfrak{p}) \in \Lab \times \lbrace 0,1\rbrace$. We introduce the operator $  \CI_{o}(\lambda_{k}^{\ell} \cdot) : \hat \CH \rightarrow \hat \CH $  that merges all the roots of the trees composing the forest into one node decorated by $(\ell,k) \in \N^2 \times \Z^{d} $.  The new decorated tree is then grafted onto a new root with no decoration. 
	If the condition $ \eqref{innerdecoration} $ is not satisfied on the argument then $\CI_{o}( \lambda_{k}^{\ell} \cdot)$ gives zero.
	If $ \ell = 0 $, then the term $  \lambda_{k}^{\ell} $ is denoted by $  \lambda_{k} $ as a short hand notation for $  \lambda_{k}^{0} $.
	The forest product between $ \CI_{o_1}(  \lambda^{\ell_1}_{k_1}F_1) $ and $ \CI_{o_2}(  \lambda^{\ell_2}_{k_2}F_2) $ is given by:
	\begin{equs}
		\CI_{o_1}(  \lambda^{\ell_1}_{k_1} F_1) \CI_{o_2}(  \lambda^{\ell_2}_{k_2} F_2) := \CI_{o_1}( \lambda^{\ell_1}_{k_1} F_1) \cdot \CI_{o_2}( \lambda^{\ell_2}_{k_2} F_2).  
	\end{equs}
The right hand side of the previous equality could be understood as a set where we can repeat elements and the forest product is the disjoint union of these sets.
	Any decorated tree $ T $ is uniquely represented as 
	\begin{equs}
		T = \CI_{o}(  \lambda^{\ell}_{k} F), \quad F \in \hat H.
	\end{equs}
	Given an iterated integral, its size is given by the number of integrations in time. Therefore, we suppose we are given a subset  $ \Lab_+ $ of  $ \Lab  $ that encodes edge decorations which correspond to time integrals that we have to approximate.
	
	\begin{example} \label{ex_1} We illustrate the definitions introduced above with decorated trees coming from the NLS equation. We consider the following decorated tree:
\begin{equs}
T =  \CI_{(\mathfrak{t}_1,0)}\Big(\lambda_{k}   \CI_{(\mathfrak{t}_2,0)}\left(\lambda_{k}    \CI_{(\mathfrak{t}_1,1)}(\lambda_{k_1}   ) \CI_{(\mathfrak{t}_1,0)}(\lambda_{k_2}   ) \CI_{(\mathfrak{t}_1,0)}(\lambda_{k_3}   )\right) \Big) =  \begin{tikzpicture}[scale=0.2,baseline=-5]
\coordinate (root) at (0,-1);
\coordinate (t3) at (0,1);
\coordinate (t4) at (0,3);
\coordinate (t41) at (-2,5);
\coordinate (t42) at (2,5);
\coordinate (t43) at (0,7);
\draw[kernels2] (t3) -- (root);
\draw[symbols] (t3) -- (t4);
\draw[kernels2,tinydots] (t4) -- (t41);
\draw[kernels2] (t4) -- (t42);
\draw[kernels2] (t4) -- (t43);
\node[not] (rootnode) at (root) {};
\node[not] (rootnode) at (t4) {};
\node[not] (rootnode) at (t3) {};
\node[var] (rootnode) at (t41) {\tiny{$ k_1 $}};
\node[var] (rootnode) at (t42) {\tiny{$ k_3 $}};
\node[var] (rootnode) at (t43) {\tiny{$  k_{2} $}};
\end{tikzpicture},
\end{equs}
where $ k = - k_1 + k_2 + k_3 $, $ \Lab = \lbrace \mathfrak{t}_1, \mathfrak{t}_2 \rbrace $, $ \Lab_+ = \lbrace \mathfrak{t}_2 \rbrace $,  $P_{\mathfrak{t}_1}(\lambda) = -\lambda^2$ and $P_{\mathfrak{t}_2}(\lambda) = \lambda^2$.  We put the frequency decorations  only on the leaves as those on the inner nodes are uniquely determined by them. In the table below, we explain the coding of the edges
\begin{equs}
\begin{tabular} { 
	| m{2.0em} | m{5.0em}| m{20.0em} |}
	\hline
	\textbf{Edge}  & \textbf{Decoration} & \textbf{Operator} \\
	\hline
	$ \<thick> $ & $ (\mathfrak{t}_1,0)$ &  $e^{i t P_{\mathfrak{t_1}}(k)} =  e^{-i t k^2}$ \\
	\hline
$\<thick2>$ & $(\mathfrak{t}_1,1)$  & $e^{-i t P_{\mathfrak{t_1}}(k)} = e^{i t k^2}$  \\
	\hline
$	\<thin>$  & $(\mathfrak{t}_2,0)$  & $-i\int^{t}_0 e^{i \xi P_{\mathfrak{t}_2}(k)} \cdots d \xi = -i\int^{t}_0 e^{i \xi k^2} \cdots d \xi$  \\
	\hline
	$	\<thin2>$  & $(\mathfrak{t}_2,1)$  & $-i \int^{t}_0 e^{-i \xi P_{\mathfrak{t}_2}(k)} \cdots d \xi=  -i \int^{t}_0 e^{-i \xi k^2} \cdots d \xi$  \\
	\hline
\end{tabular}
\end{equs}
 In the end, $ T $ is an abstract version of the following integral:
\begin{equs}
- i e^{-i t k^{2}} \int^{t}_0 e^{i \xi k^2} e^{i \xi k_1^2}  e^{-i \xi k_2^2}  e^{-i \xi k_3^2}d \xi.
\end{equs}
\end{example}
	
	The next combinatorial structure which we recall from $ \cite{BS} $  encodes abstract versions of  a discretization of an oscillatory integral. We denote by $ \CT $ the set of decorated trees $ T_{\Labe,r}^{\Labn,\Labo} = (T,\Labn,\Labo,\Labe,r)  $ where
	\begin{itemize}
		\item $ T_{\Labe}^{\Labn,\Labo} \in \hat \CT $.
		\item The decoration of the root is given by $ r \in \Z $, $ r \geq -1 $ such that
		\begin{equs} \label{condition_trees}
			r +1 \geq  \deg(T_{\Labe}^{\Labn,\Labo})
		\end{equs}
		where $ \deg $ is defined recursively by 
		\begin{equs}
			\deg(\one) & = 0, \quad \deg(F_1 \cdot F_2 )  = \max(\deg(F_1),\deg(F_2)),  \\
			\deg(\CI_{(\Labhom,\Labp)}(  \lambda^{\ell}_{k}F_1)   ) & = |\ell| +   \one_{\lbrace\Labhom \in \Lab_+\rbrace}   +  \deg(F_1)
		\end{equs}
		where $\ell = (\ell_1,\ell_2)$, $ |\ell| = \ell_1 + \ell_2 $, $ F_1, F_2  $ are forests composed of trees in $ \CT $. The quantity $\deg(T_{\Labe}^{\Labn,\Labo})$ is the maximum number of edges with type in $ \Lab_+ $, corresponding to time integrations, and of node decorations $ \Labn $ lying on the same path from one leaf to the root. 
	\end{itemize}
	We call decorated trees in $ \CT $ {\it approximated} decorated trees. The order of the approximation is encoded by a new decoration at the root $ r $. We denote by $ \CH $ the vector space spanned by  forests composed of trees in $ \CT$ and $  \lambda^\ell $, $ \ell \in \N^2 $ where $  \lambda^\ell $ is the tree with one node decorated by $\ell$. When the decoration $ \ell $  is equal to zero we identify this tree with the empty forest: $  \lambda^{0}  =\one $. 
	We now define the symbol $  \CI^{r}_{o}(\lambda_{k}^{\ell} \cdot) :  \CH \rightarrow  \CH $, as the same as $ \CI_{o}(\lambda_{k}^{\ell} \cdot) $, with the added adjunction of the decoration $ r $ which constrains the time-approximations to be of order $r$. It is given by:
	\begin{equs}
		\CI^{r}_{o}( \lambda_{k}^{\ell} (\prod_j  \lambda^{m_j}  \prod_i \CI^{r_i}_{o_i}( \lambda_{k_i}^{\ell_i} F_i))) { \, := }\CI^{r}_{o}( \lambda_{k}^{\ell+\sum_j m_j} (\prod_i \CI_{o_i}( \lambda_{k_i}^{\ell_i} F_i))).
	\end{equs}
	We  define a projection operator $\CD^{r}$ which depends on $r$ and which is used during the construction of the numerical schemes in order to only retain the terms of order at most $r$.
	We define the map $ \CD^{r} : \hat \CH \rightarrow \CH $ which assigns $ r$ to the root of  a decorated tree. This implies a projection along the identity \eqref{condition_trees}. It is given by
	\begin{equs}\label{DR}
		\CD^{r}(\one)= \one_{\lbrace 0 \leq  r+1\rbrace} , \quad \CD^r\left( \CI_{o}( \lambda_{k}^{\ell} F) \right) =  \CI^{r}_{o}( \lambda_{k}^{\ell} F)
	\end{equs}
	and we extend it multiplicatively to any forest in $ \hat \CH $. 
	
	\begin{example} We illustrate the action of the map $ \CD_r $ on the decorated tree $ T $ introduced in Example~\ref{ex_1}. One has:
		\begin{equs}
			\deg(T) = 1, \quad \CD_r(T) = 0, \, r > 1,
	\quad
			\CD_{r}\left( T \right) = \begin{tikzpicture}[scale=0.2,baseline=-5]
				\coordinate (root) at (0,-1);
				\coordinate (t3) at (0,1);
				\coordinate (t4) at (0,3);
				\coordinate (t41) at (-2,5);
				\coordinate (t42) at (2,5);
				\coordinate (t43) at (0,7);
				\draw[kernels2] (t3) -- (root);
				\draw[symbols] (t3) -- (t4);
				\draw[kernels2,tinydots] (t4) -- (t41);
				\draw[kernels2] (t4) -- (t42);
				\draw[kernels2] (t4) -- (t43);
				\node[not,label= {[label distance=-0.2em]below: \scriptsize  $ r $}] (rootnode) at (root) {};
				\node[not] (rootnode) at (t4) {};
				\node[not] (rootnode) at (t3) {};
				\node[var] (rootnode) at (t41) {\tiny{$ k_1 $}};
				\node[var] (rootnode) at (t42) {\tiny{$ k_3 $}};
				\node[var] (rootnode) at (t43) {\tiny{$  k_{2} $}};
			\end{tikzpicture}
		\end{equs}
	\end{example}
\subsection{Dominant part and polynomial interpolation}\label{sec:dominant_part_and_polynomial_interpolation}
Let us now introduce the operations used when approximating integrals represented by tree formalism as described above. We first recall \cite[Def.2.2]{BS} that select higher degree terms in a polynomial of the frequencies.
	\begin{definition} \label{Dom_projection} Let $ P(k_1,...,k_n) $ a polynomial in the $ k_i $. If the highest-degree monomials of $ P $ are of the form
		\begin{equs}
			a \sum_{i=1}^{n} (a_i k_i)^{m}, \quad a_i \in { \lbrace  0,1 \rbrace}, \, a \in \Z,
		\end{equs}
		then  we define $ \CP_{\text{\tiny{dom}}}(P) $ as
		\begin{equs} \label{physical map}
			\CP_{\text{\tiny{dom}}}(P) = a \left(\sum_{i=1}^{n} a_i k_i\right)^{m}.
		\end{equs}
		Otherwise, it is zero.
	\end{definition}

This definition is used for splitting an operator between a lower part and a dominant part. Indeed, if we consider 
the polynomial 
\begin{equs}
	P(k_1,k_2,k_3) = k^2 + k_1^2 - k_2^2 - k_3^3, \quad k = - k_1 + k_2 + k_3.
\end{equs}
coming from the NLS equation, we observe that $ P $ can be rewritten into the form:
\begin{equs}
		P(k_1,k_2,k_3) = 2 k_1^2 - 2 k_1 (k_2 + k_3) + 2 k_2 k_3.
\end{equs}
Then, we set 
\begin{equs}
	\mathcal{L}_{\text{\tiny{dom}}} = \mathcal{P}_{\text{\tiny{dom}}}(P) = 2 k_1^2, \quad 
		\mathcal{L}_{\text{\tiny{low}}} = \left( \id - \mathcal{P}_{\text{\tiny{dom}}} \right)(P).
\end{equs}
We note that $ \mathcal{L}_{\text{\tiny{dom}}}  $ asks for boundedness of two derivatives due to the factor $ k_1^2 $ and $ \mathcal{L}_{\text{\tiny{low}}} $ only one because \georginline{the latter consists only of} cross products $ k_i k_j $, $ i \neq j $. Another main reason for this splitting is to be able to map back to physical space the following integral:
\begin{equs}
	\int_0^{t} e^{i s \mathcal{L}_{\text{\tiny{dom}}}} ds = \frac{e^{i t \mathcal{L}_{\text{\tiny{dom}}}}-1}{i  \mathcal{L}_{\text{\tiny{dom}}}}.
\end{equs}
We observe that it is essential to map back to physical space the term $ \frac{1}{\mathcal{L}_{\text{\tiny{dom}}}} $ equal to $ \frac{1}{2k_1^2} $. Such a term is given by $ \Delta^{-1} $ in physical space.

The next definition extracted from \cite[Def. 2.6]{BS} allows us to compute recursively the various frequency interactions by extracting dominant and lower parts. Such a definition is required for the local error analysis and the forest formula given in the sequel.
	\begin{definition} \label{dom_freq} We recursively define $\mathscr{F}_{\text{\tiny{dom}}}, \mathscr{F}_{\text{\tiny{low}}} : \hat H_{0} \rightarrow \mathbb{R}[\Z^d]$ as:
		\begin{equs}
			\mathscr{F}_{\text{\tiny{dom}}}(\one)  = 0 \quad
			\mathscr{F}_{\text{\tiny{dom}}}(F \cdot \bar F) & =\mathscr{F}_{\text{\tiny{dom}}}(F) + \mathscr{F}_{\text{\tiny{dom}}}(\bar F) \\
			\mathscr{F}_{\text{\tiny{dom}}}\left( \CI_{(\Labhom,\Labp)}(  \lambda_{k}F) \right)  & = \left\{ \begin{aligned}
				&   \CP_{\text{\tiny{dom}}}\left( P_{(\Labhom,\Labp)}(k) +\mathscr{F}_{\text{\tiny{dom}}}(F) \right), \,
				\text{if } \Labhom \in \Lab_+  , \\
				&  P_{(\Labhom,\Labp)}(k) +\mathscr{F}_{\text{\tiny{dom}}}(F), \quad \text{otherwise}
				\\
			\end{aligned} \right.  
			\\
			\mathscr{F}_{\text{\tiny{low}}} \left( \CI_{(\Labhom,\Labp)}(  \lambda_{k}F) \right)  & =   
			\left( \id - \CP_{\text{\tiny{dom}}} \right) \left( P_{(\Labhom,\Labp)}(k) +\mathscr{F}_{\text{\tiny{dom}}}(F) \right),
		\end{equs}
	 where we recall that $ \Lab_+ $ is a subset of  $ \Lab  $ that encodes edge decorations which correspond to time integrals. We extend these two maps to $ \hat H $ by ignoring the node decorations $ \Labn $.
	\end{definition}
In a nutshell the above recursive definition means that in the set $\Lab\setminus\Lab_+$, i.e. operators that do not correspond to integration, we collect all frequency contributions, and in the set $\Lab_+$, i.e. operators that correspond to integration, we extract the dominant frequencies of the full integrand.
\begin{example}\label{ex:simple_tree_NLSE} We illustrate the previous definition on a simple decorated tree coming from the NLS equation.
	\begin{equs}
	T = 	\begin{tikzpicture}[scale=0.2,baseline=-5]
			\coordinate (root) at (0,0);
			\coordinate (tri) at (0,-2);
			\coordinate (t1) at (-2,2);
			\coordinate (t2) at (2,2);
			\coordinate (t3) at (0,3);
			\draw[kernels2,tinydots] (t1) -- (root);
			\draw[kernels2] (t2) -- (root);
			\draw[kernels2] (t3) -- (root);
			\draw[symbols] (root) -- (tri);
			\node[not] (rootnode) at (root) {};t
			\node[not,label= {[label distance=-0.2em]below: \scriptsize  $ $}] (trinode) at (tri) {};
			\node[var] (rootnode) at (t1) {\tiny{$ k_{\tiny{1}} $}};
			\node[var] (rootnode) at (t3) {\tiny{$ k_{\tiny{2}} $}};
			\node[var] (trinode) at (t2) {\tiny{$ k_3 $}};
		\end{tikzpicture} = \mathcal{I}_{(\mathfrak{t}_2,0)} \left( \lambda_k F  \right), \quad F =  \mathcal{I}_{(\mathfrak{t}_1,1)}(\lambda_{k_1})  \mathcal{I}_{(\mathfrak{t}_1,0)}(\lambda_{k_2}) \mathcal{I}_{(\mathfrak{t}_1,0)}(\lambda_{k_3}). 
		\end{equs}
with $ k = - k_1 + k_2 + k_3 $.	One has
		\begin{equs}
			\mathscr{F}_{\text{\tiny{dom}}}(T) = \CP_{\text{\tiny{dom}}}\left( P_{(\mathfrak{t}_2,0)}(k) +\mathscr{F}_{\text{\tiny{dom}}}(F) \right)
		\end{equs}
because $ \mathfrak{t}_2 \in \Lab_+ $.	Then, we use the fact that
	\begin{equs}
		P_{(\mathfrak{t}_2,0)}(k) = k^2, \quad 	P_{(\mathfrak{t}_1,0)}(k) = -k^2, \quad P_{(\mathfrak{t}_1,1)}(k) = k^2
	\end{equs}
and 
\begin{equs}
	\mathscr{F}_{\text{\tiny{dom}}}(F) &  = \mathscr{F}_{\text{\tiny{dom}}}(\mathcal{I}_{(\mathfrak{t}_1,1)}(\lambda_{k_1})) +   \mathscr{F}_{\text{\tiny{dom}}}(\mathcal{I}_{(\mathfrak{t}_1,0)}(\lambda_{k_2})) + \mathscr{F}_{\text{\tiny{dom}}}(\mathcal{I}_{(\mathfrak{t}_1,0)}(\lambda_{k_3}) )
	\\ & = 	P_{(\mathfrak{t}_2,1)}(k_1) + P_{(\mathfrak{t}_2,0)}(k_2) + P_{(\mathfrak{t}_2,0)}(k_3)
	\\ & = k_1^2 - k_2^2 - k_3^2.
\end{equs}
Therefore,
\begin{equs}
		\mathscr{F}_{\text{\tiny{dom}}}(T) & = \CP_{\text{\tiny{dom}}}\left( k^2 + k_1^2 - k_2^2 - k_3^2  \right)
		\\ & = \CP_{\text{\tiny{dom}}}\left( 2 k_1^2 - 2 k_1 (k_2 + k_3) + 2 k_2 k_3  \right)
		\\ & = 2 k_1^2
	\end{equs}
One observes that the projection $ \CP_{\text{\tiny{dom}}} $ projects to zero the cross terms $ k_i k_j$ with $ i \neq j $.
	\end{example}

A central novel idea which we introduce in our present work is that we proceed to interpolate the exponential of the lower part of the operator in place of a direct Taylor series expansion. The advantage of this procedure is firstly that it allows us to immediately arrive at stable schemes without the need for filter functions (the spectrum of $i \CP_{\text{\tiny{low}}}=\georginline{i} \CP- \georginline{i}\CP_{\text{\tiny{dom}}}$ typically lies on the imaginary axis so terms involving the exponential of the operator are all bounded). Secondly, through this interpolation process we are able to arrive at numerical schemes whose adjoint has the same functional form which is essential in the construction of symmetric methods. Classical Taylor expansion for the lower part gives:
\begin{equs}
	e^{i \xi \CL_{\text{\tiny{low}}}} = \sum_{\ell \leq r}
	\frac{\xi^\ell}{\ell !} (i \CL_{\text{\tiny{low}}})^{\ell} + \mathcal{O}( \xi^{r +1} (i \CL_{\text{\tiny{low}}})^{r+1} )
	\end{equs}
Now, for reasons of stability, we would like to use a polynomial interpolation that will give the same local error. We suppose given $ r +1 $ distinct interpolation points $0\leq a_0<a_1<\cdots<a_{r}\leq 1$ associated to the polynomials $ p_{j,r}(\cdot,\tau) $ such that
\begin{equs}
	p_{j,r}(a_m \tau,\tau) = \delta_{j,m}.
	\end{equs}
Then, we define the following approximation
\begin{equs}
	\tilde{p}_r( f,\xi) = \sum_{j=0}^r f(a_j \tau)  p_{j,r}(\xi,\tau ), \quad f(a_j \tau) =  e^{i a_j \tau \CL_{\text{\tiny{low}}}}, j=0,\dots,r,
\end{equs}
where we have suppressed the implicit $\tau$-dependency of $\tilde{p}_r(f,\xi)$ for notational simplicity. { Given any quadrature formula of order $r+1$, we at least have the following local error}
\begin{equs} \label{local_error_interpolation}
f(\xi) - 	\tilde{p}_r(f,\xi) = \mathcal{O}\left( \prod_{j=0}^r (\xi-a_j \tau)  (i \CL_{\text{\tiny{low}}})^{r+1} \right) .
\end{equs}
{We note that by judiciously choosing the quadrature rule, one can obtain the same rate in $\tau^{r+1}$ with fewer nodes. For example, taking the Gauss quadrature nodes would only require taking $(r+1)/2$ nodes to obtain the desired local error rate.}
\newline 

In the sequel, we will write the polynomial interpolation as:
\begin{equs} \label{p_hat}
		\tilde{p}_r(f,\xi) = \sum_{j=0}^r  \hat{p}_{j,r}(f,\tau) \xi^j
\end{equs}
where the  $ \hat{p}_{j,r}(f,\tau) $ are bounded in $ \tau $ because they correspond to linear combinations of terms of the form $\exp(ia_j\tau  \CL_{\text{\tiny{low}}})$. We provide below one example with two points $ 0 $ and $ \tau $
\begin{equs} \label{example_1_interpolation}
	\tilde{p}_{1}(f,\xi) =  1 + \frac{s}{\tau} \left( e^{i s \mathcal{L}_{\text{\tiny{low}}}} - 1 \right), 
\end{equs}
and 
\begin{equs}
	\quad \hat{p}_{0,1}(f,\tau) & = 1, \quad \hat{p}_{1,1}(f,\tau) = \frac{ e^{i s \mathcal{L}_{\text{\tiny{low}}}} - 1}{\tau} ,\\
	p_{0,1}(f,\xi) & = \frac{\tau - s}{\tau}, \quad 	p_{1,1}(f,\xi) = \frac{s}{\tau} e^{i s \mathcal{L}_{\text{\tiny{low}}}}.
\end{equs}
When $r=0$ we can, for example, pick
\begin{align*}
p_0(f,\xi)=\hat{p}_{0,0}(f,\tau)=p_{0,0}(f,\xi)=f\left(\frac{\tau}{2}\right).
	\end{align*}
In practice, we will also consider 
\begin{equs}
	\hat{p}_{0,0}(f,\xi) = \frac{f(0)+f(\xi)}{2}.
\end{equs}

The next definition is a slight modification of \cite[Def. 3.1]{BS} where Taylor expansions around zero are replaced by an interpolation on the interval $ [0,\tau] $ and we take into account monomials in $ \tau $ for the discretisation.

\begin{definition} \label{Taylor_exp}
	Assume that $ { G}:  \xi \mapsto \tau^m \xi^{q} e^{i \xi P(k_1,...,k_n)} $ where $ P $ is a polynomial in the frequencies $ k_1,...,k_n $ and let $ o_2 =  (\Labhom_2,p) \in \Lab_+ \times \lbrace 0,1 \rbrace$ and $ r \in \N $. Let $ k $ be a linear combination of $ k_1,...,k_n $ using coefficients in $ \lbrace -1,0,1 \rbrace $ and 
	\[
	\begin{aligned}
		\CL_{\text{\tiny{dom}}} & = \CP_{\text{\tiny{dom}}} \left(  P_{o_2}(k) + P \right), \quad \CL_{\text{\tiny{low}}}  = \CP_{\text{\tiny{low}}} \left(  P_{o_2}(k) +  P \right)
		\\
		f(\xi) & =  e^{i \xi \CL_{\text{\tiny{dom}}}}, \quad
		g(\xi)  =  e^{i \xi \CL_{\text{\tiny{low}}}}, \quad { \tilde g}(\xi) = e^{i \xi \left(  P_{o_2}(k) +  P \right)} .
	\end{aligned}
	\]
	Then, we define for $ n \in \N $, $ r \geq q $, $ \tilde{r} = r-q-m $ and $ \bar{n} = \text{deg}\left(\mathcal{L}_{\text{\tiny{dom}}}^{r+1}\right) + \alpha $
	\begin{equ}[e:Pim1]
		\CK^{k,r}_{o_2} (  { G},n)(s) = \left\{ \begin{aligned}
			&  -i \vert \nabla\vert^\alpha(k) \sum_{\ell \leq \tilde{r}} \hat{p}_{\ell, \tilde{r}}(\tilde{g},\tau) \int_0^{s} \tau^m \xi^{q+ \ell}   d \xi, \,
			\text{if }  n \geq \bar{n} , \\
			& -i\vert \nabla\vert^\alpha(k)  \sum_{\ell \leq \tilde{r}} \tau^m \hat{p}_{\ell, \tilde{r}}(g,\tau)\, \Psi^{r}_{n,q}\left( \CL_{\text{\tiny{dom}}} ,\ell\right)(s), \quad \text{otherwise}.
			\\
		\end{aligned} \right.
	\end{equ}
	Thereby we set for  $ \left(r - q - m - \ell+1 \right) \deg(\CL_{\text{\tiny{dom}}})  + \ell \deg(\CL_{\text{\tiny{low}}}) + \alpha > n  $
	\begin{equs}[e:Pim]
		\Psi^{r}_{n,q}\left( \CL_{\text{\tiny{dom}}},\ell \right)(s) =   \int_0^{s}   \xi^{q+\ell}  f(\xi)   d \xi.
	\end{equs}
	Otherwise,
	\begin{equ}[e:Pim2]
		\Psi^{r}_{n,q}\left( \CL_{\text{\tiny{dom}}},\ell \right)(s) = \sum_{j \leq \hat{r}} \hat{p}_{j,\hat{r}}(f,\tau) \int_0^{s}   \xi^{q+\ell +j} d \xi.
	\end{equ}
	Here $ \hat{r} = r - q -m - \ell $, $ \deg(\CL_{\text{\tiny{dom}}})$ and $\deg(\CL_{\text{\tiny{low}}}) $  denote the degree of the polynomial $ \CL_{\text{\tiny{dom}}} $ and $ \CL_{\text{\tiny{low}}} $, respectively and $\vert \nabla\vert^\alpha(k) = \prod_{\alpha = \sum \gamma_j <  \deg(\CL)} k_j^{\gamma_j}$. If $ r <q + m $, the map $ \CK^{k,r}_{o_2} (  { G},n)(s)$ is equal to zero.
\end{definition}

We perform an example to illustrate the polynomial interpolation.

\begin{example} 
	  We consider $P_{\Labhom_2}( \lambda) = -  \lambda^2$, $p = 0$, $ \alpha =0 $,  
	$ k = -k_1+k_2+k_3 $ and 
	\begin{equs}
		{ G}(\xi) = \xi e^{i \xi (  k_1^2 - k_2^2 - k_3^2 )}.
	\end{equs}
	With the notation of Definition \ref{Taylor_exp} we observe that
	\begin{equs}
		\CL_{\text{\tiny{dom}}}  
		& = 2 k_1^2, \quad \CL_{\text{\tiny{low }}}  =  - 2 k_1 (k_2+k_3) + 2 k_2 k_3,
	\end{equs}
	Furthermore, we observe as $\deg(\CL_{\text{\tiny{dom}}})   = 2$, $\deg(\CL_{\text{\tiny{low}}})   = 1$ and $q = 1$ that
	\begin{equs}\label{condE}
		\left(r - q - \ell+1 \right) \deg(\CL_{\text{\tiny{dom}}})  + \ell \deg(\CL_{\text{\tiny{low}}}) > n \quad \text{ if }\quad 2r - n > \ell.
	\end{equs}
We consider the polynomial interpolation given in \eqref{example_1_interpolation} and focus on some cases

	\begin{itemize}
		\item {\bf Case $r = 1$ and $n = 1$ :}   We obtain 
		\begin{equs}
		\CK^{k,1}_{o_2} ( { G},n)(s)  & = -i \hat{p}_{0,0}(f,\tau)  \Psi^{1}_{n,1}\left( \CL_{\text{\tiny{dom}}},0 \right)(s) 
	\\ &	=  -i \hat{p}_{0,0}(f,\tau)  \int_0^s \xi f(\xi) d\xi
	\\ & = \frac{1}{2i k_1^2}\left( s  e^ {2i s k_1^2}-  \frac{ e^{2 i s k_1^2}-1}{2 i k_1^2}\right) \left( \frac{1 + e^{i s \CL_{\text{\tiny{low }}}  }}{2} \right)
		\end{equs}
		as condition \eqref{condE} takes  for $\ell = 0$ the form $2-n> 0$.

		\item \noindent {\bf Case $r = 2$ and $ n=2 $:} 
		 We  have that
		\begin{equs}
			\CK^{k,2}_{\alpha} (  { G},n)(s)  & =  -i \left( \hat{p}_{0,1}(g,\tau) \Psi^{2}_{n,1}\left( \CL_{\text{\tiny{dom}}},0 \right)(s)  +
			 \hat{p}_{1,1}(g,\tau) \Psi^{2}_{n,1}\left( \CL_{\text{\tiny{dom}}},1 \right)(s) \right) 
		\end{equs}
		and condition \eqref{condE} takes  the form $4-n> \ell$.
		
		If $\ell = 1$ we thus obtain 
		\begin{equs}
			\Psi^{2}_{n ,1}\left( \CL_{\text{\tiny{dom}}},1 \right)(s)  = 
			\int_0^s \xi^2 f(\xi) d\xi 
			= \frac{ s^2 }{2 i k_2^2} \left(e^{2i s k_1^2}
			- 2  \Psi^{1}_{1,1}\left( \CL_{\text{\tiny{dom}}},0 \right) 
			\right).
		\end{equs}

		If $\ell = 0$, on the other hand,  condition \eqref{condE} holds. 
		Henceforth,  we have that
		\[
		\Psi^{2}_{n,1}\left( \CL_{\text{\tiny{dom}}},0 \right)(s)
		= \int_0^s \xi f(\xi)d\xi  .
		\]
	\end{itemize}
\end{example}

Following a similar proof as for \cite[Lem. 3.3]{BS} by using \eqref{local_error_interpolation}, one gets
\begin{lemma} \label{Taylor_bound}  We keep the notations of Definition~\ref{Taylor_exp}. We suppose that $   q + m \leq r$ then one has for $ s \in [0,\tau] $
	\begin{equation}
		- i \vert \nabla\vert^{\alpha} (k) \int_{0}^{s} \tau^m \xi^{q} { e^{i \xi \left (\CL_{\text{\tiny{dom}}} + \CL_{\text{\tiny{low}}}\right)}}  d\xi -\CK^{k,r}_{o_2} (  { G},n)(s) = \CO(\tau^{r+2} k^{\bar n})
	\end{equation} 
	where $ \bar n = \max(n, \deg(\CL_{\text{\tiny{low}}}^{r-q-m +1}) + \alpha) $.
\end{lemma}

\subsection{A forest formula for resonance based schemes}\label{sec:forest_formula}

 We recall the characters defined now on $ \CH $ and parametrised by $ n \in \N $ where $ n $ here is the a priori regularity assumed on the initial value, that is $ v \in H^n $ \georginline{where $H^n$ is the periodic Sobolev space of order $n$}. These characters give a low regularity discretisation of some iterated integrals:
\begin{equation} \label{recursive_pi_r}
	\begin{aligned}
		\Pi^n \left( F \cdot \bar F \right)(s,\tau) & = 
		\left( \Pi^n F  \right)(s,\tau)  \left( \Pi^n  \bar F \right)(s,\tau), \quad (\Pi^n  \lambda^{\ell})(s,\tau) = s^{\ell_1}\tau^{\ell_2}, \\
		(\Pi^n  \CI^{r}_{o_1}( \lambda_{k}^{\ell}  F ))(s,\tau)   & =s^{\ell_1} \tau^{\ell_2} e^{i s P_{o_1}(k)} (\Pi^n \CD^{r-|\ell|}(F))(s,\tau),  \\
		\left( \Pi^n \CI^{r}_{o_2}( \lambda^{\ell}_k F) \right) (s,\tau) & = \CK^{k,r}_{o_2} \left(  \Pi^n \left( \lambda^{\ell} \CD^{r-|\ell|-1}(F) \right)(\cdot,\tau),n \right)(s).
	\end{aligned}
\end{equation}
where $ o_2 = (\mathfrak{t}_2,p_2) $ with $ \mathfrak{t}_2 \in \Lab_+ $ and $ o_1 = (\mathfrak{t}_1,p_1) $ with $ \mathfrak{t}_1 \in \Lab \setminus \Lab_+ $ and $\ell=(\ell_1,\ell_2)\in\N^2$. We will use frequently the notations $ o_i $ in the sequel.
The main difference with \cite{BS} is the use of the polynomial interpolation in Definition~\ref{Taylor_exp}.
In the next theorem, we state a forest formula for the resonance scheme in the sense that we exhibit a general formula for the terms $ (\Pi^{n,r} F)(t) $ where $ \Pi^{n,r} $ is short hand notation for $ \Pi^{n} \mathcal{D}_r $. \georginline{This new forest formula is a significant extension of contributions in \cite{BS} since it incorporates not just the aforementioned polynomial interpolants but also allows for implicit discretisations in the unknown $v$.} We first need to introduce some notations that are needed for its formulation. We denote by $ \tilde{E}_T $ the edges of $ T $ associated to an integration in time. They carry a decoration of the type $ o_2 $ 
\begin{align*}
\tilde{E}_{T}=\{e\in E_T\,|\, \Labe(e) \in \Lab_+ \times \lbrace 0,1 \rbrace\}.
\end{align*} The notation $ T^e $ means that we consider the planted tree above the edge $ e $ in $ T $. This tree has its root connected to the rest of its nodes by the edge $ e $.
By $ F_0 \cdot T_1 ... \cdot T_m \subset F$, we mean that the forest
\begin{equs} \label{splitting_forest}
   F_0 \cdot T_1 ... \cdot T_m \subset F 
\end{equs} is a splitting of $ F $ where \label{text:introduction_of_forest_splitting}:
\begin{itemize}
	\item The $ T_i $ are planted trees with the edge connecting the root decorated by an  edge decoration of type $ o_2 $. 
	\item $ F_0 $ is a forest either empty or taking the form:
	\begin{equs}
		F_0 = \prod_{j=1}^{m_0} T_{0,j}
	\end{equs} 
where the  $  T_{0,j} $ are subtrees at the root of some trees appearing in the decomposition of the forest $ F $ into product of planted trees.
\end{itemize}

\begin{example}\label{ex:forest_splittings} We provide an example of the forests such that  $ F_0 \cdot T_1 ... \cdot T_m \subset F$. Let us consider $ F $ to be the following decorated trees coming from the NLS equation
	\begin{equs}
		 T = \begin{tikzpicture}[scale=0.2,baseline=-5]
			\coordinate (root) at (0,2);
			\coordinate (tri) at (0,0);
			\coordinate (trib) at (0,-2);
			\coordinate (t1) at (-2,4);
			\coordinate (t2) at (2,4);
			\coordinate (t3) at (0,4);
			\coordinate (t4) at (0,6);
			\coordinate (t41) at (-2,8);
			\coordinate (t42) at (2,8);
			\coordinate (t43) at (0,10);
			\draw[kernels2,tinydots] (t1) -- (root);
			\draw[kernels2] (t2) -- (root);
			\draw[kernels2] (t3) -- (root);
			\draw[symbols] (root) -- (tri);
			\draw[symbols] (t3) -- (t4);
			\draw[kernels2,tinydots] (t4) -- (t41);
			\draw[kernels2] (t4) -- (t42);
			\draw[kernels2] (t4) -- (t43);
			\draw[kernels2] (trib) -- (tri);
			\node[not] (trinode) at (trib) {};
			\node[not] (rootnode) at (root) {};
			\node[not] (rootnode) at (t4) {};
			\node[not] (rootnode) at (t3) {};
			\node[not] (trinode) at (tri) {};
			\node[var] (rootnode) at (t1) {\tiny{$ k_{\tiny{4}} $}};
			\node[var] (rootnode) at (t41) {\tiny{$ k_{\tiny{1}} $}};
			\node[var] (rootnode) at (t42) {\tiny{$ k_{\tiny{3}} $}};
			\node[var] (rootnode) at (t43) {\tiny{$ k_{\tiny{2}} $}};
			\node[var] (trinode) at (t2) {\tiny{$ k_5 $}};
		\end{tikzpicture}, \qquad \bar{T} = \begin{tikzpicture}[scale=0.2,baseline=-5]
		\coordinate (root) at (0,0);
		\coordinate (tri) at (0,-2);
		\coordinate (t1) at (-2,2);
		\coordinate (t2) at (2,2);
		\coordinate (t3) at (0,2);
		\coordinate (t4) at (0,4);
		\coordinate (t41) at (-2,6);
		\coordinate (t42) at (2,6);
		\coordinate (t43) at (0,8);
		\draw[kernels2,tinydots] (t1) -- (root);
		\draw[kernels2] (t2) -- (root);
		\draw[kernels2] (t3) -- (root);
		\draw[symbols] (root) -- (tri);
		\draw[symbols] (t3) -- (t4);
		\draw[kernels2,tinydots] (t4) -- (t41);
		\draw[kernels2] (t4) -- (t42);
		\draw[kernels2] (t4) -- (t43);
		\node[not] (rootnode) at (root) {};
		\node[not] (rootnode) at (t4) {};
		\node[not] (rootnode) at (t3) {};
		\node[not,label= {[label distance=-0.2em]below: \scriptsize  $  $}] (trinode) at (tri) {};
		\node[var] (rootnode) at (t1) {\tiny{$ k_{\tiny{4}} $}};
		\node[var] (rootnode) at (t41) {\tiny{$ k_{\tiny{1}} $}};
		\node[var] (rootnode) at (t42) {\tiny{$ k_{\tiny{3}} $}};
		\node[var] (rootnode) at (t43) {\tiny{$ k_{\tiny{2}} $}};
		\node[var] (trinode) at (t2) {\tiny{$ k_5 $}};
	\end{tikzpicture}
	\end{equs}
Because $ T $ starts with a brown edge that is an edge decorated by $ (\mathfrak{t}_1,0) $, $ F_0 $ is not empty. Below, we list all the possible splitting respecting this rule and also that the $ T_i $ with $ i \geq 1 $ must be planted trees with a blue edge (decorated by $ (\mathfrak{t}_2,0) $) at their root.
\begin{equs}
 \begin{tikzpicture}[scale=0.2,baseline=-5]
		\coordinate (root) at (0,1);
		\coordinate (tri) at (0,-1);
		\draw[kernels2] (tri) -- (root);
		\node[var] (rootnode) at (root) {\tiny{$ k $}};
		\node[not] (trinode) at (tri) {};
	\end{tikzpicture} \cdot  \begin{tikzpicture}[scale=0.2,baseline=-5]
	\coordinate (root) at (0,0);
	\coordinate (tri) at (0,-2);
	\coordinate (t1) at (-2,2);
	\coordinate (t2) at (2,2);
	\coordinate (t3) at (0,2);
	\coordinate (t4) at (0,4);
	\coordinate (t41) at (-2,6);
	\coordinate (t42) at (2,6);
	\coordinate (t43) at (0,8);
	\draw[kernels2,tinydots] (t1) -- (root);
	\draw[kernels2] (t2) -- (root);
	\draw[kernels2] (t3) -- (root);
	\draw[symbols] (root) -- (tri);
	\draw[symbols] (t3) -- (t4);
	\draw[kernels2,tinydots] (t4) -- (t41);
	\draw[kernels2] (t4) -- (t42);
	\draw[kernels2] (t4) -- (t43);
	\node[not] (rootnode) at (root) {};
	\node[not] (rootnode) at (t4) {};
	\node[not] (rootnode) at (t3) {};
	\node[not,label= {[label distance=-0.2em]below: \scriptsize  $  $}] (trinode) at (tri) {};
	\node[var] (rootnode) at (t1) {\tiny{$ k_{\tiny{4}} $}};
	\node[var] (rootnode) at (t41) {\tiny{$ k_{\tiny{1}} $}};
	\node[var] (rootnode) at (t42) {\tiny{$ k_{\tiny{3}} $}};
	\node[var] (rootnode) at (t43) {\tiny{$ k_{\tiny{2}} $}};
	\node[var] (trinode) at (t2) {\tiny{$ k_5 $}};
\end{tikzpicture}, \qquad \begin{tikzpicture}[scale=0.2,baseline=-5]
\coordinate (root) at (0,1);
\coordinate (tri) at (0,-1);
\draw[kernels2] (tri) -- (root);
\node[var] (rootnode) at (root) {\tiny{$ k $}};
\node[not] (trinode) at (tri) {};
\end{tikzpicture} \cdot \begin{tikzpicture}[scale=0.2,baseline=-5]
\coordinate (root) at (0,0);
\coordinate (tri) at (0,-2);
\coordinate (t1) at (-2,2);
\coordinate (t2) at (2,2);
\coordinate (t3) at (0,3);
\draw[kernels2,tinydots] (t1) -- (root);
\draw[kernels2] (t2) -- (root);
\draw[kernels2] (t3) -- (root);
\draw[symbols] (root) -- (tri);
\node[not] (rootnode) at (root) {};t
\node[not,label= {[label distance=-0.2em]below: \scriptsize  $  $}] (trinode) at (tri) {};
\node[var] (rootnode) at (t1) {\tiny{$ k_{\tiny{4}} $}};
\node[var] (rootnode) at (t3) {\tiny{$ \ell $}};
\node[var] (trinode) at (t2) {\tiny{$ k_5 $}};
\end{tikzpicture} \cdot \begin{tikzpicture}[scale=0.2,baseline=-5]
\coordinate (root) at (0,0);
\coordinate (tri) at (0,-2);
\coordinate (t1) at (-2,2);
\coordinate (t2) at (2,2);
\coordinate (t3) at (0,3);
\draw[kernels2,tinydots] (t1) -- (root);
\draw[kernels2] (t2) -- (root);
\draw[kernels2] (t3) -- (root);
\draw[symbols] (root) -- (tri);
\node[not] (rootnode) at (root) {};t
\node[not,label= {[label distance=-0.2em]below: \scriptsize  $  $}] (trinode) at (tri) {};
\node[var] (rootnode) at (t1) {\tiny{$ k_{\tiny{1}} $}};
\node[var] (rootnode) at (t3) {\tiny{$ k_{\tiny{2}} $}};
\node[var] (trinode) at (t2) {\tiny{$ k_3 $}};
\end{tikzpicture}, \qquad
	 \begin{tikzpicture}[scale=0.2,baseline=-5]
		\coordinate (root) at (0,2);
		\coordinate (tri) at (0,0);
		\coordinate (trib) at (0,-2);
		\coordinate (t1) at (-2,4);
		\coordinate (t2) at (2,4);
		\coordinate (t3) at (0,5);
		\draw[kernels2,tinydots] (t1) -- (root);
		\draw[kernels2] (t2) -- (root);
		\draw[kernels2] (t3) -- (root);
		\draw[kernels2] (trib) -- (tri);
		\draw[symbols] (root) -- (tri);
		\node[not] (rootnode) at (root) {};
		\node[not] (trinode) at (tri) {};
		\node[var] (rootnode) at (t1) {\tiny{$ k_{\tiny{4}} $}};
		\node[var] (rootnode) at (t3) {\tiny{$ \ell $}};
		\node[var] (trinode) at (t2) {\tiny{$ k_{\tiny{5}} $}};
		\node[not] (trinode) at (trib) {};
	\end{tikzpicture} \cdot \begin{tikzpicture}[scale=0.2,baseline=-5]
	\coordinate (root) at (0,0);
	\coordinate (tri) at (0,-2);
	\coordinate (t1) at (-2,2);
	\coordinate (t2) at (2,2);
	\coordinate (t3) at (0,3);
	\draw[kernels2,tinydots] (t1) -- (root);
	\draw[kernels2] (t2) -- (root);
	\draw[kernels2] (t3) -- (root);
	\draw[symbols] (root) -- (tri);
	\node[not] (rootnode) at (root) {};t
	\node[not,label= {[label distance=-0.2em]below: \scriptsize  $  $}] (trinode) at (tri) {};
	\node[var] (rootnode) at (t1) {\tiny{$ k_{\tiny{1}} $}};
	\node[var] (rootnode) at (t3) {\tiny{$ k_{\tiny{2}} $}};
	\node[var] (trinode) at (t2) {\tiny{$ k_3 $}};
\end{tikzpicture}, \quad  \begin{tikzpicture}[scale=0.2,baseline=-5]
	\coordinate (root) at (0,2);
	\coordinate (tri) at (0,0);
	\coordinate (trib) at (0,-2);
	\coordinate (t1) at (-2,4);
	\coordinate (t2) at (2,4);
	\coordinate (t3) at (0,4);
	\coordinate (t4) at (0,6);
	\coordinate (t41) at (-2,8);
	\coordinate (t42) at (2,8);
	\coordinate (t43) at (0,10);
	\draw[kernels2,tinydots] (t1) -- (root);
	\draw[kernels2] (t2) -- (root);
	\draw[kernels2] (t3) -- (root);
	\draw[symbols] (root) -- (tri);
	\draw[symbols] (t3) -- (t4);
	\draw[kernels2,tinydots] (t4) -- (t41);
	\draw[kernels2] (t4) -- (t42);
	\draw[kernels2] (t4) -- (t43);
	\draw[kernels2] (trib) -- (tri);
	\node[not] (trinode) at (trib) {};
	\node[not] (rootnode) at (root) {};
	\node[not] (rootnode) at (t4) {};
	\node[not] (rootnode) at (t3) {};
	\node[not] (trinode) at (tri) {};
	\node[var] (rootnode) at (t1) {\tiny{$ k_{\tiny{4}} $}};
	\node[var] (rootnode) at (t41) {\tiny{$ k_{\tiny{1}} $}};
	\node[var] (rootnode) at (t42) {\tiny{$ k_{\tiny{3}} $}};
	\node[var] (rootnode) at (t43) {\tiny{$ k_{\tiny{2}} $}};
	\node[var] (trinode) at (t2) {\tiny{$ k_5 $}};
\end{tikzpicture}
\end{equs}
where $ \ell = - k_1 + k_2 + k_3 $ and $ k = -k_4 + \ell + k_5 $. We also give below the decomposition for the decorated tree $ \bar{T} $. Notice that now $ \bar{T}_0 $ could be empty.
\begin{equs}
	\one \cdot  \begin{tikzpicture}[scale=0.2,baseline=-5]
		\coordinate (root) at (0,0);
		\coordinate (tri) at (0,-2);
		\coordinate (t1) at (-2,2);
		\coordinate (t2) at (2,2);
		\coordinate (t3) at (0,2);
		\coordinate (t4) at (0,4);
		\coordinate (t41) at (-2,6);
		\coordinate (t42) at (2,6);
		\coordinate (t43) at (0,8);
		\draw[kernels2,tinydots] (t1) -- (root);
		\draw[kernels2] (t2) -- (root);
		\draw[kernels2] (t3) -- (root);
		\draw[symbols] (root) -- (tri);
		\draw[symbols] (t3) -- (t4);
		\draw[kernels2,tinydots] (t4) -- (t41);
		\draw[kernels2] (t4) -- (t42);
		\draw[kernels2] (t4) -- (t43);
		\node[not] (rootnode) at (root) {};
		\node[not] (rootnode) at (t4) {};
		\node[not] (rootnode) at (t3) {};
		\node[not,label= {[label distance=-0.2em]below: \scriptsize  $  $}] (trinode) at (tri) {};
		\node[var] (rootnode) at (t1) {\tiny{$ k_{\tiny{4}} $}};
		\node[var] (rootnode) at (t41) {\tiny{$ k_{\tiny{1}} $}};
		\node[var] (rootnode) at (t42) {\tiny{$ k_{\tiny{3}} $}};
		\node[var] (rootnode) at (t43) {\tiny{$ k_{\tiny{2}} $}};
		\node[var] (trinode) at (t2) {\tiny{$ k_5 $}};
	\end{tikzpicture}, \qquad \one \cdot \begin{tikzpicture}[scale=0.2,baseline=-5]
		\coordinate (root) at (0,0);
		\coordinate (tri) at (0,-2);
		\coordinate (t1) at (-2,2);
		\coordinate (t2) at (2,2);
		\coordinate (t3) at (0,3);
		\draw[kernels2,tinydots] (t1) -- (root);
		\draw[kernels2] (t2) -- (root);
		\draw[kernels2] (t3) -- (root);
		\draw[symbols] (root) -- (tri);
		\node[not] (rootnode) at (root) {};t
		\node[not,label= {[label distance=-0.2em]below: \scriptsize  $  $}] (trinode) at (tri) {};
		\node[var] (rootnode) at (t1) {\tiny{$ k_{\tiny{4}} $}};
		\node[var] (rootnode) at (t3) {\tiny{$ \ell $}};
		\node[var] (trinode) at (t2) {\tiny{$ k_5 $}};
	\end{tikzpicture} \cdot \begin{tikzpicture}[scale=0.2,baseline=-5]
		\coordinate (root) at (0,0);
		\coordinate (tri) at (0,-2);
		\coordinate (t1) at (-2,2);
		\coordinate (t2) at (2,2);
		\coordinate (t3) at (0,3);
		\draw[kernels2,tinydots] (t1) -- (root);
		\draw[kernels2] (t2) -- (root);
		\draw[kernels2] (t3) -- (root);
		\draw[symbols] (root) -- (tri);
		\node[not] (rootnode) at (root) {};t
		\node[not,label= {[label distance=-0.2em]below: \scriptsize  $  $}] (trinode) at (tri) {};
		\node[var] (rootnode) at (t1) {\tiny{$ k_{\tiny{1}} $}};
		\node[var] (rootnode) at (t3) {\tiny{$ k_{\tiny{2}} $}};
		\node[var] (trinode) at (t2) {\tiny{$ k_3 $}};
	\end{tikzpicture}, \qquad \begin{tikzpicture}[scale=0.2,baseline=-5]
	\coordinate (root) at (0,0);
	\coordinate (tri) at (0,-2);
	\coordinate (t1) at (-2,2);
	\coordinate (t2) at (2,2);
	\coordinate (t3) at (0,3);
	\draw[kernels2,tinydots] (t1) -- (root);
	\draw[kernels2] (t2) -- (root);
	\draw[kernels2] (t3) -- (root);
	\draw[symbols] (root) -- (tri);
	\node[not] (rootnode) at (root) {};t
	\node[not,label= {[label distance=-0.2em]below: \scriptsize  $  $}] (trinode) at (tri) {};
	\node[var] (rootnode) at (t1) {\tiny{$ k_{\tiny{4}} $}};
	\node[var] (rootnode) at (t3) {\tiny{$ \ell $}};
	\node[var] (trinode) at (t2) {\tiny{$ k_5 $}};
\end{tikzpicture} \cdot \begin{tikzpicture}[scale=0.2,baseline=-5]
\coordinate (root) at (0,0);
\coordinate (tri) at (0,-2);
\coordinate (t1) at (-2,2);
\coordinate (t2) at (2,2);
\coordinate (t3) at (0,3);
\draw[kernels2,tinydots] (t1) -- (root);
\draw[kernels2] (t2) -- (root);
\draw[kernels2] (t3) -- (root);
\draw[symbols] (root) -- (tri);
\node[not] (rootnode) at (root) {};t
\node[not,label= {[label distance=-0.2em]below: \scriptsize  $  $}] (trinode) at (tri) {};
\node[var] (rootnode) at (t1) {\tiny{$ k_{\tiny{1}} $}};
\node[var] (rootnode) at (t3) {\tiny{$ k_{\tiny{2}} $}};
\node[var] (trinode) at (t2) {\tiny{$ k_3 $}};
\end{tikzpicture}, \quad  \begin{tikzpicture}[scale=0.2,baseline=-5]
\coordinate (root) at (0,0);
\coordinate (tri) at (0,-2);
\coordinate (t1) at (-2,2);
\coordinate (t2) at (2,2);
\coordinate (t3) at (0,2);
\coordinate (t4) at (0,4);
\coordinate (t41) at (-2,6);
\coordinate (t42) at (2,6);
\coordinate (t43) at (0,8);
\draw[kernels2,tinydots] (t1) -- (root);
\draw[kernels2] (t2) -- (root);
\draw[kernels2] (t3) -- (root);
\draw[symbols] (root) -- (tri);
\draw[symbols] (t3) -- (t4);
\draw[kernels2,tinydots] (t4) -- (t41);
\draw[kernels2] (t4) -- (t42);
\draw[kernels2] (t4) -- (t43);
\node[not] (rootnode) at (root) {};
\node[not] (rootnode) at (t4) {};
\node[not] (rootnode) at (t3) {};
\node[not,label= {[label distance=-0.2em]below: \scriptsize  $  $}] (trinode) at (tri) {};
\node[var] (rootnode) at (t1) {\tiny{$ k_{\tiny{4}} $}};
\node[var] (rootnode) at (t41) {\tiny{$ k_{\tiny{1}} $}};
\node[var] (rootnode) at (t42) {\tiny{$ k_{\tiny{3}} $}};
\node[var] (rootnode) at (t43) {\tiny{$ k_{\tiny{2}} $}};
\node[var] (trinode) at (t2) {\tiny{$ k_5 $}};
\end{tikzpicture}
\end{equs}
We can obtain these terms by iterating a Butcher-Connes-Kreimer type copoduct $ \Delta_{\text{\tiny{BCK}}} : \hat{\CH}_0 \rightarrow \hat{\CH}_0 \otimes \hat{\CH}_0 $, a simple version of the one introduced in \cite{BS}. It is defined recursively by
\begin{equs}
	\Delta_{\text{\tiny{BCK}}} \CI_{o_1}( \lambda_{k}^{\ell}  F ) & = \left(\CI_{o_1}( \lambda_{k}^{\ell}  \cdot ) \, \otimes \id \right) \Delta_{\text{\tiny{BCK}}} F,  \\ 
	\Delta_{\text{\tiny{BCK}}} \CI_{o_2}( \lambda_{k}^{\ell}  F ) & = \left(\CI_{o_2}( \lambda_{k}^{\ell}  \cdot ) \, \otimes \id \right) \Delta_{\text{\tiny{BCK}}} F + \one \otimes  \CI_{o_2}( \lambda_{k}^{\ell}  F ).
\end{equs}
and then extended multiplicatively for the forest product.
Below, we provide some examples of computations:
\begin{equs}
	\Delta_{\text{\tiny{BCK}}}  \begin{tikzpicture}[scale=0.2,baseline=-5]
		\coordinate (root) at (0,0);
		\coordinate (tri) at (0,-2);
		\coordinate (t1) at (-2,2);
		\coordinate (t2) at (2,2);
		\coordinate (t3) at (0,2);
		\coordinate (t4) at (0,4);
		\coordinate (t41) at (-2,6);
		\coordinate (t42) at (2,6);
		\coordinate (t43) at (0,8);
		\draw[kernels2,tinydots] (t1) -- (root);
		\draw[kernels2] (t2) -- (root);
		\draw[kernels2] (t3) -- (root);
		\draw[symbols] (root) -- (tri);
		\draw[symbols] (t3) -- (t4);
		\draw[kernels2,tinydots] (t4) -- (t41);
		\draw[kernels2] (t4) -- (t42);
		\draw[kernels2] (t4) -- (t43);
		\node[not] (rootnode) at (root) {};
		\node[not] (rootnode) at (t4) {};
		\node[not] (rootnode) at (t3) {};
		\node[not,label= {[label distance=-0.2em]below: \scriptsize  $  $}] (trinode) at (tri) {};
		\node[var] (rootnode) at (t1) {\tiny{$ k_{\tiny{4}} $}};
		\node[var] (rootnode) at (t41) {\tiny{$ k_{\tiny{1}} $}};
		\node[var] (rootnode) at (t42) {\tiny{$ k_{\tiny{3}} $}};
		\node[var] (rootnode) at (t43) {\tiny{$ k_{\tiny{2}} $}};
		\node[var] (trinode) at (t2) {\tiny{$ k_5 $}};
	\end{tikzpicture} & =
	\begin{tikzpicture}[scale=0.2,baseline=-5]
		\coordinate (root) at (0,0);
		\coordinate (tri) at (0,-2) ;
		\coordinate (t1) at (-2,2);
		\coordinate (t2) at (2,2);
		\coordinate (t3) at (0,2);
		\coordinate (t4) at (0,4);
		\coordinate (t41) at (-2,6);
		\coordinate (t42) at (2,6);
		\coordinate (t43) at (0,8);
		\draw[kernels2,tinydots] (t1) -- (root);
		\draw[kernels2] (t2) -- (root);
		\draw[kernels2] (t3) -- (root);
		\draw[symbols] (root) -- (tri);
		\draw[symbols] (t3) -- (t4);
		\draw[kernels2,tinydots] (t4) -- (t41);
		\draw[kernels2] (t4) -- (t42);
		\draw[kernels2] (t4) -- (t43);
		\node[not] (rootnode) at (root) {};
		\node[not] (rootnode) at (t4) {};
		\node[not] (rootnode) at (t3) {};
		\node[not,label= {[label distance=-0.2em]below: \scriptsize  $  $} ] (trinode) at (tri) {};
		\node[var] (rootnode) at (t1) {\tiny{$ k_{\tiny{4}} $}};
		\node[var] (rootnode) at (t41) {\tiny{$ k_{\tiny{1}} $}};
		\node[var] (rootnode) at (t42) {\tiny{$ k_{\tiny{3}} $}};
		\node[var] (rootnode) at (t43) {\tiny{$ k_{\tiny{2}} $}};
		\node[var] (trinode) at (t2) {\tiny{$ k_5 $}};
	\end{tikzpicture} \otimes \one 
	+ \one \otimes \begin{tikzpicture}[scale=0.2,baseline=-5]
		\coordinate (root) at (0,0);
		\coordinate (tri) at (0,-2);
		\coordinate (t1) at (-2,2);
		\coordinate (t2) at (2,2);
		\coordinate (t3) at (0,2);
		\coordinate (t4) at (0,4);
		\coordinate (t41) at (-2,6);
		\coordinate (t42) at (2,6);
		\coordinate (t43) at (0,8);
		\draw[kernels2,tinydots] (t1) -- (root);
		\draw[kernels2] (t2) -- (root);
		\draw[kernels2] (t3) -- (root);
		\draw[symbols] (root) -- (tri);
		\draw[symbols] (t3) -- (t4);
		\draw[kernels2,tinydots] (t4) -- (t41);
		\draw[kernels2] (t4) -- (t42);
		\draw[kernels2] (t4) -- (t43);
		\node[not] (rootnode) at (root) {};
		\node[not] (rootnode) at (t4) {};
		\node[not] (rootnode) at (t3) {};
		\node[not,label= {[label distance=-0.2em]below: \scriptsize  $  $}] (trinode) at (tri) {};
		\node[var] (rootnode) at (t1) {\tiny{$ k_{\tiny{4}} $}};
		\node[var] (rootnode) at (t41) {\tiny{$ k_{\tiny{1}} $}};
		\node[var] (rootnode) at (t42) {\tiny{$ k_{\tiny{3}} $}};
		\node[var] (rootnode) at (t43) {\tiny{$ k_{\tiny{2}} $}};
		\node[var] (trinode) at (t2) {\tiny{$ k_5 $}};
	\end{tikzpicture}  +
	\begin{tikzpicture}[scale=0.2,baseline=-5]
		\coordinate (root) at (0,0);
		\coordinate (tri) at (0,-2);
		\coordinate (t1) at (-2,2);
		\coordinate (t2) at (2,2);
		\coordinate (t3) at (0,3);
		\draw[kernels2,tinydots] (t1) -- (root);
		\draw[kernels2] (t2) -- (root);
		\draw[kernels2] (t3) -- (root);
		\draw[symbols] (root) -- (tri);
		\node[not] (rootnode) at (root) {};t
		\node[not,label= {[label distance=-0.2em]below: \scriptsize  $  $}] (trinode) at (tri) {};
		\node[var] (rootnode) at (t1) {\tiny{$ k_{\tiny{4}} $}};
		\node[var] (rootnode) at (t3) {\tiny{$ \ell $}};
		\node[var] (trinode) at (t2) {\tiny{$ k_5 $}};
	\end{tikzpicture} \otimes \begin{tikzpicture}[scale=0.2,baseline=-5]
		\coordinate (root) at (0,0);
		\coordinate (tri) at (0,-2);
		\coordinate (t1) at (-2,2);
		\coordinate (t2) at (2,2);
		\coordinate (t3) at (0,3);
		\draw[kernels2,tinydots] (t1) -- (root);
		\draw[kernels2] (t2) -- (root);
		\draw[kernels2] (t3) -- (root);
		\draw[symbols] (root) -- (tri);
		\node[not] (rootnode) at (root) {};t
		\node[not,label= {[label distance=-0.2em]below: \scriptsize  $ $}] (trinode) at (tri) {};
		\node[var] (rootnode) at (t1) {\tiny{$ k_{\tiny{1}} $}};
		\node[var] (rootnode) at (t3) {\tiny{$ k_{\tiny{2}} $}};
		\node[var] (trinode) at (t2) {\tiny{$ k_3 $}};
	\end{tikzpicture}  
	\\ \Delta_{\text{\tiny{BCK}}} \begin{tikzpicture}[scale=0.2,baseline=-5]
		\coordinate (root) at (0,0);
		\coordinate (tri) at (0,-2);
		\coordinate (t1) at (-2,2);
		\coordinate (t2) at (2,2);
		\coordinate (t3) at (0,3);
		\draw[kernels2,tinydots] (t1) -- (root);
		\draw[kernels2] (t2) -- (root);
		\draw[kernels2] (t3) -- (root);
		\draw[symbols] (root) -- (tri);
		\node[not] (rootnode) at (root) {};t
		\node[not,label= {[label distance=-0.2em]below: \scriptsize  $ $}] (trinode) at (tri) {};
		\node[var] (rootnode) at (t1) {\tiny{$ k_{\tiny{1}} $}};
		\node[var] (rootnode) at (t3) {\tiny{$ k_{\tiny{2}} $}};
		\node[var] (trinode) at (t2) {\tiny{$ k_3 $}};
	\end{tikzpicture} & =   \begin{tikzpicture}[scale=0.2,baseline=-5]
		\coordinate (root) at (0,0);
		\coordinate (tri) at (0,-2);
		\coordinate (t1) at (-2,2);
		\coordinate (t2) at (2,2);
		\coordinate (t3) at (0,3);
		\draw[kernels2,tinydots] (t1) -- (root);
		\draw[kernels2] (t2) -- (root);
		\draw[kernels2] (t3) -- (root);
		\draw[symbols] (root) -- (tri);
		\node[not] (rootnode) at (root) {};t
		\node[not,label= {[label distance=-0.2em]below: \scriptsize  $ $}] (trinode) at (tri) {};
		\node[var] (rootnode) at (t1) {\tiny{$ k_{\tiny{1}} $}};
		\node[var] (rootnode) at (t3) {\tiny{$ k_{\tiny{2}} $}};
		\node[var] (trinode) at (t2) {\tiny{$ k_3 $}};
	\end{tikzpicture} \otimes \one + \one \otimes  \begin{tikzpicture}[scale=0.2,baseline=-5]
		\coordinate (root) at (0,0);
		\coordinate (tri) at (0,-2);
		\coordinate (t1) at (-2,2);
		\coordinate (t2) at (2,2);
		\coordinate (t3) at (0,3);
		\draw[kernels2,tinydots] (t1) -- (root);
		\draw[kernels2] (t2) -- (root);
		\draw[kernels2] (t3) -- (root);
		\draw[symbols] (root) -- (tri);
		\node[not] (rootnode) at (root) {};t
		\node[not,label= {[label distance=-0.2em]below: \scriptsize  $ $}] (trinode) at (tri) {};
		\node[var] (rootnode) at (t1) {\tiny{$ k_{\tiny{1}} $}};
		\node[var] (rootnode) at (t3) {\tiny{$ k_{\tiny{2}} $}};
		\node[var] (trinode) at (t2) {\tiny{$ k_3 $}};
	\end{tikzpicture} 
\end{equs}
Below, we introduce recursive maps $ \psi_{\text{\tiny{BCK}}} $ and $ \tilde{\psi}_{\text{\tiny{BCK}}} $ that can compute the splitting describe above with the coproduct $ \Delta_{\text{\tiny{BCK}}}  $: 
\begin{equs} \label{splitting_coproduct} \begin{aligned}
\psi_{\text{\tiny{BCK}}} & = \left( \id \otimes \tilde{\psi}_{\text{\tiny{BCK}}} \right) \Delta_{\text{\tiny{BCK}}} , \\	\tilde{\psi}_{\text{\tiny{BCK}}} & = \mathcal{M}\left( \tilde{\psi}_{\text{\tiny{BCK}}}  \otimes P_{\one} \right) \Delta_{\text{\tiny{BCK}}} , \quad \tilde{\psi}_{\text{\tiny{BCK}}}(\one) = \one
\end{aligned}
\end{equs}
where $ \mathcal{M} $ is the forest product and $ P_{\one} = \id - \one^* $ is 
is the augmentation projector. Here $\one^*$ is the co-unit which is non-zero and equal to one only on the empty forest. The projector $ P_{\one} $ forces at least one cut at each iteration and therefore the recursion is well-defined. 
 If we apply $ \psi $ to $ T $, we obtain a linear combination of the terms of the form $ T_0 \otimes T_1 \cdot ... \cdot T_m $ that corresponds exactly to the splitting described above.
 We do not get a forest in the end but a term with a tensor product. This is for distinguishing the root as it is needed in our splitting.
  As an example of computation of those maps, one has
\begin{equs}
	\tilde{\psi} (  \begin{tikzpicture}[scale=0.2,baseline=-5]
		\coordinate (root) at (0,0);
		\coordinate (tri) at (0,-2);
		\coordinate (t1) at (-2,2);
		\coordinate (t2) at (2,2);
		\coordinate (t3) at (0,3);
		\draw[kernels2,tinydots] (t1) -- (root);
		\draw[kernels2] (t2) -- (root);
		\draw[kernels2] (t3) -- (root);
		\draw[symbols] (root) -- (tri);
		\node[not] (rootnode) at (root) {};t
		\node[not,label= {[label distance=-0.2em]below: \scriptsize  $ $}] (trinode) at (tri) {};
		\node[var] (rootnode) at (t1) {\tiny{$ k_{\tiny{1}} $}};
		\node[var] (rootnode) at (t3) {\tiny{$ k_{\tiny{2}} $}};
		\node[var] (trinode) at (t2) {\tiny{$ k_3 $}};
	\end{tikzpicture} ) & =   \begin{tikzpicture}[scale=0.2,baseline=-5]
	\coordinate (root) at (0,0);
	\coordinate (tri) at (0,-2);
	\coordinate (t1) at (-2,2);
	\coordinate (t2) at (2,2);
	\coordinate (t3) at (0,3);
	\draw[kernels2,tinydots] (t1) -- (root);
	\draw[kernels2] (t2) -- (root);
	\draw[kernels2] (t3) -- (root);
	\draw[symbols] (root) -- (tri);
	\node[not] (rootnode) at (root) {};t
	\node[not,label= {[label distance=-0.2em]below: \scriptsize  $ $}] (trinode) at (tri) {};
	\node[var] (rootnode) at (t1) {\tiny{$ k_{\tiny{1}} $}};
	\node[var] (rootnode) at (t3) {\tiny{$ k_{\tiny{2}} $}};
	\node[var] (trinode) at (t2) {\tiny{$ k_3 $}};
\end{tikzpicture}
\\ \tilde{\psi} ( \begin{tikzpicture}[scale=0.2,baseline=-5]
	\coordinate (root) at (0,0);
	\coordinate (tri) at (0,-2);
	\coordinate (t1) at (-2,2);
	\coordinate (t2) at (2,2);
	\coordinate (t3) at (0,2);
	\coordinate (t4) at (0,4);
	\coordinate (t41) at (-2,6);
	\coordinate (t42) at (2,6);
	\coordinate (t43) at (0,8);
	\draw[kernels2,tinydots] (t1) -- (root);
	\draw[kernels2] (t2) -- (root);
	\draw[kernels2] (t3) -- (root);
	\draw[symbols] (root) -- (tri);
	\draw[symbols] (t3) -- (t4);
	\draw[kernels2,tinydots] (t4) -- (t41);
	\draw[kernels2] (t4) -- (t42);
	\draw[kernels2] (t4) -- (t43);
	\node[not] (rootnode) at (root) {};
	\node[not] (rootnode) at (t4) {};
	\node[not] (rootnode) at (t3) {};
	\node[not,label= {[label distance=-0.2em]below: \scriptsize  $  $}] (trinode) at (tri) {};
	\node[var] (rootnode) at (t1) {\tiny{$ k_{\tiny{4}} $}};
	\node[var] (rootnode) at (t41) {\tiny{$ k_{\tiny{1}} $}};
	\node[var] (rootnode) at (t42) {\tiny{$ k_{\tiny{3}} $}};
	\node[var] (rootnode) at (t43) {\tiny{$ k_{\tiny{2}} $}};
	\node[var] (trinode) at (t2) {\tiny{$ k_5 $}};
\end{tikzpicture} ) & = \begin{tikzpicture}[scale=0.2,baseline=-5]
\coordinate (root) at (0,0);
\coordinate (tri) at (0,-2);
\coordinate (t1) at (-2,2);
\coordinate (t2) at (2,2);
\coordinate (t3) at (0,2);
\coordinate (t4) at (0,4);
\coordinate (t41) at (-2,6);
\coordinate (t42) at (2,6);
\coordinate (t43) at (0,8);
\draw[kernels2,tinydots] (t1) -- (root);
\draw[kernels2] (t2) -- (root);
\draw[kernels2] (t3) -- (root);
\draw[symbols] (root) -- (tri);
\draw[symbols] (t3) -- (t4);
\draw[kernels2,tinydots] (t4) -- (t41);
\draw[kernels2] (t4) -- (t42);
\draw[kernels2] (t4) -- (t43);
\node[not] (rootnode) at (root) {};
\node[not] (rootnode) at (t4) {};
\node[not] (rootnode) at (t3) {};
\node[not,label= {[label distance=-0.2em]below: \scriptsize  $  $}] (trinode) at (tri) {};
\node[var] (rootnode) at (t1) {\tiny{$ k_{\tiny{4}} $}};
\node[var] (rootnode) at (t41) {\tiny{$ k_{\tiny{1}} $}};
\node[var] (rootnode) at (t42) {\tiny{$ k_{\tiny{3}} $}};
\node[var] (rootnode) at (t43) {\tiny{$ k_{\tiny{2}} $}};
\node[var] (trinode) at (t2) {\tiny{$ k_5 $}};
\end{tikzpicture} + \tilde{\psi} \left( \begin{tikzpicture}[scale=0.2,baseline=-5]
\coordinate (root) at (0,0);
\coordinate (tri) at (0,-2);
\coordinate (t1) at (-2,2);
\coordinate (t2) at (2,2);
\coordinate (t3) at (0,3);
\draw[kernels2,tinydots] (t1) -- (root);
\draw[kernels2] (t2) -- (root);
\draw[kernels2] (t3) -- (root);
\draw[symbols] (root) -- (tri);
\node[not] (rootnode) at (root) {};t
\node[not,label= {[label distance=-0.2em]below: \scriptsize  $  $}] (trinode) at (tri) {};
\node[var] (rootnode) at (t1) {\tiny{$ k_{\tiny{4}} $}};
\node[var] (rootnode) at (t3) {\tiny{$ \ell $}};
\node[var] (trinode) at (t2) {\tiny{$ k_5 $}};
\end{tikzpicture} \right) \cdot \begin{tikzpicture}[scale=0.2,baseline=-5]
\coordinate (root) at (0,0);
\coordinate (tri) at (0,-2);
\coordinate (t1) at (-2,2);
\coordinate (t2) at (2,2);
\coordinate (t3) at (0,3);
\draw[kernels2,tinydots] (t1) -- (root);
\draw[kernels2] (t2) -- (root);
\draw[kernels2] (t3) -- (root);
\draw[symbols] (root) -- (tri);
\node[not] (rootnode) at (root) {};t
\node[not,label= {[label distance=-0.2em]below: \scriptsize  $ $}] (trinode) at (tri) {};
\node[var] (rootnode) at (t1) {\tiny{$ k_{\tiny{1}} $}};
\node[var] (rootnode) at (t3) {\tiny{$ k_{\tiny{2}} $}};
\node[var] (trinode) at (t2) {\tiny{$ k_3 $}};
\end{tikzpicture}  = \begin{tikzpicture}[scale=0.2,baseline=-5]
\coordinate (root) at (0,0);
\coordinate (tri) at (0,-2);
\coordinate (t1) at (-2,2);
\coordinate (t2) at (2,2);
\coordinate (t3) at (0,2);
\coordinate (t4) at (0,4);
\coordinate (t41) at (-2,6);
\coordinate (t42) at (2,6);
\coordinate (t43) at (0,8);
\draw[kernels2,tinydots] (t1) -- (root);
\draw[kernels2] (t2) -- (root);
\draw[kernels2] (t3) -- (root);
\draw[symbols] (root) -- (tri);
\draw[symbols] (t3) -- (t4);
\draw[kernels2,tinydots] (t4) -- (t41);
\draw[kernels2] (t4) -- (t42);
\draw[kernels2] (t4) -- (t43);
\node[not] (rootnode) at (root) {};
\node[not] (rootnode) at (t4) {};
\node[not] (rootnode) at (t3) {};
\node[not,label= {[label distance=-0.2em]below: \scriptsize  $  $}] (trinode) at (tri) {};
\node[var] (rootnode) at (t1) {\tiny{$ k_{\tiny{4}} $}};
\node[var] (rootnode) at (t41) {\tiny{$ k_{\tiny{1}} $}};
\node[var] (rootnode) at (t42) {\tiny{$ k_{\tiny{3}} $}};
\node[var] (rootnode) at (t43) {\tiny{$ k_{\tiny{2}} $}};
\node[var] (trinode) at (t2) {\tiny{$ k_5 $}};
\end{tikzpicture} +  \begin{tikzpicture}[scale=0.2,baseline=-5]
\coordinate (root) at (0,0);
\coordinate (tri) at (0,-2);
\coordinate (t1) at (-2,2);
\coordinate (t2) at (2,2);
\coordinate (t3) at (0,3);
\draw[kernels2,tinydots] (t1) -- (root);
\draw[kernels2] (t2) -- (root);
\draw[kernels2] (t3) -- (root);
\draw[symbols] (root) -- (tri);
\node[not] (rootnode) at (root) {};t
\node[not,label= {[label distance=-0.2em]below: \scriptsize  $  $}] (trinode) at (tri) {};
\node[var] (rootnode) at (t1) {\tiny{$ k_{\tiny{4}} $}};
\node[var] (rootnode) at (t3) {\tiny{$ \ell $}};
\node[var] (trinode) at (t2) {\tiny{$ k_5 $}};
\end{tikzpicture}  \cdot \begin{tikzpicture}[scale=0.2,baseline=-5]
\coordinate (root) at (0,0);
\coordinate (tri) at (0,-2);
\coordinate (t1) at (-2,2);
\coordinate (t2) at (2,2);
\coordinate (t3) at (0,3);
\draw[kernels2,tinydots] (t1) -- (root);
\draw[kernels2] (t2) -- (root);
\draw[kernels2] (t3) -- (root);
\draw[symbols] (root) -- (tri);
\node[not] (rootnode) at (root) {};t
\node[not,label= {[label distance=-0.2em]below: \scriptsize  $ $}] (trinode) at (tri) {};
\node[var] (rootnode) at (t1) {\tiny{$ k_{\tiny{1}} $}};
\node[var] (rootnode) at (t3) {\tiny{$ k_{\tiny{2}} $}};
\node[var] (trinode) at (t2) {\tiny{$ k_3 $}};
\end{tikzpicture} 
\end{equs}
\end{example}
\begin{theorem} \label{forest_formula_1}
	For every forest $ F \in \hat{H}_0$, $ (\Pi^{n,r} F)(t,\tau) $ takes the form:
	\begin{equs}	\label{forest_formula} \begin{aligned}
		&  \sum_{\mathbf{a} \in [ 0, 1]^{\tilde{E}_F}} \sum_{F_0 \cdot T_1 ... \cdot T_{{m}} \subset F} C_F \, e^{i t \mathscr{F}_{\text{\tiny{dom}}}(F_0) } \prod_{j=0}^m \prod_{e \in \tilde{E}_{T_j}} e^{i \tau a_e\mathscr{F}_{\text{\tiny{low}}}(T_j^e)} \\ &  b_{\mathbf{a}, F, F_0 \cdot ... \cdot T_{{m}}}(t,\tau,i \tau \mathscr{F}_{\text{\tiny{dom}}}(T_j), i \tau \mathscr{F}_{\text{\tiny{dom}}}(T_{0,\bar{j}}))
		\end{aligned}
	\end{equs}
with the convention $ T_0 = F_0 $. The coefficients $ C_F $ depend only on the node decorations of $ F $. The coefficients $ b $ are polynomial in $ t $  and are non zero only for a finite number of values of $ \mathbf{a} $. \georginline{They are uniformly bounded in $\tau$.} Moreover, they do not depend on the node decorations of the $ T_i, F_0$ and $ F $ that correspond to the frequencies.
	\end{theorem}
	We recall that in the above notation the parameters $n$ and $r$ in $ (\Pi^{n,r} F)(t,\tau)$ denote the regularity requirements and maximum length of trees in the approximation respectively.
\begin{proof}
	We proceed by induction on the size of the forest $ F $. For the empty forest, the sum is equal to one by convention and 
	\begin{equs}
		(\Pi^{n,r}  \one)(t,\tau) = 1.
		\end{equs}
	 Let $ F_1, F_2 $ two decorated forests with $ F = F_1 \cdot F_2 $ for which we have \eqref{forest_formula}. We apply the induction hypothesis and get
	 \begin{equs}
	 	& \left(	\Pi^{n,r}  F \right)(s,\tau) =    
	 	\left( \Pi^{n,r} F_1  \right)(s,\tau)  \left( \Pi^{n,r}   F_2 \right)(s,\tau)
	 	\\	&  =  \sum_{\mathbf{a}_1  \in [ 0, 1]^{\tilde{E}_{F_1}}} \sum_{\mathbf{a}_2 \in [ 0, 1]^{\tilde{E}_{F_2}}}  \sum_{F_{1,0} \cdot T_{1,1} ... \cdot T_{1,{m_1}} \subset F_1}\sum_{F_{2,0} \cdot T_{2,1} ... \cdot T_{2,{m_2}} \subset F_2} \, C_{F_1}  C_{F_2}
	 		\\ &  e^{i t \left( \mathscr{F}_{\text{\tiny{dom}}}(F_{1,0}) + \mathscr{F}_{\text{\tiny{dom}}}(F_{2,0}) \right) }   \prod_{j=0}^{m_1} \prod_{e \in \tilde{E}_{T_{1,j}}} e^{i \tau a_{1,e} \mathscr{F}_{\text{\tiny{low}}}(T_{1,j}^e)} \times  \prod_{j=0}^{m_2} \prod_{e \in \tilde{E}_{T_{2,j}}} e^{i \tau a_{2,e}\mathscr{F}_{\text{\tiny{low}}}(T_{2,j}^e)} 
	 		 \\ &  b_{\mathbf{a}_1, F_1, F_{1,0} \cdot ... \cdot T_{{1,m_1}}} \times b_{\mathbf{a}_2, F_2, F_{2,0} \cdot ... \cdot T_{{2,m_2}}} .
	 	\\ & 
	 \end{equs}
 By using Definition~\ref{dom_freq}, we have
 \begin{equs}
 	\mathscr{F}_{\text{\tiny{dom}}}(F_{1,0}) + \mathscr{F}_{\text{\tiny{dom}}}(F_{2,0}) = \mathscr{F}_{\text{\tiny{dom}}}(F_{0}), \quad F_0 = F_{1,0} \cdot F_{2,0}.
 	\end{equs}
 Then, we can perform the disjoint sum of $ \mathbf{a}_1 $ and $ \mathbf{a}_1 $:
 \begin{equs}
 	\mathbf{a} = \mathbf{a}_1  + \mathbf{a}_2 
 \end{equs}
by extending $ \mathbf{a}_1  $ (resp. $ \mathbf{a}_2  $) on the edges of $ F_2 $ (resp. $ F_1 $) by zero. Then, $ \mathbf{a} $ is defined on the edges of $ F $. We can gather the sum on the forests by: 
 \begin{equs}
 	 \sum_{F_{1,0} \cdot T_{1,1} ... \cdot T_{1,{m_1}} \subset F_1}\sum_{F_{2,0} \cdot T_{2,1} ... \cdot T_{2,{m_2}} \subset F_2} & =  \sum_{F_{0} \cdot T_{1,1} ... \cdot T_{1,{m_1}} \cdot  T_{2,1} ... \cdot T_{2,{m_2}}  \subset F} \\ &= \sum_{F_{0} \cdot T_{1} ... \cdot T_{m} \subset F}
 \end{equs}
and we can also set
\begin{equs}
b_{\mathbf{a}, F, F_{0} \cdot ... \cdot T_{1} \cdot ... \cdot T_{m}}  &  = 	  b_{\mathbf{a}_1, F_1, F_{1,0} \cdot ... \cdot T_{{1,m_1}}} \times b_{\mathbf{a}_2, F_2, F_{2,0} \cdot ... \cdot T_{{2,m_2}}} \\
C_F & = C_{F_1} \times C_{F_2} 
\end{equs}
and see that the properties of the coefficient $ b $ are preserved by multiplication. Indeed, we have a bijection between partitions of $ F = F_1 \cdot F_2 $ into a product of trees and the forest product of partitions of $ F_i $. For a tree of the form $ \CI_{o_1}( \lambda_{k}^{\ell}  F ) $, one has 
	\begin{equs}
			(\Pi^n  \CI^{r}_{o_1}( \lambda_{k}^{\ell}  F ))(t,\tau)   & =t^{\ell_1} \tau^{\ell_2} e^{i t P_{o_1}(k)} (\Pi^{n,r-\ell} F)(t,\tau).
	\end{equs} 
We multiply the formula for $ F $ obtained by the induction hypothesis by $ e^{i t P_{o_1}(k)} $:
\begin{equs}
	&  \sum_{\mathbf{a} \in [ 0, 1]^{\tilde{E}_F}} \sum_{F_0 \cdot T_1 ... \cdot T_{{m}} \subset F} C_F e^{i t (\mathscr{F}_{\text{\tiny{dom}}}(F_0) + P_{o_1}(k) ) } \prod_{j=0}^m \prod_{e \in \tilde{E}_{T_j}} e^{i \tau a_e\mathscr{F}_{\text{\tiny{low}}}(T_j^e)} \\ & t^{\ell} b_{\mathbf{a}, F, F_0 \cdot ... \cdot T_{{m}}}(t,\tau,i \tau \mathscr{F}_{\text{\tiny{dom}}}(T_j), i \tau \mathscr{F}_{\text{\tiny{dom}}}(T_{0,\bar{j}})).
\end{equs}  
From Definition~\ref{dom_freq}, we have
\begin{equs}
	\mathscr{F}_{\text{\tiny{dom}}}(F_0) + P_{o_1}(k)  = 
	\mathscr{F}_{\text{\tiny{dom}}}(\CI_{o_1}( \lambda_{k}^{\ell}  F_0 )).
\end{equs}
Moreover, for $ \CI_{o_1}( \lambda_{k}^{\ell}  F ) $, the forest at the root must be of the form $  \CI_{o_1}( \lambda_{k}^{\ell}  F_0 )  $.
For the coefficients $ b $, we have
\begin{equs}
	t^{\ell_1} \tau^{\ell_2} b_{\mathbf{a}, F, F_0 \cdot ... \cdot T_{{m}}} = b_{\mathbf{a}, \CI_{o_1}( \lambda_{k}^{\ell}  F ), \CI_{o_1}( \lambda_{k}^{\ell}  F_0 ) \cdot ... \cdot T_{{m}}}, \quad C_{\CI_{o_1}( \lambda_{k}^{\ell}  F )} = C_F
\end{equs}
It remains to prove the forest formula for a tree of the form $ \CI_{o_2}( \lambda_{k}^{\ell}  F ) $. We have
\begin{equs}
		\left( \Pi^n \CI^{r}_{o_2}( \lambda^{\ell}_k F) \right) (t,\tau) & = \CK^{k,r}_{o_2} \left(  \Pi^n \left( \lambda^{\ell} \CD^{r-|\ell|-1}(F) \right)(\cdot,\tau),n \right)(t).
\end{equs}
We apply the induction hypothesis on $ \Pi^{n,r-|\ell|-1} $. 
Then, the proof boils down to understand how the operator $ \CK^{k,r}_{o_2} $
acts on the forest formula. This operator computes first the dominant part of the oscillation. It is given for a fixed forest $ F_0 \cdot T_1 ... \cdot T_{{m}} \subset F $ by
\begin{equs}
\CP_{\text{\tiny{dom}}}\left( P_{o_2}(k) +\mathscr{F}_{\text{\tiny{dom}}}(F_0) \right) = \mathscr{F}_{\text{\tiny{dom}}}(\CI_{o_2}( \lambda_{k}^{\ell}  F_0 )).
\end{equs} 
If this dominant part is integrated exactly, we obtain a factor of the form
\begin{equs} 
	e^{i t \mathscr{F}_{\text{\tiny{dom}}}(\CI_{o_2}( \lambda_{k}^{\ell}  F_0 )) }.
	\end{equs}
This will correspond to forests $ \CI_{o_2}( \lambda_{k}^{\ell}  F_0 ) \cdot ... \cdot T_m $. In this exact integration, we have terms without this factor which corresponds to the forest $ \one \cdot \CI_{o_2}( \lambda_{k}^{\ell}  F_0 ) \cdot ... \cdot T_m $ as the forest connected to the root could be the empty forest. 
For the lower part given by 
\begin{equs} 
	\left( \id - \CP_{\text{\tiny{dom}}} \right) \left( P_{o_2}(k) +\mathscr{F}_{\text{\tiny{dom}}}(F) \right) = \mathscr{F}_{\text{\tiny{low}}}(\CI_{o_2}( \lambda_{k}^{\ell}  F_0 ))
\end{equs}
we perform an interpolation that produces terms of the form
\begin{equs}
	e^{i \tau a \mathscr{F}_{\text{\tiny{low}}}(\CI_{o_2}( \lambda_{k}^{\ell}  F_0 )))}, \quad a \in [0,1].
\end{equs}
The operator $  \CK^{k,r}_{o_2}  $ depends on $ n $ which is the a priori regularity assumed on the initial data. With this information, if $ n $ is sufficiently big, we can perform a full Taylor expansion via an interpolation that will produce terms of the form:
 \begin{equs}
 		e^{i \tau a \mathscr{F}_{\text{\tiny{low}}}(\CI_{o_2}( \lambda_{k}^{\ell}  F_0 )))} e^{i \tau a' \mathscr{F}_{\text{\tiny{dom}}}(\CI_{o_2}( \lambda_{k}^{\ell}  F_0 )))}, \quad a, a' \in [0,1].
 \end{equs}
These terms will be associated to a forest of the form $ \one \cdot \CI_{o_2}( \lambda_{k}^{\ell}  F_0 ) \cdot ... \cdot T_m $. 
The factor $ e^{i \tau a' \mathscr{F}_{\text{\tiny{dom}}}(\CI_{o_2}( \lambda_{k}^{\ell}  F_0 )))} $ will be inside the coefficients $ b $. The interpolation also produces monomials in $ t $ and $ \tau $ which implies the polynomial structure of the coefficients $ b $ in $ t $. Also, it produces coefficients bounded in $ \tau $ such as the $\hat{p}_{j,r}(f,\tau)$ given in \eqref{p_hat}.  The choice of the $ a \in [0,1] $ are fixed by the interpolation method and one uses only a finite number of them which implies that the  coefficients $ b $ are non-zero on a finite set of the $ \mathbf{a}$. Finally, we have
\begin{equs}
	C_{\CI_{o_2}( \lambda_{k}^{\ell}F)} = 	- i \vert \nabla\vert^{\alpha} (k) C_F.
\end{equs}
	\end{proof}

	\begin{theorem} \label{forest_formula_2}
		For every decorated tree $ T = \CI_{o_2}( \lambda_{k}  F )  $, $ (\Pi^{n,r} T)(\tau,\tau) $ takes the form:
		\begin{equs}	\label{forest_formula_21} \begin{aligned}
				\hspace{-1cm} \sum_{\mathbf{a} \in [ 0, 1]^{\tilde{E}_F}} \sum_{T_0 \cdot T_1 ... \cdot T_{{m}} \subset T} C_T \prod_{j=0}^m \prod_{e \in \tilde{E}_{T_j}}  e^{i \tau a_e \mathscr{F}_{\text{\tiny{low}}}(T_j^e)}   b_{\mathbf{a}, T, T_0 \cdot ... \cdot T_{{m}}}(\tau,i \tau \mathscr{F}_{\text{\tiny{dom}}}(T_j))
			\end{aligned}
		\end{equs}	
	where the coefficients $ b $ are polynomial in $ \tau $ with bounded coefficient in $ \tau $ and are non zero for finite values of $ \mathbf{a} $. Moreover, they do not depend on the nodes decorations of the $ T_i, F_0$ and $ F $ that correspond to the frequencies.
	\end{theorem}
\begin{proof} The proof works mostly in the same way as for Theorem~\ref{forest_formula_1} but with $ t=\tau $. The main difference is that when we apply the operator $ \CK^{k,r}_{o_2} $, we put factors of the form $ 	e^{i t \mathscr{F}_{\text{\tiny{dom}}}(\CI_{o_2}( \lambda_{k}  F_0 )) } $ in the coefficients $ b $.  
	\end{proof}

\subsection{A more general forest formula}\label{sec:more_general_forest_formula}
While we will observe that the above forest formula is sufficient for the characterisation of symmetric resonance based schemes, we can actually allow for an even larger number of degrees of freedom in the context of this algebraic structure. This idea leads to a more general forest formula, introduced in the present section, which might be exploited in finding schemes more general than the resonance based methods discussed in the present work. We first introduce a new space of decorated forests with trees having an extra decoration at the root. A tree in $ \CT_+ $ is of the form 
\begin{equs}
	\mathcal{I}^{(r,m)}_o(\lambda_k^{\ell} F), \quad \mathcal{I}^{r}_o(\lambda_k^{\ell} F) \in \CT, 
\end{equs}
with $ m \in \N^2 $ and the additional constraint that $ |m| \leq r+1 $ in order to be non zero. We also assume that $ \lambda^{\ell} $ does not appear in $ \CT_+ $. We set $ \CH_+ $ to be the linear span of $ \CH $.

We define an extension of $ \Delta_{\text{\tiny{BCK}}} $ using the symbolic notation given  by  $  \Delta : \CH \rightarrow \CH \otimes \CH_+ $ and $  \Deltap : \CH_+ \rightarrow \CH_+ \otimes \CH_+ $
\begin{equation} \label{def_deltas}
	\begin{aligned}
		\Delta \one & = \one \otimes \one, \quad \Delta  \lambda^{\ell} =  \lambda^{\ell} \otimes \one\\
		\Delta \CI^{r}_{o_1}( \lambda_{k}^{\ell} F)  & = \left(  \CI^{r}_{o_1}( \lambda_{k}^{\ell}\cdot)  \otimes \id  \right) \Delta \CD^{r-|\ell|}(F) \\
		\Delta \CI^{r}_{o_2}( \lambda_{k}^{\ell} F)  & = \left( \CI^{r}_{o_2}(  \lambda_{k}^{\ell} \cdot)  \otimes \id  \right) \Delta \CD^{r-|\ell|-1}(F) + \sum_{|m| \leq r +1} \frac{ \lambda^{m}}{m!} \otimes  \CI^{(r,m)}_{o_2}( \lambda_{k}^{\ell} F) 
		\\
		\Deltap \CI^{(r,m)}_{o_2}( \lambda_{k}^{\ell} F)  & = \left( \CI^{(r,m)}_{o_2}(  \lambda_{k}^{\ell} \cdot)  \otimes \id  \right) \Delta \CD^{r-|\ell|-1}(F) +  \one \otimes  \CI^{(r,m)}_{o_2}( \lambda_{k}^{\ell} F)
	\end{aligned}
\end{equation}
and it is extended multiplicatively for the forest product. 
 This coproduct is useful for providing a nice factorisation of the discretisation $ \Pi^n $ given in \cite{BS}:
\begin{equation} \label{antipode_def}
	\begin{aligned}
		\Pi^n = \left( \hat  \Pi^n \otimes A^n \right) \Delta
	\end{aligned} 
\end{equation}
where the character $ \hat{\Pi}^n $ singles out oscillations and it is recursively defined by
	\begin{equation} \label{def_A_B}
		\begin{aligned}
			\hat{\Pi}^n \left( F \cdot \bar F \right)(s,\tau) & = 
			\left( \hat{\Pi}^n F  \right)(s,\tau)  \left( \hat{\Pi}^n  \bar F \right)(s,\tau), \quad (\hat{\Pi}^n  \lambda^{\ell})(s,\tau) = s^{\ell_1}\tau^{\ell_2}, \\
			\hat{\Pi}^n(\CI^{r}_{o_1}( \lambda_{k}^{\ell}  F ))(s,\tau)  & = s^{\ell_1}\tau^{\ell_2 }e^{i \tau P_{o_1}(k)} \hat{\Pi}^n (\CD^{r-|\ell|}(F))(s,\tau),  \\
			\hat \Pi^n(\CI^{r}_{o_2}( \lambda_{k}^{\ell}  F )) (s,\tau)  & = \CK_{o_2,+}^{k,r}( \hat \Pi^n(  \lambda^{\ell} \CD^{r-|\ell|-1}(F))(\cdot,\tau),n   )(s)
		\end{aligned}
	\end{equation} 
with
\begin{equs}\label{K}
	\CK^{k,r}_{o_2,+}:=\left( \id - \CQ \right) \circ \CK^{k,r}_{o_2} 
\end{equs}
and $ \CQ $ is the projection that send to zero
   functions of the form 
   \begin{equs}
   	z \mapsto \sum_j Q_j(z)e^{i z P_j(k_1,...,k_n) }, \quad P_j \neq 0.
   \end{equs}
   The character $ A^n : \CH_+ \rightarrow \C  $ applied to $ \CI^{(r,m)}_{o_2}(  \lambda_{k}^{\ell} F) $ is extracting the coefficient of $ t^{m_1} \tau^{m_2} $ multiplied by $ m! $  in $ \Pi^n \CI^{r}_{o_2}(  \lambda_{k}^{\ell} F) $. In \eqref{antipode_def},  we do not need a multiplication because $ A^n(F) \in \C $ for every $ F \in \CH_+ $ and $ \CC$ is a $ \C $-vector space. Therefore, we use the identification $ \mathcal{C} \otimes \C \cong \mathcal{C} $.
One main property observed in \cite{BS} \georginline{about} $ \hat{\Pi}^n $ is the following factorisation:
	For every forest  $  F \in \hat \CH $, there exists a polynomial 
	$ B^n\left( \CD^r(F) \right){ (\xi,\tau)} $  such that
	\begin{equs} \label{simple_formula}
		\hat \Pi^n\left( \CD^r(F) \right)(\xi,\tau) =  B^n\left( \CD^r(F) \right)(\xi,\tau) e^{i \xi\mathscr{F}_{\text{\tiny{dom}}}( F)}
	\end{equs}
	where $ \mathscr{F}_{\text{\tiny{dom}}}(F)  $ is given in Definition~\ref{dom_freq}. 
	
	\begin{proposition}
		\label{forest_formula_2_bis}
		For every forest $ F = \prod_{j} T_j $ with $ T_j $ planted trees, $ (\hat{\Pi}^{n,r} F)(t,\tau) $ takes the form:
		\begin{equs}	\label{forest_formula_bis} \begin{aligned}
				&  \sum_{\mathbf{a} \in [ 0, 1]^{\tilde{E}_F}} C_F e^{i t \mathscr{F}_{\text{\tiny{dom}}}(F) }  \prod_{e \in \tilde{E}_{F}} e^{i \tau a_e\mathscr{F}_{\text{\tiny{low}}}(F^e)}   b_{\mathbf{a}, F}(t,\tau, i \tau \mathscr{F}_{\text{\tiny{dom}}}(T_{j})).
			\end{aligned}
		\end{equs}
		We assume the coefficients $ b $ are polynomial in $ t, \tau $ with bounded coefficients in $ \tau $ and that they are non zero for finite values of $ \mathbf{a} $. They also do not depend on the nodes decorations of $ F $ that correspond to the frequencies.
		\end{proposition}
	\begin{proof} The proof follows the same lines as for Theorem~\ref{forest_formula_1} by using the recursive definition \eqref{def_A_B} of the character $ \hat{\Pi}^{n} $. It can be seen as a refinement of \eqref{simple_formula}. 
		\end{proof}
	
	\begin{theorem} \label{forest_formula_1_2}
		For every forest $ F $, $ (\Pi^{n,r} F)(t,\tau) $ takes the form:
		\begin{equs}	\label{forest_formula_4_bis} \begin{aligned}
				&  \sum_{\mathbf{a} \in [ 0, 1]^{\tilde{E}_F}} \sum_{F_0 \cdot T_1 ... \cdot T_{{m}} \subset F} C_F \, e^{i t \mathscr{F}_{\text{\tiny{dom}}}(F_0) } \prod_{j=0}^m \prod_{e \in \tilde{E}_{T_j}} e^{i \tau a_e\mathscr{F}_{\text{\tiny{low}}}(T_j^e)} \\ &  b_{\mathbf{a}, F, F_0 \cdot ... \cdot T_{{m}}}(t,\tau,i \tau \mathscr{F}_{\text{\tiny{dom}}}(T_j), i \tau \mathscr{F}_{\text{\tiny{dom}}}(T_{0,\bar{j}}))
			\end{aligned}
		\end{equs}
		with the convention $ T_0 = F_0 $.
		We assume the coefficients $ b $ are polynomial in $ t, \tau $ with bounded coefficients in $ \tau $ and that they are non zero for finite values of $ \mathbf{a} $. They also do not depend on the node decorations of the $ F, F_0 $ and the $ T_j $ that correspond to the frequencies. {Now, the forests $ F_0 \cdot T_1 ... \cdot T_{{m}} $ are produced using the map $ \psi  $ below:} 
		\begin{equs} \label{splitting_coproduct_2}
		\psi = \left( \id \otimes \tilde{\psi} \right) \Delta , \quad	\tilde{\psi} = \mathcal{M}\left( \tilde{\psi}  \otimes P_{\one} \right) \Deltap , \quad \tilde{\psi}(\one) = \one.
		\end{equs}
	{The difference between \eqref{splitting_coproduct} and \eqref{splitting_coproduct_2} is that we have replaced $ \Delta_{\text{\tiny{BCK}}} $ by $ \Delta $ and $\Delta^{\!}$.}
	\end{theorem}
\begin{proof} We proceed \georginline{in the same way} as in the proof of Theorem \ref{forest_formula_1}. The main difference happens on a tree of the form $ \CI_{o_2}( \lambda_{k}^{\ell}  F ) $. We have
	\begin{equs}
		\left( \Pi^n \CI^{r}_{o_2}( \lambda^{\ell}_k F) \right) (t,\tau) & = \CK^{k,r}_{o_2} \left(  \Pi^n \left( \lambda^{\ell} \CD^{r-|\ell|-1}(F) \right)(\cdot,\tau),n \right)(t).
	\end{equs}
	We apply the induction hypothesis on $ \Pi^{n,r-\ell-1} $. 
	Then, the proof boils down to understand how the operator $ \CK^{k,r}_{o_2} $
	acts on the forest formula. If the dominant part $\mathscr{F}_{\text{\tiny{dom}}}(\CI_{o_2}( \lambda_{k}^{\ell}  F_0 ))$  is integrated exactly, we obtain a factor of the form
	\begin{equs} 
		e^{i t \mathscr{F}_{\text{\tiny{dom}}}(\CI_{o_2}( \lambda_{k}^{\ell}  F_0 )) }.
	\end{equs}
and this corresponds to forests $ \CI_{o_2}( \lambda_{k}^{\ell}  F_0 ) \cdot ... \cdot T_m $. In this exact integration, we have terms without this factor and a monomial $  t^{q_1} \tau^{q_2} $  which corresponds to the forest $ \lambda^{q} \cdot \CI^{(r,q)}_{o_2}( \lambda_{k}^{\ell}  F_0 ) \cdot ... \cdot T_m $.  Then, the coefficient $ b $ factors out in the following form:
\begin{equs}
	 & b_{\mathbf{a}, \CI^{r}_{o_2}( \lambda_{k}^{\ell} F), \lambda^{q} \cdot \CI^{(r,q)}_{o_2}( \lambda_{k}^{\ell} F_0) \cdot ... \cdot T_{{m}}}(t,\tau,i \tau \mathscr{F}_{\text{\tiny{dom}}}(T_j), i \tau \mathscr{F}_{\text{\tiny{dom}}}(T_{0,\bar{j}}))
	  =  t^{q_1} \tau^{q_2} \\ & \tilde{b}_{\mathbf{a}, \CI^{r}_{o_2}( \lambda_{k}^{\ell}F), \CI^{(r,q)}_{o_2}( \lambda_{k}^{\ell} F_0) \cdot ... \cdot T_{{m}}}(\tau,i \tau \mathscr{F}_{\text{\tiny{dom}}}(T_j), i \tau \mathscr{F}_{\text{\tiny{dom}}}(T_{0,\bar{j}}))
\end{equs}
where the term on the right hand side is bounded in $ \tau  $.
The treatment for the lower part is exactly the same.	
	 If $ n $ is sufficiently big, we can perform a full Taylor expansion via an interpolation that will produce terms of the form:
	\begin{equs}
		e^{i \tau a \mathscr{F}_{\text{\tiny{low}}}(\CI_{o_2}( \lambda_{k}^{\ell}  F_0 )))} e^{i \tau a' \mathscr{F}_{\text{\tiny{dom}}}(\CI_{o_2}( \lambda_{k}^{\ell}  F_0 )))}, \quad a, a' \in [0,1].
	\end{equs}
	These terms will be associated to  forests of the form $ \lambda^{q} \cdot \CI^{(r,q)}_{o_2}( \lambda_{k}^{\ell}  F_0 ) \cdot ... \cdot T_m $. 
	The factor $ e^{i \tau a' \mathscr{F}_{\text{\tiny{dom}}}(\CI_{o_2}( \lambda_{k}^{\ell}  F_0 )))} $ will be inside the coeffcients $ b $. 
	\end{proof}

\begin{remark}
	The forest formula given in \eqref{forest_formula_bis} contains more terms than the one given in \eqref{forest_formula}, and therefore more degrees of freedom that could be exploited for finding new schemes. In the sequel, we will use the first forest formula instead of this one as one can notice that the dominant part $ 
		\mathscr{F}_{\text{\tiny{dom}}}(F_0) $
does not change if we change monomial decorations inside $ F_0 $.
	\end{remark}
\begin{remark}\label{rmk:need_for_iterations_to_do_error_analysis}
		While the general forest formula~\eqref{eqn:general_formula} is suitable for characterising symmetric resonance based schemes (cf. Proposition~\ref{prop:symmetry_conditions}) it does not (at the current point) permit a direct analysis of the local error of the scheme, which for a general symmetric resonance based scheme has to be performed on a case-by-case basis. On the other hand, the direct construction of explicit resonance based schemes in \cite{BS} automatically allows for local error estimates, which motivates the study of general Duhamel iterates and a description of the midpoint iterates using decorated tree series in the following section. It turns out (cf. Section~\ref{sec:symmetric_schemes}) that these iterations lead to a subclass of symmetric schemes captured by the general formula \eqref{eqn:general_formula} (cf. Theorem~\ref{forest_formula_symmetric scheme}) which allows us to construct and analyse a class of symmetric low-regularity integrators of arbitrary given order in a structured way. We will see in Theorem~\ref{thm:genloc} that this iterative approach allows us to derive general estimates on the local error, even in the case of implicit schemes (implicitness is required for symmetry of the scheme).
\end{remark}

\begin{example}\label{ex:forest_splittings_general}
	Below, we provide some examples of computations with $ \Delta^{\!+} $:
	\begin{equs}
		\Delta^{\!+}  \begin{tikzpicture}[scale=0.2,baseline=-5]
			\coordinate (root) at (0,0);
			\coordinate (tri) at (0,-2);
			\coordinate (t1) at (-2,2);
			\coordinate (t2) at (2,2);
			\coordinate (t3) at (0,2);
			\coordinate (t4) at (0,4);
			\coordinate (t41) at (-2,6);
			\coordinate (t42) at (2,6);
			\coordinate (t43) at (0,8);
			\draw[kernels2,tinydots] (t1) -- (root);
			\draw[kernels2] (t2) -- (root);
			\draw[kernels2] (t3) -- (root);
			\draw[symbols] (root) -- (tri);
			\draw[symbols] (t3) -- (t4);
			\draw[kernels2,tinydots] (t4) -- (t41);
			\draw[kernels2] (t4) -- (t42);
			\draw[kernels2] (t4) -- (t43);
			\node[not] (rootnode) at (root) {};
			\node[not] (rootnode) at (t4) {};
			\node[not] (rootnode) at (t3) {};
			\node[not,label= {[label distance=-0.2em]below: \scriptsize  $ (r,r') $}] (trinode) at (tri) {};
			\node[var] (rootnode) at (t1) {\tiny{$ k_{\tiny{4}} $}};
			\node[var] (rootnode) at (t41) {\tiny{$ k_{\tiny{1}} $}};
			\node[var] (rootnode) at (t42) {\tiny{$ k_{\tiny{3}} $}};
			\node[var] (rootnode) at (t43) {\tiny{$ k_{\tiny{2}} $}};
			\node[var] (trinode) at (t2) {\tiny{$ k_5 $}};
		\end{tikzpicture} & =
		\begin{tikzpicture}[scale=0.2,baseline=-5]
			\coordinate (root) at (0,0);
			\coordinate (tri) at (0,-2) ;
			\coordinate (t1) at (-2,2);
			\coordinate (t2) at (2,2);
			\coordinate (t3) at (0,2);
			\coordinate (t4) at (0,4);
			\coordinate (t41) at (-2,6);
			\coordinate (t42) at (2,6);
			\coordinate (t43) at (0,8);
			\draw[kernels2,tinydots] (t1) -- (root);
			\draw[kernels2] (t2) -- (root);
			\draw[kernels2] (t3) -- (root);
			\draw[symbols] (root) -- (tri);
			\draw[symbols] (t3) -- (t4);
			\draw[kernels2,tinydots] (t4) -- (t41);
			\draw[kernels2] (t4) -- (t42);
			\draw[kernels2] (t4) -- (t43);
			\node[not] (rootnode) at (root) {};
			\node[not] (rootnode) at (t4) {};
			\node[not] (rootnode) at (t3) {};
			\node[not,label= {[label distance=-0.2em]below: \scriptsize  $(r,r') $} ] (trinode) at (tri) {};
			\node[var] (rootnode) at (t1) {\tiny{$ k_{\tiny{4}} $}};
			\node[var] (rootnode) at (t41) {\tiny{$ k_{\tiny{1}} $}};
			\node[var] (rootnode) at (t42) {\tiny{$ k_{\tiny{3}} $}};
			\node[var] (rootnode) at (t43) {\tiny{$ k_{\tiny{2}} $}};
			\node[var] (trinode) at (t2) {\tiny{$ k_5 $}};
		\end{tikzpicture} \otimes \one 
		+ \one \otimes \begin{tikzpicture}[scale=0.2,baseline=-5]
			\coordinate (root) at (0,0);
			\coordinate (tri) at (0,-2);
			\coordinate (t1) at (-2,2);
			\coordinate (t2) at (2,2);
			\coordinate (t3) at (0,2);
			\coordinate (t4) at (0,4);
			\coordinate (t41) at (-2,6);
			\coordinate (t42) at (2,6);
			\coordinate (t43) at (0,8);
			\draw[kernels2,tinydots] (t1) -- (root);
			\draw[kernels2] (t2) -- (root);
			\draw[kernels2] (t3) -- (root);
			\draw[symbols] (root) -- (tri);
			\draw[symbols] (t3) -- (t4);
			\draw[kernels2,tinydots] (t4) -- (t41);
			\draw[kernels2] (t4) -- (t42);
			\draw[kernels2] (t4) -- (t43);
			\node[not] (rootnode) at (root) {};
			\node[not] (rootnode) at (t4) {};
			\node[not] (rootnode) at (t3) {};
			\node[not,label= {[label distance=-0.2em]below: \scriptsize  $ (r,r')  $}] (trinode) at (tri) {};
			\node[var] (rootnode) at (t1) {\tiny{$ k_{\tiny{4}} $}};
			\node[var] (rootnode) at (t41) {\tiny{$ k_{\tiny{1}} $}};
			\node[var] (rootnode) at (t42) {\tiny{$ k_{\tiny{3}} $}};
			\node[var] (rootnode) at (t43) {\tiny{$ k_{\tiny{2}} $}};
			\node[var] (trinode) at (t2) {\tiny{$ k_5 $}};
		\end{tikzpicture}  + \sum_{m} \frac{1}{m!}
		\begin{tikzpicture}[scale=0.2,baseline=-5]
			\coordinate (root) at (0,0);
			\coordinate (tri) at (0,-2);
			\coordinate (t1) at (-2,2);
			\coordinate (t2) at (2,2);
			\coordinate (t3) at (0,3);
			\draw[kernels2,tinydots] (t1) -- (root);
			\draw[kernels2] (t2) -- (root);
			\draw[kernels2] (t3) -- (root);
			\draw[symbols] (root) -- (tri);
			\node[not] (rootnode) at (root) {};t
			\node[not,label= {[label distance=-0.2em]below: \scriptsize  $ (r,r') $}] (trinode) at (tri) {};
			\node[var] (rootnode) at (t1) {\tiny{$ k_{\tiny{4}} $}};
			\node[var1] (rootnode) at (t3) {\tiny{$ ^m_{\ell} $}};
			\node[var] (trinode) at (t2) {\tiny{$ k_5 $}};
		\end{tikzpicture} \otimes \begin{tikzpicture}[scale=0.2,baseline=-5]
			\coordinate (root) at (0,0);
			\coordinate (tri) at (0,-2);
			\coordinate (t1) at (-2,2);
			\coordinate (t2) at (2,2);
			\coordinate (t3) at (0,3);
			\draw[kernels2,tinydots] (t1) -- (root);
			\draw[kernels2] (t2) -- (root);
			\draw[kernels2] (t3) -- (root);
			\draw[symbols] (root) -- (tri);
			\node[not] (rootnode) at (root) {};t
			\node[not,label= {[label distance=-0.2em]below: \scriptsize  $ (r-1,m) $}] (trinode) at (tri) {};
			\node[var] (rootnode) at (t1) {\tiny{$ k_{\tiny{1}} $}};
			\node[var] (rootnode) at (t3) {\tiny{$ k_{\tiny{2}} $}};
			\node[var] (trinode) at (t2) {\tiny{$ k_3 $}};
		\end{tikzpicture}  
		\\ \Delta^{\!+}  \begin{tikzpicture}[scale=0.2,baseline=-5]
			\coordinate (root) at (0,0);
			\coordinate (tri) at (0,-2);
			\coordinate (t1) at (-2,2);
			\coordinate (t2) at (2,2);
			\coordinate (t3) at (0,3);
			\draw[kernels2,tinydots] (t1) -- (root);
			\draw[kernels2] (t2) -- (root);
			\draw[kernels2] (t3) -- (root);
			\draw[symbols] (root) -- (tri);
			\node[not] (rootnode) at (root) {};t
			\node[not,label= {[label distance=-0.2em]below: \scriptsize  $(r,r')$}] (trinode) at (tri) {};
			\node[var] (rootnode) at (t1) {\tiny{$ k_{\tiny{1}} $}};
			\node[var] (rootnode) at (t3) {\tiny{$ k_{\tiny{2}} $}};
			\node[var] (trinode) at (t2) {\tiny{$ k_3 $}};
		\end{tikzpicture} & =   \begin{tikzpicture}[scale=0.2,baseline=-5]
			\coordinate (root) at (0,0);
			\coordinate (tri) at (0,-2);
			\coordinate (t1) at (-2,2);
			\coordinate (t2) at (2,2);
			\coordinate (t3) at (0,3);
			\draw[kernels2,tinydots] (t1) -- (root);
			\draw[kernels2] (t2) -- (root);
			\draw[kernels2] (t3) -- (root);
			\draw[symbols] (root) -- (tri);
			\node[not] (rootnode) at (root) {};t
			\node[not,label= {[label distance=-0.2em]below: \scriptsize  $(r,r')$}] (trinode) at (tri) {};
			\node[var] (rootnode) at (t1) {\tiny{$ k_{\tiny{1}} $}};
			\node[var] (rootnode) at (t3) {\tiny{$ k_{\tiny{2}} $}};
			\node[var] (trinode) at (t2) {\tiny{$ k_3 $}};
		\end{tikzpicture} \otimes \one + \one \otimes  \begin{tikzpicture}[scale=0.2,baseline=-5]
			\coordinate (root) at (0,0);
			\coordinate (tri) at (0,-2);
			\coordinate (t1) at (-2,2);
			\coordinate (t2) at (2,2);
			\coordinate (t3) at (0,3);
			\draw[kernels2,tinydots] (t1) -- (root);
			\draw[kernels2] (t2) -- (root);
			\draw[kernels2] (t3) -- (root);
			\draw[symbols] (root) -- (tri);
			\node[not] (rootnode) at (root) {};t
			\node[not,label= {[label distance=-0.2em]below: \scriptsize  $ (r,r')$}] (trinode) at (tri) {};
			\node[var] (rootnode) at (t1) {\tiny{$ k_{\tiny{1}} $}};
			\node[var] (rootnode) at (t3) {\tiny{$ k_{\tiny{2}} $}};
			\node[var] (trinode) at (t2) {\tiny{$ k_3 $}};
		\end{tikzpicture} 
	\end{equs}
where a node of the form $ \begin{tikzpicture}[scale=0.2,baseline=-5]
	\coordinate (root) at (0,0);
	\node[var1] (rootnode) at (root) {\tiny{$ _\ell^m $}};
\end{tikzpicture} $ in the example above corresponds to the frequency $ \ell $ and the monomial $ \lambda^m $.
As an example of computation of $ \tilde{\psi} $, one has
\begin{equs}
	\tilde{\psi} \left( \begin{tikzpicture}[scale=0.2,baseline=-5]
		\coordinate (root) at (0,0);
		\coordinate (tri) at (0,-2);
		\coordinate (t1) at (-2,2);
		\coordinate (t2) at (2,2);
		\coordinate (t3) at (0,3);
		\draw[kernels2,tinydots] (t1) -- (root);
		\draw[kernels2] (t2) -- (root);
		\draw[kernels2] (t3) -- (root);
		\draw[symbols] (root) -- (tri);
		\node[not] (rootnode) at (root) {};t
		\node[not,label= {[label distance=-0.2em]below: \scriptsize  $ (r,r') $}] (trinode) at (tri) {};
		\node[var] (rootnode) at (t1) {\tiny{$ k_{\tiny{4}} $}};
		\node[var1] (rootnode) at (t3) {\tiny{$ _\ell^m $}};
		\node[var] (trinode) at (t2) {\tiny{$ k_5 $}};
	\end{tikzpicture} \right) & =   \begin{tikzpicture}[scale=0.2,baseline=-5]
	\coordinate (root) at (0,0);
	\coordinate (tri) at (0,-2);
	\coordinate (t1) at (-2,2);
	\coordinate (t2) at (2,2);
	\coordinate (t3) at (0,3);
	\draw[kernels2,tinydots] (t1) -- (root);
	\draw[kernels2] (t2) -- (root);
	\draw[kernels2] (t3) -- (root);
	\draw[symbols] (root) -- (tri);
	\node[not] (rootnode) at (root) {};t
	\node[not,label= {[label distance=-0.2em]below: \scriptsize  $ (r,r') $}] (trinode) at (tri) {};
	\node[var] (rootnode) at (t1) {\tiny{$ k_{\tiny{4}} $}};
	\node[var1] (rootnode) at (t3) {\tiny{$ _\ell^m $}};
	\node[var] (trinode) at (t2) {\tiny{$ k_5 $}};
\end{tikzpicture} 
	\\ \tilde{\psi} ( \begin{tikzpicture}[scale=0.2,baseline=-5]
		\coordinate (root) at (0,0);
		\coordinate (tri) at (0,-2);
		\coordinate (t1) at (-2,2);
		\coordinate (t2) at (2,2);
		\coordinate (t3) at (0,2);
		\coordinate (t4) at (0,4);
		\coordinate (t41) at (-2,6);
		\coordinate (t42) at (2,6);
		\coordinate (t43) at (0,8);
		\draw[kernels2,tinydots] (t1) -- (root);
		\draw[kernels2] (t2) -- (root);
		\draw[kernels2] (t3) -- (root);
		\draw[symbols] (root) -- (tri);
		\draw[symbols] (t3) -- (t4);
		\draw[kernels2,tinydots] (t4) -- (t41);
		\draw[kernels2] (t4) -- (t42);
		\draw[kernels2] (t4) -- (t43);
		\node[not] (rootnode) at (root) {};
		\node[not] (rootnode) at (t4) {};
		\node[not] (rootnode) at (t3) {};
		\node[not,label= {[label distance=-0.2em]below: \scriptsize  $ (r,r') $}] (trinode) at (tri) {};
		\node[var] (rootnode) at (t1) {\tiny{$ k_{\tiny{4}} $}};
		\node[var] (rootnode) at (t41) {\tiny{$ k_{\tiny{1}} $}};
		\node[var] (rootnode) at (t42) {\tiny{$ k_{\tiny{3}} $}};
		\node[var] (rootnode) at (t43) {\tiny{$ k_{\tiny{2}} $}};
		\node[var] (trinode) at (t2) {\tiny{$ k_5 $}};
	\end{tikzpicture} ) & = \begin{tikzpicture}[scale=0.2,baseline=-5]
		\coordinate (root) at (0,0);
		\coordinate (tri) at (0,-2);
		\coordinate (t1) at (-2,2);
		\coordinate (t2) at (2,2);
		\coordinate (t3) at (0,2);
		\coordinate (t4) at (0,4);
		\coordinate (t41) at (-2,6);
		\coordinate (t42) at (2,6);
		\coordinate (t43) at (0,8);
		\draw[kernels2,tinydots] (t1) -- (root);
		\draw[kernels2] (t2) -- (root);
		\draw[kernels2] (t3) -- (root);
		\draw[symbols] (root) -- (tri);
		\draw[symbols] (t3) -- (t4);
		\draw[kernels2,tinydots] (t4) -- (t41);
		\draw[kernels2] (t4) -- (t42);
		\draw[kernels2] (t4) -- (t43);
		\node[not] (rootnode) at (root) {};
		\node[not] (rootnode) at (t4) {};
		\node[not] (rootnode) at (t3) {};
		\node[not,label= {[label distance=-0.2em]below: \scriptsize  $ (r,r') $}] (trinode) at (tri) {};
		\node[var] (rootnode) at (t1) {\tiny{$ k_{\tiny{4}} $}};
		\node[var] (rootnode) at (t41) {\tiny{$ k_{\tiny{1}} $}};
		\node[var] (rootnode) at (t42) {\tiny{$ k_{\tiny{3}} $}};
		\node[var] (rootnode) at (t43) {\tiny{$ k_{\tiny{2}} $}};
		\node[var] (trinode) at (t2) {\tiny{$ k_5 $}};
	\end{tikzpicture} + \sum_{m} \frac{1}{m!} \tilde{\psi} \left( \begin{tikzpicture}[scale=0.2,baseline=-5]
		\coordinate (root) at (0,0);
		\coordinate (tri) at (0,-2);
		\coordinate (t1) at (-2,2);
		\coordinate (t2) at (2,2);
		\coordinate (t3) at (0,3);
		\draw[kernels2,tinydots] (t1) -- (root);
		\draw[kernels2] (t2) -- (root);
		\draw[kernels2] (t3) -- (root);
		\draw[symbols] (root) -- (tri);
		\node[not] (rootnode) at (root) {};t
		\node[not,label= {[label distance=-0.2em]below: \scriptsize  $ (r,r') $}] (trinode) at (tri) {};
		\node[var] (rootnode) at (t1) {\tiny{$ k_{\tiny{4}} $}};
		\node[var1] (rootnode) at (t3) {\tiny{$ _\ell^m $}};
		\node[var] (trinode) at (t2) {\tiny{$ k_5 $}};
	\end{tikzpicture} \right) \cdot \begin{tikzpicture}[scale=0.2,baseline=-5]
		\coordinate (root) at (0,0);
		\coordinate (tri) at (0,-2);
		\coordinate (t1) at (-2,2);
		\coordinate (t2) at (2,2);
		\coordinate (t3) at (0,3);
		\draw[kernels2,tinydots] (t1) -- (root);
		\draw[kernels2] (t2) -- (root);
		\draw[kernels2] (t3) -- (root);
		\draw[symbols] (root) -- (tri);
		\node[not] (rootnode) at (root) {};t
		\node[not,label= {[label distance=-0.2em]below: \scriptsize  $ (r-1,m) $}] (trinode) at (tri) {};
		\node[var] (rootnode) at (t1) {\tiny{$ k_{\tiny{1}} $}};
		\node[var] (rootnode) at (t3) {\tiny{$ k_{\tiny{2}} $}};
		\node[var] (trinode) at (t2) {\tiny{$ k_3 $}};
	\end{tikzpicture} 
\\ & = \begin{tikzpicture}[scale=0.2,baseline=-5]
	\coordinate (root) at (0,0);
	\coordinate (tri) at (0,-2);
	\coordinate (t1) at (-2,2);
	\coordinate (t2) at (2,2);
	\coordinate (t3) at (0,2);
	\coordinate (t4) at (0,4);
	\coordinate (t41) at (-2,6);
	\coordinate (t42) at (2,6);
	\coordinate (t43) at (0,8);
	\draw[kernels2,tinydots] (t1) -- (root);
	\draw[kernels2] (t2) -- (root);
	\draw[kernels2] (t3) -- (root);
	\draw[symbols] (root) -- (tri);
	\draw[symbols] (t3) -- (t4);
	\draw[kernels2,tinydots] (t4) -- (t41);
	\draw[kernels2] (t4) -- (t42);
	\draw[kernels2] (t4) -- (t43);
	\node[not] (rootnode) at (root) {};
	\node[not] (rootnode) at (t4) {};
	\node[not] (rootnode) at (t3) {};
	\node[not,label= {[label distance=-0.2em]below: \scriptsize  $ (r,r') $}] (trinode) at (tri) {};
	\node[var] (rootnode) at (t1) {\tiny{$ k_{\tiny{4}} $}};
	\node[var] (rootnode) at (t41) {\tiny{$ k_{\tiny{1}} $}};
	\node[var] (rootnode) at (t42) {\tiny{$ k_{\tiny{3}} $}};
	\node[var] (rootnode) at (t43) {\tiny{$ k_{\tiny{2}} $}};
	\node[var] (trinode) at (t2) {\tiny{$ k_5 $}};
\end{tikzpicture} + \sum_{m} \frac{1}{m!}  \begin{tikzpicture}[scale=0.2,baseline=-5]
	\coordinate (root) at (0,0);
	\coordinate (tri) at (0,-2);
	\coordinate (t1) at (-2,2);
	\coordinate (t2) at (2,2);
	\coordinate (t3) at (0,3);
	\draw[kernels2,tinydots] (t1) -- (root);
	\draw[kernels2] (t2) -- (root);
	\draw[kernels2] (t3) -- (root);
	\draw[symbols] (root) -- (tri);
	\node[not] (rootnode) at (root) {};t
	\node[not,label= {[label distance=-0.2em]below: \scriptsize  $ (r,r') $}] (trinode) at (tri) {};
	\node[var] (rootnode) at (t1) {\tiny{$ k_{\tiny{4}} $}};
	\node[var1] (rootnode) at (t3) {\tiny{$ _\ell^m $}};
	\node[var] (trinode) at (t2) {\tiny{$ k_5 $}};
\end{tikzpicture}  \cdot \begin{tikzpicture}[scale=0.2,baseline=-5]
	\coordinate (root) at (0,0);
	\coordinate (tri) at (0,-2);
	\coordinate (t1) at (-2,2);
	\coordinate (t2) at (2,2);
	\coordinate (t3) at (0,3);
	\draw[kernels2,tinydots] (t1) -- (root);
	\draw[kernels2] (t2) -- (root);
	\draw[kernels2] (t3) -- (root);
	\draw[symbols] (root) -- (tri);
	\node[not] (rootnode) at (root) {};t
	\node[not,label= {[label distance=-0.2em]below: \scriptsize  $ (r-1,m) $}] (trinode) at (tri) {};
	\node[var] (rootnode) at (t1) {\tiny{$ k_{\tiny{1}} $}};
	\node[var] (rootnode) at (t3) {\tiny{$ k_{\tiny{2}} $}};
	\node[var] (trinode) at (t2) {\tiny{$ k_3 $}};
\end{tikzpicture}.
\end{equs}
\end{example}

\subsection{Midpoint general resonance based schemes}
\label{Duhamel_iteration_mid_point}
In this section, we introduce new resonance based schemes where we iterate Duhamel's formula in slightly different manner. The iteration chosen follows the mid point rule. \georginline{These schemes turn out to be a subclass of the general forest formula~\eqref{eqn:general_formula}. In this subclass it is also} possible to work with a fairly general framework \georginline{more closely aligned with \cite{BS} which allows for an automatic handle on the local error of these schemes. In particular, in order to incorporate the midpoint iterations and subsequent lower part interpolations we need to add more edge decorations on the trees representing the iterated integrals.} For example, we can consider decorated trees of the form 
  $ 
T_{\Labe, \chi}^{\Labn, \Labo} =  (T,\Labn,\Labo,\Labe,\chi) $ where $ \chi : E_T \rightarrow \mathfrak{D}$. Here  $ \mathfrak{D} $ is  a finite set and encodes the following information for an edge $ e \in N_T$ decorated by $ (\mathfrak{t},p) $ :
\begin{itemize}
	\item For $ \mathfrak{t} \notin \Lab_+ $, the edge $ e $ is associated with a term of the form $ e^{i P_{(\mathfrak{t},p)}(k_v)} $. Then, $ \chi(e) $ corresponds of a way of iterating Duhamel's formula not only using the leftmost point of the interval but a weighted sum of iterations on various points in $ [0,\tau] $.
	\item  For $ \mathfrak{t} \in \Lab_+ $, the edge $ e $ is associated with a term of the form $ \int^{t}_{a \tau} e^{i P_{(\mathfrak{t},p)}(k_v)} \cdots ds $ where $ a \in [0,1] $ and it corresponds to a different Duhamel's iteration. Now, $ \chi(e) $ gives a choice of a polynomial interpolation for the lower part of the resonance in the discretisation.
\end{itemize}
For reasons of presentation we focus on the midpoint rule, however, we could also follow other types of Duhamel iterations. \georginline{Moreover, in this subclass we take the} polynomial interpolation to be fixed.

To illustrate the central idea of a different way of iterating Duhamel's formula, let us first consider the Nonlinear Schrödinger equation. The usual iteration is given by
\begin{equs}
u(t_n + s) = e^{i \tau   \Delta} u(t_n) - i e^{i  \tau \Delta} \int_0^\tau e^{-i s \Delta}
\left(\vert u(t_n + s)\vert^2 u(t_n + s)\right)ds.
\end{equs}
In Fourier space, we obtain
\begin{equs} \label{eqn:original_Duhamel_NLSE}
	u_k(t_n + \tau) & = e^{- i \tau  k^2} u_k(t_n) \\ & - i e^{-i  \tau k^2} \sum_{k = - k_1 + k_2 + k_3} \int_0^\tau e^{i s k^2}
	\overline{\hat{u}_{k_1}(t_n + s)} u_{k_2}(t_n + s) u_{k_3}(t_n + s) ds.
\end{equs}
We can now choose to iterate this expression using two possible ways:
\begin{equs}\label{eqn:left_endpoint_duhamel}
		& u_k(t_n + s)  = e^{- i s  k^2} u_k(t_n) \\ & - i e^{-i  s k^2} \sum_{k = - k_1 + k_2 + k_3} \int_0^s e^{i \tilde{s} k^2}
		\overline{u_{k_1}(t_n + \tilde{s})} u_{k_2}(t_n + \tilde{s}) u_{k_3}(t_n + \tilde{s}) d\tilde{s}.
\end{equs}
and
\begin{equs} \label{eqn:right_end_point_duhamel}
	&	u_k(t_n + s)  = e^{- i  (s-\tau)  k^2} u_k(t_n+\tau) \\ & - i e^{-i  s k^2} \sum_{k = - k_1 + k_2 + k_3} \int_{\tau}^{s} e^{i \tilde{s} k^2}
	\overline{u_{k_1}(t_n + \tilde{s})} u_{k_2}(t_n + \tilde{s}) u_{k_3}(t_n + \tilde{s}) d\tilde{s}.
\end{equs}
The iteration \eqref{eqn:left_endpoint_duhamel} corresponds to the left end point of the interval $ [0,\tau] $ while \eqref{eqn:right_end_point_duhamel} is the right end point.
We now have a choice over each term in \eqref{eqn:original_Duhamel_NLSE} if we want the iteration of Duhamel's formula to begin with \eqref{eqn:left_endpoint_duhamel} or \eqref{eqn:right_end_point_duhamel}. The average of the two iterations gives the midpoint rule. 

There are quite a lot of degrees of freedom as one can choose various linear combinations of Duhamel's formulae in different points.
Let us mention that the tree structure is not modified if one changes the iteration but the definition of $ \Pi $ has to reflect this new formulation. The midpoint rule oscillatory integrals are given by
\begin{equation} \label{recursive_pi_r_mid_point}
	\begin{aligned}
		\Pi_{\text{\tiny{mid}}} \left( F \cdot \bar F \right)(s, \tau) & = 
		\left( \Pi_{\text{\tiny{mid}}} F  \right)(s,\tau)  \left( \Pi_{\text{\tiny{mid}}}  \bar F \right)(s,\tau), \\
		(\Pi_{\text{\tiny{mid}}}  \CI_{o_1}( \lambda_{k}  F ))(s,\tau)   & =\frac{1}{2} e^{i s P_{o_1}(k)} \left( (\Pi_{\text{\tiny{mid}},1} F)(s,\tau) +  (\Pi_{\text{\tiny{mid}},2} F)(s,\tau) \right),  \\
		\left( \Pi_{\text{\tiny{mid}},1} \CI_{o_2}( \lambda_k F) \right) (s,\tau) & =	  -i  \vert \nabla\vert^{\alpha} (k)
		 \int_{\tau}^{s} e^{i \xi P_{o_2}(k)} (\Pi_{\text{\tiny{mid}}} F)(\xi,\tau) d \xi, 
		 \\
		 \left( \Pi_{\text{\tiny{mid}},2} \CI_{o_2}( \lambda_k F) \right) (s,\tau) & =	  -i  \vert \nabla\vert^{\alpha} (k)
		 \int_{0}^{s} e^{i \xi P_{o_2}(k)} (\Pi_{\text{\tiny{mid}}} F)(\xi,\tau) d \xi.
	\end{aligned}
\end{equation}
We notice that, in this definition, we have to keep track of the time step $ \tau $ in order to remember the interval $ [0,\tau] $.
In this definition, we have assumed that the Duhamel iteration corresponds to edges decorated by $ \mathfrak{t}_1 $. Moreover, we have supposed that $ F $ is not empty for $ (\Pi_{\text{\tiny{mid}}}  \CI_{o_1}( \lambda_{k}  F ))(s,\tau) $ and $ s \neq \tau $. If $ F $ is empty, we set
\begin{equs} \label{leaves}
	(\Pi_{\text{\tiny{mid}}}  \CI_{o_1}( \lambda_{k} \one ))(s,\tau) = e^{i s P_{o_1}(k)}.
\end{equs}
If $ \tau = s $, we set
\begin{equs} \label{first_iteration}
		(\Pi_{\text{\tiny{mid}}}  \CI_{o_1}( \lambda_{k}  F ))(\tau,\tau) =  (\Pi_{\text{\tiny{mid}}}  \CI_{o_1}( \lambda_{k}  F ))(\tau) = e^{i \tau P_{o_1}(k)} (\Pi_{\text{\tiny{mid}},2} F)(\tau,\tau).\quad 
	\end{equs}
The last two specific cases are necessary for building up the scheme. Indeed, \eqref{leaves} corresponds to the leaves of our trees or when we terminate on an initial data $ u_{k_i} $. Here, we will apply the midpoint rule in the sequel (see \eqref{mid_point_rule} in the definition of $ \Upsilon^{p}_{\text{\tiny{mid}}}( T)(v,\tau) $) 
\begin{equs}
	\frac{1}{2}  u_{k_j}(0) + e^{- i \tau P_{o_1}(k_j)} \frac{1}{2} u_{k_j}(\tau). 
\end{equs}
The second condition \eqref{first_iteration} corresponds to the fact that the first is not the midpoint rule approximation as we do  not need to perform it as $ \tau = s $.

\begin{remark}
	This approach also works for the more general scheme given in \cite{BBS}. The main difference is that now
	$ \Upsilon^{p}_{\text{\tiny{mid}}}( T)(v,\tau) $ defined in the sequel is part of the definition of $ \Pi_{\text{\tiny{mid}}}  $.
	\end{remark}

The scheme $ \Pi_{\text{\tiny{mid}}}^{n,r} $ is defined as the same as for  $ \Pi_{\text{\tiny{mid}}} $ but now we discretise the time integrals:
\begin{equs}
	\left( \Pi_{\text{\tiny{mid}},j}^{ n,r} \CI_{o_2}( \lambda_k F) \right) (s,\tau)  = \CK^{k,r}_{o_2,j} \left(  \Pi_{\text{\tiny{mid}}}^{ n,r} \left( F \right)(\cdot,\tau),n \right)(s,\tau), \quad j \in \left\lbrace 1,2  \right\rbrace
	\end{equs}
	where the map $ \CK^{k,r}_{o_2,1}(\cdot)(s,\tau) $ uses the exact integration $ \int_{s}^{\tau} ...d\xi $ and $  \CK^{k,r}_{o_2,2}  $ the  one given by $ \int_{0}^{s} ...d\xi $.
We first introduce some notations:
\begin{defi}\label{def:notation_for_numerical_schemes_based_on_forest_formula}
	\begin{itemize}
	\item For a decorated tree $ T_{\Labe} = (T,\Labe)$ with only edge decorations, we define the symmetry 
	factor $S(T_{\Labe})$  inductively by  $S(\one)\,  { =} 1$, while if 
	$T$ is of the form
	\begin{equs}  
		\prod_{i,j}  \mathcal{I}_{(\Labhom_{t_i},p_i)}\left( T_{i,j}\right)^{\beta_{i,j}}, \quad   
	\end{equs}
	with $T_{i,j} \neq T_{i,\ell}$ for $j \neq \ell$, then 
	\begin{align}\label{S}
		S(T)
		\,  { :=} 
		\Big(
		\prod_{i,j}
		S(T_{i,j})^{\beta_{i,j}}
		\beta_{i,j}!
		\Big)\;.
	\end{align}
	We extend this definition to any tree $ T_{\Labe}^{\Labn,\Labo} $ in $ \CT $ by setting:
	\begin{equs}
		S(T_{\Labe}^{\Labn,\Labo} )\,  { :=}  S(T_{\Labe} ).
	\end{equs}
Let us stress that the symmetric factor depends only on the edges decorations but not on the nodes decorations given by the frequencies. 
	\item Then, we define the map $ \Upsilon^{p}_{\text{\tiny{mid}}}(T)(v,\tau) $ for 
	\begin{equs}
		T  = 
		\CI_{(\Labhom_1,a)}\left( \lambda_k \CI_{(\Labhom_2,a)}( \lambda_k   \prod_{i=1}^n \CI_{(\Labhom_1,0)}( \lambda_{k_i} T_i) \prod_{j=1}^m \CI_{(\Labhom_1,1)}( \lambda_{\tilde k_j} \tilde T_j)  ) \right), \quad a \in \lbrace 0,1 \rbrace  
	\end{equs}
	by
	\begin{equs}\label{upsi}
		\Upsilon^{p}_{\text{\tiny{mid}}}(T)(v,\tau)& \,  { :=}  \partial_v^{n} \partial_{\bar v}^{m} p_a(v,\bar v) \prod_{i=1}^n  \Upsilon^p_{\text{\tiny{mid}}}( \CI_{(\Labhom_1,0)}\left( \lambda_{k_i}  T_i \right) )(v,\tau) \\ & \prod_{j=1}^m \Upsilon^p_{\text{\tiny{mid}}}( \CI_{(\Labhom_1,1)}( \lambda_{\tilde k_j}\tilde T_j ) )(v,\tau)
	\end{equs}
	and 
	\begin{equs} 
		\label{mid_point_rule}
		\begin{aligned}
		\Upsilon^{p}_{\text{\tiny{mid}}}(\CI_{(\Labhom_1,0)}( \lambda_{k})  )(v,\tau)  \, & { :=}  \frac{1}{2} v_k(0) + \frac{1}{2} e^{-i P_{o_1}(k) \tau} v_k(\tau)
		\\ 
		\Upsilon^{p}_{\text{\tiny{mid}}}(\CI_{(\Labhom_1,1)}( \lambda_{k})  )(v,\tau)  \, & { :=} \frac{1}{2} \bar{v}_k(0) + \frac{1}{2} e^{i P_{o_1}(k) \tau} \bar{v}_k(\tau).
		\end{aligned}
	\end{equs}
	Above, we have used the notation:
	\begin{equs}
		p_0(v,\bar{v}) = p(v,\bar{v}), \quad p_{1}(v,\bar{v}) = \overline{p(v, \bar{v})} 
	\end{equs}
	In the sequel, we will use the following short hand notation:
	\begin{equs}
		\overline{\Upsilon^{p}_{\text{\tiny{mid}}}(T)}(v,\tau) = \bar{\Upsilon}^{p}_{\text{\tiny{mid}}}(T)(v,\tau).
	\end{equs}
	\item We set 
	\begin{equs}
		\hat \CT_0(R)&  = \lbrace \CI_{(\Labhom_1,0)}( \lambda_k \CI_{(\Labhom_2,0)}( \lambda_k   \prod_{i=1}^N T_i \prod_{j=1}^M  \tilde T_j  ) ), \CI_{(\Labhom_1,0)}(\lambda_k) \,  \\ & \quad T_i \in \hat{\CT}_0(R), \, \tilde{T}_j \in \bar{\hat{\CT}}_0(R), \, k \in \Z^{d}  \rbrace \\
		\bar{\hat{\CT}}_0(R)&  = \lbrace \CI_{(\Labhom_1,1)}( \lambda_k \CI_{(\Labhom_2,1)}( \lambda_k   \prod_{i=1}^N T_i \prod_{j=1}^M  \tilde T_j  ) ), \CI_{(\Labhom_1,1)}(\lambda_k) \,  \\ & \quad T_i \in \bar{\hat{\CT}}_0(R), \, \tilde{T}_j \in \hat{\CT}_0(R), \, k \in \Z^{d}  \rbrace \\ 
		\hat \CT_2(R)&  = \lbrace \CI_{(\Labhom_2,0)}( \lambda_k   \prod_{i=1}^N T_i \prod_{j=1}^M  \tilde T_j  ) , \, T_i \in \hat{\CT}_0(R), \, \tilde{T}_j \in \bar{\hat{\CT}}_0(R), \, k \in \Z^{d}  \rbrace
	\end{equs}
	For a fixed $ k \in \Z^d$, we denote the set $ \hat{\CT}^k_0(R) $ (resp. $ \bar{\hat{\CT}}^k_0(R) $ and $ \hat{\CT}^k_2(R) $) as the subset of $ \hat{\CT}_0(R) $ (resp. $\bar{\hat{\CT}}_0(R) $ and $ \bar{\hat{\CT}}^k_2(R) $) whose decorated trees have decorations on the node connected to the root given by $ k $. For $ r \in \Z $, $ r \geq -1 $, we set:
	\begin{equs}
		\hat \CT_0^{r,k}(R) = \lbrace  T_{\Labe}^{\Labo} \in \hat \CT^{k}_0{ (R)}   \,  , n_+(T_{\Labe}^{\Labo}) \leq r +1  \rbrace. 
	\end{equs}
	In the previous space, we disregard iterated integrals which have more than $ r+1 $ integrals and will be of order $ \mathcal{O}(\tau^{r+2}) $. The set $ \hat \CT_2^{r,k}(R) $ is defined as the same from $ \hat \CT_2^{k}(R) $. In the sequel, we will use the short hand notation for $ T \in \hat \CT_2^{k}(R) $:
	\begin{equs}
			\Upsilon^{p}_{\text{\tiny{mid}}}(T)(v,\tau) = \Upsilon^{p}_{\text{\tiny{mid}}}(\mathcal{I}_{(\mathfrak{t}_1,0)}(\lambda_k T))(v,\tau).
		\end{equs}
\end{itemize}
\end{defi}
\georginline{This truncation leads exactly to the current local error behaviour as shown in the following proposition which forms the basis of our local error analysis in Theorem~\ref{thm:genloc}.}
\begin{proposition}
 The tree series given by 
\begin{equs}\label{genscheme}
	U_{\text{\tiny{mid}}, k}^{r}(\tau, v) =   \sum_{T \in \hat{\CT}^{r,k}_{0}(R)} \frac{\Upsilon^{p}_{\text{\tiny{mid}}}( T)(v,\tau)}{S(T)} (\Pi_{\text{\tiny{mid}}}   T )(\tau)
\end{equs} 
 is the $k$-th Fourier coefficient of a solution of $ \eqref{eqn:original_Duhamel_NLSE} $ with the midpoint rule expansion up to order $ r +1 $.
	\end{proposition}
\begin{proof}
	The proof follows the same lines as the one given in \cite[Prop. 4.3]{BS}.
	\end{proof}
We are now able to define the main resonance based scheme:
\begin{definition} \label{main_scheme}
	The midpoint   resonance based scheme is given by:
	\begin{equs} \label{mid_point_low_reg_scheme}
			U_{\text{\tiny{mid}}, k}^{n,r}(\tau, v) =    \sum_{T \in \hat{\CT}^{r,k}_{0}(R)} \frac{\Upsilon^{p}_{\text{\tiny{mid}}}( T)(v,\tau)}{S(T)} (\Pi_{\text{\tiny{mid}}}^{n,r}   T )(\tau)
	\end{equs}
It is obtained by replacing the character $ \Pi_{\text{\tiny{mid}}}^{r} $ by $ \Pi_{\text{\tiny{mid}}}^{n,r} $ in \eqref{genscheme}.
	\end{definition}
The new scheme  \eqref{mid_point_low_reg_scheme} can be described by the same type of forest formula introduced before.
\begin{proposition}
	For every forest $ F $, $ (\Pi^{n,r}_{\text{\tiny{mid}}} F)(t,\tau) $ is of the form of	\eqref{forest_formula}. For every decorated tree $ T = \CI_{o_1}( \lambda_{k} \CI_{o_2}( \lambda_{k}  F ) )  $, $ (\Pi^{n,r}_{\text{\tiny{mid}}} T)(\tau) $ is of the form	\eqref{forest_formula_21}.
\end{proposition}
\begin{proof} The proof follow by induction as in Theorem-\ref{forest_formula_1} and Theorem~\ref{forest_formula_2}.
\end{proof}
Before stating our main result \georginline{connecting the midpoint resonance based schemes to our earlier forest formula}, we need to introduce a new map $$ \Upsilon^p_{\chi}(T)(u_{k_v}^{n +\chi_v}, v \in L_T, \tau)$$ defined as the same as  $\Upsilon_{\text{\tiny{mid}}}^p(T)(v,\tau)$ except that for the leaves we use 
\begin{equs} \label{new_upsilon}
 e^{- i \chi_u\tau P_{o_{e_u}}(k_u)} u_{k_u}^{n +\chi_u}, \end{equs}
where $ e_u $ is the outgoing edge of $ u $ in $ T $ and $ o_{e_u} $ corresponds to the edge decoration of $ e_u $.  The map $ \Upsilon^p_{\chi}(T) $ allows to parametrise implicit schemes as the scheme given by the midpoint rule.

\begin{theorem} \label{forest_formula_symmetric scheme}
	The low regularity scheme $ U_{\text{\tiny{mid}},k}^{n,r} $ is of the form:
	\begin{equs}
		\begin{split}\label{eqn:general_formula}
			u_k^{n+1} & =  e^{i \tau P_{o_1}(k)} u_k^n+  e^{i \tau P_{o_1}(k)} \sum_{T \in \hat{\CT}^{r,k}_{2}(R)}  \sum_{\mathbf{a} \in [ 0, 1]^{\tilde{E}_F}}  \sum_{\chi \in \lbrace 0, 1\rbrace^{ L_T}}\sum_{T_0 \cdot T_1 ... \cdot T_{{m}} \subset T} \\ &  C_T \,b_{\mathbf{a}, \chi, T, T_0 \cdot ... \cdot T_{{m}}}(\tau,i \tau \mathscr{F}_{\text{\tiny{dom}}}(T_j), j \in \lbrace 0,...,{m} \rbrace)
			\\ &	 \prod_{e \in \tilde{E}_{T_j}} e^{i \tau a_e\mathscr{F}_{\text{\tiny{low}}}(T_j^e)} \frac{\Upsilon^p_{\chi}(T)(u_{k_v}^{n +\chi_v}, v \in L_T, \tau)}{S(T)},\end{split}
	\end{equs}
	where $\hat{\CT}^{r,k}_{2}(R)$ was introduced in Definition~\ref{def:notation_for_numerical_schemes_based_on_forest_formula}.
\end{theorem}
\begin{proof}
	First, we notice that
	\begin{equs}
		U_{\text{\tiny{mid}}, k}^{n,r}(\tau, v)  = e^{i \tau P_{o_1}(k)} u_k^n+  e^{i \tau P_{o_1}(k)}  \sum_{T \in \hat{\CT}^{r,k}_{2}(R)} \frac{\Upsilon^{p}_{\text{\tiny{mid}}}( T)(v,\tau)}{S(T)} (\Pi_{\text{\tiny{mid}}}^{n,r}   T )(\tau).
	\end{equs}
	Then, the result is just a consequence of Theorem~\ref{forest_formula_2} applied to each of the $ (\Pi_{\text{\tiny{mid}}}^{n,r}   T )(\tau) $. Indeed, one multiplies the coefficients for a decorated trees \eqref{forest_formula_21} with $\Upsilon^{p}_{\text{\tiny{mid}}}( T)(v,\tau)$. 
\end{proof}

For the local error, we can adapt \cite[Def. 3.11]{BS}.
\begin{definition}\label{def:Llow}
	Let $ n \in \N $, $ r \in \Z $. We recursively define $ \mathcal{L}^{r}_{\text{\tiny{low}}}(\cdot,n)$ as
	\begin{equs}
		\mathcal{L}^{r}_{\text{\tiny{low}}}(F,n) = 1, \quad r < 0.
	\end{equs}
	Else\georginline{, when $r\geq0$, we let:}
	\begin{equs}
		\mathcal{L}^{r}_{\text{\tiny{low}}}(\one,n) = 1, \quad
		\mathcal{L}^{r}_{\text{\tiny{low}}}(F \cdot \bar F,n)  = 
		\mathcal{L}^{r}_{\text{\tiny{low}}}(F,n ) + \mathcal{L}^{r}_{\text{\tiny{low}}}( \bar F,n) \\
		\mathcal{L}^{r}_{\text{\tiny{low}}}(\CI_{o_1}( \lambda_{k}^{\ell}  F ),n)  = \mathcal{L}^{r-|\ell|}_{\text{\tiny{low}}}(  F,n ) \\
		\mathcal{L}^{r}_{\text{\tiny{low}}}(\CI_{o_2}( \lambda^{\ell}_{k}  F ),n)  =   k^{\alpha} \mathcal{L}^{r-|\ell|-1}_{\text{\tiny{low}}}(  F,n ) + \one_{\lbrace r-|\ell| \geq 0 \rbrace} \sum_j k^{\bar n_j}
	\end{equs}
	where 
	\begin{equs}
		\bar n_j   =  \max_{ m}\left(n,\deg\left( P_{(F^{(1)}_j, F^{(2)}_j,m)}  \mathscr{F}_{\text{\tiny{low}}} (\CI_{(\Labhom_2,p)}( \lambda^{\ell}_{k}  F^{(1)}_j ))^{r-|\ell| +1- m} + \alpha \right) \right)
	\end{equs}
	with
	\begin{equs}
		\Delta \CD^{r-|\ell|-1}(F) & = \sum_{j} F^{(1)}_j \otimes F^{(2)}_j, \\ \quad A^n(F^{(2)}_j)  B^n\left( F^{(1)}_j \right)(\xi,\tau)  & = \sum_{|m| \leq r-|\ell| -1} \frac{P_{(F^{(1)}_j, F^{(2)}_j,m)}}{Q_{(F^{(1)}_j, F^{(2)}_j,m)}}\xi^{m_1} \tau^{m_2}
	\end{equs}
	and $ \mathscr{F}_{\text{\tiny{low}}} $ is defined in Definition \ref{dom_freq}. 
\end{definition}

\begin{remark}
	The main difference between the definition above and  \cite[Def. 3.11]{BS} is the fact that we deal with monomials of the form $ s^{m_1} \tau^{m_2} $ due to the fact that $ \tau $ appears in the exact integrations. These modifications are minor from the original structure as the formalism is robust from moving from decorations on the edges in $ \N $ to $ \N^2 $.
	\end{remark}
\medskip
	
With the previous definition, one is able to give the local error of the approximations of the oscillatory integrals and for the schemes. The proofs are exactly the same as in \cite[Section 3.3]{BS}. { Nevertheless, we recall from \cite{BS} the definition of the big-$\CO$ notation used to express the local errors.
\begin{definition}\label{def:localBS}
Let $\Phi^\tau(u_0) =u^1 \approx u(\tau)$ denote the numerical solution at time $t=\tau$. We write 
\begin{equs}\label{deflocalErr}
u(\tau) -\Phi^\tau(u_0) = \CO_{\| \cdot\|}(\tau^m  \tilde{\mathcal{L}} u_0)
\end{equs}
if in a suitable norm $\| \cdot\|$, it holds that 
\begin{equs}\label{deflocalErrBound}
\| u(\tau) -\Phi^\tau(u_0)\| \le C(T,d) \tau^m \sup_{0\le t\le\tau} \left\|q\left( \tilde{\mathcal{L}} u(t)\right) \right\|,
\end{equs}
for some polynomial $q$, differential operator $\tilde{\mathcal{L}} $, and constant $C$ independent of $\tau$.
\end{definition} 
If \eqref{deflocalErr} holds, it is said that the numerical solution $u^1$ approximates the exact solution $u(t)$ at time $t=\tau$ with a local error of order $\CO(\tau^m  \tilde{\mathcal{L}} u_0)$, where we omit the dependance on the norm $\| \cdot\|$, for notational convenience.
With the above definition we express the local error terms as follows.}

\begin{theorem} \label{approxima_tree}
	For every $ T \in \CT $ one has,
	\begin{equs}
		\left(\Pi_{\text{\tiny{mid}}} T - \Pi^{n,r}_{\text{\tiny{mid}}} T \right)(\tau)  = \mathcal{O}\left( \tau^{r+2} \mathcal{L}^{r}_{\text{\tiny{low}}}(T,n) \right).
	\end{equs}
\end{theorem}

The numerical scheme \eqref{genscheme}  approximates  the exact solution locally up to order $r+2$. More precisely, the following Theorem holds:
\begin{theorem}[Local error]\label{thm:genloc} 
	The numerical scheme \eqref{genscheme}  with initial value $v = u(0)$ approximates the exact solution $U_{k}(\tau,v) $ up to a  local error of type
	\begin{equs}
		U_{\text{\tiny{mid}},k}^{n,r}(\tau,v) - U_{k}(\tau,v) = \sum_{T \in \tilde{\CT}^{r,k}_{0}(R)} \CO\left(\tau^{r+2} \CL^{r}_{\text{\tiny{low}}}(T,n) \Upsilon^{p}_{\text{\tiny{mid}}}( \lambda_kT)(v,\tau) \right)
	\end{equs}
	where the operator $\CL^{r}_{\text{\tiny{low}}}(T,n)$, given in Definition \ref{def:Llow}, embeds the necessary regularity of the solution.
\end{theorem}
\medskip

{
Obtaining the above local error term is an important first step towards obtaining convergence of the scheme. However, to obtain a global convergence result, one needs to fix the space and norm $\| \cdot\|$ in which to work with, rigorously bound the local error term in this norm (as in \eqref{deflocalErrBound}) and couple this with a stability argument to conclude, see also Remark \ref{impl-comp}. 
Deducing global error bounds from the local error terms is in general non trivial for PDE problems, and we refer to \cite{WZ22, MS22, AB23, BMS22, FMW24} where this analysis is made for low-regularity symmetric approximation to specific dispersive equations. This analysis is not dealt with in this paper, where we focus on the derivation and classification of the schemes and of their local error terms.
}

\begin{remark}
	As in \cite[Prop. 3.18]{BS}, one can always map back to physical space the scheme $ 	U_{\text{\tiny{mid}},k}^{n,r} $. This due to the fact that the structure of the resonances and their exact integration is the same in this context. 
	\end{remark}

\section{Symmetric schemes}\label{sec:symmetric_schemes}
Having introduced the general forest formula in \eqref{forest_formula} \& \eqref{forest_formula_bis} and the general subclass of midpoint general resonance based schemes, we now seek to answer the central question of this manuscript: ``Which schemes in these classes are symmetric in the sense of definition~\ref{def:time-symmetric_method}?'' For this we take two routes: Firstly, for the subclass of midpoint general resonance based schemes it turns out that symmetry of the interpolation nodes is sufficient for the symmetry of the schemes. Secondly, for schemes captured by the forest formula \eqref{forest_formula} we can study the form of their adjoint method and find conditions on the coefficients of these schemes under which the methods are symmetric. We recall the adjoint method of a numerical scheme $v^{n+1}=\Phi_{\tau} v^{n}$ is defined by $\widehat{\Phi}_{\tau}:=\Phi_{-\tau}^{-1}$ and the method is said to be symmetric if $\widehat{\Phi}_{\tau}=\Phi_{\tau}$. We can find the adjoint method of a scheme simply by the operations $n\leftrightarrow n+1$ and $\tau\leftrightarrow -\tau$.
The swapping of $ n $  and $ n+1 $ corresponds in our case to changing $ v_k(0) $ into $ v_k(\tau) $. We define $ \tilde{\Upsilon}^{p}_{\text{\tiny{mid}}}( T)(v,\tau) $ as the same as $ \Upsilon^{p}_{\text{\tiny{mid}}}( T)(v,\tau) $ exept that we \georginline{exchange}  $ v_k(0) $ and $ v_k(\tau) $ in the definition:
\begin{equs}
		\tilde{\Upsilon}^{p}_{\text{\tiny{mid}}}(\CI_{(\Labhom_1,0)}( \lambda_{k})  )(v,\tau)  \, & { :=}  \frac{1}{2} v_k(\tau) + \frac{1}{2} e^{i P_{o_1}(k) \tau} v_k(0)
	\\ 
	\tilde{\Upsilon}^{p}_{\text{\tiny{mid}}}(\CI_{(\Labhom_1,1)}( \lambda_{k})  )(v,\tau)  \, & { :=} \frac{1}{2} \bar{v}_k(\tau) + \frac{1}{2} e^{-i P_{o_1}(k) \tau} \bar{v}_k(0).
\end{equs}

\subsection{Symmetric interpolation}\label{sec:symmetric_interpolation}
We prove in the next proposition that the Duhamel's midpoint iteration truncated up to order $ r+1 $ gives a symmetric scheme. The proof uses the recursive construction of the iterated integrals.
\begin{proposition} \label{duhamel_symmetric_1}
	The scheme defined by \eqref{genscheme} is symmetric.
	\end{proposition}
\begin{proof} We first observe that the scheme is given by 
	\begin{equs}
		u_k(\tau) = e^{i \tau P_{o_1}(k)} u_k(0) + \sum_{T \in \hat{\CT}^{r,k}_{0}(R) \setminus \lbrace \CI_{(\mathfrak{t}_1,0)}( \lambda_{k} \one ) \rbrace } \frac{\Upsilon^{p}_{\text{\tiny{mid}}}( T)(u,\tau)}{S(T)} (\Pi_{\text{\tiny{mid}}}   T )(\tau).
		\end{equs}
	Now we swap $ n $ and $ n+1 $, and we also send $ \tau  $ onto $ -\tau $, we obtain
	\begin{equs}
		u_k(\tau) = e^{i \tau P_{o_1}(k)} u_k(0) -  e^{i \tau P_{o_1}(k)} \sum_{T \in \hat{\CT}^{r,k}_{0}(R)\setminus \lbrace \CI_{(\mathfrak{t}_1,0)}( \lambda_{k} \one ) \rbrace } \frac{\tilde{\Upsilon}^{p}_{\text{\tiny{mid}}}( T)(u,-\tau)}{S(T)} (\Pi_{\text{\tiny{mid}}}   T )(-\tau).
		\end{equs}
	Then, one has to show that two sums coincide for proving that the scheme is symmetric. We prove that this is the case for each term of the sum namely, one has:
	\begin{equs} \label{sum_term_equal}
		-  e^{-i \tau P_{o_1}(k)}  \frac{\tilde{\Upsilon}^{p}_{\text{\tiny{mid}}}( T)(u,-\tau)}{S(T)} (\Pi_{\text{\tiny{mid}}}   T )(-\tau)
		=  \frac{\Upsilon^{p}_{\text{\tiny{mid}}}( T)(u,\tau)}{S(T)} (\Pi_{\text{\tiny{mid}}}  T )(\tau).
	\end{equs}
We proceed by induction on the construction of the trees for showing \eqref{sum_term_equal}. Decorated trees in $ \hat{\CT}^{r,k}_{0}(R)\setminus \lbrace \CI_{(\mathfrak{t}_1,0)}( \lambda_{k} \one ) \rbrace  $ are necessarily of the form 
\begin{equs}
T =	\CI_{(\mathfrak{t}_1,0)}( \lambda_{k} \CI_{(\mathfrak{t}_2,0)}( \lambda_{k} F ) ).
\end{equs}
 We notice that
 \begin{equs}
 	(\Pi_{\text{\tiny{mid}}}   T )(-\tau) = 	  -i  \vert \nabla\vert^{\alpha} (k)  e^{i \tau P_{(\mathfrak{t}_1,0)}(k)} 
 	\int_{0}^{-\tau} e^{-i s P_{(\mathfrak{t}_1,0)}(k)} (\Pi F)_{\text{\tiny{mid}}}(s,-\tau) d s .
 	\end{equs}
 By performing the change of variable $ s = s + \tau $, one gets
\begin{equs}
		(\Pi_{\text{\tiny{mid}}}   T )(-\tau) 	   =  	i  \vert \nabla\vert^{\alpha} (k) 
	\int_{0}^{\tau} e^{i s P_{(\mathfrak{t}_1,0)}(k)} (\Pi_{\text{\tiny{mid}}}  F)(s-\tau,-\tau) d s .
\end{equs}
It remains to show that 
\begin{equs}
\frac{\tilde{\Upsilon}^{p}_{\text{\tiny{mid}}}( T_j)(u,-\tau)}{S(T_j)} 	(\Pi_{\text{\tiny{mid}}}   T_j ) (s-\tau,-\tau) = \frac{\Upsilon^{p}_{\text{\tiny{mid}}}( T_j)(u,\tau)}{S(T_j)} (\Pi_{\text{\tiny{mid}}} T_j)(s,\tau),
	\end{equs}
where $ T_j = \CI_{o_1}(  \lambda_{k_j} F_j ) $ is a decorated tree appearing in the decomposition of $ F $ into a product of planted trees.
If $ F_j = \one $, with loss of generality, we suppose that $ o_1 = (\mathfrak{t}_1,0) $, then
\begin{equs}
 \frac{\Upsilon^{p}_{\text{\tiny{mid}}}( T_j)(u,\tau)}{S(T_j)} (\Pi_{\text{\tiny{mid}}} T_j)(s,\tau)  & = \left( \frac{1}{2} e^{-i \tau P_{o_1}(k_j)} u_{k_j}(\tau) + \frac{1}{2} u_{k_j}(0) \right) e^{i s P_{o_1}(k_j)}
 \\ & =   \frac{1}{2} e^{i (s-\tau) P_{o_1}(k_j)} u_{k_j}(\tau) + \frac{1}{2} e^{i s P_{o_1}(k_j)} u_{k_j}(0)
\end{equs}
and
\begin{equs}
	\frac{\tilde{\Upsilon}^{p}_{\text{\tiny{mid}}}( T_j)(u,-\tau)}{S(T_j)} 	(\Pi_{\text{\tiny{mid}}}   T_j ) (s-\tau,-\tau)   & = \left( \frac{1}{2} e^{i \tau P_{o_1}(k_j)} u_{k_j}(0) + \frac{1}{2} u_{k_j}(\tau) \right) e^{i (s-\tau) P_{o_1}(k_j)}
	\\ & =  \frac{1}{2} e^{i (s-\tau) P_{o_1}(k_j)} u_{k_j}(\tau) + \frac{1}{2} e^{i s P_{o_1}(k_j)} u_{k_j}(0).
\end{equs}
For $ (\mathfrak{t}_1,1) $, we proceed analogously with the conjugate. For a more general $ F_j $, we have:
\begin{equs}
	(\Pi_{\text{\tiny{mid}}}  T_j)(s-\tau,-\tau)   & =\frac{1}{2} e^{i (s-\tau) P_{o_1}(k_j)}  \left( (\Pi_{\text{\tiny{mid}},1} F_j)(s-\tau,-\tau) +  (\Pi_{\text{\tiny{mid}},2} F_j)(s-\tau,0) \right).  
\end{equs}
Then $ F_j $ is of the form $ \CI_{o_2}( \lambda_{k_j} \hat{F}_j ) $. Thus, we have
\begin{equs}
	\left( \Pi_{\text{\tiny{mid}},1} \CI_{o_2}( \lambda_{k_j} \hat{F}_j) \right) (s-\tau,-\tau) & =	  -i  \vert \nabla\vert^{\alpha} (k_j)
	\int_{-\tau}^{s-\tau} e^{i \xi P_{o_2}(k_j)} (\Pi_{\text{\tiny{mid}}} \hat{F}_j)(\xi,-\tau) d \xi, 
	\\ & = -i  \vert \nabla\vert^{\alpha} (k_j)
	\int_{0}^{s} e^{i (\xi-\tau) P_{o_2}(k_j)} (\Pi_{\text{\tiny{mid}}} \hat{F}_j)(\xi-\tau,-\tau) d \xi
\end{equs}
and 
\begin{equs}
\left( \Pi_{\text{\tiny{mid}},2} \CI_{o_2}( \lambda_{k_j} \hat{F}_j) \right) (s-\tau,0) & =	  -i  \vert \nabla\vert^{\alpha} (k_j)
\int_{0}^{s-\tau} e^{i \xi P_{o_2}(k)} (\Pi_{\text{\tiny{mid}}} \hat{F}_j)(\xi,-\tau) d \xi
\\ & =  	  -i  \vert \nabla\vert^{\alpha} (k_j)
\int_{\tau}^{s} e^{i (\xi-\tau) P_{o_2}(k)} (\Pi_{\text{\tiny{mid}}} \hat{F}_j)(\xi-\tau,-\tau) d \xi.
\end{equs}
We conclude by applying the induction hypothesis on $ \hat{F}_j $ that is
\begin{equs}
	\frac{\tilde{\Upsilon}^{p}_{\text{\tiny{mid}}}( F_j)(u,-\tau)}{S(F_j)} 	(\Pi_{\text{\tiny{mid}}}   \hat{F}_j ) (s-\tau,-\tau) = \frac{\Upsilon^{p}_{\text{\tiny{mid}}}( F_j)(u,\tau)}{S(F_j)} (\Pi_{\text{\tiny{mid}}} \hat{F}_j)(s,\tau).
\end{equs} 
	\end{proof}
We recall the scheme given by the midpoint rule \eqref{mid_point_low_reg_scheme}
	\begin{equs} \label{mid_point_low_reg_scheme_2}
	U_{\text{\tiny{mid}}, k}^{n,r}(\tau, v) =   \sum_{T \in \hat{\CT}^{r,k}_{0}(R)} \frac{\Upsilon^{p}_{\text{\tiny{mid}}}( T)(v,\tau)}{S(T)} (\Pi_{\text{\tiny{mid}}}^{n,r}   T )(\tau).
\end{equs}
The terms $ (\Pi_{\text{\tiny{mid}}}^{n,r}   T )(\tau) $ are constructed in a similar way as  $ (\Pi_{\text{\tiny{mid}}}   T )(\tau) $. The main difference happens for the computation of the time integrals. Indeed, $  (\Pi_{\text{\tiny{mid}}}^{n,r}   T )(\tau)  $ performs an approximation with a polynomial interpolation and we need to do it in a symmetric way.
 We need the following lemma on the polynomial interpolation in order to guarantee this property:
\begin{lemma} If the interpolation nodes $a_j\in[0,1], j=0,\dots, r$, are symmetrically distributed, i.e. $a_j=1-a_{r-j}, j=0,\dots,r$, then
\begin{equs} \label{condition_symmetric}
	\tilde{p}_r(s-\tau,-\tau) = \sum_{j=0}^r e^{-i a_j \tau \CL_{\text{\tiny{low}}}} p_j(s -\tau,-\tau) = 
	e^{-i  \tau \CL_{\text{\tiny{low}}}}  \tilde{p}_r(s,\tau),
\end{equs}
where we have used the short hand notation
\begin{equs}
	\tilde{p}_r(s,\tau):=\tilde{p}_r(\exp(i s \CL_{\text{\tiny{low}}}),\tau).
	\end{equs}
\end{lemma}
\begin{proof} To begin with, for any $j=0,\dots, r$ we have by definition of the interpolating polynomials $\tilde{p}_r(s,\tau)$ and $\tilde{p}_r(s,-\tau)$
	\begin{align*}
		\tilde{p}_r(a_j\tau,\tau)=e^{i a_j \tau \CL_{\text{\tiny{low}}}}, \quad \tilde{p}_r(-a_j\tau,-\tau)=e^{i a_j \tau \CL_{\text{\tiny{low}}}}.
	\end{align*}
Thus in particular we have
\begin{align*}
	\tilde{p}_r(a_j\tau-\tau,-\tau) =\tilde{p}_r(-a_{r-j}\tau,-\tau)&=e^{i a_{r-j} \tau \CL_{\text{\tiny{low}}}}\\
	&=e^{i \tau \CL_{\text{\tiny{low}}}}e^{-i a_{j} \tau \CL_{\text{\tiny{low}}}}=e^{i \tau \CL_{\text{\tiny{low}}}}\tilde{p}_r(a_j\tau,\tau),
\end{align*}
for each $j=0,\dots, r$. Thus, for any given $\tau$, $\tilde{p}_r(s-\tau,-\tau)$ and 	$e^{-i  \tau \CL_{\text{\tiny{low}}}}  \tilde{p}_r(s,\tau)$ are two polynomials in $s$ of degree $\leq r$ which match at $r+1$ distinct points, so they are identical.
\end{proof}
\begin{theorem} \label{symmetric_schemme_low_regularity}
	The scheme given by \eqref{mid_point_low_reg_scheme} is symmetric. 
	\end{theorem}
\begin{proof}
	The proof works \georginline{in the same manner as} for Proposition~\ref{duhamel_symmetric_1}. The main difference is the use of the operator  $ \CK^{k,r}_{o_2,j} $. We suppose that 
	\begin{equs} \label{induction_hypothesis}
		\frac{\tilde{\Upsilon}^{p}_{\text{\tiny{mid}}}( F_j)(u,-\tau)}{S(F_j)} 	(\Pi_{\text{\tiny{mid}}}^{n,r}   \hat{F}_j ) (s-\tau,-\tau) = \frac{\Upsilon^{p}_{\text{\tiny{mid}}}( F_j)(u,\tau)}{S(F_j)} (\Pi_{\text{\tiny{mid}}}^{n,r} \hat{F}_j)(s,\tau)
	\end{equs} 
and we consider
\begin{equs}
		\left( \Pi_{\text{\tiny{mid}},1}^{n,r} \CI_{o_2}( \lambda_{k_j} \hat{F}_j) \right) (s-\tau,-\tau)
		= \CK^{k_j,r}_{o_2,1} \left( (\Pi_{\text{\tiny{mid}}}^{n,r} \hat{F}_j)(\cdot,-\tau) ,n \right)(s-\tau,-\tau).
\end{equs} 
From \eqref{induction_hypothesis}, we know that $ \Pi_{\text{\tiny{mid}}}^{n,r} \hat{F}_j $ is of the form 
\begin{equs}
	(\Pi_{\text{\tiny{mid}}}^{n,r} \hat{F}_j)(s,\tau) = 
	e^{i (s-\tau) \mathcal{F}_{\text{\tiny{dom}}}(\hat{F}_j) } A(s-\tau) + 	e^{i s \mathcal{F}_{\text{\tiny{dom}}}(\hat{F}_j) } A(s).
	\end{equs} 
Now, when we apply the operator $ \CK^{k,r}_{o_2,j} $, we get among various cases the exact integration
\begin{equs}
& \sum_{\ell=0}^r\int_{-\tau}^{s-\tau}
	e^{i \xi  \mathcal{L}_{\text{\tiny{dom}}} }  e^{-i a_j \tau \CL_{\text{\tiny{low}}}} p_{\ell,r}(\xi,-\tau)	 \left( e^{i \tau \mathcal{F}_{\text{\tiny{dom}}}(\hat{F}_j) } A(\xi+\tau) + 	 A(\xi) \right) d \xi
	\\ & =\sum_{\ell=0}^r\int_{0}^{s}
	e^{i (\xi-\tau)  \mathcal{L}_{\text{\tiny{dom}}} }  e^{-i a_j \tau \CL_{\text{\tiny{low}}}} p_{\ell,r}(\xi- \tau,-\tau)	 \left( e^{i \tau \mathcal{F}_{\text{\tiny{dom}}}(\hat{F}_j) } A(\xi) + 	 A(\xi-\tau) \right) d \xi
	\\ &  = \int_{0}^{s}
	e^{i (\xi-\tau)  \mathcal{L}_{\text{\tiny{dom}}} }   	e^{-i  \tau \CL_{\text{\tiny{low}}}}  \tilde{p}_r(s,\tau) \left( e^{i \tau \mathcal{F}_{\text{\tiny{dom}}}(\hat{F}_j) } A(\xi) + 	 A(\xi-\tau) \right) d \xi
	\end{equs}
where from the second to the third line, we have used the assumption \eqref{condition_symmetric} and we have
\begin{equs}
	 \mathcal{L}_{\text{\tiny{dom}}} +  \mathcal{L}_{\text{\tiny{low}}} = P_{o_2}(k_j) + \mathcal{F}_{\text{\tiny{dom}}}(\hat{F}_j) .
\end{equs}
Therefore, we obtain in the end
\begin{equs}
	e^{-i \tau P_{o_2}(k_j)  }	\int_{0}^{s}
	e^{i \xi \mathcal{L}_{\text{\tiny{dom}}} }   	 \tilde{p}_r(s,\tau) \left(  A(\xi) + e^{-i \tau \mathcal{F}_{\text{\tiny{dom}}}(\hat{F}_j) }	 A(\xi-\tau) \right) d \xi
\end{equs}
which allows us to conclude the symmetry of the method.
\end{proof}

\subsection{Conditions for symmetry}\label{sec:conditions_for_symmetry_general_framework}
Based on the general expression of the scheme \eqref{eqn:general_formula} we can arrive at sufficient conditions for the methods to be symmetric. The following observation is crucial:
\begin{lemma}\label{lem:total_operator_is_sum_of_low_dom_bis} Let $ T $ a decorated tree in $ \hat{\CT}^{k}_{0}(R) $ as introduced in \eqref{splitting_forest}. We have
	\begin{align*}
		\mathscr{F}_{\text{\tiny{dom}}}(T) + \sum_{e \in \tilde{E}_{T}}  \mathscr{F}_{\text{\tiny{low}}}(T^e) = \sum_{v \in L_T} P_{o_{e_v}}(k_v)
	\end{align*}
where $ e_v $ is the outgoing edge of $ v $ in $ T $ and $ o_{e_v} $ corresponds to the edge decoration of $ e_v $. The $k_v$ are the leaves decorations corresponding to the frequencies. The dominant part $ \mathscr{F}_{\text{\tiny{dom}}}(T) $ and the lower parts $ \mathscr{F}_{\text{\tiny{low}}}(T^e) $ depend on them.
\end{lemma}
Before proving this statement let us briefly exhibit the meaning based on a simple example already introduced in Example~\ref{ex:simple_tree_NLSE}.
	\begin{example} We consider the simple decorated tree from the NLSE
		\begin{equs}
				T= \mathcal{I}_{(\mathfrak{t}_1,0)}(\lambda_k \tilde{T}) =  \begin{tikzpicture}[scale=0.2,baseline=-5]
				\coordinate (root) at (0,2);
				\coordinate (tri) at (0,0);
				\coordinate (trib) at (0,-2);
				\coordinate (t1) at (-2,4);
				\coordinate (t2) at (2,4);
				\coordinate (t3) at (0,5);
				\draw[kernels2,tinydots] (t1) -- (root);
				\draw[kernels2] (t2) -- (root);
				\draw[kernels2] (t3) -- (root);
				\draw[kernels2] (trib) -- (tri);
				\draw[symbols] (root) -- (tri);
				\node[not] (rootnode) at (root) {};
				\node[not] (trinode) at (tri) {};
				\node[var] (rootnode) at (t1) {\tiny{$ k_{\tiny{1}} $}};
				\node[var] (rootnode) at (t3) {\tiny{$ k_{\tiny{2}} $}};
				\node[var] (trinode) at (t2) {\tiny{$ k_{\tiny{3}} $}};
				\node[not] (trinode) at (trib) {};
			\end{tikzpicture}, 
		\end{equs}
	with $k=-k_1+k_2+k_3$
	we use the fact that
	\begin{equs}
	P_{(\mathfrak{t}_1,0)}(k) = -k^2, \quad P_{(\mathfrak{t}_1,1)}(k) = k^2
	\end{equs}
to find
\begin{equs}
\sum_{v \in L_T} P_{o_{e_v}}(k_v)=P_{(\mathfrak{t}_1,1)}(k_1)+P_{(\mathfrak{t}_1,0)}(k_2) +P_{(\mathfrak{t}_1,0)}(k_3)=k_1^2-k_2^2-k_3^2. 
\end{equs}
Moreover we have already established in Example~\ref{ex:simple_tree_NLSE} that 
\begin{align*}
	\mathscr{F}_{\text{\tiny{dom}}}(T) & = \mathscr{F}_{\text{\tiny{dom}}}(\tilde{T})  + P_{(\mathfrak{t}_1,0)}(k)   = 2 k_1^2 - k^2
\end{align*}
On the other hand, we have
\begin{equs}
	\mathscr{F}_{\text{\tiny{low}}}(\tilde{T}) = k^2 - k_1^2 - k_2^2 - k_3^2
\end{equs}
and the set $ \tilde{E}_T $ is composed of only one edge which is the only blue edge in $ T $. We have $ T^e = \tilde{T} $. In the end
\begin{equs}	\mathscr{F}_{\text{\tiny{dom}}}(T) + \sum_{e \in \tilde{E}_{T}}  \mathscr{F}_{\text{\tiny{low}}}(T^e)  & = \mathscr{F}_{\text{\tiny{dom}}}(T) + \mathscr{F}_{\text{\tiny{low}}}(T)
	\\ & =  2 k_1^2 - k^2 + k^2 - k_1^2 - k_2^2 - k_3^2
	\\ &  = k_1^2 - k_2^2 - k_3^2
	\\ & = \sum_{v \in L_T} P_{o_{e_v}}(k_v).
\end{equs}

	\end{example}
\begin{proof}[ of Lemma \ref{lem:total_operator_is_sum_of_low_dom}]
	We proceed by induction on the size of the trees. If $ T $ is of the form $ \mathcal{I}_{(\mathfrak{t}_1,a)}(\lambda_k \one) $ then 
	\begin{equs}
		\mathscr{F}_{\text{\tiny{dom}}}( \mathcal{I}_{(\mathfrak{t}_1,a)}(\lambda_k \one) ) = P_{(\mathfrak{t}_1,a)}(k)
	\end{equs}
which allows us to conclude the desired result since $ \tilde{E}_T $ is empty and this tree has only one leaf giving a contribution $ P_{(\mathfrak{t}_1,a)}(k) $.
If $ T $ is of the form $ \mathcal{I}_{(\mathfrak{t}_1,a)}(\lambda_k \mathcal{I}_{(\mathfrak{t}_2,a)}(\lambda_k F) ) $ then 
\begin{equs}
		\mathscr{F}_{\text{\tiny{dom}}}( T ) + 	 \sum_{e \in \tilde{E}_{T}}  \mathscr{F}_{\text{\tiny{low}}}(T^e) & =  P_{(\mathfrak{t}_1,a)}(k)  + \mathscr{F}_{\text{\tiny{dom}}}( \CI_{(\mathfrak{t}_2,a)}(\lambda_k F) ) \\ & +\mathscr{F}_{\text{\tiny{low}}}(  \CI_{(\mathfrak{t}_2,a)}(\lambda_k F) )  + \sum_{e \in \tilde{E}_{T} \setminus \lbrace \bar{e} \rbrace }  \mathscr{F}_{\text{\tiny{low}}}(T^e)
\end{equs}
where $ \bar{e} $ denotes the edge such that $ T^{\bar{e}} = \CI_{(\mathfrak{t}_2,a)}(\lambda_k F) $. Now, we use Definition~\ref{dom_freq} to notice that
\begin{equs}
	\mathscr{F}_{\text{\tiny{dom}}}( \CI_{(\mathfrak{t}_2,a)}(\lambda_k F) )  +\mathscr{F}_{\text{\tiny{low}}}(  \CI_{(\mathfrak{t}_2,a)}(\lambda_k F) ) = P_{(\mathfrak{t_2},a)}(k) + \mathscr{F}_{\text{\tiny{dom}}}( F ).
\end{equs}
By definition, we have also
\begin{equs} \label{cancellation_1}
	P_{(\mathfrak{t_2},a)}(k) + P_{(\mathfrak{t_1},a)}(k)
= 0.
\end{equs}
We got in the end
\begin{equs}
		\mathscr{F}_{\text{\tiny{dom}}}( T ) + 	 \sum_{e \in \tilde{E}_{T}}  \mathscr{F}_{\text{\tiny{low}}}(T^e) & = \mathscr{F}_{\text{\tiny{dom}}}( F ) +  \sum_{e \in \tilde{E}_{T} \setminus \lbrace  \bar{e} \rbrace}  \mathscr{F}_{\text{\tiny{low}}}(T^e)
		\\ & =  \mathscr{F}_{\text{\tiny{dom}}}( F ) +  \sum_{e \in \tilde{E}_{F} }  \mathscr{F}_{\text{\tiny{low}}}(F^e).
\end{equs}
We continue the induction by observing that $ F $ is a product of trees in $ \tilde{\CT}^{r,k}_{0}(R)$. We apply the induction hypothesis on each of these trees and use the fact that $ \mathscr{F}_{\text{\tiny{dom}}} $ is additive for the forest product in order to conclude.
	\end{proof}
\begin{lemma}\label{lem:total_operator_is_sum_of_low_dom} Let $ T_0 \cdot T_1 ... \cdot T_m \subset \mathcal{I}_{(\mathfrak{t}_2,0)}(\lambda_k F) \in \hat{\CT}^{k}_{2}(R) $ be a splitting of $ F $ as introduced in \eqref{splitting_forest}. Then we have
	\begin{align}\label{eqn:total_operator_sum_of_lower_tree}
		\sum_{j=0}^m  \left(\mathscr{F}_{\text{\tiny{dom}}}(T_j) + \sum_{e \in \tilde{E}_{T_j}}  \mathscr{F}_{\text{\tiny{low}}}(T_j^e)\right)=\sum_{v \in L_F} P_{o_{e_v}}(k_v) + P_{(\mathfrak{t}_2,0)}(k).
	\end{align}
\end{lemma}
\begin{proof}
	This is a consequence of Lemma~\ref{lem:total_operator_is_sum_of_low_dom_bis} applied to each of the $ T_j $. In the end, we do not get all the leaves of the $ T_j $ but only the ones in $ F $ because the root of the $ T_j $ ($ j \geq 0  $) is associated with a leaf in a $ T_i $. Indeed, this introduced a cancellation of the type \eqref{cancellation_1}.  
\end{proof}

\begin{example} Let us consider the following forest
		\begin{align*}
			F=\begin{tikzpicture}[scale=0.2,baseline=-5]
				\coordinate (root) at (0,-2);
				\coordinate (t3) at (0,0);
				\draw[kernels2,tinydots] (root) -- (t3);
				\node[var] (rootnode) at (t3) {\tiny{$ k_{\tiny{4}} $}};
				\node[not] (rootnode) at (root) {};
			\end{tikzpicture}\cdot\begin{tikzpicture}[scale=0.2,baseline=-5]
				\coordinate (root) at (0,-2);
				\coordinate (t3) at (0,0);
				\coordinate (t4) at (0,2);
				\coordinate (t41) at (-2,4);
				\coordinate (t42) at (2,4);
				\coordinate (t43) at (0,6);
				\draw[kernels2] (t3) -- (root);
				\draw[symbols] (t3) -- (t4);
				\draw[kernels2,tinydots] (t4) -- (t41);
				\draw[kernels2] (t4) -- (t42);
				\draw[kernels2] (t4) -- (t43);
				\node[not] (rootnode) at (root) {};
				\node[not] (rootnode) at (t4) {};
				\node[not] (rootnode) at (t3) {};
				\node[var] (rootnode) at (t41) {\tiny{$ k_{\tiny{1}} $}};
				\node[var] (rootnode) at (t42) {\tiny{$ k_{\tiny{3}} $}};
				\node[var] (rootnode) at (t43) {\tiny{$ k_{\tiny{2}} $}};
			\end{tikzpicture}\cdot\begin{tikzpicture}[scale=0.2,baseline=-5]
				\coordinate (root) at (0,-2);
				\coordinate (t3) at (0,0);
				\draw[kernels2] (root) -- (t3);
				\node[var] (rootnode) at (t3) {\tiny{$ k_{\tiny{5}} $}};
				\node[not] (rootnode) at (root) {};
			\end{tikzpicture}
		\end{align*} 
	such that $T=\mathcal{I}_{(\mathfrak{t}_2,0)}(\lambda_k F) $ is identified with 
	\begin{align*}
	 T=  \begin{tikzpicture}[scale=0.2,baseline=-5]
		\coordinate (root) at (0,0);
		\coordinate (tri) at (0,-2);
		\coordinate (t1) at (-2,2);
		\coordinate (t2) at (2,2);
		\coordinate (t3) at (0,2);
		\coordinate (t4) at (0,4);
		\coordinate (t41) at (-2,6);
		\coordinate (t42) at (2,6);
		\coordinate (t43) at (0,8);
		\draw[kernels2,tinydots] (t1) -- (root);
		\draw[kernels2] (t2) -- (root);
		\draw[kernels2] (t3) -- (root);
		\draw[symbols] (root) -- (tri);
		\draw[symbols] (t3) -- (t4);
		\draw[kernels2,tinydots] (t4) -- (t41);
		\draw[kernels2] (t4) -- (t42);
		\draw[kernels2] (t4) -- (t43);
		\node[not] (rootnode) at (root) {};
		\node[not] (rootnode) at (t4) {};
		\node[not] (rootnode) at (t3) {};
		\node[not,label= {[label distance=-0.2em]below: \scriptsize  $  $}] (trinode) at (tri) {};
		\node[var] (rootnode) at (t1) {\tiny{$ k_{\tiny{4}} $}};
		\node[var] (rootnode) at (t41) {\tiny{$ k_{\tiny{1}} $}};
		\node[var] (rootnode) at (t42) {\tiny{$ k_{\tiny{3}} $}};
		\node[var] (rootnode) at (t43) {\tiny{$ k_{\tiny{2}} $}};
		\node[var] (trinode) at (t2) {\tiny{$ k_5 $}};
	\end{tikzpicture}
	\end{align*}
with $k=-k_1-k_4+k_2+k_3+k_5$. We consider the following forest splitting of $T$, with $T_0\cdot T_1\subset T$:
\begin{equs}
	T_0= \begin{tikzpicture}[scale=0.2,baseline=-5]
		\coordinate (root) at (0,0);
		\coordinate (tri) at (0,-2);
		\coordinate (t1) at (-2,2);
		\coordinate (t2) at (2,2);
		\coordinate (t3) at (0,3);
		\draw[kernels2,tinydots] (t1) -- (root);
		\draw[kernels2] (t2) -- (root);
		\draw[kernels2] (t3) -- (root);
		\draw[symbols] (root) -- (tri);
		\node[not] (rootnode) at (root) {};t
		\node[not,label= {[label distance=-0.2em]below: \scriptsize  $  $}] (trinode) at (tri) {};
		\node[var] (rootnode) at (t1) {\tiny{$ k_{\tiny{4}} $}};
		\node[var] (rootnode) at (t3) {\tiny{$ \ell $}};
		\node[var] (trinode) at (t2) {\tiny{$ k_5 $}};
	\end{tikzpicture}, T_1= \begin{tikzpicture}[scale=0.2,baseline=-5]
		\coordinate (root) at (0,0);
		\coordinate (tri) at (0,-2);
		\coordinate (t1) at (-2,2);
		\coordinate (t2) at (2,2);
		\coordinate (t3) at (0,3);
		\draw[kernels2,tinydots] (t1) -- (root);
		\draw[kernels2] (t2) -- (root);
		\draw[kernels2] (t3) -- (root);
		\draw[symbols] (root) -- (tri);
		\node[not] (rootnode) at (root) {};t
		\node[not,label= {[label distance=-0.2em]below: \scriptsize  $  $}] (trinode) at (tri) {};
		\node[var] (rootnode) at (t1) {\tiny{$ k_{\tiny{1}} $}};
		\node[var] (rootnode) at (t3) {\tiny{$ k_{\tiny{2}} $}};
		\node[var] (trinode) at (t2) {\tiny{$ k_3 $}};
	\end{tikzpicture}
\end{equs}
with $k=k_4-\ell-k_5$ and $\ell=-k_1+k_2+k_3$. Let us compute both sides of the identity \eqref{eqn:total_operator_sum_of_lower_tree} for this forest splitting: Beginning with the left hand side where we have
\begin{align*}
	\mathscr{F}_{\text{\tiny{dom}}}(T_0) & = 2k_4^2.
\end{align*}
Moreover, since the tree has just one single blue edge the sum simplifies and we find
\begin{align*}
	\sum_{e \in \tilde{E}_{T_0}}  \mathscr{F}_{\text{\tiny{low}}}(T_0)&= \mathscr{F}_{\text{\tiny{low}}}(T_0^e)=-k_4^2-k_5^2-\ell^2 +k^2,
\end{align*}
and, similarly for $T_1$ we have
\begin{align*}
	\mathscr{F}_{\text{\tiny{dom}}}(T_1) &= 2k_1^2,\\
	\sum_{e \in \tilde{E}_{T_1}}  \mathscr{F}_{\text{\tiny{low}}}(T_1^e) &= -k_1^2-k_2^2-k_3^2+\ell^2.
\end{align*}
For the right hand side on the other hand we obtain
\begin{align*}
	\sum_{v \in L_F} P_{o_{e_v}}(k_v)& =k_4^2+k_1^2-k_2^2-k_3^2-k_5^2,\\
	P_{(\mathfrak{t}_2,0)}(k)&=-k^2.
\end{align*}
Combining all of the above expressions clearly shows that the identity \eqref{eqn:total_operator_sum_of_lower_tree} is indeed satisfied in the present example.
\end{example}
\begin{remark}
The appearance of the term $P_{(\mathfrak{t}_2,0)}(k)$ in \eqref{eqn:total_operator_sum_of_lower_tree} is due to the fact that here we consider forest splittings of trees inside $\hat{\CT}^{k}_{2}(R)$. Had we instead chosen to work with trees inside $\hat{\CT}^{k}_{0}(R)$ this term would disappear from the above identity. Essentially, like in the previous proof, if we want to express the sum of the dominant and lower order parts in a forest splitting we can take advantage of cancellations of the form \eqref{cancellation_1}, meaning as soon as a blue dotted and brown solid edge are adjacent this leads to cancellation of the contribution from the nodal decoration in the overall identity. This means the only terms left are those which cannot be paired with an edge of conjugate colour, in particular the root and all the leaves of the resulting tree.
\end{remark}

\begin{proposition}\label{prop:symmetry_conditions}
If the coefficients $b$, satisfy the following simple relation
\begin{equs} \label{symmetry_condition}
	-\left(\prod_{j=0}^me^{z_j}\right)b_{\mathbf{a}, \chi, T, T_0 \cdot ... \cdot T_{{m}}}(-\tau,-z_j  )=b_{1-\mathbf{a}, 1-\chi, T, T_0 \cdot ... \cdot T_{{m}}}(\tau,z_j )\quad 
\end{equs}
for every $\mathbf{a} \in [ 0, 1]^{\tilde{E}_F},\chi \in \lbrace 0, 1\rbrace^{ L_T}$ and any splitting $T_0 \cdot T_1 ... \cdot T_{{m}} \subset T$, then the method \eqref{eqn:general_formula} is symmetric.
\end{proposition}
\begin{proof}
	Let us consider the adjoint method $\widehat{\Phi}_{n\to n+1}=\Phi_{n+1\to n}^{-1}$, which can be expressed as follows
	\begin{equs}
		\begin{split}
			u_k^{n+1} & =  e^{i \tau P_{o_1}(k)} u_k^n -  \sum_{T \in \hat{\CT}^{r,k}_{2}(R)}  \sum_{\mathbf{a} \in [ 0, 1]^{\tilde{E}_F}}  \sum_{\chi \in \lbrace 0, 1\rbrace^{ L_T}}\sum_{T_0 \cdot T_1 ... \cdot T_{{m}} \subset T} C_T \\ &  b_{\mathbf{a}, \chi, T, T_0 \cdot ... \cdot T_{{m}}}(\tau,i \tau \mathscr{F}_{\text{\tiny{dom}}}(T_j), j \in \lbrace 0,...,{m} \rbrace)
			\\ &	 \prod_{e \in \tilde{E}_{T_j}} e^{i \tau a_e\mathscr{F}_{\text{\tiny{low}}}(T_j^e)} \frac{\Upsilon^p_{\chi}(T)(u_{k_v}^{n +\chi_v}, v \in L_T),\tau}{S(T)}.\end{split}
	\end{equs}
By using Lemma~\ref{lem:total_operator_is_sum_of_low_dom}, we have
	\begin{equs} \label{symme_scheme}
		\begin{split}
			u_k^{n+1} & = e^{i \tau P_{o_1}(k)} u_k^{n}- e^{i \tau P_{o_1}(k)} \sum_{T \in \tilde{\CT}^{r,k}_{0}(R)} \sum_{\mathbf{a} \in [ 0, 1]^{\tilde{E}_F}}  \sum_{\chi \in \lbrace 0, 1\rbrace^{ L_T}}\sum_{T_0 \cdot T_1 ... \cdot T_{{m}} \subset T} C_T \\ & \left(\prod_{j=0}^{m}e^{i\tau \mathscr{F}_{\text{\tiny{dom}}}(T_j)}\right) b_{\mathbf{a}, \chi, T, T_0 \cdot ... \cdot T_{{m}}}(-\tau,-i \tau \mathscr{F}_{\text{\tiny{dom}}}(T_j), j \in \lbrace 0,...,{m} \rbrace)
			\\ &	 \prod_{e \in \tilde{E}_{T_j}} e^{i \tau (1-a_e)\mathscr{F}_{\text{\tiny{low}}}(T_j^e)} \frac{\Upsilon^p_{\chi}(T)(u_{k_v}^{n +1-\chi_v}, v \in L_T, \tau)}{S(T)} \prod_{v \in L_T} e^{-i \tau P_{o_{e_v}}(k_v)}.
		\end{split}
	\end{equs}
Here we have used the fact that
\begin{equs}
	\Upsilon^p_{\chi}(T)(u_{k_v}^{n +1-\chi_v}, v \in L_T,\tau) \prod_{v \in L_T} e^{-i \tau P_{o_{e_v}}(k_v)} & =  \Upsilon^p_{\chi}(T)(u_{k_v}^{n +\chi_v}, v \in L_T,\tau)
\end{equs}
which can be proved easily by induction on $ T $ using the recursive definition of $ \Upsilon^p_{\chi} $.
We have also used the identity
\begin{equs}
	 P_{o_1}(k)   =  - P_{(\mathfrak{t}_2,0)}(k).
\end{equs}
	By comparing \eqref{eqn:general_formula} and \eqref{symme_scheme} we immediately obtain the conditions \eqref{symmetry_condition} for symmetry.
	\end{proof}

\subsection{Examples}\label{sec:examples}
	In this section we illustrate the general framework introduced previously on two examples: the nonlinear Schr\"odinger equation (see Section \ref{sec:NLS}) and the Korteweg-de Vries equation (see Section \ref{sec:KDV}). 
\subsubsection{The Nonlinear Schr\"odinger equation}\label{sec:NLS}
As a first example let us consider the  cubic nonlinear Schr\"odinger (NLS) equation
\begin{equation}\label{NLS}
	i  \partial_t u(t,x) +  \Delta u(t,x) =  \vert u(t,x)\vert^2 u(t,x)\quad (t,x) \in \R \times  \mathbb{T}^d 
\end{equation}
with an initial condition
\begin{equation}
	\label{init}
	u_{|t=0}= u_{0}.
\end{equation}
We start with the construction of a first-order symmetric low-regularity scheme for \eqref{NLS} and illustrate how our general framework covers both the previous explicit low-regularity schemes \cite{BS}, \cite{BBS} and the case of symmetric low-regularity schemes for NLS which was recently introduced and studied in \cite{AB23}, \cite{FMS23}, \cite{MS22}. We then exhibit our new symmetric midpoint rule framework \eqref{mid_point_low_reg_scheme}, which in particular allows for a symmetric second order scheme which is optimal in the sense of regularity.
\newline

Note that the Schrödinger equation \eqref{NLS} fits into the general framework \eqref{dis} with
\begin{equation*}
	\begin{aligned}
		\mathcal{L}\left(\nabla \right)  =\Delta, \quad \alpha = 0   \quad \text{and}\quad  p(u,\overline u) = u^2 \overline u.
	\end{aligned}
\end{equation*} 

Here $ \CL = \lbrace \Labhom_1, \Labhom_2 \rbrace $, $ P_{\Labhom_1} = - \lambda^2 $ and $ P_{\Labhom_2} =  \lambda^2 $, and the structure constant $C_T=1$ for all $T\in \hat{\CT}^{r,k}_{0}(R)$, for any $r\in\N$. Then, we denote by~$ \<thick> $ an edge decorated by $ (\Labhom_1,0) $, $ \<thick2> $ an edge denoted by $ (\Labhom_1,1) $ by $\<thin>$ an edge decorated by $ (\Labhom_2,0) $ and by $\<thin2>$ an edge decorated by $ (\Labhom_2,1) $. The set $ \hat{\CT}^{0,k}_{0}(R) $ is given by:
\begin{equs}\label{eqn:definition_ct120_NLSE1}
	\hat{\CT}^{0,k}_{0}(R) = \left\lbrace T_0, T_1, \, k_i \in \Z^d \right\rbrace, \quad T_0 =  \begin{tikzpicture}[scale=0.2,baseline=-5]
		\coordinate (root) at (0,1);
		\coordinate (tri) at (0,-1);
		\draw[kernels2] (tri) -- (root);
		\node[var] (rootnode) at (root) {\tiny{$ k $}};
		\node[not] (trinode) at (tri) {};
	\end{tikzpicture} , \quad T_1 =  \begin{tikzpicture}[scale=0.2,baseline=-5]
	\coordinate (root) at (0,2);
	\coordinate (tri) at (0,0);
	\coordinate (trib) at (0,-2);
	\coordinate (t1) at (-2,4);
	\coordinate (t2) at (2,4);
	\coordinate (t3) at (0,5);
	\draw[kernels2,tinydots] (t1) -- (root);
	\draw[kernels2] (t2) -- (root);
	\draw[kernels2] (t3) -- (root);
	\draw[kernels2] (trib) -- (tri);
	\draw[symbols] (root) -- (tri);
	\node[not] (rootnode) at (root) {};
	\node[not] (trinode) at (tri) {};
	\node[var] (rootnode) at (t1) {\tiny{$ k_{\tiny{1}} $}};
	\node[var] (rootnode) at (t3) {\tiny{$ k_{\tiny{2}} $}};
	\node[var] (trinode) at (t2) {\tiny{$ k_{\tiny{3}} $}};
	\node[not] (trinode) at (trib) {};
\end{tikzpicture},  
\end{equs}
and $ \hat{\CT}^{1,k}_{0}(R) $ is given by:
\begin{equs}\label{eqn:definition_ct120_NLSE2}
\hat{\CT}^{1,k}_{0}(R)  = \left\lbrace T_0,T_1,T_2,T_3, \, k_i \in \Z^d \right\rbrace, \quad T_2 = \begin{tikzpicture}[scale=0.2,baseline=-5]
		\coordinate (root) at (0,2);
		\coordinate (tri) at (0,0);
		\coordinate (trib) at (0,-2);
		\coordinate (t1) at (-2,4);
		\coordinate (t2) at (2,4);
		\coordinate (t3) at (0,4);
		\coordinate (t4) at (0,6);
		\coordinate (t41) at (-2,8);
		\coordinate (t42) at (2,8);
		\coordinate (t43) at (0,10);
		\draw[kernels2,tinydots] (t1) -- (root);
		\draw[kernels2] (t2) -- (root);
		\draw[kernels2] (t3) -- (root);
		\draw[symbols] (root) -- (tri);
		\draw[symbols] (t3) -- (t4);
		\draw[kernels2,tinydots] (t4) -- (t41);
		\draw[kernels2] (t4) -- (t42);
		\draw[kernels2] (t4) -- (t43);
		\draw[kernels2] (trib) -- (tri);
		\node[not] (trinode) at (trib) {};
		\node[not] (rootnode) at (root) {};
		\node[not] (rootnode) at (t4) {};
		\node[not] (rootnode) at (t3) {};
		\node[not] (trinode) at (tri) {};
		\node[var] (rootnode) at (t1) {\tiny{$ k_{\tiny{4}} $}};
		\node[var] (rootnode) at (t41) {\tiny{$ k_{\tiny{1}} $}};
		\node[var] (rootnode) at (t42) {\tiny{$ k_{\tiny{3}} $}};
		\node[var] (rootnode) at (t43) {\tiny{$ k_{\tiny{2}} $}};
		\node[var] (trinode) at (t2) {\tiny{$ k_5 $}};
	\end{tikzpicture}, \quad T_3 = \begin{tikzpicture}[scale=0.2,baseline=-5]
	\coordinate (root) at (0,2);
	\coordinate (tri) at (0,0);
	\coordinate (trib) at (0,-2);
	\coordinate (t1) at (-2,4);
	\coordinate (t2) at (2,4);
	\coordinate (t3) at (0,4);
	\coordinate (t4) at (0,6);
	\coordinate (t41) at (-2,8);
	\coordinate (t42) at (2,8);
	\coordinate (t43) at (0,10);
	\draw[kernels2] (t1) -- (root);
	\draw[kernels2] (t2) -- (root);
	\draw[kernels2,tinydots] (t3) -- (root);
	\draw[symbols] (root) -- (tri);
	\draw[symbols,tinydots] (t3) -- (t4);
	\draw[kernels2] (t4) -- (t41);
	\draw[kernels2,tinydots] (t4) -- (t42);
	\draw[kernels2,tinydots] (t4) -- (t43);
	\draw[kernels2] (trib) -- (tri);
	\node[not] (trinode) at (trib) {};
	\node[not] (rootnode) at (root) {};
	\node[not] (rootnode) at (t4) {};
	\node[not] (rootnode) at (t3) {};
	\node[not] (trinode) at (tri) {};
	\node[var] (rootnode) at (t1) {\tiny{$ k_{\tiny{4}} $}};
	\node[var] (rootnode) at (t41) {\tiny{$ k_{\tiny{1}} $}};
	\node[var] (rootnode) at (t42) {\tiny{$ k_{\tiny{3}} $}};
	\node[var] (rootnode) at (t43) {\tiny{$ k_{\tiny{2}} $}};
	\node[var] (trinode) at (t2) {\tiny{$ k_5 $}};
\end{tikzpicture}.
\end{equs}
 If we take all coefficients equal to zero whenever $T$ is not given by
 \begin{equs}
		T =  \begin{tikzpicture}[scale=0.2,baseline=-5]
			\coordinate (root) at (0,0);
			\coordinate (tri) at (0,-2);
			\coordinate (t1) at (-2,2);
			\coordinate (t2) at (2,2);
			\coordinate (t3) at (0,3);
			\draw[kernels2,tinydots] (t1) -- (root);
			\draw[kernels2] (t2) -- (root);
			\draw[kernels2] (t3) -- (root);
			\draw[symbols] (root) -- (tri);
			\node[not] (rootnode) at (root) {};t
			\node[not,label= {[label distance=-0.2em]below: \scriptsize  $ $}] (trinode) at (tri) {};
			\node[var] (rootnode) at (t1) {\tiny{$ k_{\tiny{1}} $}};
			\node[var] (rootnode) at (t3) {\tiny{$ k_{\tiny{2}} $}};
			\node[var] (trinode) at (t2) {\tiny{$ k_3 $}};
		\end{tikzpicture}, 
	\end{equs}
and we consider only the forest $ F = T $, 
	then the general formula \eqref{eqn:general_formula} reduces to a single term of the form:
	\begin{equs}
		u_k^{n+1} & = e^{- i \tau k^2} u_k^{n}+  e^{-i \tau k^2}\sum_{a \in [ 0, 1]}  \sum_{\chi \in \lbrace 0, 1\rbrace^{L_T}}  b_{a,\chi}(\tau,{i \tau \mathscr{F}_{\text{\tiny{dom}}}(T)})
		\\ &\quad\quad\quad\quad\quad\quad\quad\quad\quad\quad\quad	 e^{i \tau a \mathscr{F}_{\text{\tiny{low}}}(T)} \frac{\Upsilon^{p}_{\chi}(T)(u_{k_v}^{n +\chi_v}, v \in L_T)}{S(T)}.
	\end{equs}
	Note here $|L_T|=3$ so we can equivalently write the above expression in the following form in Fourier coordinates. Indeed, by noting that  $\mathscr{F}_{\text{\tiny{dom}}}(T)=2k_1^2$, $\mathscr{F}_{\text{\tiny{low}}}(T)=2k_2k_3-2k_1k_2-2k_1k_3$ we have
	\begin{align}\begin{split}\label{eqn:specific_formula_1st_orderr_NLSE}
			u_k^{n+1}&= e^{-i \tau k^2} u^n_k+ e^{-i \tau k^2}\sum_{k=-k_1+k_2+k_3}\sum_{a\in[0,1]} \sum_{\chi \in \lbrace 0, 1\rbrace^{3}}  b_{a,\chi}(\tau,{2i \tau k_1^2 })\\
			&\quad\quad\quad\quad\quad\quad\quad\quad\quad\quad  e^{i \tau a 2(k_2k_3-k_1k_2-k_1k_3)}\overline{\hat{v}_{k_1}^{n+\chi_1}}\hat{v}_{k_2}^{n+\chi_2}\hat{v}_{k_3}^{n+\chi_3}.\end{split}
	\end{align}
		First of all we note that the first order integrator developed in \cite{OS18} falls in this category: Take $b_{0,(0,0,0)}(\tau,z)=-i\tau\varphi_1(z)$ and all other coefficients to zero then we find
	\begin{align*}
		u_k^{n+1}=e^{-i \tau k^2}u^n_k-i e^{-i \tau k^2} \sum_{k=-k_1+k_2+k_3} \tau \varphi_1(2i\tau k_1^2)\overline{u_{k_1}^n} u_{k_2}^n u_{k_3}^n,
	\end{align*}
	which is exactly equal to the integrator introduced in \cite[(4)]{OS18}. In physical space the above scheme is given by
	\begin{equs}
		\label{scheme_NLS_order1_1}
			u^{n+1} = \Phi_{\text{\tiny NLS},1}^\tau(u^n)=&  e^{i\tau\Delta} u^n - i{\tau}  e^{i\tau\Delta} \left( (u^n)^2 \varphi_1(-2i\tau\Delta) \overline{u^n} \right),
	\end{equs}
	where the filter function $\varphi_1$ is defined as $\varphi_1(\sigma) = \frac{e^\sigma - 1}{\sigma}$. 
\newline
	
	 Let us now consider {\it symmetric} schemes. Following Proposition \ref{prop:symmetry_conditions} the scheme \eqref{eqn:general_formula} is symmetric if the following equality  is satisfied for all $a\in[0,1],\chi \in \lbrace 0, 1\rbrace^{3}$:
	\begin{align}\label{eqn:specific_symmetry_conditions_first_order_scheme}
		-e^zb_{a, \chi}(-\tau,-z)=b_{1-a, 1-\chi}(\tau,z).
\end{align}
Intuitively speaking the above equations provide a sufficient condition relating the coefficients $b_{a, \chi}$ and $b_{1-a, 1-\chi}$ therefore allowing us to find symmetric schemes if we specify one of the two for each value of $a,\chi$. There are a large class of first order schemes of this form, but perhaps the most intuitive one is the following symmetrised version of the above integrator \eqref{scheme_NLS_order1_1}, which was recently introduced in \cite{AB23}: 
Take $b_{0,(0,0,0)}(\tau,z)=i/2\tau \varphi_1(z/2)$, then by \eqref{eqn:specific_symmetry_conditions_first_order_scheme} we should choose $b_{1,(1,1,1)}(\tau, z)=i/2\tau\varphi_1(-z/2)$. By letting all the other coefficients equal to zero, we recover the following integrator
\begin{equs} \label{schemeNLS_first scheme} \begin{aligned}
	u^{n+1} = \widehat{\Phi}_{\text{\tiny NLS},1} \circ\Phi_{\text{\tiny NLS},1}(u^n)=&  e^{i\tau\Delta} u^n - i \frac{\tau}{2} e^{i\tau\Delta} \left( (u^n)^2 \varphi_1(-i\tau\Delta) \overline{u^n} \right) \\
	&- i\frac{\tau}{2} \left( (u^{n+1})^2\varphi_1(i\tau \Delta)\overline{u^{n+1}} \right),
	\end{aligned}
\end{equs}
{ given by the composition of the non-symmetric resonance-based scheme \eqref{scheme_NLS_order1_1} with its adjoint. Higher-order resonance-based symmetric schemes can analogously be obtained by composing a higher order asymmetric resonance based scheme $\Phi_{\text{\tiny NLS}}$ with its adjoint $\widehat{\Phi}_{\text{\tiny NLS}} \circ \Phi_{\text{\tiny NLS}}$, and fits in the framework of formula \eqref{eqn:general_formula}. Nevertheless, there are simpler higher-order symmetric resonance-based schemes that involve fewer terms than those written in this compositional form, and are based on {\it midpoint} approximations. We discuss this next.}
\newline

Another symmetric first order integrator can be found by taking \begin{equs}
	b_{0,(1,0,0)}(\tau,z)=i/2\tau\varphi_1(z/2), \quad b_{1,(0,1,1)}(\tau,z)=i/2\tau\varphi_1(-z/2),
\end{equs}
 and all other coefficients to zero.
Next we  choose the coefficients 
\begin{equs}
b_{a,(\chi_1,\chi_2,\chi_3)}(\tau,z)=   - i \frac{\tau}{16} \varphi_1(z)
\end{equs}
for every $ a, \chi_j \in \left\lbrace 0, 1 \right\rbrace $. The other coefficients are set to be zero which leads to the following symmetric scheme
for \eqref{NLS} :
	\begin{equs}\label{schemeNLS3}
		u^{n+1} & = \Phi_{\text{\tiny NLS},3}^\tau(u^n) \\ & =  e^{i\tau\Delta} u^n
		- i \frac{\tau}{16} e^{i\tau\Delta} \left( (u^n + e^{-i\tau\Delta}u^{n+1})^2 \varphi_1(-2i\tau\Delta) \left(\overline{u^n} + e^{i\tau\Delta}\overline{u^{n+1}}\right)\right) \\
		&- i\frac{\tau}{16} \left( (e^{i\tau\Delta}u^n + u^{n+1})^2\varphi_1(2i\tau \Delta)\left( e^{-i\tau\Delta}\overline{u^n}+ \overline{u^{n+1}}\right) \right).
	\end{equs}
	The above scheme can also be recursively derived by  the general framework of the {\it midpoint} rule \eqref{mid_point_low_reg_scheme} and therefore allows for higher order symmetric counterparts which are optimal in the sense of regularity. Our characterisation of symmetric schemes in Proposition~\ref{prop:symmetry_conditions} immediately confirm this method to be symmetric, since for all $a,\chi$ we have
	\begin{align*}
		-e^zb_{a, \chi}(-\tau,-z)=  - i \frac{\tau}{16} e^z\varphi_1(-z)= - i \frac{\tau}{16} \frac{1-e^{z}}{-z}=\varphi_1(z)b_{1-a, 1-\chi}(\tau,z).
\end{align*}
	
	\begin{proposition}\label{prop:deriviation_of_schemeNLS3_from_trees}
		The scheme \eqref{schemeNLS3} can be derived from the general tree series expansion \eqref{mid_point_low_reg_scheme}.
		\end{proposition}
	\begin{proof}
	At first order it follows from \eqref{mid_point_low_reg_scheme} that we have
\begin{equs}
	U_{k}^{r}(\tau, u) & = \sum_{T \in \hat{\CT}^{0,k}_{0}(R)} \frac{\Upsilon^{p}_{\text{\tiny{mid}}}( T)(u,\tau)}{S(T)} (\Pi^{n,r}_{\text{\tiny{mid}}}   T_0 )(\tau)
	\\  &  = \frac{\Upsilon^{p}_{\text{\tiny{mid}}}( T_0)(u,\tau)}{S(T_0)} (\Pi^{n,r}_{\text{\tiny{mid}}}   T_0 )(\tau) + \sum_{k = -k_1 + k_2 + k_3} \frac{\Upsilon^{p}_{\text{\tiny{mid}}}( T_1)(u,\tau)}{S(T_1)} (\Pi^{n,r}_{\text{\tiny{mid}}}   T_1 )(\tau)
\end{equs}
From the definition of the symmetry factor, one has
\begin{equs}
	S(T_0) = 1, \quad S(T_1) = 2
\end{equs}
for $ S(T_1) $ the factor two is due to the fact that we have two solid brown edges $ \<thick> $ attached to the same node in $ T_1 $. Moreover, we have:
\begin{equs}
	\Upsilon^{p}_{\text{\tiny{mid}}}( T_0)(u,\tau) & = u_k^{\ell}
	\\ \hat{\Upsilon}^{p}_{\text{\tiny{mid}}}( T_0)(u,\tau)  & = \frac{1}{2} ( e^{i \tau k^2} u_{k}^{\ell+1} + u_{k}^{\ell} ) 
\end{equs}
where we have used $ \ell $ instead of $ n $ such that to not create confusion with $ \Pi^{n,r} $.
 By multiplicativity, for the following tree:
\begin{equs}
	\tilde{T}_1 = 	\begin{tikzpicture}[scale=0.2,baseline=-5]
		\coordinate (root) at (0,-2);
		\coordinate (t1) at (-2,0);
		\coordinate (t2) at (2,0);
		\coordinate (t3) at (0,1);
		\draw[kernels2,tinydots] (t1) -- (root);
		\draw[kernels2] (t2) -- (root);
		\draw[kernels2] (t3) -- (root);
		\node[not] (rootnode) at (root) {};
		\node[var] (rootnode) at (t1) {\tiny{$ k_{\tiny{1}} $}};
		\node[var] (rootnode) at (t3) {\tiny{$ k_{\tiny{2}} $}};
		\node[var] (trinode) at (t2) {\tiny{$ k_3 $}}; 
	\end{tikzpicture}, 
\end{equs}
we have
\begin{equs}
		\Upsilon^{p}_{\text{\tiny{mid}}}( T_1)(u,\tau) & = 	\hat{\Upsilon}^{p}_{\text{\tiny{mid}}}( \tilde{T}_1)(u,\tau) 
		\\ &  = 2  \frac{1}{2} ( e^{-i \tau k_1^2} \bar{u}_{k_1}^{\ell+1} + \bar{u}_{k_1}^{\ell} ) \frac{1}{2} ( e^{i \tau k_2^2} u_{k_2}^{\ell+1} + u_{k_2}^{\ell} ) \frac{1}{2} ( e^{i \tau k_3^2} u_{k_3}^{\ell+1} + u_{k_3}^{\ell} ) 
\end{equs}
On the other hand, we have
\begin{equs}
		(\Pi_{\text{\tiny{mid}}}^{n,r}   T_0 )(\tau) = e^{-i \tau k^2}, \quad (\Pi_{\text{\tiny{mid}}}^{n,r}   \tilde{T}_1 )(s,\tau) = e^{i (k_1^2 - k_2^2 - k_3^2)},
	\end{equs}
Then, 
\begin{equs}
		(\Pi_{\text{\tiny{mid}}}^{n,r}   T_1 )(\tau)  = e^{-i \tau k^2} \CK^{k,r}_{o_2} ( (\Pi_{\text{\tiny{mid}}}^{n,r-1}   \tilde{T}_1 )(\cdot,\tau),n)(\tau) 
\end{equs}
We compute the scheme for $ n=1 $ and $ r =  0$. We obtain the following term:
\begin{equs}
\CK^{k,r}_{o_2} ( (\Pi_{\text{\tiny{mid}}}^{n,r-1}   \tilde{T}_1 )(\cdot,\tau),n)(\tau) = -i	\int_0^{\tau} e^{i s \mathcal{L}_{\text{\tiny{dom}}} } ds \left( \frac{1 +e^{i \tau \mathcal{L}_{\text{\tiny{low}}} } }{2}  \right)
\end{equs}
where
\begin{equs}
	\mathcal{L}_{\text{\tiny{dom}}} = 2 k_1^2, \quad 	\mathcal{L}_{\text{\tiny{low}}} =  k^2 - k_1^2 -k_2^2 -k_3^2.
\end{equs}
In the end, we have
\begin{equs}
	(\Pi_{\text{\tiny{mid}}}^{n,r}   T_1 )(\tau)  =-i \tau \phi_1(2 i \tau k_1^2 ) \left( \frac{e^{-i \tau k^2} +e^{-i \tau (k_1^2 + k_2^2 + k_3^2)  } }{2}  \right).
\end{equs}
We note that
\begin{equs}
	\phi_1(2 i \tau k_1^2 ) e^{- 2 i \tau k_1^2 }  = 	\phi_1(-2 i \tau k_1^2 ). 
\end{equs}
Therefore, we have
\begin{equs}
&	\sum_{k = -k_1 + k_2 + k_3}  \frac{\Upsilon^{p}_{\text{\tiny{mid}}}( T_1)(u,\tau)}{S(T_1)} (\Pi^{n,r}_{\text{\tiny{mid}}}   T_1 )(\tau)  = - i	\sum_{k = -k_1 + k_2 + k_3} \frac{e^{-i \tau k^2}}{2} \tau \phi_1(2 i \tau k_1^2 ) \\ &  
	\frac{1}{2} ( e^{-i \tau k_1^2} \bar{u}_{k_1}^{\ell+1} + \bar{u}_{k_1}^{\ell} ) \frac{1}{2} ( e^{i \tau k_2^2} u_{k_2}^{\ell+1} + u_{k_2}^{\ell} ) \frac{1}{2} ( e^{i \tau k_3^2} u_{k_3}^{\ell+1} + u_{k_3}^{\ell} ) 
	\\ & + \tau \frac{\phi_1(-2 i \tau k_1^2 )}{2}  \frac{1}{2} (  \bar{u}_{k_1}^{\ell+1} + e^{i \tau k_1^2} \bar{u}_{k_1}^{\ell} ) \frac{1}{2} ( u_{k_2}^{\ell+1} +  e^{-i \tau k_2^2}u_{k_2}^{\ell} ) \frac{1}{2} (  u_{k_3}^{\ell+1} +e^{-i \tau k_3^2} u_{k_3}^{\ell} ) .
	\end{equs}
In physical space this leads to the first order symmetric low regularity integrator for the NLS equation \eqref{schemeNLS3}.
\end{proof}

\begin{remark} The derivation of the scheme \eqref{schemeNLS3} from the general midpoint Duhamel iterations exhibits an interesting recipe for resonance based schemes that are constructed for equations of the form
		\begin{align*}
			& i \partial_t u(t,x) +   \mathcal{L}\left(\nabla, \tfrac{1}{\varepsilon}\right) u(t,x) =\vert \nabla\vert^\alpha p\left(u(t,x), \overline u(t,x)\right), \\
			& u(0,x) = v(x).
		\end{align*}
Indeed, suppose we have already obtained an explicit asymmetric \textit{first-order} resonance based scheme for the above equation (cf. \cite{OS18,HS17} etc.) of the general form
		\begin{align*}
			u^{n+1}=\Phi_{\tau}(u^n),
		\end{align*}
	where $\Phi_{\tau}$ is a general nonlinear map representing the time step, then this can be easily converted to a symmetric method  by simply considering
	\begin{align*}
		u^{n+1}=\Phi_{\tau}\left(\frac{e^{-i\tau \mathcal{L}\left(\nabla, \tfrac{1}{\varepsilon}\right)}u^{n+1}+u^{n}}{2}\right).
	\end{align*}
	It can be shown that the above scheme is of first order with the same regularity assumptions as required by the asymmetric method $\Phi$, and is of second order for smoother solutions (see Remark \ref{rem:evenOrder}). 
		\end{remark}
\begin{remark}
In similar vein to Proposition~\ref{prop:deriviation_of_schemeNLS3_from_trees} we could derive the scheme~\eqref{schemeNLS_first scheme} from a generalised tree series expansion. Indeed, instead of using a midpoint iteration of Duhamel's formula as introduced in Section~\ref{Duhamel_iteration_mid_point} by averaging \eqref{eqn:left_endpoint_duhamel} and \eqref{eqn:right_end_point_duhamel}, we would iterate in 
all appearances of $u_j(t_n+\tilde{s})$ 
the left endpoint Duhamel formula \eqref{eqn:left_endpoint_duhamel} in the terms integrated over the interval $[0,s]$ 
and the right endpoint Duhamel formula \eqref{eqn:right_end_point_duhamel} in the terms integrated over the interval $[s, \tau]$.
\end{remark}

\begin{proposition}
	The schemes \eqref{scheme_NLS_order1_1}, \eqref{schemeNLS_first scheme} and \eqref{schemeNLS3}  have a local error of order $\CO(\tau^2\nabla u)$. 
	\end{proposition}
	
\begin{proof}
First, for the explicit scheme \eqref{scheme_NLS_order1_1} the local error directly follows from Theorem \ref{thm:genloc} and can be computed using  Definition~\ref{def:Llow} as it is performed in the proof of \cite[Cor. 5.1]{BS}. In order to obtain the local error bounds of the implicit schemes \eqref{schemeNLS_first scheme}, \eqref{schemeNLS3} one needs to apply Theorem \ref{thm:genloc} and to combine it with a fixed-point argument on the numerical flow. To go further, the first order convergence of these schemes follow by combining the local error bound with a stability argument, we refer to the works \cite{MS22, AB23} which perform this analysis in full detail.
\end{proof}

\begin{remark}
The symmetric scheme  \eqref{schemeNLS_first scheme} was first rigorously analysed in \cite{AB23}. In particular  it was shown in \cite{AB23} that the local error of the scheme is of order $\CO(\tau^2\nabla u)$, which is optimal in regard of the regularity assumptions. Indeed, the scheme \eqref{schemeNLS_first scheme}  does not require more regularity on the solution than \georginline{previously constructed} asymmetric low regularity integrators such as \eqref{scheme_NLS_order1_1} introduced in \cite{OS18, BS}. As the scheme \eqref{schemeNLS_first scheme} is symmetric it is naturally also of second order; however, not under optimal regularity assumptions (see also Remark \ref{rem:evenOrder}). More precisely, by exploiting the tools presented in \cite{AB23} one can show that the scheme \eqref{schemeNLS_first scheme} (as well as \eqref{schemeNLS3}) is of second order with a local error of order $\CO(\tau^3\nabla \Delta u)$. This error structure imposes more regularity on the solution than asymmetric low-regularity integrators such as the ones proposed in \cite{BS} which only require the boundedness of two additional derivatives instead of three due to the local error of the form $\CO(\tau^3 \Delta u)$.

Our new symmetric midpoint rule framework \eqref{mid_point_low_reg_scheme} allows for a symmetric second order scheme which is optimal in the sense of regularity, i.e., has a local error structure of the form $\CO(\tau^3 \Delta u)$, see the scheme \eqref{schemeNLS2} below.
\end{remark}

\begin{proposition}\label{prop:analysis_schemeNLS2}	The second order scheme coming from  \eqref{mid_point_low_reg_scheme} is given by:
		\begin{equs}
			\label{schemeNLS2}
			u^{n+1} & = \varphi_{\text{\tiny NLS},4}^\tau(u^n)=  e^{i\tau\Delta} u^n \\
			&- i \frac{\tau}{8} e^{i\tau\Delta} \left( (u^n + e^{-i\tau\Delta}u^{n+1})^2 \left( \varphi_1(-2i\tau\Delta) - \varphi_2(-2i\tau\Delta) \right)\left(\overline{u^n} + e^{i\tau\Delta}\overline{u^{n+1}}\right)\right) \\
			&- i\frac{\tau}{8} \left( (e^{i\tau\Delta}u^n + u^{n+1})^2\varphi_2(-2i\tau \Delta)\left( e^{i\tau\Delta}\overline{u^n}+e^{2i\tau\Delta} \overline{u^{n+1}}\right) \right)
		\end{equs}
	with a local error structure of the form $\CO(\tau^3 \Delta u)$ and $\varphi_2(\sigma) = \frac{e^{\sigma}-\varphi_1(\sigma)}{\sigma}$.
	\end{proposition}
\begin{remark}\label{impl-comp}
	In practice the computational effort required to compute $u^{n+1}$ in \eqref{schemeNLS2} is not significantly larger than the solution of \eqref{schemeNLS3}, \eqref{schemeNLS2}, or other previous symmetric low-regularity methods for the NLSE. In particular, a similar analysis to that presented in \cite[Section~3]{AB23}, \cite[Appendix~A]{BMS22} and \cite{MS22} shows that the implicit equations can be solved efficiently using fixed point iterations, and that the number of iterations required is independent of the number of spatial discretisation points.
	\end{remark}

\begin{remark}
	One can find the coefficients $ b $ for the scheme \eqref{schemeNLS2} such that it is of the form given by \eqref{eqn:general_formula}.
\end{remark}

\begin{proof}[of Proposition~\ref{prop:analysis_schemeNLS2}]
 For the scheme of order two $ (r=1) $ and $ n =2 $, one has to consider:
\begin{equs}
U_{k}^{n,1}(\tau, u) & = \sum_{T \in \hat{\CT}^{1,k}_{0}(R)} \frac{\Upsilon^{p}_{\text{\tiny{mid}}}( T)(u,\tau)}{S(T)} (\Pi_{\text{\tiny{mid}}}^{n,1}   T_0 )(\tau)
\\  &  = \frac{\Upsilon^{p}_{\text{\tiny{mid}}}( T_0)(u,\tau)}{S(T_0)} (\Pi_{\text{\tiny{mid}}}^{n,1}   T_0 )(\tau) + \sum_{k = -k_1 + k_2 + k_3} \frac{\Upsilon^{p}_{\text{\tiny{mid}}}( T_1)(u,\tau)}{S(T_1)} (\Pi_{\text{\tiny{mid}}}^{n,1}   T_1 )(\tau)
\\ &  +\sum_{k=-k_1+k_2+k_3 -k_4+k_5} \frac{\Upsilon^{p}_{\text{\tiny{mid}}}( T_2)(u,\tau)}{S(T_2)} \left( \Pi^{n,1}_{\text{\tiny{mid}}}   T_2 \right)(\tau)
\\& +\sum_{k_1-k_2-k_3 +k_4+k_5= k} \frac{\Upsilon^{p}_{\text{\tiny{mid}}}(  T_3)(u,\tau)}{S(T_3)} \left(\Pi^{n,1}_{\text{\tiny{mid}}}   T_3 \right)(\tau).
\end{equs}
From the definition of the symmetry factor, we have
\begin{equs}
	S(T_2) =  1 \times 2  = 2, \quad S(T_3) = 2 \times 2= 4,
\end{equs}
for $ S(T_2) $ the factor one corresponds to the fact that for the node on top of the first blue edges the symmetry factor is one. Indeed, the trees on top of the brown edges are different: a leaf decorated by $ k_5 $ is different from a tree having three leaves.
 Moreover, we have:
\begin{equs}
	\Upsilon^{p}_{\text{\tiny{mid}}}( T_j)(u,\tau) & =\Upsilon^{p}( T_j)(\frac{1}{2} ( e^{i \tau k^2} u^{n+1} + u^{n} )), \quad j \in \left\lbrace 2,3 \right\rbrace,
\end{equs}
where
\begin{equs}
	\Upsilon^{p}( T_2)(u)  = 4 \bar{u}_{k_1} u_{k_2} u_{k_3} \bar{u}_{k_4} u_{k_5},
	\quad
	 \Upsilon^{p}( T_3)(u)  = 4 u_{k_1} \bar{u}_{k_2} \bar{u}_{k_3} u_{k_4} u_{k_5}.
\end{equs}
The factor $ 4 $ in both expressions comes from the two brown edges that appear twice inside the decorated trees $ T_2 $ and $ T_3 $. For the term $ (\Pi_{\text{\tiny{mid}}}^{2,1}   T_1 )(\tau)  $, we proceed with interpolation at two nodes ($a_0=0,a_1=1$):
\begin{equs}
	p_2(s,\tau) = 1 + \frac{s}{\tau} \left( e^{i s \mathcal{L}_{\text{\tiny{low}}}} - 1 \right),
\end{equs}
where
\begin{equs}
	\mathcal{L}_{\text{\tiny{dom}}} = 2 k_1^2, \quad 	\mathcal{L}_{\text{\tiny{low}}} =  k^2 - k_1^2 -k_2^2 -k_3^2.
\end{equs}
We obtain
\begin{equs}
	\CK^{k,1}_{o_2} ( (\Pi_{\text{\tiny{mid}}}^{2,0}   \tilde{T}_1 )(\cdot,\tau),2)(\tau) & = -i	\int_0^{\tau} e^{i s \mathcal{L}_{\text{\tiny{dom}}} } ds
	  -i	\int_0^{\tau} s e^{i s \mathcal{L}_{\text{\tiny{dom}}} } ds \left( \frac{e^{i \tau \mathcal{L}_{\text{\tiny{low}}} } -1 }{\tau}  \right)
	  \\ & - i \tau \varphi_1(2 i \tau k_1^2) - i \tau \varphi_2(2 i \tau k_1^2) \left(e^{ik^2 - k^2_1 - k_2^2 - k_3^2} - 1 \right).
\end{equs}
Therefore, 
\begin{equs}
	& \sum_{k = -k_1 + k_2 + k_3} \frac{\Upsilon^{p}_{\text{\tiny{mid}}}( T_1)(u,\tau)}{S(T_1)} (\Pi_{\text{\tiny{mid}}}^{2,1}   T_1 )(\tau)  \\ &= \left( - i \tau \varphi_1(2 i \tau k_1^2) - i \tau \varphi_2(2 i \tau k_1^2) \left(e^{ik^2 - k^2_1 - k_2^2 - k_3^2} - 1 \right) \right)
	\\ & \times   ( e^{-i \tau k_1^2} \bar{u}_{k_1}^{n+1} + \bar{u}_{k_1}^{n} ) \frac{1}{2} ( e^{i \tau k_2^2} u_{k_2}^{n+1} + u_{k_2}^{n} ) \frac{1}{2} ( e^{i \tau k_3^2} u_{k_3}^{n+1} + u_{k_3}^{n} ) 
\end{equs}
For the decorated tree $ T_2 $, we have 
\begin{equs} \label{F_2}
\left(	\Pi^{2,1}_{\text{\tiny{mid}}}   T_2 \right)(\tau) = e^{-i \tau k^2} 	\CK^{k,1}_{o_2} \left((\Pi^{2,0}_{\text{\tiny{mid}}}   F_2 )(\cdot,\tau) ,2 \right)(0,\tau)	
\end{equs}
where 
\begin{equs}
	 F_2 = \mathcal{I}_{(\mathfrak{t}_1,1)}(\lambda_{k_4}) \mathcal{I}_{(\mathfrak{t}_1,0)}(\lambda_{k_5})  T_1.
\end{equs}
Then,
\begin{equs}
	(\Pi^{2,0}_{\text{\tiny{mid}}}   \mathcal{I}_{(\mathfrak{t}_1,1)}(\lambda_{k_4}) \mathcal{I}_{(\mathfrak{t}_1,0)}(\lambda_{k_5}) )(s,\tau) =  e^{i s (k_4^2 -k_5^2)}
\end{equs}
and for $ \tilde{k} = -k_1 + k_2 + k_3 $
\begin{equs}
	(\Pi^{2,0}_{\text{\tiny{mid}}}   T_1 )(s,\tau) & = \frac{1}{2}   e^{-i s \tilde{k}^2} \left(  \CK^{\tilde{k},0}_{o_2} ( (\Pi_{\text{\tiny{mid}}}^{2,-1}   \tilde{T}_1 )(\cdot,\tau),2)(s,\tau) \right. \\ & \left. +  \CK^{\tilde{k},0}_{o_2} ( (\Pi_{\text{\tiny{mid}}}^{2,-1}   \tilde{T}_1 )(\cdot,\tau),2)(s,0) \right),
\end{equs}
Now, because of $ n=2 $, we perform a \georginline{direct interpolation of the full operator} which gives
\begin{equs}
\CK^{\tilde{k},0}_{o_2} ( (\Pi_{\text{\tiny{mid}}}^{2,-1}   \tilde{T}_1 )(\cdot,\tau),2)(s,\tau) = -i \int_{\tau}^{s} ds \left( \frac{1 + e^{i \tau (\tilde{k}^2 + k_1^2 - k_2^2 - k_3^2)}}{2} \right).
\end{equs}
We obtain
\begin{equs}	(\Pi^{2,0}_{\text{\tiny{mid}}}   T_1 )(s,\tau) = - i \frac{(2s -\tau)}{2} e^{-i s \tilde{k}^2}  \left( \frac{1 + e^{i \tau (\tilde{k}^2 + k_1^2 - k_2^2 - k_3^2)}}{2} \right).
\end{equs}
Therefore, we find
\begin{equs}
	(\Pi^{2,0}_{\text{\tiny{mid}}}   F_2 )(s,\tau) = - i \frac{(2s-\tau)}{2} e^{i s( k_4^2 - k_5^2 -\tilde{k}^2)}  \left( \frac{1 + e^{i \tau (\tilde{k}^2 + k_1^2 - k_2^2 - k_3^2)}}{2} \right)
\end{equs}
and by performing again an interpolation of the full operator:
\begin{equs}
\left(	\Pi^{2,1}_{\text{\tiny{mid}}}   T_2 \right)(\tau) & = - \int^{\tau}_0	\frac{(2s -\tau)}{2} ds \\ &\quad\quad e^{-i \tau k^2 } 	\left( \frac{1 + e^{i \tau (k^2 + k_4^2 - k_5^2 - \tilde{k}^2)}}{2} \right)	\left( \frac{1 + e^{i \tau (\tilde{k}^2 + k_1^2 - k_2^2 - k_3^2)}}{2} \right)
\\ &  = 0
\end{equs}
because we have
\begin{equs}
	\int^{\tau}_0	\frac{(2s -\tau)}{2} ds = 0.
\end{equs}
A similar computation shows that
\begin{equs}
\left(	\Pi^{2,1}_{\text{\tiny{mid}}}   T_3 \right)(\tau) & = 0.
\end{equs}
The local error analysis follows from  the proof of \cite[Cor. 5.3]{BS}. 
\end{proof}

\begin{example} To illustrate how the general formula can be used to express more general higher order resonance based schemes, let us consider the following low-regularity integrator
		\begin{align}\begin{split}\label{eqn:second_order_symmetric_scheme}
		u^{n+1}&=e^{i\tau\Delta}u-i\frac{\tau}{2}e^{i\tau\Delta}\left(\left(u^{n}\right)^2\left(\varphi_1(-i\tau\Delta)-\varphi_2(-i\tau\Delta)\right)\overline{u^n}\right)\\
		&\quad -i\frac{\tau}{2}\left(\left(u^{n+1}\right)^2\left(\varphi_1(i\tau\Delta)-\varphi_2(i\tau\Delta)\right)\overline{u^{n+1}}\right)\\
		&\quad-i\frac{\tau}{2}e^{i\frac{\tau}{2} \Delta}\left(e^{i\frac{\tau}{2}\Delta}u^{n}\right)^2\varphi_2(-i\tau\Delta)e^{i\frac{\tau}{2}\Delta}\overline{u^{n}}\\
		&\quad-i\frac{\tau}{2}e^{i\frac{\tau}{2}\Delta}\left(e^{-i\frac{\tau}{2}\Delta}u^{n+1}\right)^2\varphi_2(i\tau\Delta)e^{-i\frac{\tau}{2}\Delta}\overline{u^{n+1}}\\
		&\quad -\frac{\tau^2}{8}\left|u^n\right|^4e^{-i\tau\Delta}u^{n+1}+\frac{\tau^2}{8}\left|u^{n+1}\right|^4e^{i\tau\Delta}u^{n}\end{split}
		\end{align}
which is motivated from \cite[(5.16)]{BS} and has a local error of the form $\mathcal{O}(\tau^3\Delta)$. We will see that this scheme is in the form of the general formula \eqref{eqn:general_formula} with $r=1$.	Indeed, we already saw at the beginning of this section that
		\begin{equ}
			\hat{\CT}^{1,k}_{0}(R)  = \left\lbrace T_0,T_1,T_2,T_3, \, k_i \in \Z^d \right\rbrace
		\end{equ}
		where $T_i, i=0,1,2,3,$ were defined in \eqref{eqn:definition_ct120_NLSE1} and \eqref{eqn:definition_ct120_NLSE2}. In the interest of brevity we will not derive all the coefficients $b$ in the expression of \eqref{eqn:second_order_symmetric_scheme} in the form \eqref{eqn:general_formula}. Instead let us focus on the coefficients arising from the contribution to \eqref{eqn:general_formula} arising from $T=\hat{T_2}$ where $ \hat{T_2} $ is obtained by removing the brown edge attached to the root. To understand this we first recall the forest splittings of this choice of $T$ from Example \ref{ex:forest_splittings}:
		\begin{equs}
			\one \cdot  \begin{tikzpicture}[scale=0.2,baseline=-5]
				\coordinate (root) at (0,0);
				\coordinate (tri) at (0,-2);
				\coordinate (t1) at (-2,2);
				\coordinate (t2) at (2,2);
				\coordinate (t3) at (0,2);
				\coordinate (t4) at (0,4);
				\coordinate (t41) at (-2,6);
				\coordinate (t42) at (2,6);
				\coordinate (t43) at (0,8);
				\draw[kernels2,tinydots] (t1) -- (root);
				\draw[kernels2] (t2) -- (root);
				\draw[kernels2] (t3) -- (root);
				\draw[symbols] (root) -- (tri);
				\draw[symbols] (t3) -- (t4);
				\draw[kernels2,tinydots] (t4) -- (t41);
				\draw[kernels2] (t4) -- (t42);
				\draw[kernels2] (t4) -- (t43);
				\node[not] (rootnode) at (root) {};
				\node[not] (rootnode) at (t4) {};
				\node[not] (rootnode) at (t3) {};
				\node[not,label= {[label distance=-0.2em]below: \scriptsize  $  $}] (trinode) at (tri) {};
				\node[var] (rootnode) at (t1) {\tiny{$ k_{\tiny{4}} $}};
				\node[var] (rootnode) at (t41) {\tiny{$ k_{\tiny{1}} $}};
				\node[var] (rootnode) at (t42) {\tiny{$ k_{\tiny{3}} $}};
				\node[var] (rootnode) at (t43) {\tiny{$ k_{\tiny{2}} $}};
				\node[var] (trinode) at (t2) {\tiny{$ k_5 $}};
			\end{tikzpicture}, \qquad \one \cdot \begin{tikzpicture}[scale=0.2,baseline=-5]
				\coordinate (root) at (0,0);
				\coordinate (tri) at (0,-2);
				\coordinate (t1) at (-2,2);
				\coordinate (t2) at (2,2);
				\coordinate (t3) at (0,3);
				\draw[kernels2,tinydots] (t1) -- (root);
				\draw[kernels2] (t2) -- (root);
				\draw[kernels2] (t3) -- (root);
				\draw[symbols] (root) -- (tri);
				\node[not] (rootnode) at (root) {};t
				\node[not,label= {[label distance=-0.2em]below: \scriptsize  $  $}] (trinode) at (tri) {};
				\node[var] (rootnode) at (t1) {\tiny{$ k_{\tiny{4}} $}};
				\node[var] (rootnode) at (t3) {\tiny{$ \ell $}};
				\node[var] (trinode) at (t2) {\tiny{$ k_5 $}};
			\end{tikzpicture} \cdot \begin{tikzpicture}[scale=0.2,baseline=-5]
				\coordinate (root) at (0,0);
				\coordinate (tri) at (0,-2);
				\coordinate (t1) at (-2,2);
				\coordinate (t2) at (2,2);
				\coordinate (t3) at (0,3);
				\draw[kernels2,tinydots] (t1) -- (root);
				\draw[kernels2] (t2) -- (root);
				\draw[kernels2] (t3) -- (root);
				\draw[symbols] (root) -- (tri);
				\node[not] (rootnode) at (root) {};t
				\node[not,label= {[label distance=-0.2em]below: \scriptsize  $  $}] (trinode) at (tri) {};
				\node[var] (rootnode) at (t1) {\tiny{$ k_{\tiny{1}} $}};
				\node[var] (rootnode) at (t3) {\tiny{$ k_{\tiny{2}} $}};
				\node[var] (trinode) at (t2) {\tiny{$ k_3 $}};
			\end{tikzpicture}, \qquad
			 \begin{tikzpicture}[scale=0.2,baseline=-5]
			\coordinate (root) at (0,0);
			\coordinate (tri) at (0,-2);
			\coordinate (t1) at (-2,2);
			\coordinate (t2) at (2,2);
			\coordinate (t3) at (0,3);
			\draw[kernels2,tinydots] (t1) -- (root);
			\draw[kernels2] (t2) -- (root);
			\draw[kernels2] (t3) -- (root);
			\draw[symbols] (root) -- (tri);
			\node[not] (rootnode) at (root) {};t
			\node[not,label= {[label distance=-0.2em]below: \scriptsize  $  $}] (trinode) at (tri) {};
			\node[var] (rootnode) at (t1) {\tiny{$ k_{\tiny{4}} $}};
			\node[var] (rootnode) at (t3) {\tiny{$ \ell $}};
			\node[var] (trinode) at (t2) {\tiny{$ k_5 $}};
		\end{tikzpicture} \cdot \begin{tikzpicture}[scale=0.2,baseline=-5]
				\coordinate (root) at (0,0);
				\coordinate (tri) at (0,-2);
				\coordinate (t1) at (-2,2);
				\coordinate (t2) at (2,2);
				\coordinate (t3) at (0,3);
				\draw[kernels2,tinydots] (t1) -- (root);
				\draw[kernels2] (t2) -- (root);
				\draw[kernels2] (t3) -- (root);
				\draw[symbols] (root) -- (tri);
				\node[not] (rootnode) at (root) {};t
				\node[not,label= {[label distance=-0.2em]below: \scriptsize  $  $}] (trinode) at (tri) {};
				\node[var] (rootnode) at (t1) {\tiny{$ k_{\tiny{1}} $}};
				\node[var] (rootnode) at (t3) {\tiny{$ k_{\tiny{2}} $}};
				\node[var] (trinode) at (t2) {\tiny{$ k_3 $}};
			\end{tikzpicture}, \quad  \begin{tikzpicture}[scale=0.2,baseline=-5]
			\coordinate (root) at (0,0);
			\coordinate (tri) at (0,-2);
			\coordinate (t1) at (-2,2);
			\coordinate (t2) at (2,2);
			\coordinate (t3) at (0,2);
			\coordinate (t4) at (0,4);
			\coordinate (t41) at (-2,6);
			\coordinate (t42) at (2,6);
			\coordinate (t43) at (0,8);
			\draw[kernels2,tinydots] (t1) -- (root);
			\draw[kernels2] (t2) -- (root);
			\draw[kernels2] (t3) -- (root);
			\draw[symbols] (root) -- (tri);
			\draw[symbols] (t3) -- (t4);
			\draw[kernels2,tinydots] (t4) -- (t41);
			\draw[kernels2] (t4) -- (t42);
			\draw[kernels2] (t4) -- (t43);
			\node[not] (rootnode) at (root) {};
			\node[not] (rootnode) at (t4) {};
			\node[not] (rootnode) at (t3) {};
			\node[not,label= {[label distance=-0.2em]below: \scriptsize  $  $}] (trinode) at (tri) {};
			\node[var] (rootnode) at (t1) {\tiny{$ k_{\tiny{4}} $}};
			\node[var] (rootnode) at (t41) {\tiny{$ k_{\tiny{1}} $}};
			\node[var] (rootnode) at (t42) {\tiny{$ k_{\tiny{3}} $}};
			\node[var] (rootnode) at (t43) {\tiny{$ k_{\tiny{2}} $}};
			\node[var] (trinode) at (t2) {\tiny{$ k_5 $}};
		\end{tikzpicture}
		\end{equs}
		Let us now consider the contribution from the second of these splittings, $\tilde{T}_0\cdot \tilde{T}_1\cdot \tilde{T}_2$ with
		\begin{align*}
			\tilde{T}_0=\one, \quad \tilde{T}_1=\begin{tikzpicture}[scale=0.2,baseline=-5]
				\coordinate (root) at (0,0);
				\coordinate (tri) at (0,-2);
				\coordinate (t1) at (-2,2);
				\coordinate (t2) at (2,2);
				\coordinate (t3) at (0,3);
				\draw[kernels2,tinydots] (t1) -- (root);
				\draw[kernels2] (t2) -- (root);
				\draw[kernels2] (t3) -- (root);
				\draw[symbols] (root) -- (tri);
				\node[not] (rootnode) at (root) {};t
				\node[not,label= {[label distance=-0.2em]below: \scriptsize  $  $}] (trinode) at (tri) {};
				\node[var] (rootnode) at (t1) {\tiny{$ k_{\tiny{4}} $}};
				\node[var] (rootnode) at (t3) {\tiny{$ \ell $}};
				\node[var] (trinode) at (t2) {\tiny{$ k_5 $}};
			\end{tikzpicture}, \tilde{T}_2=\begin{tikzpicture}[scale=0.2,baseline=-5]
				\coordinate (root) at (0,0);
				\coordinate (tri) at (0,-2);
				\coordinate (t1) at (-2,2);
				\coordinate (t2) at (2,2);
				\coordinate (t3) at (0,3);
				\draw[kernels2,tinydots] (t1) -- (root);
				\draw[kernels2] (t2) -- (root);
				\draw[kernels2] (t3) -- (root);
				\draw[symbols] (root) -- (tri);
				\node[not] (rootnode) at (root) {};t
				\node[not,label= {[label distance=-0.2em]below: \scriptsize  $  $}] (trinode) at (tri) {};
				\node[var] (rootnode) at (t1) {\tiny{$ k_{\tiny{1}} $}};
				\node[var] (rootnode) at (t3) {\tiny{$ k_{\tiny{2}} $}};
				\node[var] (trinode) at (t2) {\tiny{$ k_3 $}};
			\end{tikzpicture}
		\end{align*}
		where from Kirchhoff's law we have $k=-k_4+l+k_5$ and $l=-k_1+k_2+k_3$. The dominant and lower order operators arising in these splittings are given by
		\begin{align*}
			\mathscr{F}_{\text{\tiny{dom}}}(\tilde{T}_0)&=0,\mathscr{F}_{\text{\tiny{low}}}(\tilde{T}_0)=0\\
			\mathscr{F}_{\text{\tiny{dom}}}(\tilde{T}_1)&=2k_4^2,\mathscr{F}_{\text{\tiny{low}}}(\tilde{T}_1)=k^2-k_4^2-l^2-k_5^2=2(-k_4l-k_4k_5+lk_5)\\
			\mathscr{F}_{\text{\tiny{dom}}}(\tilde{T}_2)&=2k_1^2,\mathscr{F}_{\text{\tiny{low}}}(\tilde{T}_2)=l^2-k_1^2-k_2^2-k_3^2=2(-k_1k_2-k_1k_3+k_2k_3),
		\end{align*}
	Moreover we have $|\tilde{E}_T|$ (there are only two blue edges in $T$ corresponding to time-integration) and $|L_T|=5$ ($T$ has 5 leaves), and given the above splitting, $|\tilde{E}_{\tilde{T}_0}|=0,|\tilde{E}_{\tilde{T}_1}|=|\tilde{E}_{\tilde{T}_2}|=1$. Thus the contribution from this term to the overall sum in \eqref{eqn:general_formula} is of the form (recall that $C_T=1$ for all $T$ in the NLSE case)
	\begin{align*}
		& \sum_{\mathbf{a} \in [ 0, 1]^2}  \sum_{\chi \in \lbrace 0, 1\rbrace^{ 5}} \,b_{\mathbf{a}, \chi, T, \tilde{T}_0 \cdot \tilde{T}_1\cdot \tilde{T}_{2}}(\tau,i \tau \mathscr{F}_{\text{\tiny{dom}}}(\tilde{T}_0),i \tau \mathscr{F}_{\text{\tiny{dom}}}(\tilde{T}_1),i \tau \mathscr{F}_{\text{\tiny{dom}}}(\tilde{T}_2))
		\\ &\hspace{3cm}	 e^{i \tau a_1\mathscr{F}_{\text{\tiny{low}}}(\tilde{T}_1)}e^{i \tau a_2\mathscr{F}_{\text{\tiny{low}}}(\tilde{T}_2)} \frac{\Upsilon^p_{\chi}(T)(u_{k_v}^{n +\chi_v}, v \in L_T, \tau)}{S(T)}
	\end{align*}
From the derivation in the proof of Proposition \ref{prop:analysis_schemeNLS2} and \eqref{new_upsilon} we have that
\begin{align*}
	&\frac{\Upsilon^p_{\chi}(T)(u_{k_v}^{n +\chi_v}, v \in L_T, \tau)}{S(T)}\\&\hspace{2cm}=2e^{-i\tau \chi_1 k_1^2}\overline{u}_{k_1}^{\chi_1}e^{i\tau \chi_2 k_2^2}u_{k_2}^{\chi_2}e^{i\tau \chi_3 k_3^2}u_{k_3}^{\chi_3}e^{-i\tau \chi_4 k_4^2}\overline{u}_{k_4}^{\chi_4}e^{i\tau \chi_5 k_5^2}u_{k_5}^{\chi_5}
\end{align*}
Thus the contribution to \eqref{eqn:general_formula} equals
\begin{align*}
&\sum_{\mathbf{a} \in [ 0, 1]^2}  \sum_{\chi \in \lbrace 0, 1\rbrace^{ 5}} \,b_{\mathbf{a}, \chi, T, \tilde{T}_0 \cdot \tilde{T}_1\cdot \tilde{T}_{2}}(\tau,0,2i \tau k_4^2,2i \tau k_1^2)
\\ &\hspace{3cm}	 e^{i \tau a_1(k^2-k_4^2-(-k_1+k_2+k_3)^2-k_5^2)}e^{i \tau a_2((-k_1+k_2+k_3)^2-k_1^2-k_2^2-k_3^2)}\\
&\hspace{3.0cm}2e^{-i\tau \chi_1 k_1^2}\overline{u}_{k_1}^{\chi_1}e^{i\tau \chi_2 k_2^2}u_{k_2}^{\chi_2}e^{i\tau \chi_3 k_3^2}u_{k_3}^{\chi_3}e^{-i\tau \chi_4 k_4^2}\overline{u}_{k_4}^{\chi_4}e^{i\tau \chi_5 k_5^2}u_{k_5}^{\chi_5}
\end{align*}
Let us now show that the quintic terms in \eqref{eqn:second_order_symmetric_scheme} arise precisely from these contributions: Indeed suppose we choose
\begin{align}\label{eqn:b_coeff_quintic_terms}
	b_{\mathbf{0},(0,0,0,0,1),T,\tilde{T}_0\cdot\tilde{T}_1\cdot\tilde{T}_2}(\tau,z_0,z_1,z_2)&=-\frac{\tau^2}{16},\\
	b_{(1,1),(1,0,0,0,0),T,\tilde{T}_0\cdot\tilde{T}_1\cdot\tilde{T}_2}(\tau,z_0,z_1,z_2)&=-e^{z_0+z_1+z_2}\frac{\tau^2}{16}
\end{align}
and all other coefficients $b$ in the above expression equal to zero then we arrive precisely at the contributions of the form
\begin{align*}
-\frac{\tau^2}{8}\left|u^n\right|^4e^{-i\tau\Delta}u^{n+1}+\frac{\tau^2}{8}\left|u^{n+1}\right|^4e^{i\tau\Delta}u^{n}
\end{align*}
in the overall scheme, corresponding to the quintic terms in \eqref{eqn:second_order_symmetric_scheme}. The remaining terms in the scheme can be expressed similarly from contributions from lower rank trees $T_0,T_1$. Moreover, we note that the coefficients as given by \eqref{eqn:b_coeff_quintic_terms} clearly satisfy \eqref{symmetry_condition} and that the same holds for the coefficients of lower order contributions, thus confirming that the scheme \eqref{eqn:second_order_symmetric_scheme} is symmetric.
\end{example}
\subsubsection{The Korteweg--de Vries equation}\label{sec:KDV}

The Korteweg--de Vries (KdV) equation is given by
\begin{equs}\label{kdv}
	\partial_t u + \partial_x^{3} u = \frac12 \partial_x u^2
\end{equs}
It fits into the general framework with
\begin{equation*}\label{kgrDo}
	\begin{aligned}
		\mathcal{L}\left(\nabla, \right)  = i \partial_x^3, \quad \alpha = 1   \quad \text{and}\quad   p(u,\overline u) =  p(u) = i  \frac12 u^2.
	\end{aligned}
\end{equation*} 

Here $ \CL = \lbrace \Labhom_1, \Labhom_2 \rbrace $, $ P_{\Labhom_1} = - \lambda^3 $ and $ P_{\Labhom_2} =  \lambda^3 $. Moreover, in this case the structure constant $C_T$ reflects the presence of the Burger's nonlinearity in the iterations of Duhamel's formula, which means
\begin{align*}
	C_T=\prod_{\substack{e=(v,u) \in \tilde{E}_{T}\\
			u\in N_T \setminus \lbrace\varrho_T \rbrace}} (-1)^{\mathfrak{p}(e) }i\Labo(u)
\end{align*}
where we recall $\Labe(e) = (\Labhom(e),\mathfrak{p}(e))$ is the edge decoration of $e$ with $\Labhom(e)\in  \Lab$ and $\mathfrak{p}(e)\in \lbrace 0,1\rbrace$. Note by Kirchhoff's law the above definition is invariant \georginline{under the choice of} node $u$ or $v$ for an edge $e=(u,v)$ in the product, so long as the node is an interior one. Then, we denoted by $ \<thick> $ an edge decorated by $ (\Labhom_1,0) $ and by $\<thin>$ an edge decorated by $ (\Labhom_2,0) $.
Following the formalism given in \cite{BHZ}, we can provide the rules that generate the trees obtained by iterating Duhamel's formula:
\begin{equ}
	R(\<thin>) = \{(\<thick>,\<thick>) \}\;, \quad
	R(\<thick>) = \{(\<thin>), ()\}\;.
\end{equ}

The general framework~\eqref{genscheme} derived in Section \ref{Duhamel_iteration_mid_point} builds the foundation of the  first- and second-order resonance based schemes presented below for the KdV equation~\eqref{kdv}. The structure of the schemes depends  on the regularity of the solution.

\begin{corollary}\label{corKdV} For the KdV equation \eqref{kdv} the general midpoint scheme~ \eqref{genscheme} takes at first order the form
	\begin{equation}
		\begin{aligned}\label{schemeKdV1}
			u^{\ell+1} &=  e^{-\tau \partial_x^3} u^\ell + \frac{1}{24} \left(e^{-\tau\partial_x^3 }\partial_x^{-1}  u^{\ell} + \partial_x^{-1} u^{\ell+1} \right)^2 \\ & - \frac{1}{24} e^{-\tau\partial_x^3} \left(\partial_x^{-1}  u^{\ell} + e^{\tau \partial_x^3} \partial_x^{-1} u^{\ell+1} \right)^2\end{aligned}
	\end{equation}
	with a local error  of order $\mathcal{O}\Big(
	\tau^2 \partial_x^2 u
	\Big)$
	at  first-order
	and
	with a local error  of order $\mathcal{O}\Big(
	\tau^3 \partial_x^4 u
	\Big)$ at second order. 
\end{corollary}
\begin{remark}
	Note that this schemes has been obtained in \cite{MS22}. It was shown that this scheme is of even order for higher regularity in $ H^4 $  (see \cite[Thm 5.2]{MS22}). By embedding this scheme into our general framework, we know that it has the same local error analysis as the second-order scheme introduced in \cite{BS}. \begin{equation}
		\begin{aligned}\label{schemeKdV}
			u^{\ell+1} &=  e^{-\tau \partial_x^3} u^\ell + \frac16 \left(e^{-\tau\partial_x^3 }\partial_x^{-1} u^\ell\right)^2 - \frac16 e^{-\tau\partial_x^3} \left(\partial_x^{-1} u^\ell\right)^2\\& +\frac{\tau^2}{4} e^{- \tau \partial_x^3}\Psi\big(i \tau \partial_x^2\big)  \Big(\partial_x \Big(u^\ell \partial_x (u^\ell u^\ell)\Big)\Big)
		\end{aligned}
	\end{equation}
	with a local error  of order $\mathcal{O}\Big(
	\tau^3 \partial_x^4 u
	\Big)$ and  a suitable filter function $\Psi$ satisfying
	\begin{equs}
		\Psi= \Psi\left(i \tau \partial_x^2 \right), \quad \Psi(0) = 1, \quad \Vert \tau   \Psi \left(i \tau \partial_x^2\right) \partial_x^2 \Vert_r \leq 1.
	\end{equs}
\end{remark}
\begin{proof}
	The proof follows the line of argumentation to the analysis for the Schrödinger equation.
	For the first-order scheme, we have
	\begin{equs}\label{kdvU}
		U_{k}^{n,0}(\tau, v) & = \frac{\Upsilon^{p}_{\text{\tiny{mid}}}( T_0 )(\tau, v)}{S(T_0)} \Pi^{n,0}_{\text{\tiny{mid}}} \left( T_0 \right)(\tau)\\
		& +  \sum_{k = k_1 + k_2} \frac{\Upsilon^{p}_{\text{\tiny{mid}}}(  T_1 )(\tau, v)}{S(T_1)} \Pi^{n,0}_{\text{\tiny{mid}}} \left(  T_1 \right)(\tau).
	\end{equs}
	where the trees of interest are 
	\begin{equs}
		\hat{\CT}^{0,k}_{0}(R)= \lbrace  T_0, T_1, \, k_i \in \Z^{d} \rbrace, \quad T_0 = \begin{tikzpicture}[scale=0.2,baseline=-5]
			\coordinate (root) at (0,1);
			\coordinate (tri) at (0,-1);
			\draw[kernels2] (tri) -- (root);
			\node[var] (rootnode) at (root) {\tiny{$ k $}};
			\node[not] (trinode) at (tri) {};
		\end{tikzpicture} \quad\text{and}\quad T_1 =   \begin{tikzpicture}[scale=0.2,baseline=-5]
			\coordinate (root) at (0,2);
				\coordinate (root1) at (0,0);
			\coordinate (tri) at (0,-2);
			\coordinate (t1) at (-1,4);
			\coordinate (t2) at (1,4);
			\draw[kernels2] (t1) -- (root);
			\draw[kernels2] (t2) -- (root);
			\draw[symbols] (root) -- (root1);
			\draw[kernels2] (tri) -- (root1);
			\node[not] (rootnode) at (root) {};
			\node[not] (rootnode) at (root1) {};
			\node[not] (trinode) at (tri) {};
			\node[var] (rootnode) at (t1) {\tiny{$ k_{\tiny{1}} $}};
			\node[var] (trinode) at (t2) {\tiny{$ k_2 $}};
		\end{tikzpicture}  
	\end{equs}
	and in symbolic notation  takes the form
	\begin{equs}
		T_1 =\CI_{(\Labhom_1,0)}( \CI_{(\Labhom_2,0)}\left ( \lambda_{k} F_1\right)) \quad F_1 = \CI_{(\Labhom_1,0)}(  \lambda_{k_1})  \CI_{(\Labhom_1,0) }( \lambda_{k_2}) \quad \text{with } k = k_1+k_2.
	\end{equs}
	For the first term we readily obtain that
	\begin{equs}
	\frac{\Upsilon^{p}_{\text{\tiny{mid}}}( T_0 )(\tau, v)}{S(T_0)} \Pi^{n,0}_{\text{\tiny{mid}}} \left( T_0 \right)(\tau)
		= e^{-i \tau k^3} \hat{v}_k.
	\end{equs}
	It remains to compute the second term.  Note that thanks to \eqref{recursive_pi_r} we have that
	\begin{equs}\label{kdvST}
		\Pi^{n,0} \left( T_1 \right)(\tau) &=  e^{-i \tau k^3} \Pi^{n,0} (\CI_{(\Labhom_2,0)}( \lambda_k F_1) )(\tau)
		\\ & =  e^{-i \tau k^3} \mathcal{K}_{(\Labhom_2,0)}^{k,0} \left( \Pi^{n,-1} (F_1), n\right)(\tau)
		\\ & = e^{-i \tau k^3} \mathcal{K}_{(\Labhom_2,0)}^{k,0} \left( e^{i \xi (-k_1^3 - k_2^3)}, n \right)(\tau).
	\end{equs}
	where we have used for the third line	
	\begin{equs}
		(\Pi^{n,-1} F_1)(\tau)   = (\Pi^{n,-1}  \CI_{(\Labhom_1,0)} ( \lambda_{k_1})) (\tau) (\Pi^{n,-1} \CI_{(\Labhom_1,0)}(  \lambda_{k_2})) (\tau)  = e^{-i \tau k_1^3} e^{-i \tau k_2^3}.
	\end{equs}
	Next  we observe that
	\begin{equs}
		P_{(\Labhom_2,0)}(k) - k_1^3-k_2^3  =  k^3 - k_1^3 - k_2^3 = 3 k_1 k_2 (k_1+k_2)
	\end{equs}
	such that
	\begin{equs}
		\frac{1}{ P_{(\Labhom_2,0)}(k) - k_1^3-k_2^3  } 
	\end{equs}
	can be mapped back to physical space. Therefore, we set
	\begin{equs}
		\mathcal{L}_{\tiny\text{dom}} =  P_{(\Labhom_2,0)}(k) - k_1^3-k_2^3  = 3 k_1 k_2 (k_1+k_2)
	\end{equs}
	and integrate all frequencies exactly. This implies
	\begin{equs}
		\Pi^{n,0} ( T_1)(\tau)
		& = e^{-i \tau k^3} \frac{i (k_1+k_2) }{3i  k_1 k_2 (k_1+k_2)} \left(e^{i \tau (k^3 - k_1^3 -k_2^3)}-1\right)\\
		& =\frac{1 }{3  k_1 k_2 } \left(e^{- i \tau ( k_1^3 + k_2^3)}- e^{-i \tau k^3} \right).
	\end{equs}
	Together with \eqref{kdvU} this yields the scheme \eqref{schemeKdV1}.
	For the second-order scheme, we first notice that
	\begin{equs}
		 \Pi^{n,0}_{\text{\tiny{mid}}} \left( T_0 \right)(\tau) = \Pi^{n,1}_{\text{\tiny{mid}}} \left( T_0 \right)(\tau), \quad \Pi^{n,0}_{\text{\tiny{mid}}} \left( T_1 \right)(\tau) = \Pi^{n,1}_{\text{\tiny{mid}}} \left( T_1 \right)(\tau).
		\end{equs}
	Indeed, for the tree $ T_1 $, we perform an exact integration without any discretisation.
Then, we need to take into account the following trees
	\begin{equs} \hat{\CT}^{1,k}_{0}(R) = \lbrace  T_0, T_1, T_2, \, k_i \in \Z^{d} \rbrace, \quad 
		T_2 = \begin{tikzpicture}[scale=0.2,baseline=-5]
			\coordinate (root) at (0,2);
				\coordinate (root1) at (0,0);
			\coordinate (tri) at (0,-2);
			\coordinate (t1) at (-1,4);
			\coordinate (t11) at (-2,6);
			\coordinate (t12) at (-3,8);
			\coordinate (t13) at (-1,8);
			\coordinate (t2) at (1,4);
			\draw[kernels2] (t11) -- (t13);
			\draw[kernels2] (t11) -- (t12);
			\draw[kernels2] (t1) -- (root);
			\draw[symbols] (t1) -- (t11);
			\draw[kernels2] (t2) -- (root);
			\draw[symbols] (root) -- (root1);
			\draw[kernels2] (root1) -- (tri);
			\node[not] (rootnode) at (root) {};
				\node[not] (rootnode) at (root1) {};
			\node[not] (trinode) at (tri) {};
			\node[not] (trinode) at (t1) {};
			\node[var] (rootnode) at (t12) {\tiny{$ k_{\tiny{1}} $}};
			\node[var] (rootnode) at (t13) {\tiny{$ k_{\tiny{2}} $}};
			\node[var] (trinode) at (t2) {\tiny{$ k_3 $}};
		\end{tikzpicture} 
	\end{equs}
	Then we can proceed as in the second-order schemes for the Schrödinger equation to show its contribution is zero that is 
	\begin{equs}
		\Pi^{4,1}_{\text{\tiny{mid}}} \left( T_2 \right)(\tau) = 0.
	\end{equs}
\end{proof}

\section{Numerical Experiments}\label{sec:numerical_experiments}
We now test the practical performance of our new symmetric schemes in practical experiments evaluating both their low-regularity convergence properties and their ability to correctly preserve constants of motion in the relevant equations. In fitting with our above construction our spatial discretisation is a Fourier spectral method throughout with $M$ modes. In  order to understand the low-regularity convergence properties of our methods we follow \cite{OS18} and consider the following types of initial data:
\begin{enumerate}
	\item Smooth initial data, \begin{align}\label{eqn:numerical_experiments_smooth_data}
		u_0(x)=\frac{\cos(x)}{2+\sin(x)}.
	\end{align}
	\item Low-regularity initial data $u_0$ of the following form. Firstly we choose a vector sampled from a uniform distribution $U_{m}\sim U([0,1]+i[0,1]),\  m=-M/2+1,\dots, M/2$ and then we define
	\begin{align}\label{eqn:numerical_experiments_low_reg_data}
		u_0(x):=U_0+\sum_{\substack{m=-M/2+1\\m\neq 0}}^{M/2}e^{imx}|m|^{-\vartheta}U_m,
	\end{align}
for a given value of $\vartheta>0$, which corresponds to $u_0\in H^{\vartheta}$.
\end{enumerate}
Both choices of initial data {are projected on the sphere in $L^2$} via the map $u_0\mapsto u_0/\|u_0\|_{L^2}$.
\subsection{The Nonlinear Schr\"odinger equation}
To begin with we look at our new symmetric integrators for the NLSE (\eqref{schemeNLS3} and \eqref{schemeNLS2}). In the following numerical experiments we compare the performance of our new schemes to the following state-of-the-art reference schemes for the NLSE:
\begin{itemize}
	\item The Strang splitting \cite{McLacQ02}, as an example of a classical symmetric numerical technique;
	\item The first and second order resonance based integrators introduced by Ostermann \& Schratz \cite{OS18} and Bruned \& Schratz \cite[Section 5.1.2]{BS} respectively, as examples of asymmetric low-regularity schemes;
	\item The symmetrised low-regularity integrator introduced by Alama Bronsard \cite{AB23}, as an example of previous structure preserving low-regularity schemes.
\end{itemize}
In the following numerical experiments we focus on the 1d case, i.e. the NLSE formulated on $\mathbb{T}$, but our methods equally apply to higher dimensional settings where their favourable performance can also be observed.

In the first instance we consider the long-time structure preservation properties of our newly designed symmetric low-regularity integrators. For this we consider two first integrals of the cubic NLSE, the normalisation
\begin{equs}
	I_0^{[NLSE]}[u]=\int_{\mathbb{T}}|u|^2 dx,
\end{equs}
and the energy
\begin{equs}
	I_1^{[NLSE]}[u]=\int_{\mathbb{T}} |\nabla u|^2+\frac{1}{2}|u|^4 dx.
\end{equs}
Symmetric numerical schemes are typically unable to preserve such conservation laws exactly, however it is known for the ODE case \cite[Chapter XI]{HLW} (and also observed numerically for the PDE case, for example in \cite{CCO08}) that symmetric methods can exhibit very good approximate long-time preservation of such first integrals. In the following numerical experiments we test this behaviour by looking at the error in these quantities for a fixed time step $\tau=0.02$ and highest frequency $M=1024$, over a long time interval, much larger than $\mathcal{O}(1/\tau)$.

\begin{figure}[h!]
	\centering
	\begin{subfigure}{0.9\textwidth}
		\centering
		\includegraphics[width=0.84\linewidth]{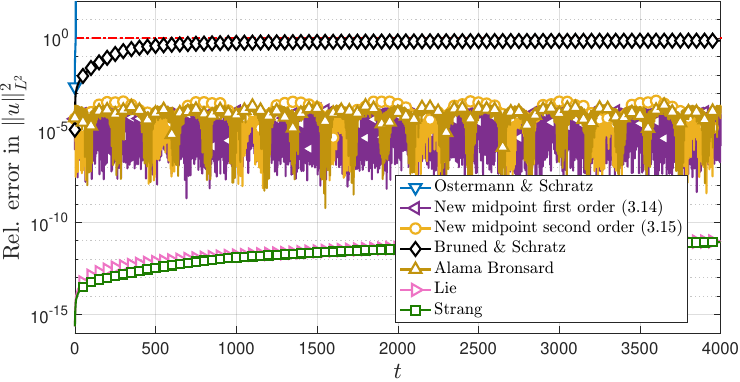}
		\caption{Long-time interval $t=n\tau\in[0,4000]$, $M=128$.}
	\end{subfigure}
	\begin{subfigure}{0.9\textwidth}
		\centering
		\includegraphics[width=0.84\linewidth]{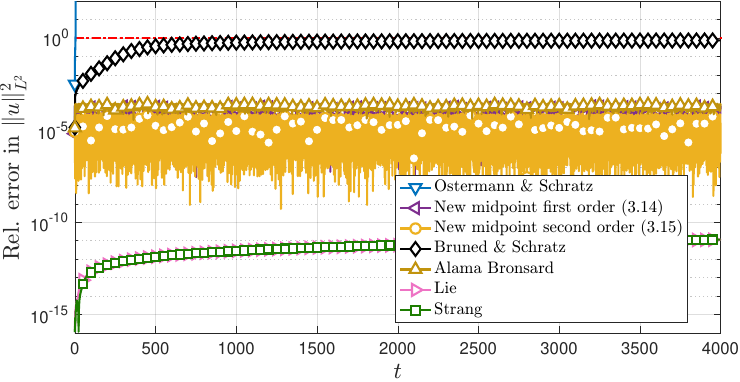}
		\caption{Long-time interval $t=n\tau\in[0,4000]$, $M=1024$.}
	\end{subfigure}
	\begin{subfigure}{0.9\textwidth}
		\centering
		\includegraphics[width=0.84\linewidth]{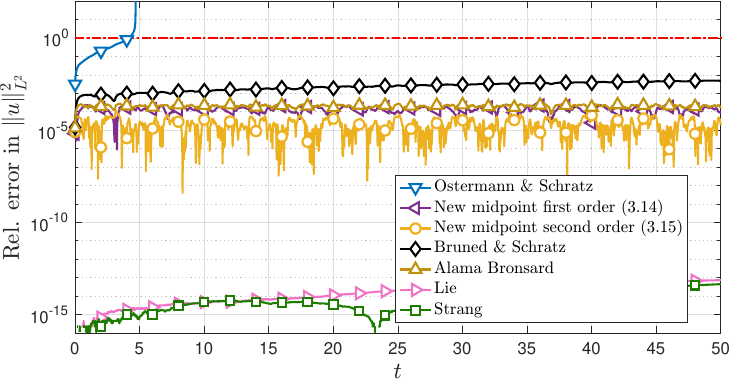}
		\caption{Magnification of $t=n\tau\in[0,50]$, $M=1024$.}
	\end{subfigure}
	\caption{Error in the normalisation $\|u^{n}\|_{L^{2}}$, for time step $\tau=0.02$ and $H^2$ data {for multiple spatial resolutions $M$}. {The figure shows that the new methods (3.14) \& (3.15) are able to preserve the normalisation over long time intervals even in the low-regularity setting, while previous explicit low-regularity integrators (Ostermann \& Schratz and Bruned \& Schratz) accumulate error rapidly already after short simulation times. This favorable long-time behaviour is observed both for coarse spatial resolutions ($M=128$) and for fine ones ($M=1024$) suggesting the behaviour is not overly sensitive towards the relationship between $\tau$ and $M$ (i.e. no strong CFL-type condition is required).}}
	\label{fig:NLSE_error_normalisation_rough}\vspace{-0.5cm}
\end{figure}
Firstly, in figures~\ref{fig:NLSE_error_normalisation_rough}\&\ref{fig:NLSE_error_normalisation_smooth} we observe that the normalisation appears to be preserved really well, in particular (the example is representative of a host of numerical experiments for various time steps which we performed) the preservation is much better than previous asymmetric resonance based schemes. We note that the Strang splitting conserves quadratic first integrals, i.e. the normalisation, to machine accuracy and thus undoubtedly outperforms our schemes on the level of normalisation preservation.
\begin{figure}[h!]
	\begin{subfigure}{0.99\textwidth}
		\centering
		\includegraphics[width=0.9\linewidth]{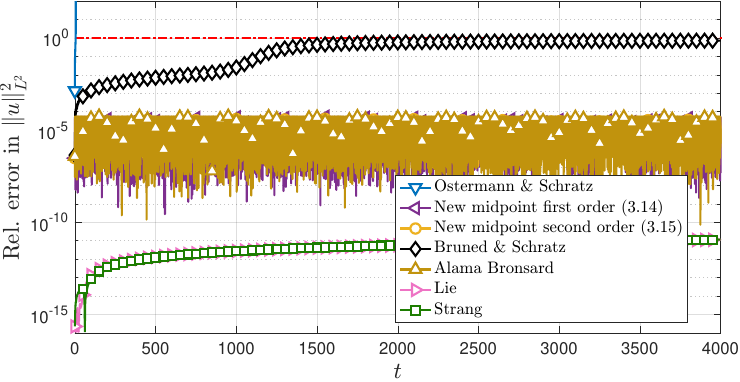}
		\caption{Long-time interval $t=n\tau\in[0,4000]$, $M=1024$.}
	\end{subfigure}
	\begin{subfigure}{0.99\textwidth}
		\centering
		\includegraphics[width=0.9\linewidth]{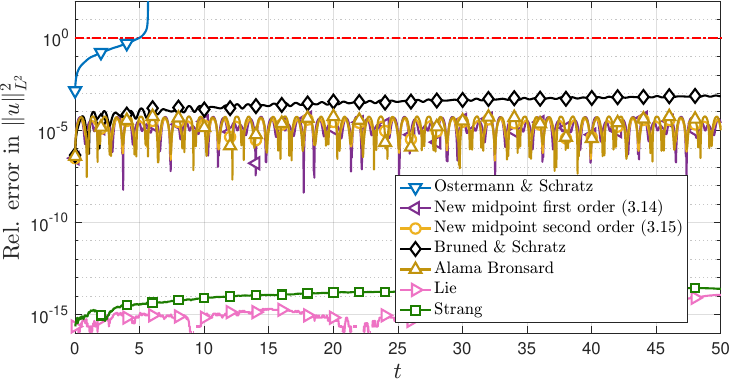}
		\caption{Magnification of $t=n\tau\in[0,50]$, $M=1024$.}
	\end{subfigure}
	\caption{Error in the normalisation $\|u^{n}\|_{L^{2}}$, for time step $\tau=0.02$ and $C^\infty$ data, {highlighting that the proposed methods perform similarly well in the high- as in the low-regularity setting.}}
	\label{fig:NLSE_error_normalisation_smooth}\vspace{-0.3cm}
\end{figure}

Our next numerical result for the NLSE, presented in Figure~\ref{fig:error_energy_rough}\&\ref{fig:error_energy_smooth}, shows the error in the NLSE energy over a long time interval, for a fixed time step $\tau=0.02$. Albeit rigorous theory exists for the ODE case \cite{HLW}, there is again no theoretical guarantee for the long-time preservation of the energy under symmetric methods. Indeed in practical experiments it can be seen that symmetric schemes are able to clearly outperform asymmetric integrators in the long-time approximate energy preservation. For the Strang splitting this behaviour was rigorously analysed in \cite{Faou12} where a CFL condition was necessary to guarantee long-time approximate energy preservation beyond the realms of forward error analysis. This CFL condition is indeed observed even for smooth data in our experiments. Perhaps somewhat surprisingly our new symmetric resonance based scheme do not appear to suffer from comparable CFL conditions and, as expected, perform well for both smooth and low-regularity solutions.

\begin{figure}[h!]
	\centering
	\begin{subfigure}{0.9\textwidth}
		\centering
		\includegraphics[width=0.83\linewidth]{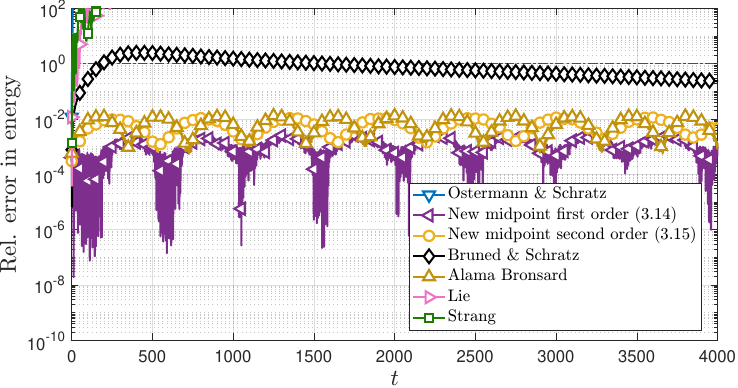}
		\caption{$M=128$.}
	\end{subfigure}
	\begin{subfigure}{0.9\textwidth}
		\centering
		\includegraphics[width=0.83\linewidth]{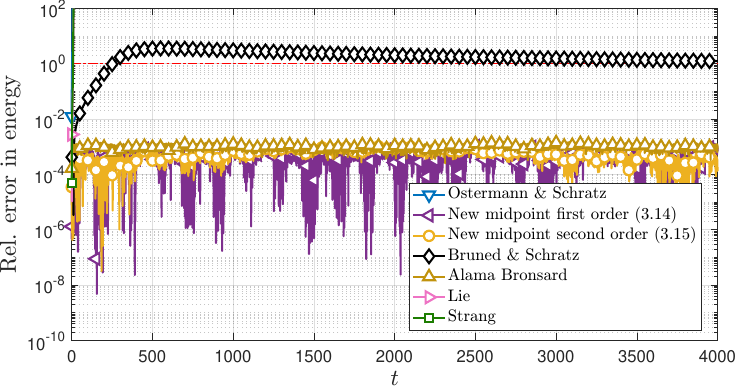}
		\caption{$M=1024$.}
	\end{subfigure}
	\begin{subfigure}{0.9\textwidth}
		\centering
		\includegraphics[width=0.83\linewidth]{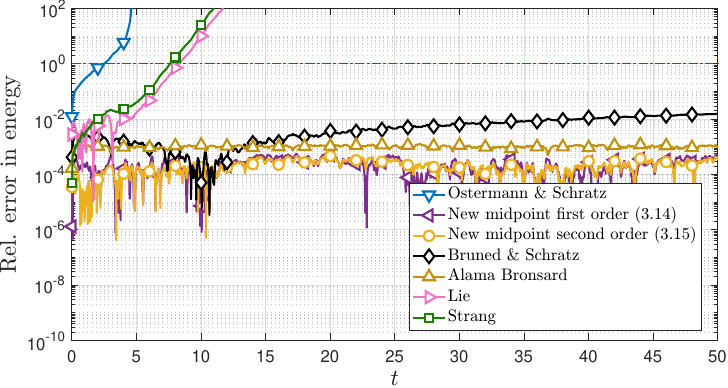}
		\caption{$M=1024$.}
	\end{subfigure}
	\caption{Error in the Hamiltonian, for time step $\tau=0.02$ and $H^2$ data {and multiple spatial resolutions $M$}. {We observe that, like the normalisation in Figure~\ref{fig:NLSE_error_normalisation_rough}, the energy is preserved approximately over long times with our proposed symmetric low-regularity integrators, while previous explicit low-regularity integrators again quickly accummulate error in the energy. In addition, for the energy we observe that in the low-regularity regime both splitting methods (Lie and Strang) fail to achieve energy preservation beyond short times.}}
	\label{fig:error_energy_rough}
\end{figure}
\begin{figure}[h!]
	\centering
	\begin{subfigure}{0.9\textwidth}
		\centering
		\includegraphics[width=0.88\linewidth]{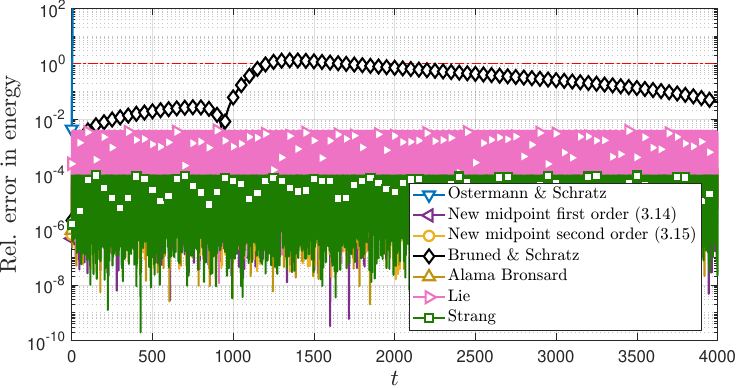}
		\caption{$M=128$.}
	\end{subfigure}
	\begin{subfigure}{0.9\textwidth}
		\centering
		\includegraphics[width=0.88\linewidth]{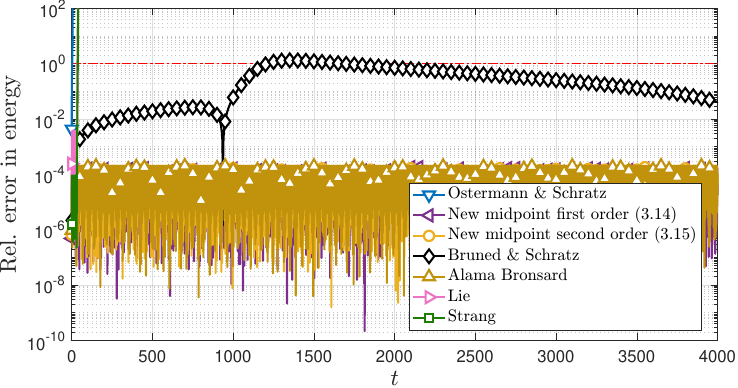}
		\caption{$M=1024$.}
	\end{subfigure}
	\begin{subfigure}{0.9\textwidth}
		\centering
		\includegraphics[width=0.88\linewidth]{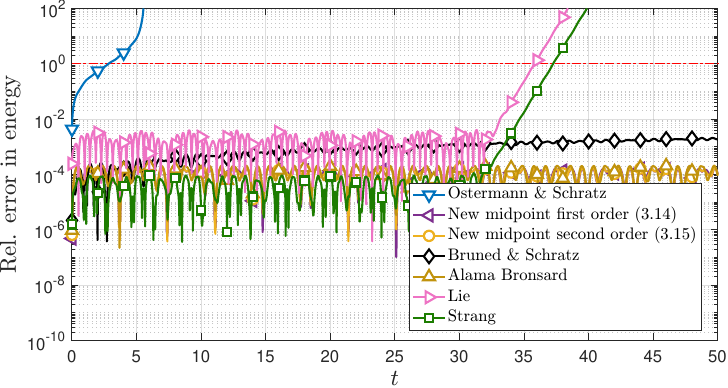}
		\caption{$M=1024$.}
	\end{subfigure}
	\caption{Error in the Hamiltonian, for time step $\tau=0.02$ and $C^\infty$ data. {We observe that in the smooth case, the splitting methods (Lie and Strang) achieve reasonably good long-time approximate energy conservation although this behaviour is sensitive to a rather stringent CFL condition being satisfied (cf. \cite{Faou12}). On the other hand our proposed methods perform very well for a range of spatial resolutions and in particular appear not to suffer from a stringent CFL-type constraint.}}\vspace{-0.5cm}
	\label{fig:error_energy_smooth}
\end{figure}

 Finally, we performed experiments to confirm that the low-regularity convergence properties of our symmetric schemes are at least as good as in prior asymmetric methods. In the following numerical experiments our reference solutions were computed with $M=2^{14}$ Fourier modes and a time step $\tau=10^{-6}$ with the symmetrised method from \cite{AB23}. In Figures~\ref{fig:NLSE_low_regularity_convergence_plots} \& \ref{fig:NLSE_high_regularity_convergence_plots} we choose to measure the error in $H^1$-norm, and observe the convergence properties of our methods for initial data of various levels of regularity. In all of these experiments the number of spatial discretisation modes was taken to be $M=1024$ and the initial data chosen according to \eqref{eqn:numerical_experiments_low_reg_data} \& \eqref{eqn:numerical_experiments_smooth_data}. In Figure~\ref{fig:NLSE_low_regularity_convergence_plots} we observe that our new methods have exactly the predicted convergence properties at those levels of regularity: The integrator \eqref{schemeNLS3} is optimal for first order convergence in the sense of regularity, meaning it converges at first order in $H^1$ with $H^2$ data, while the integrator \eqref{schemeNLS2} is optimally convergent in the sense of regularity up to second order meaning it converges at $\mathcal{O}(\tau)$ in $H^1$ for data in $H^2$ and at $\mathcal{O}(\tau^2)$ in $H^1$ for data in $H^3$. For smaller values of $\tau$ the error forms a plateau around $10^{-4}$ and $10^{-7}$ in Figures \ref{fig:NLSE_low_regularity_convergence_plots_H2} and \ref{fig:NLSE_low_regularity_convergence_plots_H3} respectively. This is due to the error made by the pseudo-spectral space discretisation, which decreases as the regularity of the initial data is increased.
\begin{figure}[h!]
	\centering
	\begin{subfigure}{0.5\textwidth}
		\includegraphics[width=0.99\linewidth]{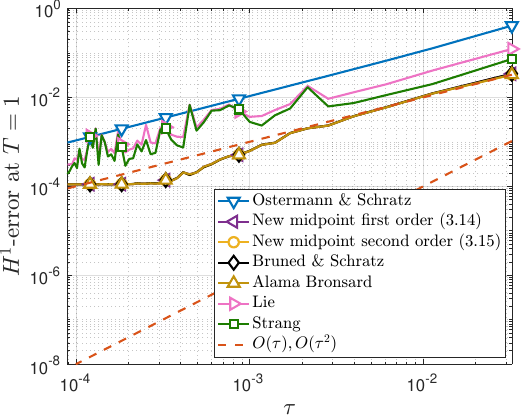}
		\caption{$H^2$-initial data, $\vartheta=2$ in \eqref{eqn:numerical_experiments_low_reg_data}.}
		\label{fig:NLSE_low_regularity_convergence_plots_H2}
	\end{subfigure}%
	\begin{subfigure}{0.5\textwidth}
		\includegraphics[width=0.99\linewidth]{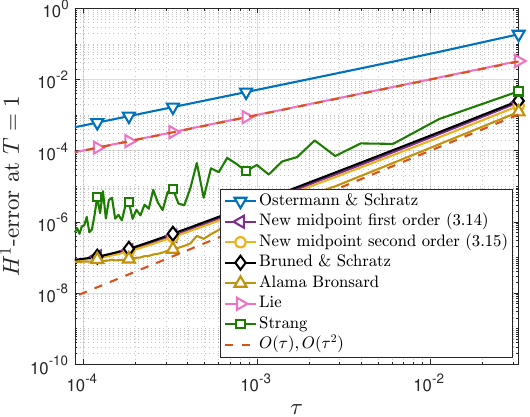}
		\caption{$H^3$-initial data, $\vartheta=3$ in \eqref{eqn:numerical_experiments_low_reg_data}.}
		\label{fig:NLSE_low_regularity_convergence_plots_H3}
	\end{subfigure}
	\caption{$H^1$-error at $T=1$ as a function of the time step $\tau$ for low-regularity initial data. {We observe that our proposed methods (Alama Bronsard, (3.14) and (3.15)) have similarly favourable low-regularity convergence properties as prior resonance-based schemes (Ostermann \& Schratz and Bruned \& Schratz), while the splitting methods (Lie and Strang) suffer from order reduction.}}
	\label{fig:NLSE_low_regularity_convergence_plots}
\end{figure}

 The behaviour of the splitting methods observed in these experiments matches exactly with the convergence analysis given by \cite{R3, Tha12} and suffers from significant order reduction in low-regularity regimes.

\begin{figure}[h!]
	\begin{subfigure}{0.5\textwidth}
	\includegraphics[width=0.99\linewidth]{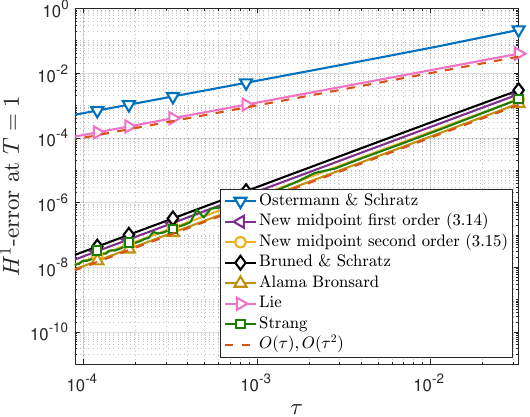}
	\caption{$H^4$-initial data, $\vartheta=4$ in \eqref{eqn:numerical_experiments_low_reg_data}.}
	\label{fig:NLSE_high_regularity_convergence_plots_H4}
\end{subfigure}
\begin{subfigure}{0.5\textwidth}
\includegraphics[width=0.99\linewidth]{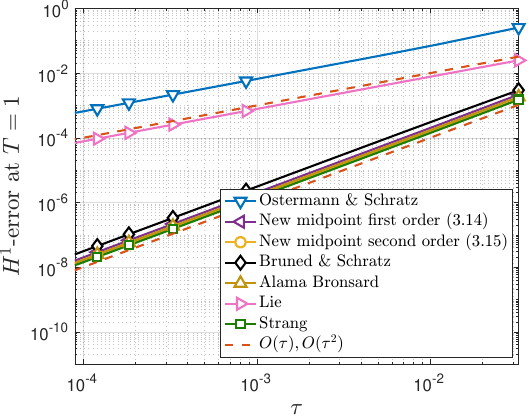}
\caption{$C^\infty$-initial data \eqref{eqn:numerical_experiments_smooth_data}.}
\end{subfigure}
\caption{$H^1$-error at $T=1$ as a function of the time step $\tau$ for more regular initial data{, highlighting that our proposed schemes also converge well for highly regular initial data.}}
\label{fig:NLSE_high_regularity_convergence_plots}
\end{figure}

\newpage

\subsection{The Korteweg--de Vries equation}

As a final numerical example we consider the resonance based midpoint rule \eqref{schemeKdV1} introduced in \cite{MS22} which our midpoint iterates (section~\ref{Duhamel_iteration_mid_point}) are able to recover. This rule has excellent low-regularity convergence properties and at the same time is able to conserve momentum and energy of the KdV equation over long times even in the low-regularity regime. With permission we reproduce some of the numerical results presented in \cite{MS22} to recall the favourable properties of the method \eqref{schemeKdV1}, for further evaluation and results the interested reader is referred to \cite{MS22}. Here we compare the performance of our new schemes to the following state-of-the-art reference schemes for the KdV:
\begin{itemize}
	\item The Strang splitting \cite{McLacQ02}, as an example of a classical symmetric numerical technique, by splitting the KdV equation into the two subproblems $\partial_t u=-\partial_x^3u$ and $\partial_t u=u\partial_x u$. In our implementation the resulting Burgers equation is solved using an explicit RK4 scheme with small time step $\tau_{RK4}=\tau10^{-3}$ such that in essence the error observed in the following experiments is only due to the splitting of the problem.
	\item The first and second order low-regularity schemes introduced by Hofmanova \& Schratz \cite{HS17} and Bruned \& Schratz \cite[Section 5.2]{BS} respectively, as examples of asymmetric low-regularity integrators.
\end{itemize}
\begin{figure}[h!]
	\centering
	\begin{subfigure}{0.99\textwidth}
		\centering
		\includegraphics[width=0.9\linewidth]{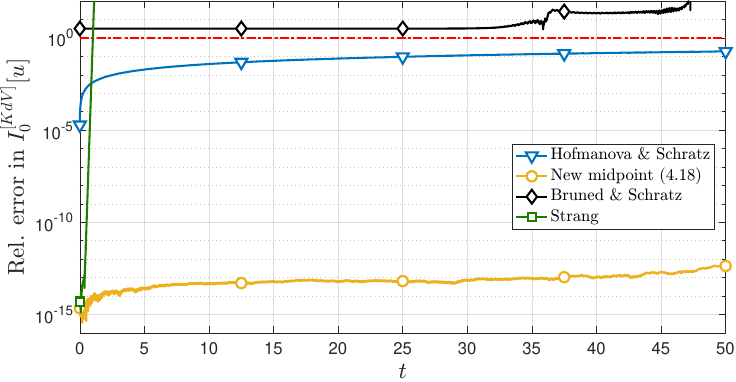}
		\caption{$H^5$-initial data, $\vartheta=5$ in \eqref{eqn:numerical_experiments_low_reg_data}.}
		\label{fig:KdV_momentum_H5_solutions}
	\end{subfigure}
	\begin{subfigure}{0.99\textwidth}
		\centering
		\includegraphics[width=0.9\linewidth]{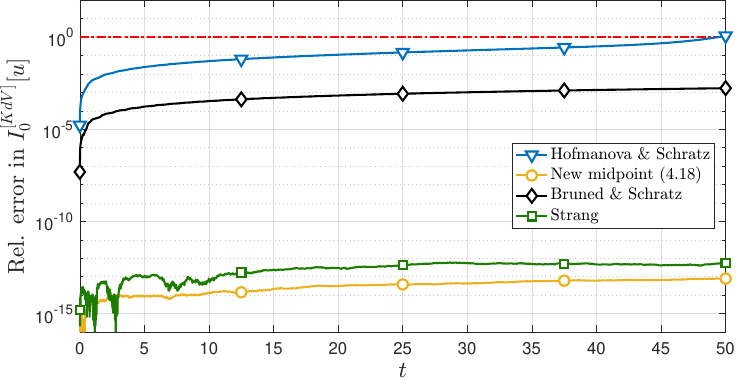}
		\caption{$C^\infty$-initial data \eqref{eqn:numerical_experiments_smooth_data}.}
		\label{fig:KdV_momentum_Cinfy_solutions}
	\end{subfigure}
	\caption{
		Error in the momentum $I_0^{[KdV]}[u]$ for $\tau=0.005$ and $M=64$. {The momentum is observed to be preserved to machine accuracy with our proposed method (New midpoint (4.18)) even in the low-regularity setting, while the Strang splitting can preserve this quantity only for highly regular data. Both prior explicit resonance-based schemes (Hofmanov{\'a} \& Schratz and Bruned \& Schratz) rapidly accumulate error in the momentum already after short times.}}
\end{figure}

To begin with, we can consider the momentum, which is a quadratic first integral in the flow. \vspace{-0.5cm}

\begin{align*}
	I_0^{[KdV]}[u]=\int_{\mathbb{T}} u^2  d x.
\end{align*}\vspace{-0.5cm}

Seeing as the Strang splitting is a symplectic integrator we expect the momentum to be preserved to machine precision for $C^{\infty}$ solutions which is indeed observed in Figure~\ref{fig:KdV_momentum_Cinfy_solutions}. Note in \cite{MS22} it was actually shown that the method \eqref{schemeKdV1} is symplectic and preserves the momentum exactly, which explains the preservation of momentum to machine accuracy in that same graph. On the other hand, for low-regularity data we see in Figure~\ref{fig:KdV_momentum_H5_solutions} that this preservation is lost in the Strang splitting for rougher data, but that the symmetric scheme \eqref{schemeKdV1} is able to deal with rough solutions as well.

In addition to this favourable long-time behaviour, in Figures~\ref{fig:KdV_low_regularity_convergence_plots}\&\ref{fig:KdV_high_regularity_convergence_plots} we observe the excellent low-regularity convergence properties of our method. For this example all reference values were computed with $M=2^{14}$ Fourier modes and a time step $\tau=10^{-6}$ with the second order method from \cite{BS}. The classical Strang splitting suffers in this case from a CFL condition, which requires $\tau\lesssim 1/M$ in order to resolve the Burgers nonlinearity, as a result we chose to include numerical results with only a small number of Fourier modes, since the classical integrator was found to be unstable whenever the aforementioned CFL condition is not satisfied. Perhaps somewhat surprisingly we see that in practice the integrator \eqref{schemeKdV1} appears to converge even faster than the method introduced by Bruned \& Schratz for solutions in $H^3$, this was also observed in \cite{MS22}.

\begin{figure}[h!]
	\centering
	\begin{subfigure}{0.5\textwidth}
		\includegraphics[width=0.99\linewidth]{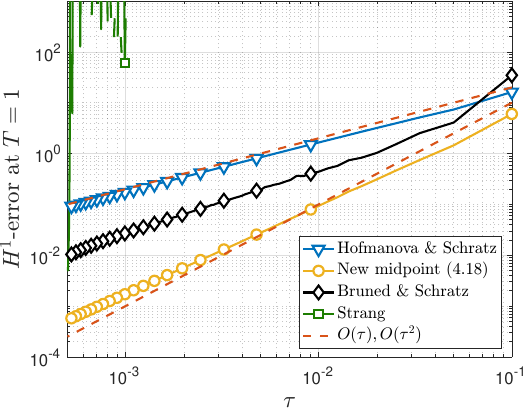}
		\caption{$H^3$-initial data, $\vartheta=3$ in \eqref{eqn:numerical_experiments_low_reg_data}.}
	\end{subfigure}%
	\begin{subfigure}{0.5\textwidth}
		\includegraphics[width=0.99\linewidth]{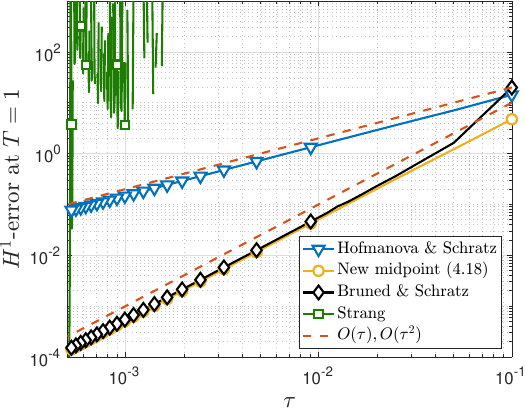}
		\caption{$H^5$-initial data, $\vartheta=5$ in \eqref{eqn:numerical_experiments_low_reg_data}.}
		\label{fig:KdV_low_regularity_convergence_plots_H3}
	\end{subfigure}
	\caption{$H^1$-error at $T=1$ as a function of the timestep $\tau$ for low-regularity initial data. {Our proposed scheme (4.18) provides faster low-regularity convergence rates than prior state-of-the-art (Hofmanov{\'a} \& Schratz and Bruned \& Schratz), while the splitting method suffers from a strong CFL constraint and order reduction.}}
	\label{fig:KdV_low_regularity_convergence_plots}
\end{figure}

\begin{figure}[h!]
	\begin{subfigure}{0.5\textwidth}
		\includegraphics[width=0.99\linewidth]{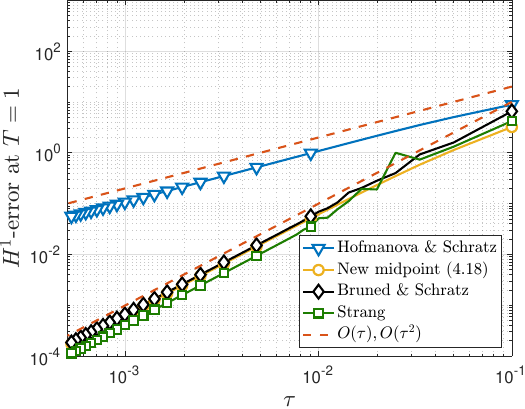}
		\caption{$M=64$.}
		\label{fig:KdV_high_regularity_convergence_plots_small-K}
	\end{subfigure}
	\begin{subfigure}{0.5\textwidth}
		\includegraphics[width=0.99\linewidth]{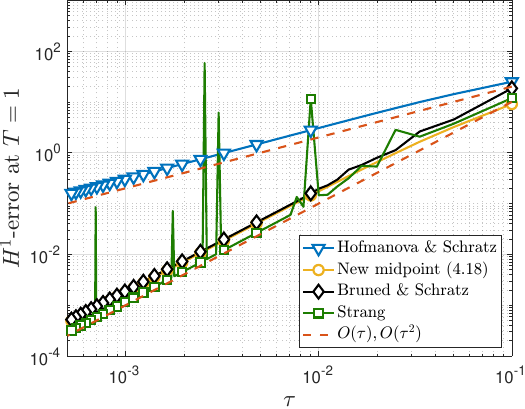}
		\caption{$M=128$.}
	\end{subfigure}
	\caption{$H^1$-error at $T=1$ as a function of the timestep $\tau$ for the smooth initial data \eqref{eqn:numerical_experiments_smooth_data}. {We observe that the proposed method performs equally well in the smooth case. The performance of the Strang splitting is observed to be highly sensitive to the number of spatial modes $M$ taken in the simulation.}}
	\label{fig:KdV_high_regularity_convergence_plots}
\end{figure}

\appendix

\endappendix

\end{document}